%% file: AdsMS4.tex
\documentclass[10pt]{amsart}

\usepackage{amsmath,amsthm,amssymb,verbatim}
\usepackage{tikz}
\usepackage{fancyhdr}

\DeclareMathOperator {\sgn} {sgn}
\DeclareMathOperator {\supp} {supp}
\DeclareMathOperator {\Id} {Id} 
\DeclareMathOperator {\Diag} {Diag} 
\DeclareMathOperator {\Rea} {Re} 
\DeclareMathOperator {\Img} {Im} 
\DeclareMathOperator {\Res} {Res} 
\DeclareMathOperator {\Diff} {Diff} 
\DeclareMathOperator {\WF} {WF}

 \newcommand{\RR}{\mathbb{R}}  
 \newcommand{\ZZ}{\mathbf{Z}}  
 \newcommand{\QQ}{\mathbf{Q}}
  \newcommand{\sfS}{\mathsf{S}}
 \renewcommand{\SS}{\mathbf{S}}  

   \newtheorem{theorem}    {Theorem}       [section]
    \newtheorem{lemma}        {Lemma} [section]
    \newtheorem{corollary}  [theorem]       {Corollary}

    \theoremstyle{definition}

    \theoremstyle{definition}
    \newtheorem{remark} [theorem]    {Remark}

\newcommand{\nl}{\newline}

\title{A simple diffractive boundary value problem on an asymptotically anti-de 
Sitter space.}
\author{Ha Pham}
\address{}  
\curraddr{Stanford University} 
\email{hnpham@math.stanford.edu}
\thanks{This work is partially supported by the National Science Foundation via grant DMS-0801226} 

\begin{document}
\maketitle{}
\begin{abstract}
In this paper, we study the propagation of singularities (in the sense of $\mathcal{C}^{\infty}$ wave front set) of the solution of a model case initial-boundary value problem with glancing rays for a concave domain on an asymptotically anti-de Sitter manifold. The main result addresses the diffractive problem and establishes that there is no propagation of singularities into the shadow for the solution, i.e. the diffractive result for codimension-1 smooth boundary holds in this setting. The approach adopted is motivated by the work done for a conformally related diffractive model problem by Friedlander, in which an explicit solution was constructed using the Airy function. This work was later generalized by Melrose and by Taylor, via the method of parametrix construction. Our setting is a simple case of asympotically anti-de Sitter spaces, which are Lorentzian manifolds modeled on anti-de Sitter space at infinity but whose boundary are not totally geodesic (unlike the exact anti-de Sitter space). Most technical difficulties of the problem reduce to studying and constructing a global resolvent for a semiclassical ODE on $\RR^+$, which at one end is a b-operator (in the sense of Melrose) while having a scattering behavior at infinity. We use different techniques near $0$ and infinity to analyze the local problems: near infinity we use local resolvent bounds and near zero we build a local semiclassical parametrix. After this step, the `gluing' method by Datchev-Vasy serves to combine these local estimates to get the norm bound for the global resolvent. 

\end{abstract}

\tableofcontents
\newpage

\input{Introduction-ProjetEntier-2-noheader.tex}
\input{ApproximateSolutionConstruction-5-noheader.tex}
\input{smlparametrixconstruction-pro5a-noheader.tex}
\input{SingularitiesComputation_4_-noheader}

\end{document}

%% file: Introduction-ProjetEntier-2-noheader.tex
\section{Introduction}
 In this paper, we study the propagation of singularities (in the sense of $\mathcal{C}^{\infty}$ wave front set) for the solution of an initial-boundary value problem in the presence of glancing rays on an asymptotically anti-de Sitter (AdS) manifold whose boundary is not totally geodesic\footnote{Totally geodesics: all of the geodesics of the manifold through points on the boundary are also geodesics in the boundary}.  Asympotically AdS spaces are Lorentzian manifolds modeled on anti-de Sitter space at infinity. \newline Singularities of solutions of linear partial differential equations are described in terms of wave front sets of distributions. In the boundaryless setting, a special case of a result by H\"ormander \cite{Hor01} gives that for $P$ a linear differential operator with real principal symbol $p$, the solution $u$ of $Pu = 0$ has its wavefront set contained in the characteristic set $p^{-1}(0)$ and invariant under the flow generated by Hamiltonian vector field $H_p$. In the presence of boundary, the rays containing wavefront set which hit the boundary transversally are reflected according to Snell's law, i.e. with energy and tangential momentum conserved (see figure \ref{transref}). See \cite{Hor02} for result on codimension-1 boundaries and e.g \cite{Andras01} for higher codimension boundaries.
  \begin{figure}
  \centering
  \includegraphics[scale=0.50]{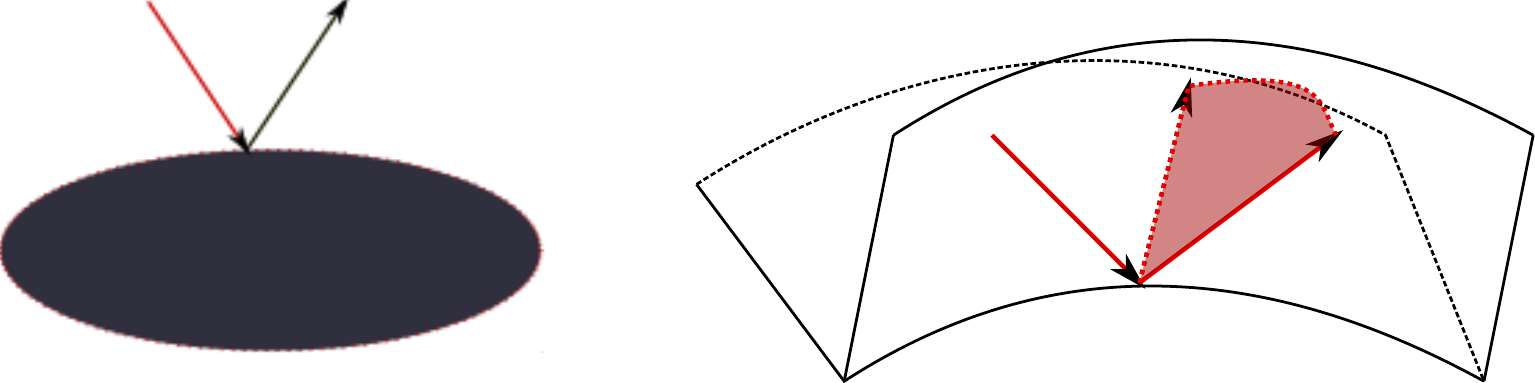}
  \caption{The rays containing wavefront set which hit the boundary transversally are reflected according to Snell's law i.e. with energy and tangential momentum conserved. \nl Single incident ray hitting an edge (corner) generates a whole cone of reflected rays with angle given by the angle between the incident ray and the edge.}
 \label{transref}
  \end{figure}
  The situation becomes more complicated in the presence of boundary and glancing rays (i.e. tangency of the bicharacteristics). For codimension-1 smooth boundary, one has theorems of propagation of singularities along generalized broken bicharacteristics (GBB)\footnote{ This definition of GBB distinguishes tangential rays according to their second order of tangency the boundary. If $p$ is the principal symbol of the operator and $\phi$ the defining function of the boundary, tangential rays are categorized(i.e where $p=0, \phi=0$ and $H_p \phi = 0$) according to the sign of $H_p^2 \phi$, see for e.g. definition 24.3.7 in \cite{Hor02}} that include the diffractive result, by Melrose \cite{Melrose01}, Melrose-Sj\"ostrand (\cite{Mel-Jos01} and \cite{Mel-Jos02}) and Taylor \cite{Taylor01}. Their results rule out the propagation of singularities into the `shadow region', the bicharacteristically inaccessible region; in fact at a diffractive point, the ray carrying wavefront set does not stick to the boundary (see figure \ref{diffobj}). This is not the case if we consider analytic singularities\footnote{Gevrey-3 is the borderline case between the behavior associated with $\mathcal{C}^{\infty}$ singularities and that with analytic singularities, which is Gevrey-1.}, a result given by Lebeau \cite{Lebeau01}. 
 \begin{figure}
  \centering
    \includegraphics[scale=0.75]{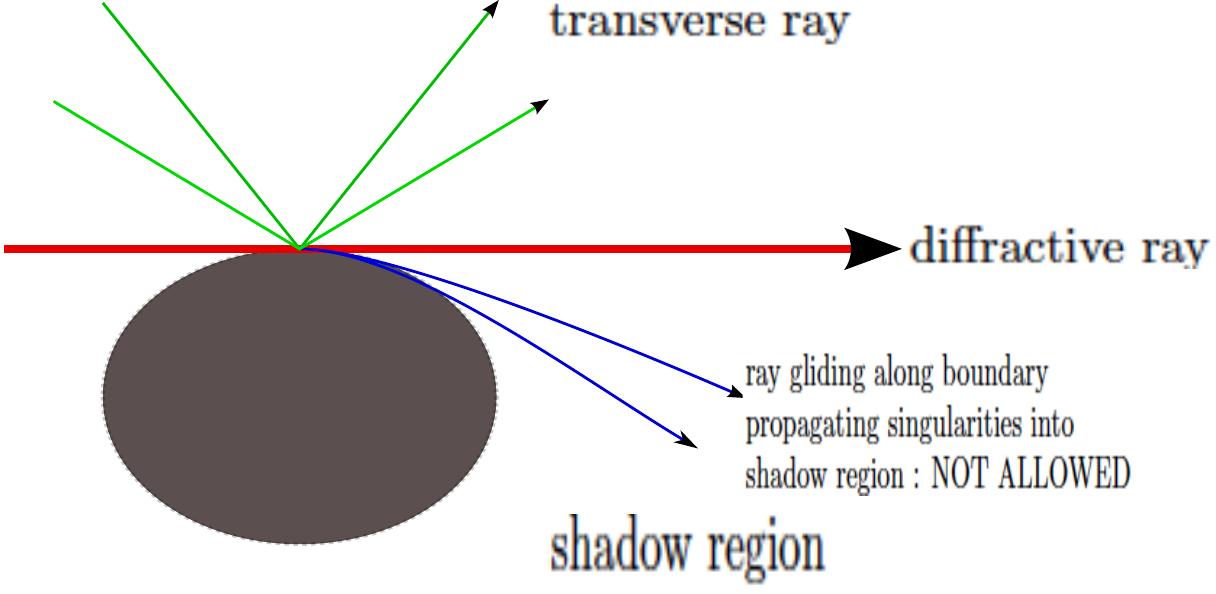}
    \caption{Diffractive phenomenon in the case of tangency to a convex obstacle}
    \label{diffobj}
 \end{figure}
For settings with higher codimension boundaries, in the presence of tangency,  we have general theorems (i.e without diffractive improvement) of propagation of singularities along GBB for corners by Vasy \cite{Andras01} and for edges by Melrose-Wunsch-Vasy \cite{Melrose-Wu-Va} using a relatively permissive\footnote{This definition of GBB does not consider higher order of tangency of tangential rays at higher codimension boundary. The same definition was used for the analogous work done for analytic singularities by Lebeau \cite{Lebeau01}. See section 5 in \cite{Andras01} for more details on the definition and discussion of GBB} notion of GBB. The diffractive problem, one aspect of which studies the phenomenon of propagation of singularities into the `shadow region', remains open for higher codimension smooth settings. \newline The interest in solving the diffractive problem on edges provides a motivation to study the diffractive problem on asymptotically AdS space, since asymptotically AdS space can be considered analogous to a `reduced problem' on an edge but with `negative dimension'. In order to see this, we consider a simple example of an edge which is the product of a sector with an interval: \begin{center} \includegraphics{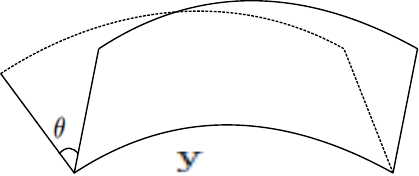}\end{center} 
Introduce polar coordinates $(r,\theta)$ on the sector, and let $y$ be the variable parametrizing the edge. In $(r,\theta,y)$ coordinates the volume form on this space is : 
  $$\sim r^{\text{codimension of edge} -1 } dr d\theta dy$$
  When the codimension of the edge is 1, i.e. there is no $\theta$, we get back the smooth boundary setting. However, when the codimension is higher than 1, even if there is no dependence on $\theta$ (e.g when the coefficients of the metric are independent of $\theta$ and we solve the problem on fixed Fourier modes), the diffractive problem is quite different from the smooth boundary setting\footnote{This is because of the term $r^{\text{codimension of edge} -1 }$ in the volume form, which in turn affects the first order term in the Laplace-Beltrami operator. This is seen as follows:
  $$\text{dVol}_g = \sqrt{\lvert \det g\rvert } dz ,\  \text{with}\  z = (r, \theta, y) ,\ \text{and} $$
  $$\Delta = \dfrac{1}{\sqrt{\lvert \det g\rvert}} \partial_{z_i} \big(\sqrt{\lvert \det g\rvert} G^{ij} \partial_{z_j}\big)
   = G^{ij} \partial_{z_i} \partial_{z_j} + ( \partial_{z_i} G^{ij}) \partial_{z_j} + \dfrac{(\partial_{z_i} \lvert \det g\rvert ) }{2\lvert \det g\rvert} G_{ij} \partial_{z_j}$$ }. On the other hand, the volume form on asymptotically anti-de Sitter space is 
  $$\sim  r^{-(\text{total dimension of the space})} dr dy$$
 Viewed in this way, asymptotically anti-de Sitter space is analogous to a `reduced'\footnote{See remark \ref{generalization} for a direction to generalize the problem which removes the `reduced' condition} problem on edges with `negative dimension'. In particular, the volume form for our model operator on $\RR^+_x \times \RR_y^n$ equipped with an asymptotically anti-de Sitter metric (\ref{modelMetric}) is $x^{-(n+1)} (1+ x)^{-1/2} \,dx\, dy$, where $n+1$ is the dimension of the space. \newline
  Since we already have the general theorem (without diffractive improvements) for propagation of singularities by Vasy \cite{Andras02} on asymptotically AdS spaces for both hyperbolic and glancing cases, we can study the diffractive problem for this setting. The metric for our model case is:
   \begin{equation}\label{modelMetric}
   \text{on}  \ \mathbb{R}^+_x \times \mathbb{R}^n_y :  g = \dfrac{-dx^2+ (1+x)^{-1} dy_n^2 - \sum_{j=1}^{n-1} dy_j^2}{x^2}  \end{equation}
For convenience, we work with a modification of the Klein-Gordon operator by a first order derivative 
$$P := \Box_g + \dfrac{x}{2(x+1)} x\partial_x + \lambda .$$
This modification does not change the problem in an essential way\footnote{This is because the added term does not affect the normal operator and the 0-principal symbol. This also means the bicharacteristics remain unchanged. Hence the results (propagation of singularities and well-posedness) obtained for the Klein-Gordon operator $\Box_g+\lambda$ by Vasy in \cite{Andras02} still hold true for $P$}. The diffractive condition for $P$ is satisfied with $H_{\hat{p}}^2 x  > 0$ where $\hat{p}$ is the principal symbol for conformal operator $\hat{P}$ ie $\hat{p} =-\xi^2  + [ (1+x)  \theta_n^2 - \lvert \theta'\rvert^2] $.\newline
By Lemma \ref{ExistenceNBdy}, we have existence and uniqueness\footnote{See also the proof of  Lemma \ref{ExistenceNBdy} for the motivation of the form of the `boundary condition' (*). The term $g$ is given in terms of $\delta(y)$ by expression (\ref{solveInseries}). For definition of $H^{-1,\alpha}_{0,b,\text{loc}}, \alpha \in \RR$ see definition (5.7) in \cite{Andras02}} for the solution to: 
\begin{equation}\label{originalProblem}\begin{cases} Pu = 0   \\
 \supp u \subset \{ y_n \geq 0\} \\
  u =  x^{s_-} \delta(y) +  x^{s_-+1} g(x,y) + h(x,y) ; & g(x,y)\in \mathcal{C}^{\infty}(\RR_x, \mathcal{D}'_y); \supp g \subset\{y_n=0\} \\ & h\in H^{1,\alpha}_{0,b,\text{loc}}; \text{for some}\ \alpha\in \RR
 \end{cases} \end{equation}
Here $ s_{\pm}(\lambda) = \tfrac{n}{2} \pm \sqrt{\tfrac{n^2}{4} - \lambda}$ are the indicial roots of the normal operator for $P$. The main goal of the paper is to study the singularity structure of the solution to (\ref{originalProblem}) and draw conclusions about the presence of singularities in the `shadow region'.\nl
The approach used to study problem (\ref{originalProblem}) is motivated by the work done for a conformally related diffractive model problem by Friedlander \cite{Fried01}, in which an explicit solution is constructed using the Airy function. The result is later greatly generalized by Melrose using a parametrix construction in \cite{Melrose01} and \cite{Melrose02} and by Taylor \cite{Taylor01}. After taking the Fourier transform in $y$ and denoting by $\theta$ the dual variable of $y$, we need to construct a polyhomogeous conormal solution (modulo a smooth function in $\dot{\mathcal{C}}^{\infty}(\mathbb{R}^{n+1}_+)$) which satisfies the following boundary condition at $x = 0$:
\begin{equation}\label{originalProAfterF}
\begin{cases}
\hat{L} \hat{u} \in \dot{\mathcal{C}}^{\infty}(\mathbb{R}^{n+1}_+) \\
x^{-s_-} \hat{u} \ |_{x=0} = 1   & s_{\pm}(\lambda) = \dfrac{n}{2} \pm \sqrt{\dfrac{n^2}{4} - \lambda} \\
\hat{u} \in \exp(i\phi_{\text{in}})  \mathcal{L}_{ph}(C) & \phi_{\text{in}} = -\tfrac{2}{3} \theta_n^{-1} \big[ (1 +x - \lvert\hat{\theta}'\rvert^2)^{3/2} - (1 - \lvert\hat{\theta}'\rvert^2)^{3/2}\big] \text{sgn} \theta_n
\end{cases} \end{equation}
Here $\mathcal{L}_{ph}(C)$ is the set of polyhomogeneous conormal functions on some blown-up space $C$ of $\overline{\mathbb{R}}_{\theta}^n \times [0,1)_x$. The chosen oscillatory behavior is modeled on that possessed by the specific solution (\ref{fundamentalAiry}) Friedlander worked with in \cite{Fried01}.\nl
 The difficulties for our problem start with the fact that (\ref{originalProblem}) under the Fourier transform does not have a well-understood solution as \cite{Fried01} does in terms of Airy functions. However, following the same change of variable as in \cite{Fried01}, the problem is reduced to studying the following semiclassical ODE on $\RR^+$:
 $$ \mathbf{Q}  = h^2(z \partial_z)^2   + h^2(\lambda-\dfrac{n^2}{4})  + z^3+ z^2 .$$  
  This family of semiclassical ODE is of regular singular type, hence a `b'-operator\footnote{Totally characteristic operators or `b-operators' in the sense of Melrose \cite{Melrose04} are differential operators generated by $b$ vector fields $\mathcal{V}_b$, which is the set of all smooth vector fields tangent to all boundaries. See subsection \ref{adsPropsing} } at the zero end and scattering\footnote{in sense of Melrose scattering calculus \cite{Melrose03}} behavior at the infinity one.

    We use different techniques near zero and infinity to analyze the local problems: near infinity we use local resolvent bounds given by Vasy-Zworski \cite{Andras-Zworksi}, and near zero we build a local semiclassical parametrix. 
   The next technical difficulties arise from the fact that the operator $\mathbf{Q}$ is not semiclassical elliptic, which prevents us from simply patching the two above local resolvents by a partition of unity, which otherwise results in an error that is not trivial semiclassically. To overcome this problem, we adopt the method of Datchev-Vasy \cite{Andras-Kiril} to get a global approximate resolvent with an $\mathsf{O}(h^{\infty})$ error. \newline
   The remaining technical details of the problem involve constructing a local resolvent near zero for $\mathbf{Q}$. For this we first need to construct a certain blown-up space of $[0,1)_z\times [0,1)_{z'} \times [0,1)_h$. The usual method (see e.g. \cite{M-SB-AV}) involves resolving the $b$-singularity at the corner of $\{z = z' = 0\}$ and the semiclassical singularity. In our case, we incur non-uniformity in the behavior of the normal operator on the semiclassical front face as we approach the b front face (i.e. as $z' \rightarrow 0$). This gives motivation to do a blow-up at $\mathsf{Z}_h \rightarrow \infty$ and $z' \rightarrow 0$, where $\mathsf{Z} = \tfrac{z}{z'}$ is the variable associated with the b-blowup and $\mathsf{Z}_h = \tfrac{\mathsf{Z}-1}{h}$ the variable associated with the blow-up along the lifted diagonal in the semiclassical face. The blown-up space we will work on results from the same idea of singularity resolution but with a different order of blowing up; namely, after the b-blowup we blow up the intersection of the b-front face with semiclassical face, then blow up the intersection of the lifted diagonal and the semiclassical face. There will be two additional blow-ups to desingularize the flow and to create a transition region from the scattering behavior to the boundary behavior prescribed by the indicial roots. We then construct a polyhomogeous cornomal function $U$ on this blown-up space so that    
     $$\mathbf{Q} U  - \text{SK}_{\text{Id}} \in    h^{\infty} z^{\infty} (z')^{\sqrt{\tfrac{n^2}{4}-\lambda}  }\mathcal{C}^{\infty};\ \ \text{SK}_{\text{Id}} \ \text{schwartz kernel of Id}$$
   After removing singularities at the lifted diagonal, the remaining difficulty  involves choosing the order in which one will remove singularities of the error to achieve the above effect.\newline
      Once constructed, a parametrix provides among other things information on the properties of the fundamental solution, which enables one to make very precise statements on the propagation of singularities, in particular in the `shadow regions' (see Corollary \ref{mainCor2}). We state below the main results of the paper, whose proofs are contained in subsection \ref{proofOfmainResult}.      
      \begin{theorem}\label{maintheorem}[Wavefront set statement]\indent\par\noindent
   Let $U$ be the exact solution to to (\ref{originalProblem}) on $\RR^+_x \times \RR^n_y$  then 
  $$\text{WF}\,(U ) \subset \Sigma $$
  where
    $$\Sigma = \Big\{(x,y,\xi,\eta) : \xi = \partial_x \left(y\cdot\theta - \phi \right), \eta = d_y \left(y\cdot\theta-\phi\right) , d_{\theta}\left(y\cdot\theta -\phi\right)  = 0 , x > 0 , \lvert\theta_n\rvert \geq \lvert\theta'\rvert > 0 \Big\}.$$
    with $\phi := \tfrac{2}{3}( Z^{3/2} - Z_0^{3/2})\sgn\theta_n$, $\theta = (\theta', \theta_n)$ being the dual variable of $y$ under the Fourier transform.


  \end{theorem}
 
  \begin{theorem}[Conormality result 1]\label{resultOnCorClass1}\
  
  On $y_n> 0$, if we are close enough to the boundary $\{x=0\}$, then $U\in x^{s_+}  \mathcal{C}^{\infty}_{x,y}$. 
  \end{theorem}
   As an immediate consequence of Theorem \ref{resultOnCorClass1} we have
  \begin{corollary}[Conormality result 2]\label{resultOnCorClass2}\indent\par\noindent
 On $y_n> 0$, if we are close enough to the boundary $\{x=0\}$, then $U$ is conormal relative to $H^1_0$, i.e. $U \in H^{1,\infty}_{0,b}$.  (See subsection \ref{adsPropsing} for definition).
  \end{corollary}
    As a result of Theorem \ref{maintheorem} and Corollary \ref{resultOnCorClass2}, we have:
  \begin{corollary}\label{mainCor2}
  There are no singularities\footnote{The type of regularity measured at the boundary is $H^{1,\infty}_{0,b}$. See subsection (\ref{adsPropsing}) for definitions and for further discussion} along the boundary or in the shadow region, i.e. figure \ref{noshadowprop} holds for $U$. 
  \end{corollary}
  \begin{figure}
\centering
   \includegraphics[scale=0.7]{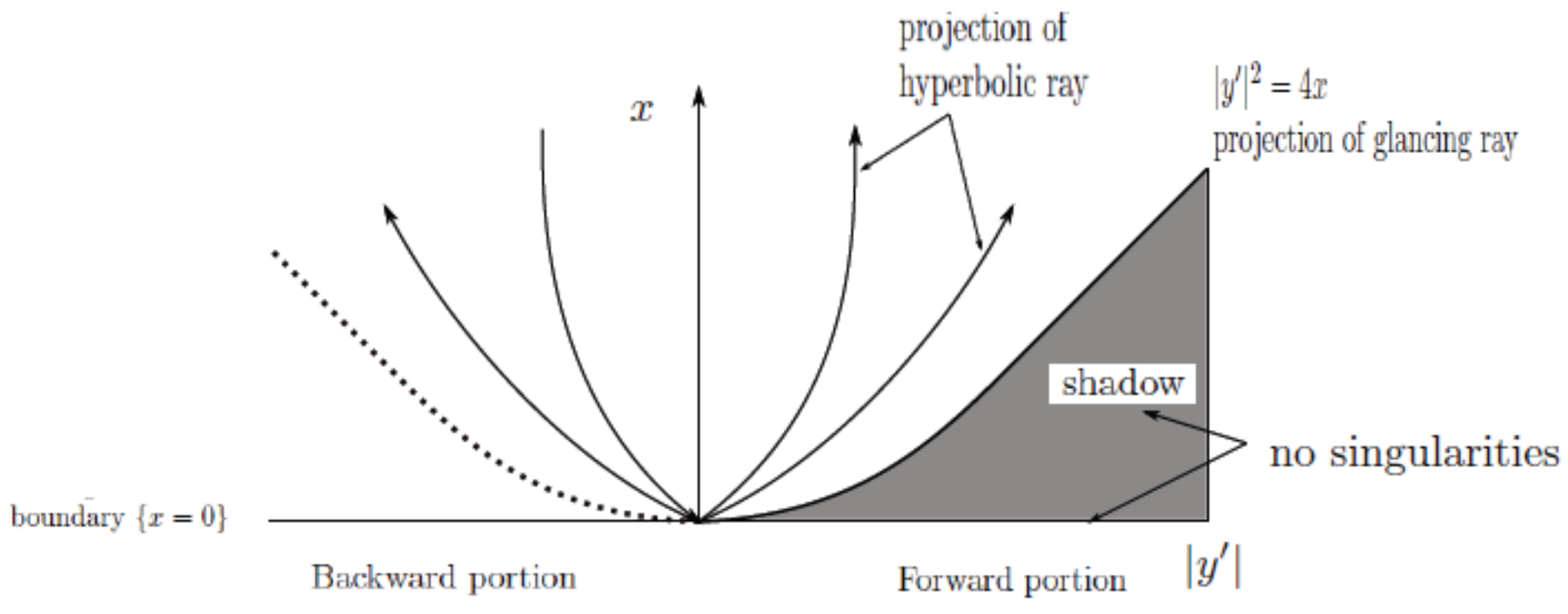}
\caption{Singulartities are neither propagated into the shadow nor along the boundary}
\label{noshadowprop}
  \end{figure}
  The paper is broadly organized into five parts. 
  \begin{enumerate}
  \item Part 1 states the model problem together with the history, background and motivations of the problem. 
    \item The second part of the paper (section \ref{reductionSml}) shows the reduction of the original problem to a semiclassical ODE on $\RR^+$. To motivate the technical details that are required in solving the problem, section \ref{mainIdeaApprox} gives an outline of the construction of an approximate solution to (\ref{originalProAfterF}), i.e. the associated problem obtained after Fourier transform. The technical tools we will need are a local semiclassical parametrix construction and a global resolvent. 
  \item The third part is the technical heart of the paper where we carry out the construction of a local semiclassical parametrix near zero (in section \ref{ParametrixConstruction})) and the construction of a global resolvent (in section \ref{globalResEst}). 
    \item With all the provided technical ingredients, in the fourth part, in section \ref{detailofCon} we fill out the details of section \ref{mainIdeaApprox} and construct the approximate solution to (\ref{originalProAfterF}). 
  \item In the last part (section \ref{SingComp}) we study the wavefront set and conormality behavior at the boundary of the solution to the original problem (\ref{originalProblem}) and prove Theorem \ref{maintheorem}, Corollary \ref{resultOnCorClass1} and Corollary \ref{resultOnCorClass2} .
  \end{enumerate}
   
   \begin{remark}\label{generalization}[Open problems]\
   
  Since asymptotically AdS can be considered a `reduced' problem on edge (as explain in the section), one can generalize the current model problem (\ref{originalProblem}) to one that is conformal to the problem on edges by adding the angular variable $\omega$. Specifically on $\RR^+_x \times \RR^n_y \times \SS_{\omega}$  where $y = (y_1, \ldots, y_n)$ we consider the following metric, which we call `fibred'-AdS:
$$g = \dfrac{-dx^2+ (1+x)^{-1} dy_n^2 - \sum_{j=1}^{n-1} dy_j^2}{x^2}  - d\omega^2$$
$$-\Box_g = (1+x)x^2 \partial_{y_n}^2 - x^2 \sum_{j=1}^{n-1} \partial_{y_j}^2 - x^2 \partial_x^2 + (n-1) x\partial_x + \dfrac{x}{2(x+1)} x\partial_x - \Delta_{\omega}$$
We can study the diffractive problem for 
$$x^2 \partial_x^2 - (n-1)x\partial_x  + \Delta_{\omega}  + x^2 [ (1+x)  \theta_n^2 - \lvert \theta'\rvert^2]  + \lambda$$
Another direction to obtain new problems is to remove the translation invariance in the $y$ variable of (\ref{originalProblem}). In the new problem, the coefficients will be dependent both on $x$ and $y$ so we cannot use the Fourier transform to reduce to a family of ODEs. This direction is analogous to the work done by Melrose (\cite{Melrose01}, \cite{Melrose02}) to generalize the model diffractive problem of Friedlander in \cite{Fried01}.
 \end{remark}
  
  \textbf{ACKNOWLEDGEMENT:}  The author would like to thank her advisor Andr\'as Vasy, without whose thorough, patient and dedicated guidance the realization of this project would not have been possible. 

%% file: ApproximateSolutionConstruction-5-noheader.tex
\section{The model diffractive problem by Friedlander in `layered medium'}\label{FriedProblem}

In this section, we briefly visit the model problem studied by Friedlander which were the motivation for the techniques used in our case. In \cite{Fried01}, Friedlander studied the wavefront set of the solution to the following initial-boundary value problem:
 \begin{theorem}[ Friedlander]
On  $\RR^+_x \times \RR^n_{y}$ consider $u$ the solution to 
  \begin{equation}\label{FriedOp}
\begin{cases}
   Pu = (1 + x) \dfrac{\partial^2 u}{\partial y_n^2} -  \dfrac{\partial^2 u}{\partial y_1^2} \ldots  \dfrac{\partial^2 u}{\partial y_{n-1}^2} - \dfrac{\partial^2 u}{\partial x^2} = 0\\
   u|_{x=0} = f \\
    u = 0 \   ;\  y_n < 0
    \end{cases}
    \end{equation}
    then $\WF (U )$ is contained in the union of forward bicharacteristic strips that emanate from the conormal bundle of the origin over the domain $\{x\geq 0\}$. 
 \end{theorem}
 Note that he does not explicitly make a statement about singularities at the boundary. However from other results (e.g Melrose-Sj\"ostrand \cite{Mel-Jos01}, Taylor \cite{Taylor01}) we know that there are no singularities at the boundary. Hence the following figure \ref{noshadowprop} holds. 
%
   Friedlander's method was to construct the fundamental forward solution $K_x$ to
  \begin{equation}\label{forwardfundsoln}
   \begin{cases} PK_x  = 0  & x >  0 \\ K_x\ |_{x=0} = \delta  \\ K_x = 0 & y_n <0    \end{cases}
   \end{equation} 
 and study its wavefront set. To construct $K_x$, he first takes the Fourier transform in $y$ and looks for solution satisfying:  \begin{equation}\label{Airyboundary}
  \begin{cases} \tfrac{\partial^2 \hat{K}_x}{\partial x^2} = \big[\lvert\theta'\rvert^2 - (1 + x) \theta_n^2 \big] \hat{K}_x  \\
  \hat{K}\ |_{x=0} = 1\end{cases}
  \end{equation}
$\theta$ is the dual variable to $y$. He then uses the following change of variable, which is valid as a smooth change of variables for $\theta_n \neq 0$
  \begin{equation}\label{firstcv}
  x\mapsto \zeta =\theta_n^{-4/3} \lvert\theta'\rvert^2 - (1 + x)\theta_n^{2/3}
  \end{equation}
to get the Airy to equation 
 \begin{equation}
      \big( \partial_{\zeta}^2  -\zeta \big) F = 0  ;\ \ \hat{K}_x = F(\zeta)\end{equation}
    The Airy equation has as solution the Airy function with the properties:
  $\text{Ai}(\zeta)$ is an entire function, whose zeros are on the negative real axis
  $$\text{Ai}(\zeta) = \exp(-\tfrac{2}{3}\zeta^{3/2}) \Phi(\zeta)$$
  $$\Phi(\zeta) \sim \sum_k a_k \zeta^{-\tfrac{1}{4} - \tfrac{3}{2}k}; \ \text{uniformly as}\  \zeta\rightarrow \infty \ \text{in} \ \{\zeta : \lvert\arg\zeta\rvert < \pi - \epsilon, 0 < \epsilon < \pi\}$$
  The unique solution to (\ref{forwardfundsoln}) is the inverse Fourier-Laplace transform of 
  \begin{equation}\label{fundamentalAiry}
  \hat{K}_x(\theta) = \dfrac{\text{Ai}(\zeta)}{\text{Ai}(\zeta_0)};  \   \zeta_0  = \zeta |_{x=0} = \theta_n^{-4/3} \lvert\theta'\rvert^2 - \theta_n^{2/3} 
  \end{equation}
From here, after undoing the Fourier transform, it is standard to draw conclusion on the behavior of singularities of the actual solution in terms of its wavefront set.
Melrose generalizes this work tremendously by using a parametrix construction and generalized $K_x$ to a class of Fourier-Airy integral operators in \cite{Melrose01},\cite{Melrose02}. 

\section{Asymptotically anti-de Sitter space }\label{ads}
\subsection{Description of actual anti-de Sitter space}
  As a justification of terminology we include a brief discussion (borrowed from Vasy's paper \cite{Andras02}) of actual anti-de Sitter space to show the form of the AdS metric in a collar neighborhood of the boundary. The form of the AdS metric is generalized in the definition of asymptotically AdS space in the following subsection.\nl
   In $\RR^{n+1}_z, z = (z_1, \ldots, z_{n+1})$ consider the hyperboloid :
  $$z_1^2 + \ldots + z_{n-1}^2 - z_n^2 - z_{n+1}^2 = -1$$
   living in the ambient space $\RR^{n+1}_z$ which is equipped with the pseudo-Riemanian metric of signature $(2,n-1)$ given by 
 $$-dz_1^2 - \ldots - dz_{n-1}^2 + dz_n^2 + d z_{n+1}^2$$
 Since $z_n^2 + z_{n+1}^2 > 0$, we can introduce polar coordinates $(R,\theta)$ in the two variables $z_n$ and $z_{n+1}$.  Then the hyperboloid is of the form: 
  $$z_1^2 + \ldots+ z_{n-1}^2 - R^2 = -1$$ inside $\RR^{n-1} \times \RR^+_R \times S^1_{\theta}$
 and thus can be identified with $\RR^{n-1} \times \mathbb{S}^1_{\theta}$. \nl
 Next we compatify $\RR^{n-1}$ to a ball $\overline{\mathbb{B}^{n-1}} $ using the inverse polar coordinates $(x,\omega)$ with $x = r^{-1}$, in order to study the form of the metric close to the boundary. Denote by 
  $$X = \overline{\mathbb{B}^{n-1}}\times \mathbb{S}^1$$
   A collar neigborhood of $\partial X$ is identified with 
  $$[0,1)_x \times \mathbb{S}_{\omega}^{n-2} \times \mathbb{S}^1_{\theta},$$ where the Lorentzian metric takes the form
  $$ g = \dfrac{ -(1+x^2)^{-1} dx^2   + \big[ -d\omega^2 + (1 + x^2)d\theta^2 \big]}{x^2}$$
The boundary is time-like and $-d\omega^2 + (1 + x^2)d\theta^2 $ is lorentzian of signature $(1,n-2)$.
  \begin{center}
  \includegraphics[scale=3]{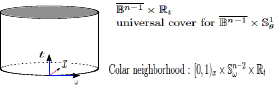}
  \end{center}
  Note AdS space is totally geodesic i.e. all of the geodesics of the manifold through points on the boundary are also geodesics in the boundary, hence we cannot talk about `shadow region'. However, this property does not hold in general for arbitrary asymptotically AdS spaces, whose definition is given in the next section.

\subsection{Definition of asymptotically anti-de Sitter space } 
These are Lorentzian manifolds modeled on anti-de Sitter space at infinity. In more details, they are manifolds with boundary $X^n, n\geq 2$, such that the interior $X^o$ is equipped with a pseudo-Riemannian metric $g$ of signature $(1, n-1)$ which near the boundary $X$ has the form 
 $$g = \dfrac{-dx^2 +h }{x^2},$$
 here $h$ is a smooth symmetric 2-cotensor on $X$ with respect to some product decomposition of $X$ near $Y$, $X = Y \times [0,\epsilon)_x$; in addition, $h|_Y$ is a Lorentzian with signature $(1,n-2)$.\newline
 
 \subsection{Well-posedness results}
  A special case of the well-posedness results for solution of the Klein-Gordon equation on asymptotically AdS space by Vasy in \cite{Andras02},
  \begin{theorem}\label{AdsBdyOrig}[Vasy]
  Suppose $f \in H^{-1,m+1}_{0,b,\text{loc}} (X), m'\leq m \in \RR$  and $\supp f \subset \{y_n \geq 0\}$. $\lambda < \tfrac{n^2}{4}$. Then there exists a unique $U \in H^{1,m'}_{0,loc}$  which in fact lies in $H^{1,m}_{0,b,\text{loc}}(X)$ with 
  $$(\Box_g + \lambda)U = f , \supp U \subset \{y_n \geq 0\}$$
  For all compact $K \subset X$ there exists a compact $K' \subset X$ and a constant $C > 0$ such that
  $$\lVert U \rVert_{H^{1,m}_0 (K)} \leq C \lVert f\rVert_{H^{-1,m+1}_{0,b} (K')}$$
  \end{theorem}
   \textbf{Notation} : For definition of $H^{-1,\alpha}_{0,b,\text{loc}}, \alpha \in \RR$ see page 2 in \cite{Andras02}. 
 
 \begin{theorem}[Vasy]
$X$ asymptotically anti-de Sitter space of dimension $n$.  $(\Box_g + \lambda)u = f$ where $f\in \dot{\mathcal{C}}^{\infty}(X)$ then 
$$ u = x^{s_+(\lambda)} v , v\in \mathcal{C}^{\infty}(X), s_{\pm}(\lambda) = \tfrac{n-1}{2} \pm \sqrt{\tfrac{(n-1)^2}{4} - \lambda}$$
Note $s_{\pm}(\lambda)$ are the indicial roots of $\Box_g + \lambda$ .
 \end{theorem}
 
 \begin{remark}\label{behaviorAtBdy}
 In Friedlander's diffractive problem, the fundamental solution can be considered as a smooth function in $x$ having value in distributions of $y$ i.e. $\mathcal{C}^{\infty}(\RR^+_x, \mathcal{D}'_y)$. However as indicated by the well-posedness results by Vasy quoted above, the solutions we will work with are not smooth up to the boundary and will have boundary behavior prescribed by the indicial roots $s_{\pm}(\lambda)$. This will affect how one formulates the initial boundary value problem for our model operator (see section \ref{wellposedness}, in particular Lemma \ref{AdsBdy} and Lemma \ref{ExistenceNBdy}). 
 \end{remark}
 
\subsection{Propagation of singularities}\label{adsPropsing} 
 \begin{theorem}[Vasy]
 On asymptotically AdS space, we have the propagation of singularities along `generalized broken bicharacteristics' (GBB) of $\hat{g}$ where $\hat{g} = x^2 g$.
 \end{theorem}
 See \cite{Andras02} for more description of the GBB\footnotemark[4].\nl
 In contrast with the boundaryless settings, the type of regularities one measures here is not  $\mathcal{C}^{\infty}$ but `b'-regularity conormal to the boundary (as explained next). This is as expected from the lack of smoothness for the solution of the Klein-Gordon at the boundary (as noted in remark (\ref{behaviorAtBdy})). Instead of $\mathcal{C}^{\infty}$, the `nice' functions are $H^{1,\infty}_{0,b}$ with:
  $$u\in H^{1,\infty}_{0,b} \Leftrightarrow u \in H^1_0(X)\  \text{and}\  Qu \in H^1_0(X) \ \forall Q\in \Diff_b(X) \ \text{of any order}  $$
Here,
 $$u \in H^k_0, \text{if } \ L u \in L^2(X, x^{-(n+1)}\, d\hat{g})  \ \ \forall L  \in \Diff^k_0(X)$$
  $\Diff_b$ is the set of differential operators generated by $\mathcal{V}_b(X)$, which in coordinates have the form
  $$ a x\partial_x + \sum b_j \partial_{y_j} : \text{ smooth vector fields tangent at the boundary} $$
  $\Diff_0$ is the set of differential operators generated by $\mathcal{V}_0(X)$, which in coordinates have the form
  $$ a x\partial_x + \sum b_j (x\partial_{y_j}) : \text{ smooth vector fields vanishing at the boundary} $$

\section{Statement of our model diffractive problem on an asymptotically anti-de Sitter space}\label{statement}
 
On $\RR^+_x \times \RR^n_y $  where $y = (y_1, \ldots, y_n)$ for our model problem, we consider the metric
$$ g = \dfrac{-dx^2+ (1+x)^{-1} dy_n^2 - \sum_{j=1}^{n-1} dy_j^2}{x^2} $$
For such a metric, $J = \lvert \det g\rvert^{1/2} = x^{-(n+1)} (1 + x)^{-1/2}$. We consider the Beltrami-Laplace operator associated with $g$
$$-\Box_g = J^{-1} \partial_i (J G^{ij} \partial_j) = (1+x)x^2 \partial_{y_n}^2 - x^2 \sum_{j=1}^{n-1} \partial_{y_j}^2 - x^2 \partial_x^2 + (n-1) x\partial_x + \dfrac{x}{2(x+1)} x\partial_x $$
For simpler computation in the model case, we will modify $\Box_g$ by a first-order tangential derivative. We will work with the operator: 
\begin{equation}
\begin{aligned}
P &:= \Box_g + \dfrac{x}{2(x+1)} x\partial_x + \lambda \\
 &= x^2 \partial_x^2 - (n-1)x\partial_x    - (1+x) x^2 \partial_{y_n}^2 + x^2 \sum_{j=1}^{n-1} \partial_{y_j}^2 + \lambda
 \end{aligned}
\end{equation}
As noted in the introduction section, this modification does not change the problem in an essential way.\newline
Taking the Fourier transform in $y$ we obtain the operator $\hat{L}$. Denote by $\theta = (\theta_1, \ldots, \theta_n)$ to be the dual variables to $y$ under the Fourier transform.
\begin{equation}\label{mainopFourier}
\begin{aligned}
\hat{L} &:= x^2 \partial_x^2 - (n-1)x\partial_x    + x^2 [ (1+x)  \theta_n^2 - \lvert \theta'\rvert^2]  + \lambda \\
&= x^2 \big[\partial_x^2 - (n-1)\dfrac{\partial_x}{x}    + [ (1+x)  \theta_n^2 - \lvert \theta'\rvert^2]  + x^{-2} \lambda\big]
\end{aligned}
\end{equation}
The principal symbol of $x^{-2} \hat{L}$ is 
$$\hat{l}(x,\xi,\theta) = -\xi^2 + [ (1+x)  \theta_n^2 - \lvert \theta'\rvert^2]  $$
The diffractive condition is satisfied since $H_{\hat{l}}^2 x  > 0$\footnote{On the characteristic set, $\lvert\theta_n\rvert = \lvert\theta'\rvert$ so $\theta_n > 0$. $H_{\hat{l}} x = \partial_{\xi} p = -2\xi$. $H_{\hat{l}}^2 x = 2\partial_x \hat{l} = 2  \theta_n^2$. $(x=0,y=0,\xi=0,\theta)$ with $\lvert\theta_n\rvert = \lvert\theta'\rvert > 0$ satisfies the diffractive condition: $ H_{\hat{l}} x = \xi = 0$ and $H_{\hat{l}}^2 x = \theta_n^2  > 0$ }. Lemma \ref{ExistenceNBdy} gives existence and uniqueness to the following boundary problem:
 $$\begin{cases} Pu = 0 \\
 \supp u \subset \{ y_n \geq 0\}\\
 u \in  x^{s_-} \delta(y) +  x^{s_-+1} g(x,y) + h(x,y) ; &g \in \mathcal{C}^{\infty}(\RR_x, \mathcal{D}'_y); \supp g \subset \{y_n =0\} \\  & h \in H^{1,\alpha}_{0,b,\text{loc}}, \ \text{for some} \ \alpha \in \RR
  \end{cases} $$
  Our main goal is to study the singularity structure of this solution in terms of wavefront set and investigate the presence of singularities in the `shadow region'. In order to do this, after taking the Fourier transform in $y$, we need to construct a polyhomogeous conormal solution modulo smooth function in $\dot{\mathcal{C}}^{\infty}(\RR^{n+1}_+) $ which satisfies certain the boundary condition at $x = 0$ as follows:
\begin{equation}
\begin{cases}
\hat{L} \hat{u} \in \dot{\mathcal{C}}^{\infty}(\RR^{n+1}_+) \\
x^{-s_-} \hat{u} \ |_{x=0} = 1   & s_{\pm}(\lambda) = \dfrac{n}{2} \pm \sqrt{\dfrac{n^2}{4} - \lambda}\\
\hat{u} \in \exp(-i\phi_{\text{in}})  \mathcal{L}_{ph}(C) & \phi_{\text{in}} = \tfrac{2}{3} \theta_n^{-1} \big[ (1 +x - \lvert\hat{\theta}'\rvert^2)^{3/2} - (1 - \lvert\hat{\theta}'\rvert^2)^{3/2}\big] \sgn\theta_n
\end{cases} 
\end{equation}
 $\mathcal{L}_{ph}(C)$ is the set of polyhomogeneous conormal functions on some blown-up space $C$ of $\overline{\RR}_{\theta} \times [0,1)_x$. $ s_{\pm}(\lambda)$ come from the indicial roots of the operator and prescribes the asymptotic behavior of solution of the Klein Gordon on asymptotically AdS space (Vasy \cite{Andras02}). In our case, the dimension is $n+1$. This construction is carried out in section \ref{detailofCon} and the singuralities are studied in section \ref{SingComp} and \ref{mainresultcomp}.\nl
 \textbf{The bicharacteristics} : the bicharacteristics are the same as those for the conformal operator  $x^2P$, which is the operator considered by Friedlander \cite{Fried01} (up to leading order). We will use Friedlander's description of bicharacteristics (see page 146-148 of \cite{Fried01}).

\subsection{Geometry of singularity structure}
There are similarities in the geometry of phase space $T^*( \RR^+_x \times \RR^n_y )$ between our model problem and Friedlander's operator (\ref{FriedOp}) in \cite{Fried01}. Our operator $P$  in (\ref{originalProblem}) is a 0-differential operator\footnote{0-differential operators are operators generated by $\mathcal{V}_0(X)$ - smooth vector fields vanishing at the boundary. $\mathcal{V}_0(X)$ is also the set of all smooth sections of a vector bundle called $^0TX$, whose dual bundle is $^0T^*X$. The quotient of $^0T^*X$ under the dilation action of $\RR^+$ gives the 0-cosphere bundle.} so its geometry lives in the 0-cosphere bundle\footnotemark[5]. On the other hand, there is a natural identification between 0-cosphere bundle and the natural cosphere bundle. Up to an identification of the cotangent bundle near the boundary, the principal symbol for our model problem (\ref{originalProblem}) and that for Friedlander's problem \cite{Fried01} are both $-\xi^2 + q(x,y,\theta)$. Both problems will have the same geometry in the cotangent/cosphere bundle. In particular, as shown in figure \ref{Lagrangian1}, away from glancing, singularities are carried by a Lagrangian. The boundaries of the flow are the same in both cases: what happens at the boundary of the Lagrangian comes from the glancing point at the boundary of the domain from which it emanates (see figure \ref{propsing}). However, the analytic objects that show up at the boundary will be different in each case. In Friedlander's case, this is associated to $\hat{K}_x$ in (\ref{fundamentalAiry}), while in our case, we will have to construct a family of $u_{\text{asym-Ads}}$ that plays the same role as $\hat{K}_x$.\newline
       \begin{figure}
  \includegraphics[scale=1.5]{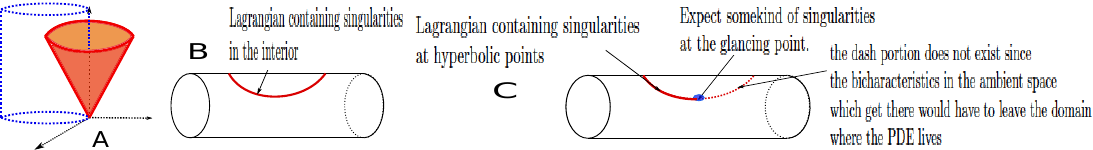}
 \caption{\small In figure A, the problem has no boundary; initial data has singularities contained in the conormal bundle of the origin (the dotted cylinder is for visual comparison to the case with boundary). Figure B shows the structure of the flow out from the conormal bundle of the diagonal. Figure C shows the flow out in the situation with boundary, where away from glancing, singularities are still carried by a Lagrangian while at glancing the Lagrangian has a boundary.}
  \label{Lagrangian1}
 \end{figure}
 \begin{figure}
 \includegraphics[scale=1.5]{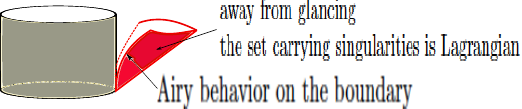}
 \caption{\small Sideview of the problem with boundary (the PDE lives outside of the cylinder) and initial data having singularities contained in the conormal bundle at the origin}
 \label{propsing}
 \end{figure}

 \section{Various results on well-posedness and formulation of boundary condition problem on asymptotically AdS space}\label{wellposedness}\
In this section we establish the existence and uniqueness for solution to (\ref{originalProblem}). The proof also explains the appearance of indicial roots $s_-(\lambda)$ in the behavior of the solution at the boundary as stated in (\ref{originalProblem}).
  \begin{lemma}\label{AdsBdy}
 \begin{enumerate}\indent\par\noindent
\item  For $g\in \mathcal{D}'_y$ there exists $\tilde{u}_{\pm} \in x^{s_{\pm}} \mathcal{C}^{\infty}(\RR_+, \mathcal{D}'_y)$ satisfying for some $\alpha \in \RR$:
  $$\begin{cases}  P\tilde{u}_{\pm} \in H^{-1,\alpha}_{0,b}  \\ \tilde{u}_{\pm} \in  x^{s_{\pm}} g_{\pm} + x^{s_{\pm} + 1} \mathcal{C}^{\infty}(\RR^+_x, \mathcal{D}'_y) + H^{1,\alpha}_{0,b}  \end{cases}$$
 \item 
 For $g\in \mathcal{D}'_y$ there exists $\tilde{u}_{\pm} \in x^{s_{\pm}} \mathcal{C}^{\infty}(\RR_+, \mathcal{D}'_y)$ satisfying
  $$\begin{cases}  P\tilde{u}_{\pm} \in \dot{\mathcal{C}}^{\infty} (\RR_x , \mathcal{D}'_y) \\ \tilde{u}_{\pm} \in  x^{s_{\pm}} g_{\pm} + x^{s_{\pm} + 1} \mathcal{C}^{\infty}(\RR^+_x, \mathcal{D}'_y)  \end{cases};\ \ \dot{\mathcal{C}}^{\infty} (\RR_x , \mathcal{D}'_y) = x^{\infty} \mathcal{C}^{\infty}(\RR^+_x, \mathcal{D}'_y)$$
  \end{enumerate}
 \end{lemma}
 \begin{proof}
  \begin{remark}\label{sobolev}
   By a special case of definition of $H^{k,-m}_{0,b}, m > 0, m\in \RR$ in \cite{Andras02}, if $v = \sum Q_j v_j$ where $v_j \in H^{k,0}_{0,b} (X), Q_j \in \text{Diff}^m_b(X),  H^{k,0}_{0,b}= H^k_0$ then $v \in H^{k,-m}_{0,b}$. (\emph{For general definition of $H^{k,m}_{0,b}, m\in \RR$ see definition (5.7) in \cite{Andras02}}). Hence $\varphi_j \in H^{0,-\alpha_j}_{0,b}$ for some $\alpha_j > 0, \alpha_j \in \RR$.
  \end{remark}
  \begin{remark}
 $g\in \mathcal{D}'_y$ the space of tempered distributions ie $g$ is a finite sum of derivatives (with constant coefficients) in $y$ of continuous functions. And derivatives in $y$ commute with $P$. Hence the proof for arbitrary $g$ is completely analogous to that for $g = \delta(y)$.
\end{remark}
   \begin{enumerate}
%
 \item Consider $\varphi_0 \in \mathcal{C}^0_y$. We consider $\varphi_j \in \mathcal{D}'_y$ being $b$ derivatives of $\varphi_0$.  With such $\varphi_j$-s, consider $\tilde{u}$ of the form:
   $$  \tilde{u}  =  x^{s_-} \varphi_0 +  \sum_{k=1}^{ k_0} x^{s_- + k} \varphi_k(y)$$
   Write $P$ as
  $$P = Q_1 + x^2 Q_2 $$
  $$\text{where}\ \ Q_1 = (x\partial_x)^2 -nx\partial_x + \lambda ; \ \ Q_2 = - (1+x) x^2 \partial_{y_n}^2 + x^2 \sum_{j=1}^{n-1} \partial_{y_j}^2$$
  Note that $Q_1$ has indicial roots $s_{\pm}$ with $s_{\pm}(\lambda) = \tfrac{n}{2} \pm \sqrt{\tfrac{n^2}{4} - \lambda}$ then $P\tilde{u}$ is of the form
   $$P\tilde{u} =   x^{s_- + 1} (2s_-+1) \varphi_1  +  \sum_{k=2}^{k_0-1} x^{s_- + k}  \big[    (2s_- + k)k \varphi_k  + Q_2 \varphi_{k-2} \big]  + x^{s_- + k_0+1 } Q_2 \varphi_{k_0-1}   $$
 With $k_0$ large enough then the last term  $ x^{s_- + k_0+1 } Q_2 \varphi_{k_0} \in H^{0,\alpha}_{0,b,\text{loc}}$ for some $\alpha \in \RR$  \footnote{$k_0$ is large enough to cancel out the weight in the volume form $L^2(x^{-(n+1)} d\hat{g})$. With this choice of $k_0$, the term in consideration is then some b-derivatives of a function in $L^2_0$. Note that $Q_2$ consists of only b derivatives (in fact only derivatives in $y$). By remark \ref{sobolev}, it is in $H^{1,\alpha}_{0,b}$ for some $\alpha \in \RR$}. If we only care of the decay in $x$ then this term is in $x^{s_- + k_0+1 } \mathcal{C}^{\infty}(\RR_x, \mathcal{D}'_y)$. 
  \nl We can solve for $\varphi_j$ with $j = 1, \ldots, k_0 - 1$ so that $P\tilde{u} \in  H^{0,\alpha}_{0,b,\text{loc}}$ by requiring
 \begin{equation}\label{solveInseries}
 \varphi_1 = 0 ; \ \text{For}\  j = 2, \ldots, k_0 - 1:  \varphi_k = \dfrac{-1}{  (2s_- + k)k}  Q_2 \varphi_{k-2}
 \end{equation}
Hence we have constructed $\tilde{u}$ satisfying
 $$\begin{cases} \tilde{u} = x^{s_-} \varphi(y) + x^{s_-+1} \sum_{k=0}^{k_0-2} x^k \varphi_k  ; \ \varphi_k \ \text{being b-derivatives in $y$ of $\varphi$} \\ 
P\tilde{u} \in H^{-1, \alpha}_{0,b,\text{loc}} \end{cases}$$
 Since $H^{0, \alpha}_{0,b,\text{loc}} \subset H^{-1, \alpha}_{0,b,\text{loc}}$, $P\tilde{u} \in H^{-1, \alpha}_{0,b,\text{loc}}$. 
 \item   Let  $\varphi$ is such that $(\Delta_y + 1)^m \varphi = \delta(y)$ for some $m$. We have $(\Delta_y + 1)^m$ commutes with $P$. Also $\supp P\tilde{u}  \subset \{y_n \geq 0\}$. Define
   $$u_{\text{bdy}} = (\Delta_y + 1)^m \tilde{u} $$
then $u_{\text{bdy}} $ satisfies for some $\tilde{\alpha} \in \RR$: 
   $$\begin{cases} Pu_{\text{bdy}}  \in H^{-1,\tilde{\alpha}}_{0,b,\text{loc}}   \\  u_{\text{bdy}}  \in  x^{s_-} (\Delta_y + 1)^m \varphi +  x^{s_-+1} \sum_{k=0}^{k_0-1} x^k (\Delta_y + 1)^m \tilde{\varphi}_k    
  \end{cases} $$ 
    where $\varphi_k $ are derivatives $b$ in $y$ of $\varphi$ and thus in $\mathcal{D}'_y$. Thus we obtain the first statement of the lemma.
         
   \item Instead of stopping at $k_0$ as in (\ref{solveInseries}) we can solve in series and ask that 
  $$
 \varphi_1 = 0 ; \ \text{For}\  j = 2, \ldots, :  \varphi_k = \dfrac{-1}{  (2s_- + k)k}  Q_2 \varphi_{k-2}$$
We then take an asymptotic summation to get $u_{\text{bdy};\infty}$, which satisfies:
   $$\begin{cases} Pu_{\text{bdy};\infty}  \in x^{\infty} \mathcal{C}^{\infty} (\RR_x , \mathcal{D}'_y)  \\ u_{\text{bdy};\infty}   \in  x^{s_-} (\Delta_y + 1)^m \varphi +  x^{s_-+1} \mathcal{C}^{\infty} (\RR_x , \mathcal{D}'_y)  \\
 \end{cases} $$ 
 
 \item The same proof goes through if we replace $s_- $ by $s_+$.
  \end{enumerate}
 
\end{proof}

  As a corollary, we can formulate a boundary problem for our model operator:
\begin{lemma}\label{ExistenceNBdy}
We have the existence and uniqueness for the solution to 
 $$\begin{cases} Pu = 0 \\
 \supp u \subset \{ y_n \geq 0\}\\
 u \in  x^{s_-} \delta(y) +  x^{s_-+1} g(x,y) + h(x,y) ; &g \in \mathcal{C}^{\infty}(\RR_x, \mathcal{D}'_y); \supp g \subset \{y_n =0\} \\  & h \in H^{1,\alpha}_{0,b,\text{loc}}, \ \text{for some} \ \alpha \in \RR
  \end{cases} $$
\end{lemma}
  
  \begin{proof}
 Denote by $f = Pu_{\text{bdy}}$ where $u_{\text{bdy}} $ is as constructed in the proof for Lemma \ref{AdsBdy}. For some $\alpha\in \RR$, $f \in H^{-1,\alpha}_{0,b}$. By Theorem \ref{AdsBdyOrig} (Vasy), we have the existence and uniqueness of $\tilde{\mathsf{U}} \in H^{1,\tilde{\alpha}}_{0,b,\text{loc}}$ for some $\tilde{\alpha}\in \RR$ with $\tilde{\mathsf{U}}$ satisfying:
 $$\begin{cases} P\tilde{\mathsf{U}} = f \\ \supp \tilde{\mathsf{U}}\subset \{ y_n \geq 0\} \end{cases}$$ 
 Lets remind ourselves of the properties of $u_{\text{bdy}}$ from the proof of Lemma \ref{AdsBdy}
  $$u_{\text{bdy}} = (\Delta_y+1)^m \tilde{u}; \  \tilde{u} = x^{s_-} \varphi(y) + x^{s_-+1} \sum_{k=0}^{k_0-1} x^k \varphi_k $$
 where $ \varphi_k $ are derivatives in $y$ of $\varphi$ and  $\varphi$ is such that $(\Delta_y + 1)^m \varphi = \delta(y)$ for some $m$. Hence $(\Delta_y+1)^m \varphi_k$ are also derivatives of $\delta(y)$. Thus $\supp u_{\text{bdy}} \subset \{y_n= 0\}$. 
 Let $u$ be defined by 
  $$u := u_{\text{bdy}} - \tilde{\mathsf{U}}$$ 
 then $u$ as constructed solves :
  $$\begin{cases} Pu = 0 \\
 u \in  x^{s_-} \delta(y) +  x^{s_-+1} \mathcal{C}^{\infty}(\RR_x, \mathcal{D}'_y) +H^{1,\alpha}_{0,b,\text{loc}} \\
 \supp u \subset \{ y_n \geq 0 \} \end{cases} $$
 Uniqueness follows from the fact the $g$ is uniquely determined as shown in the proof for Lemma \ref{AdsBdy}. There, we construct $g$ as a Taylor series in $x$ with coefficients given explicitly in terms of $\delta(y)$ by \eqref{solveInseries}. 
\end{proof}

\section{Reduction to the semiclassical problem on $\RR^+$}\label{reductionSml}
 
 As described in section \ref{FriedProblem}, in \cite{Fried01} Friedlander studies a diffractive problem conformally related to ours (\ref{originalProblem}). After the Fourier transform in the tangential variables, he uses a change of variable (\ref{firstcv}) to get the Airy equation. Although we will not get the Airy equation, using the same change of variable, we show in this section how our model problem is reduced to the following semiclassical ODE on $\RR^+$ of the form:
$$ \mathbf{Q}  = h^2(z \partial_z)^2   + h^2(\lambda-\dfrac{n^2}{4})  + z^3+ z^2 .$$ 
This family is at one end of regular singular type (hence a b-operator in the sense of Melrose) while at infinity has a scattering behavior.\nl
\textbf{Step 1:} Follow the first step of Friedlander and use the change of variable $$x\mapsto  Z = - \lvert\theta_n\rvert^{-4/3} \lvert\theta'\rvert^2  + (1+x) \lvert\theta_n\rvert^{2/3} \in \RR$$
Also denote by
\begin{equation}\label{defOfZ0}
Z_0 = - \lvert\theta_n\rvert^{-4/3} \lvert\theta'\rvert^2  +  \lvert\theta_n\rvert^{2/3}\in \RR \Rightarrow Z_0 = Z|_{x=0}
\end{equation}

 Denote
 $$\tilde{\QQ} = x^{-n/2} \hat{L} x^{n/2} = (x \partial_x)^2   + x^2 [ (1+x)  \theta_n^2 - \lvert \theta'\rvert^2]  + \lambda - \dfrac{n^2}{4}$$
 In terms of $Z, Z_0$ , $\tilde{\QQ}$ becomes
   \begin{equation}\label{op2}
   \tilde{\QQ} =  (Z-Z_0)^2\partial_Z^2 + (Z-Z_0) \partial_Z + \lambda - \tfrac{n^2}{4}  + Z(Z-Z_0)^2 
  \end{equation}
  
   \textbf{Geometric interpretation of $Z_0$:} 
  $Z_0 = \lvert\theta_n\rvert^{2/3} \big( 1 - \tfrac{\lvert\theta'\rvert^2}{\lvert\theta_n\rvert^2}\big)$ : the nonhomogeneous blowing up of the corner in phase space at glancing $\lvert\theta_n\rvert = \lvert\theta'\rvert$ and infinity $\lvert\theta_n\rvert^{-1}$ (see figure \ref{compactification2}). The functions in terms $x$ and $\theta$ live on the blowup space B as given in figure \ref{compactification2}.
\begin{figure}
\centering
 \includegraphics[scale=0.80]{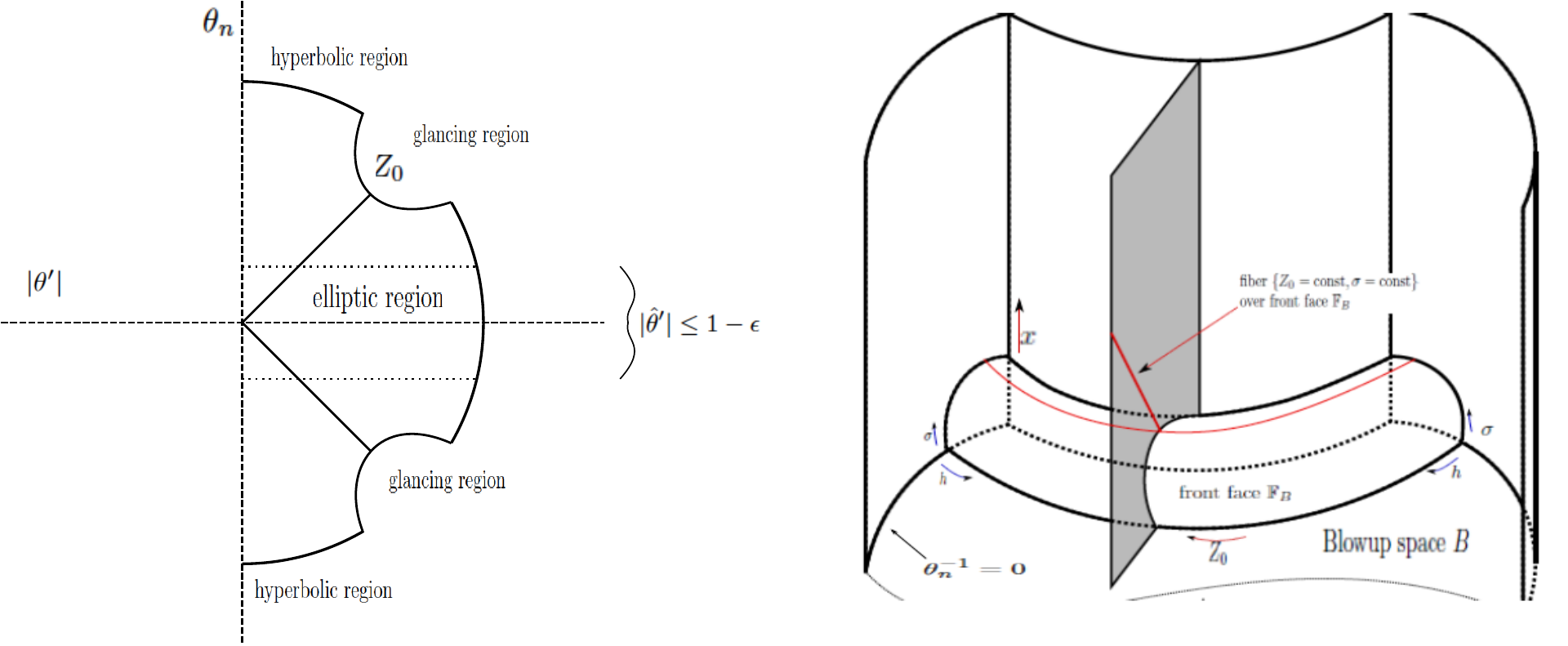}
 \caption{$Z_0$ is obtained by blowing up the corner in phase space at glancing $\lvert\theta_n\rvert = \lvert\theta'\rvert$ and infinity $\lvert\theta\rvert^{-1} = 0$. Figure on the left show the base of the blowup space B. Figure on the right shows the portion of the blow-up space B corresponding to $\theta > 0$. Both portions for $\theta > 0$ and $\theta< 0$ are similar.}
 \label{compactification2}
 \end{figure}

\textbf{Step 2: } With $S = Z-Z_0$ 
\begin{equation}
\tilde{\QQ} =  S^2\partial_S^2 + S\partial_S + \lambda - \tfrac{n^2}{4}  +S^3 + Z_0 S^2 ; 
\end{equation}
This is an family of ODE in  $\sfS \in [0,\infty)$ with parameter $Z_0 \in (-\infty, \infty)$. \newline
\textbf{Goal:} Need to know the behavior of the solution at $\mathsf{S} =0$ and $\mathsf{S}^{-1} = 0$ for bounded $\lvert Z_0 \rvert$, uniformly as $\mathsf{Z}_0 \rightarrow \infty$ and uniformly as $\mathsf{Z}_0 \rightarrow -\infty$. 

 \textbf{Step 3: } We rewrite the problem as $\mathsf{Z}_0 \rightarrow \infty$ and $\mathsf{Z}_0 \rightarrow -\infty$ as semiclassical problem using the following variables $z, h$. 
 
    \begin{equation}\label{definitionOfz}
  z := \chi_0 (Z_0) ( Z-Z_0) + \chi_+(Z_0)  \big( \dfrac{Z-Z_0}{Z_0}\big) + \chi_-(Z_0)  \big( \dfrac{Z-Z_0}{-Z_0}\big)
  \end{equation}
  where $\chi_0 , \chi_+, \chi_-$ be smooth function on $\RR$ such that for some $\tilde{\delta}_2 > 0$, 
  $$\begin{cases}  \chi_0 \equiv 1 \ \text{on}\ \lvert Z_0\rvert \leq \tilde{\delta}_2\\
  \supp \chi_0 \subset \{\lvert Z_0\rvert \leq 2\tilde{\delta}_2\} \end{cases} ;\ \ 
  \begin{cases} \chi_+ \equiv 1 \ \text{on}\ Z_0\geq2\tilde{\delta}_2\\
  \supp \chi_+ \subset \{ Z_0 \geq \tilde{\delta}_2\} \end{cases} ;\ \  \begin{cases} \chi_-\equiv 1 \ \text{on}\ Z_0\leq-2\tilde{\delta}_2\\
  \supp \chi_- \subset \{ Z_0 \leq -\delta_2\} \end{cases}$$
   On the two more interesting regions,
 $$z = \begin{cases}  =    x ( 1 - \lvert \hat{\theta}'\rvert^2)^{-1}      & Z_0 \geq 2\tilde{\delta}_2\\
  =    x (  \lvert \hat{\theta}'\rvert^2 - 1)^{-1}   & Z_0 \leq -2\tilde{\delta}_2
 \end{cases}.$$

   Similarly define 
 \begin{equation}\label{defOfh}
 h := \chi_0 (Z_0) Z_0  + \chi_+(Z_0) Z_0^{-3/2} + \chi_-(Z_0) (-Z_0)^{3/2}
 \end{equation}
On the two more interesting regions,
 $$ h =   \begin{cases} Z_0^{-3/2} & Z_0 > 2\tilde{\delta}_2 \\ (-Z_0)^{-3/2} & Z_0 < -2\tilde{\delta}_2 \end{cases}.$$
We will also use
   \begin{equation}\label{definitionSig}
  \sigma := \chi_0 (Z_0) ( Z-Z_0)^{3/2} + \chi_+(Z_0)  \big( \dfrac{Z-Z_0}{Z_0}\big)^{3/2} + \chi_-(Z_0)  \big( \dfrac{Z-Z_0}{-Z_0}\big)^{3/2}
  \end{equation}
    Thus, on the two more interesting regions, $\sigma = \begin{cases}    x^{3/2} ( 1 - \lvert \hat{\theta}'\rvert^2)^{-3/2}    & Z_0 \geq 2\tilde{\delta}_2\\
   x^{3/2} (  \lvert \hat{\theta}'\rvert^2 - 1)^{-3/2}   & Z_0 \leq -2\tilde{\delta}_2
 \end{cases}.$
 
   The space the solution lives on blowup space C, see figure \ref{blowupspaceCnB}.
 \begin{figure}
\centering
   \includegraphics[scale=0.80]{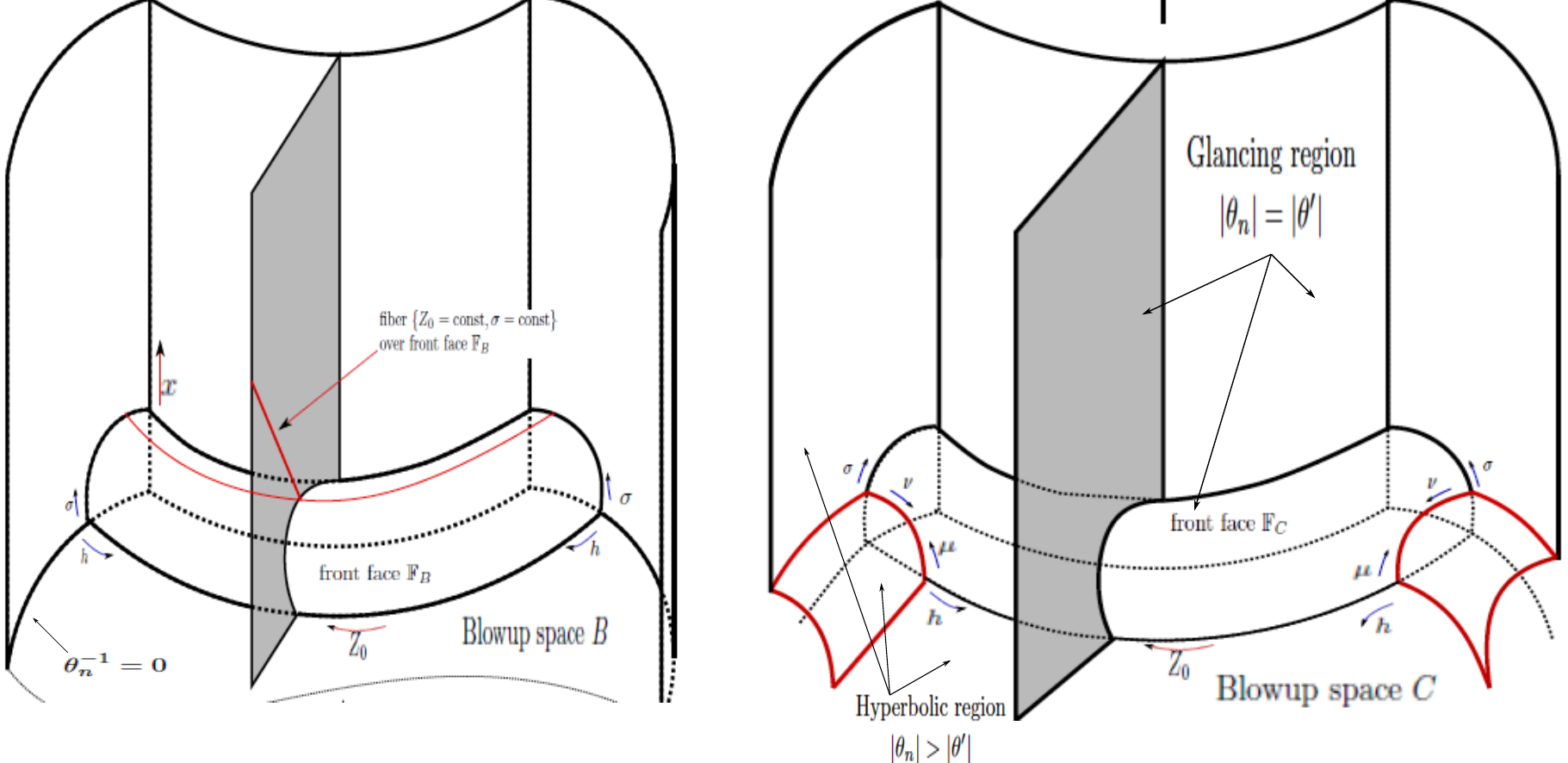}
  \caption{The pictures show the portion of blowup space C and  blown-up space B for $\theta_n > 0$ with the fibration structure $\{Z_0 = \text{const} , \sigma = \text{const}  \}$ over front face $\mathbb{F}_B$. This fibration structure will be used to extend the solution on $\mathbb{F}_B$ into the interior of the space. ($z =\sigma^{2/3}$) \nl The other portion for $\lvert \theta_n\rvert < 0$ is similar. \nl The finite region is labeled by $\mathbb{I} = \{\lvert \theta\rvert \leq 2\}$. For picture of the base (in terms of $\lvert\theta\rvert$) see figure \ref{compactification2} }
  \label{blowupspaceCnB}
   \end{figure}
   \begin{remark}[On the blow-up at $z$ and $h=0$ : transition region]
    In the finite regime, the boundary behavior (at $z =0$) is transcribed by the indicial roots depending on $\lambda$; while the boundary behavior (i.e. at $z =0$) of the solution of the transport equations is independent of $\lambda$. Hence a transition region is needed between the oscillatory behavior in the infinite regime and the regular singular behavior (with no oscillation) in the finite one. See figure \ref{transitionarea}.
    \begin{figure}
 \centering
   \includegraphics[scale=0.6]{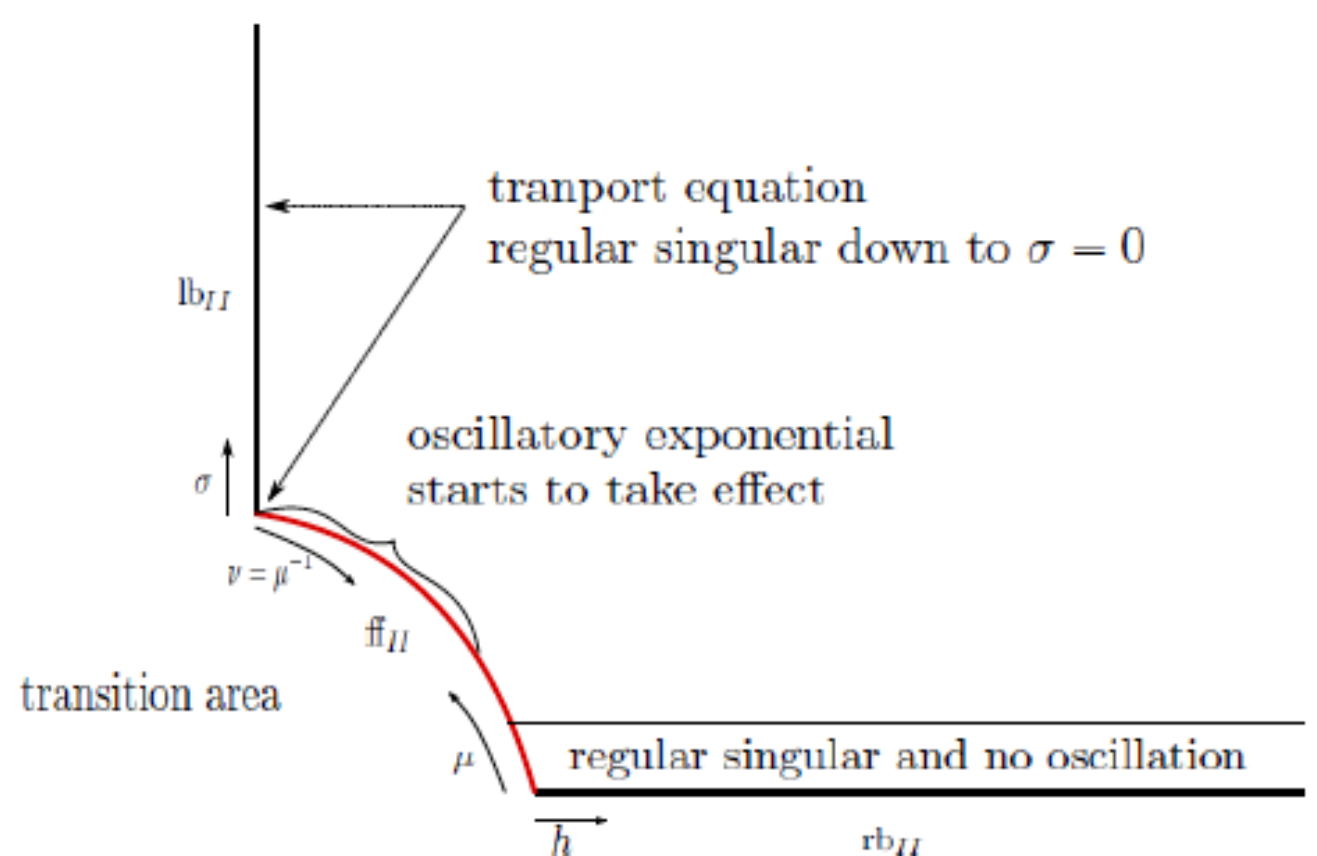}
   \caption{The blowup is associated to $z = 0, h=0$. The projective coordinates are $t = \tfrac{z}{h}, h$ (In the picture $t = \mu^{2/3} , \sigma = z^{2/3}$. See (\ref{definitionOfz}) and (\ref{defOfh}) for definition of $z$ and $h$ )}
   \label{transitionarea}
   \end{figure}
     \end{remark}
We will need to use the following variables associated to these blowup. The boundary defining function for $\mathbb{F}_C^{(1)}$ 
 \begin{equation}\label{defOft}
\rho_{ \mathbb{F}_C^{(1)}} = t := \chi_0 (Z_0) ( Z-Z_0) + \chi_+(Z_0)  \big( \dfrac{Z-Z_0}{Z_0}\big)Z_0^{3/2} + \chi_-(Z_0)  \big( \dfrac{Z-Z_0}{-Z_0}\big)(-Z_0)^{3/2}
 \end{equation}

\textbf{Properties of $\tilde{\QQ}$ in terms in $z,h$}\newline
Consider the region on which $Z_0 > 2\delta_2$. On this region
 $$z := (Z-Z_0) Z_0^{-1} ; \ \ h: = Z_0^{-3/2}$$
 In these coordinates $\tilde{\QQ}$ has the form:
  \begin{equation}
  \begin{aligned} \tilde{\QQ}_+ &=  h^{-2} \Big[ h^2(z \partial_z)^2  +  (\lambda  - \tfrac{n^2}{4}) h^2 +  z^2 + z^3  \Big] \\
  &\stackrel{ \mathsf{t} =z^{-3/2}}{=} \mathsf{t}^{-2} \Big[ h^2(\mathsf{t}^2 \partial_{\mathsf{t}})^2 + \mathsf{t}^{2/3}V^+ + \tfrac{4}{9}    \Big]; 
  \ V^+ =  - h^2 \mathsf{t}^{7/3} \partial_{\mathsf{t}} + \tfrac{4}{9} (\lambda  - \tfrac{n^2}{4}) h^2\mathsf{t}^{4/3} + \tfrac{4}{9}    ; 
  \end{aligned}
  \end{equation}
  This is a family of semiclassical ODE with `b'\footnote{Totally characteristic operators or `b-operators' in the sense of Melrose \cite{Melrose04} are differential operators generated by $b$ vector fields $\mathcal{V}_b$, which is the set of all smooth vector fields tangent to all boundaries. See subsection \ref{adsPropsing} } / regular singular behavior at one end (at $z = 0$) and scattering\footnote{in sense of Melrose scattering calculus \cite{Melrose03}} behavior at infinity.
  
\begin{remark}
In terms of $h,z$ the solution used by Friedlander on this region has the form
$$ \hat{K}_x = \dfrac{\text{Ai}(\zeta)}{\text{Ai}(\zeta_0)}  \sim \exp\Big(-\tfrac{2}{3}i h^{-1} \big[(z + 1)^{3/2} -1 \big]\Big) (z+1)^{-1/2} (1 + f) ;  f\in \mathcal{C}^{\infty}(h,z)$$
This is the model of phase function we will use for the semiclassical local construction and the construction of local approximate solution (for infinity region $\lvert\theta\rvert \geq 1$).
\end{remark}

 As for $Z_0 < 0$ we have another semiclassical problem however this time much simpler. 
 \begin{equation}\label{smlBnSc}
\begin{aligned}
\tilde{\QQ}_- &= 
h^{-2} \Big[ h^2(z \partial_z)^2  +  (\lambda  - \tfrac{n^2}{4}) h^2 -  z^2 + z^3   \Big]\\
  &\stackrel{ \mathsf{t} = z^{-3/2}}{=} \mathsf{t}^{-2} \Big[ h^2(\mathsf{t}^2 \partial_{\mathsf{t}})^2 + \mathsf{t}^{2/3} V^- + \tfrac{4}{9}    \Big]; 
  \ V^- =  - h^2 a\mathsf{t}^{7/3} \partial_{\mathsf{t}} + \tfrac{4}{9} (\lambda  - \tfrac{n^2}{4}) h^2\mathsf{t}^{4/3} - \tfrac{4}{9}    ; a = 1
  \end{aligned}
  \end{equation}
   \textbf{Semiclassical b- symbol} : $-\xi^2 + z^3 - z^2$. On $z \geq 0$ the characteristic set has an isolated point at $z=0$. \nl\textbf{Local behavior at infinity}  : as in the case $Z_0 > 0$ and we can use the result of Vasy - Zworski to obtain semiclassical local resolvent estimate. 
\begin{remark}
The `phase function' of $\hat{K}_x = \tfrac{\text{Ai}(\zeta)}{\text{Ai}(\zeta_0)}$ has the form 
$$\exp\big( -\tfrac{2}{3} (\zeta^{3/2} - \zeta^{-3/2}_0) \big)  = 
\begin{cases}
\exp\big( -i -\tfrac{2}{3} h^{-1} (z - 1)^{3/2} \sgn\theta_n\big) \exp(-h^{-1})  & z  > 1\\
\exp\big(-\tfrac{2}{3}h^{-1} \big[ 1 - (1 - z^{2/3})^{3/2} \big]\big) & z < 1
\end{cases}$$ 
\end{remark}

\section{Construction of approximate solution on the blownup space $C$ : Main idea}\label{mainIdeaApprox}

As motivated in section \ref{statement}, we need to construct a polyhomogeous conormal solution modulo smooth function in $\dot{\mathcal{C}}^{\infty}(\RR^{n+1}_+) $ which satisfies the following boundary condition at $x = 0$:
$$\begin{cases}\hat{L} \hat{u} \in \dot{\mathcal{C}}^{\infty}(\RR^{n+1}_+) \\
x^{-s_-} \hat{u} \ |_{x=0} = 1   & s_{\pm}(\lambda) = \dfrac{n}{2} \pm \sqrt{\dfrac{n^2}{4} - \lambda}\\
\hat{u} \in \exp(-i\phi_{\text{in}})  \mathcal{L}_{ph}(C) & \phi_{\text{in}} = \tfrac{2}{3} \theta_n^{-1} \big[ (1 +x - \lvert\hat{\theta}'\rvert^2)^{3/2} - (1 - \lvert\hat{\theta}'\rvert^2)^{3/2}\big] \sgn\theta_n
\end{cases} $$
We show in section \ref{reductionSml} that this problem has a reduction to a family of semiclassical ODE on $\RR^+$, living on the front face $\mathbb{F}_C$ ( see figure \ref{ffFC(version1)} ) of the blownup space $C$ where the solution in terms of $(x,\theta)$ lives (see also figure \ref{compactification2} and \ref{blowupspaceCnB}). The details of the construction are in section \ref{detailofCon}. In this section, we list the main ideas in order to motivate the technical tools involved: a global resolvent and a local parametrix.\nl
We construct an exact solution for $\hat{L} \hat{u} = 0$ near infinity ($\lvert\theta\rvert \geq 1$) and cut this off away from $\lvert\theta\rvert \leq1$, which now gives an approximate solution to (\ref{originalProAfterF}). To achieve this, one needs an exact polyhomogeous cornomal solution on $\mathbb{F}_C$ (see figure \ref{blowupspaceCnB}) i.e. with no error since this function will be extended into the interior of blowup space $C$ using the fibration structure over front face $\mathbb{F}_B$. This gives a solution with the required asymptotic behavior at infinity where $\lvert\theta\rvert \geq 1$. 
 
 Partition the front face $\mathbb{F}_C$ into $\mathbb{F}_{C,0}$, $\mathbb{F}_{C,+} =  \mathbb{F}_{C,+}^{\infty} \cup \mathbb{F}_{C,+}^0$ and $\mathbb{F}_{C,-} =  \mathbb{F}_{C,-}^0 \cup  \mathbb{F}_{C,-}^{\infty} $ as shown in figure \ref{ffFC(version1)}. 
\begin{figure}
 \centering
  \includegraphics[scale=0.7]{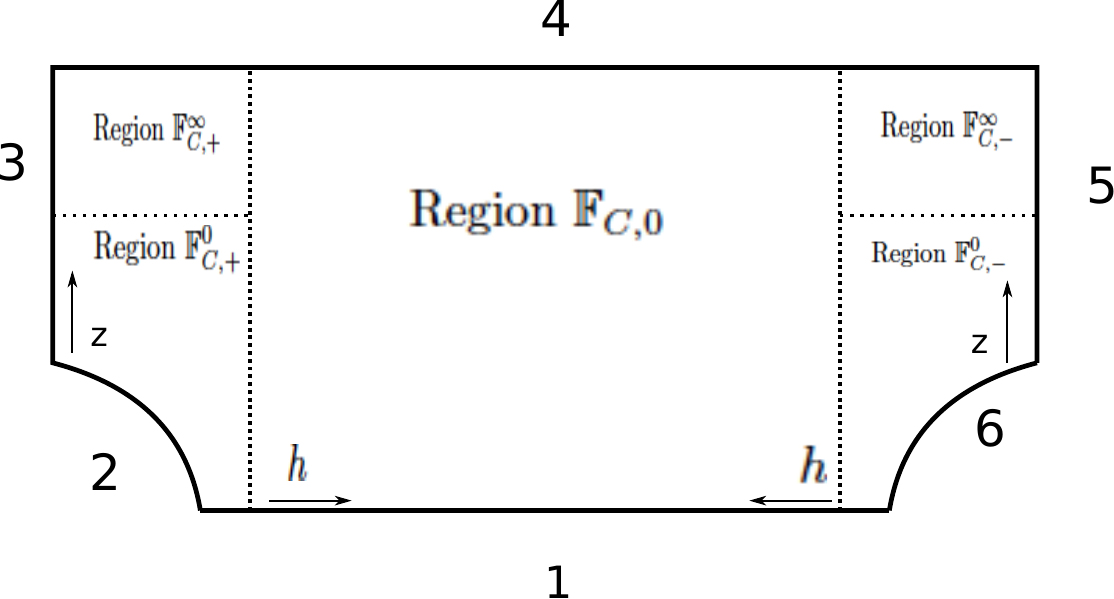}
  \caption{Front face $\mathbb{F}_C$ of blowup space $C$ (see figure \ref{blowupspaceCnB}). This represents the resolution of the infinite regime $\lvert\theta\rvert \geq 1$ at glancing behavior $\lvert\theta_n\rvert = \lvert\theta'\rvert$. The number $i$ indicates the boundary face $\mathbb{F}_C^{(i)}$ for $i=1,\ldots,6$.\nl 
   $\mathbb{F}_{C,+}^{0,0}$ denotes the neighborhood in region $\mathbb{F}_{C,+}^0$ close to the corner of $t = 0, h=0$ where $t$ is the projective blow-up with $t = \tfrac{z}{h}$. $\mathbb{F}_{C,+}^{0,\infty}$ is the neighborhood close to $z = 0, t^{-1} = 0$. Similar notation for $\mathbb{F}_{C,-}^{0,\infty}$ and $\mathbb{F}_{C,-}^{0,0}$ for the side of $\mathbb{F}_C$ where $Z_0 < 0$. Denote by $\mathbb{F}_C^0$ the region of $\mathbb{F}_C$ away from boundary face $\mathbb{F}_C^{(4)}$.\nl
  See (\ref{defOft}) for the definition of $\rho_{\mathbb{F}_C^{(1)}}$ and (\ref{definitionOfz}) for the definition of $z$.}
  \label{ffFC(version1)}
  \end{figure}
  This construction is carried out in two steps:\nl\
  
 \textbf{STEP 1}:
 First we construct an approximate solution that is polyhomogeous conormal to all boundaries of this face, with an error vanishing to infinite order at all boundaries of $\mathbb{F}_C$. This is simpler for the interior where the fibers are not singular. In the region where $\lvert Z_0\rvert\rightarrow 0$, we have the option of constructing directly a polyhomogeneous conormal approximate solution by first constructing it on region $\mathbb{F}_{C,+}^0, \mathbb{F}_{C,-}^0$ then solving along the global transport equations away to infinity. In fact, as we  see very soon that since we will have to build a local parametrix for this region anyway (ie for $\sigma \sim 0$), we will  use this option to obtain the approximate solution from the parametrix by taking the limit as $z' \rightarrow 0$ of the parametrix (which is a function of $h,z$ and $z'$) after some rescaling. After this, we will solve the error along the global transport equations away to infinity. Also see remark (\ref{difference}) for further discussion on the difference between the two chosen methods in dealing with region near $z = 0$ and near $z\rightarrow \infty$. \nl Denoted by $\hat{U}_{\mathbb{F}_C}$ the constructed function in this manner. $\QQ \hat{U}_{\mathbb{F}_C} = E_h$ vanishes to infinite order at all boundary faces of $\mathbb{F}_C$.
 \begin{remark}[On the difference between the methods used for construction of approximate solution near $z = 0$ and that near $z\rightarrow \infty$]\label{difference}
 We only construct a local parametrix for the local region $z\sim 0$ and not for the infinite region (see remark (\ref{motivationResolvent}) for further explication). Because of this, we can read off readily information of the local solution from the parametrix, while having to work out directly the asymptotics of the approximate solution for the infinite region.
 \end{remark}

\textbf{STEP 2 :} 
  In this step, we solve the error $E_h$ exactly by using a global resolvent on $\mathbb{F}_C$. Denoted by $\hat{\mathsf{U}}_{\mathbf{F}_C}$ the exact solution on $\mathbb{F}_C$. \nl We use the gluing method by Datchev-Vasy \cite{Andras-Kiril} to obtain the global resolvent $G$ on $\mathbb{F}_C$ with a  polynomial bounded in $h$ weighted $L^2$ norm. $\hat{U}_{\mathbb{F}_C}$ and $\hat{\mathsf{U}}_{\mathbb{F}_C}$ are then different by
 $$ \hat{U}_{\mathbb{F}_C} - \hat{\mathsf{U}}_{\mathbb{F}_C} =  G E_h$$
where the mapping property of $G$ given by Lemma \ref{globalresolvent}. $G$ maps $\dot{\mathcal{C}}^{\infty}( \mathbb{F}_C)$ into the space of oscillatory functions which have the same asymptotic behavior as $\hat{U}_{\mathbb{F}_C}$ and $GE_h = \mathsf{O}(h^{\infty})$. 
As a result of this, $\hat{\mathsf{U}}_{\mathbb{F}_C}$ has the same asymptotic properties as $\hat{U}_{\mathbb{F}_C}$ does at infinite $\lvert\theta\rvert$. 
 \begin{remark}\label{motivationResolvent}[Ingredients and methods to obtain the global resolvent] Since our operator is not semiclassical elliptic, one cannot simply patch together \emph{by partition of unity} results from the local region around zero and the infinite region. To overcome this problem, we use the gluing method by Datchev-Vasy \cite{Andras-Kiril}. In order to apply this microlocal gluing method, on each region, one either needs a local resolvent that is polynomially bounded in $h$ or at least a local parametrix with a semiclassically trivial error.\nl For the infinity region, we have the semiclassical local resolvent estimate by Vasy - Zworski \cite{Andras-Zworksi} for asymptotically euclidean scattering settings. On the other hand, the local problem near $z = 0$ is more degenerate due to the $b$-geometry of the bicharacteristics as shown in figure \ref{singularBichar} and we do not have results on resolvent estimate for this region. Hence we will have to construct a local parametrix for the region close to zero. In fact, this situation gives rise to a new research direction: to obtain semiclassical resolvent bound directly without a parametrix construction. 
   \begin{figure}
 \centering
 \includegraphics[scale=1.2]{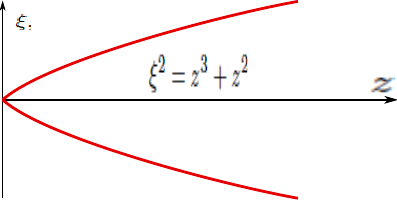}
 \caption{The semiclassical b- symbol for $\tilde{\QQ}$ is $-\xi^2 + z^3 + z^2$. The semiclassical characteristic set is singular at $z = 0$.}
 \label{singularBichar}
\end{figure}

 \end{remark}

%% file: smlparametrixconstruction-pro5a-noheader.tex
\section{Some preparations needed for the semiclassical parametrix construction}
 Recall the family of ODEs on $\RR^+$ obtained from the model problem $P$ through various changes of variable in section \ref{reductionSml}, which we denoted by $\QQ $ where $\QQ = h^2\tilde{\QQ}$, 
 $$\QQ_+ = h^2(z \partial_z)^2  + h^2(\lambda - \dfrac{n^2}{4})+ z^3+ z^2   $$
 $$ \QQ_- = h^2(z \partial_z)^2 + h^2(\lambda - \dfrac{n^2}{4})+ z^3-  z^2   $$
$\QQ_+$ corresponds to $Z_0 > 0$ and $\QQ_-$ to $Z_0 < 0$. For definition of $Z_0$, $h$ and $z$ see (\ref{defOfZ0}), (\ref{defOfh}) and  (\ref{definitionOfz}). 

\subsection{Phase function}
In terms of $h,z$ the solution used by Friedlander on this region has the form
$$ \hat{K}_x = \dfrac{\text{Ai}(\zeta)}{\text{Ai}(\zeta_0)}  \sim \exp\Big(-\tfrac{2}{3}i h^{-1} \big[(z + 1)^{3/2} -1 \big]\Big) (z+1)^{-1/2} (1 + f) ;  f\in \mathcal{C}^{\infty}(h,z)$$
This is the phase function we will use
$$\varphi_{\text{in}} =   h^{-1}\lvert \varphi (z) - \varphi(z')\rvert; \ \ \varphi(z) = \tfrac{2}{3}\big[ (z+1)^{3/2} -1\big] \sgn\theta_n$$
This function can also be obtained directly from solving the eikonal equation for $\QQ$
 $$ h^2 (\mathsf{Z}\partial_{\mathsf{Z}})^2 + h^2 (\lambda-\dfrac{n^2}{4})   + (\mathsf{Z}\rho)^3+ (\mathsf{Z}\rho)^2 $$
this is in b blow-up coordinates : $Z = \tfrac{z}{z'} , \rho = z' , h$.\nl
\textbf{Consider the conjugation of} $\QQ_b$ \textbf{by} $e^{i\tfrac{\varphi}{h}}$ :\
First we compute :
$$\mathsf{Z} \partial_{\mathsf{Z}} \left(e^{i\tfrac{\varphi}{h}} u \right)
 =   ih^{-1} e^{i\tfrac{\varphi}{h}} u \mathsf{Z} \partial_{\mathsf{Z}} \varphi +  e^{i\tfrac{\varphi}{h}} \mathsf{Z} \partial_{\mathsf{Z}} u    $$
 $$(\mathsf{Z} \partial_{\mathsf{Z}})^2 \left(e^{i\tfrac{\varphi}{h}} u \right)
=  - h^{-2} e^{i\tfrac{\varphi}{h}} u (\mathsf{Z} \partial_{\mathsf{Z}} \varphi )^2 
+  ih^{-1} e^{i\tfrac{\varphi}{h}} u (\mathsf{Z}  \partial_{\mathsf{Z}})^2 \varphi
+ 2 ih^{-1} e^{i\tfrac{\varphi}{h}} (\mathsf{Z}  \partial_{\mathsf{Z}} u)(\mathsf{Z}  \partial_{\mathsf{Z}} \varphi)
+ e^{i\tfrac{\varphi}{h}} (\mathsf{Z} \partial_{\mathsf{Z}})^2 u  $$
The eikonal equation comes from coefficients of $h^0$ :
$$ u \left[ -(\mathsf{Z} \partial_{\mathsf{Z}} \varphi )^2 +  (\mathsf{Z}\rho)^3+ (\mathsf{Z}\rho)^2\right] = 0$$
$$\Rightarrow \partial_{\mathsf{Z}} \varphi = \pm\left[  \mathsf{Z} \rho^3+ \rho^2\right]^{1/2}$$
$$\Rightarrow \varphi = \pm\int \left[  \tilde{\mathsf{Z}} \rho^3+ \rho^2\right]^{1/2} \, d\tilde{\mathsf{Z}} + C(\rho)$$
$$\Rightarrow \varphi (\mathsf{Z},\rho)= \pm \dfrac{2}{3} \dfrac{ \left[  \mathsf{Z} \rho^3+ \rho^2\right]^{3/2}}{\rho^3} + C(\rho)$$
We can check again that 
$$\varphi (\mathsf{Z},\rho)=\pm \dfrac{2}{3} \dfrac{ \left[  \mathsf{Z} \rho^3+ \rho^2\right]^{3/2}}{\rho^3} + C(\rho) = \pm \dfrac{2}{3}\left[\mathsf{Z} \rho + 1\right]^{3/2} + C(\rho) $$
solves $-(\mathsf{Z} \partial_{\mathsf{Z}} \varphi )^2 +  (\mathsf{Z}\rho)^3+ (\mathsf{Z}\rho)^2 = 0$

\textbf{In semiclassical blow-up coordinates} :  $Z_h = \tfrac{\mathsf{Z} -1}{h},h,\rho$ then $\mathsf{Z} = h Z_h + 1$ . In this coordinate the phase function is
\begin{align*}
 h^{-1} \left[\pm\dfrac{2}{3}\left[(h Z_h \rho +\rho + 1\right]^{3/2} + C(\rho)\right]
&= h^{-1} \left[\pm \dfrac{2}{3}(\rho+1)^{3/2} \left( 1 + hZ_h \dfrac{\rho}{\rho+1}\right)^{3/2} + C(\rho)\right]\\
&= \dfrac{2}{3}(\rho+1)^{3/2} h^{-1} \left[\pm  \left( 1 + hZ_h \dfrac{\rho}{\rho+1}\right)^{3/2} + \dfrac{3}{2} C(\rho)(\rho+1)^{-3/2}\right]
\end{align*}
which we need to be equal to $\pm \rho  Z_h $ when restricted to $h=0$
$$\lim_{h\rightarrow 0}  \dfrac{2}{3}(\rho+1)^{3/2} h^{-1} \left[ \pm \left( 1 + hZ_h \dfrac{\rho}{\rho+1}\right)^{3/2} + \dfrac{3}{2} C(\rho)(\rho+1)^{-3/2}\right]
=\pm Z_h ( \rho^3 + \rho^2)^{1/2}$$
Hence we need 
$$\dfrac{3}{2} C(\rho) (\rho+1)^{-3/2 } = \mp1 \Rightarrow C(\rho) =\mp \dfrac{2}{3}(\rho+1)^{3/2}$$
In that case then each fixed $Z_h$ and $\rho$ the limit is
\begin{align*}
&\lim_{h\rightarrow 0}  \dfrac{2}{3}(\rho+1)^{3/2} h^{-1} \left[  \pm\left( 1 + hZ_h \dfrac{\rho}{\rho+1}\right)^{3/2} \mp \dfrac{3}{2} C(\rho)(\rho+1)^{-3/2}\right]\\
&= \pm\dfrac{1}{3} (\rho+1)^{3/2} 3 Z_h \dfrac{\rho}{\rho+1} = \pm Z_h \rho(\rho+1)^{1/2} = \pm Z_h ( \rho^3 + \rho^2)^{1/2}
\end{align*}

The phase function should be symmetric in $z,z'$ and singular across the diagonal. Hence we choose to work with \textbf{the following phase function}, which is shown in various coordinates   :
\begin{itemize}
\item In coordinates on the double space crossed with $h$ : $z,z' , h$
$$\varphi_{\text{in}} = \dfrac{2}{3}h^{-1} \big| (z+1)^{3/2} - (z'+1)^{3/2} \big|$$
The phase $-\varphi_{\text{in}}$ satisfies the eikonal equation 
$$-(z\partial_z)^2 + z^3 + z^3$$   for the operator
$$\QQ(h,\lambda) = z^{-\tfrac{n}{2}} \left(\dfrac{9}{4}z^{\tfrac{1}{3}}Q_1\right) z^{\tfrac{n}{2}}=   h^2(z \partial_z)^2   + h^2(\lambda-\dfrac{n^2}{4}) + z^3+ z^2 $$

\item  In b blow-up coordinates $\mathsf{Z}=\tfrac{z}{z'}, \rho = z' ,h$ 
$$\varphi_{\text{in}} = \dfrac{2}{3}h^{-1}\big| (\mathsf{Z}\rho + 1)^{3/2} - (\rho+1)^{3/2} \big| $$
\item In b semiclassical blow-up $Z_h = \tfrac{\mathsf{Z}-1}{h},\rho=z',h$
$$\varphi_{\text{in}} = \dfrac{2}{3}h^{-1}\big| (hZ_h\rho + 1+\rho)^{3/2} - (\rho+1)^{3/2} \big| $$

\end{itemize}

\subsubsection{Computation to show that the phase function $\varphi$ is smooth up to $\mathfrak{B}_1$ in $\tilde{\mathfrak{M}}_{b,h}$}\

Coordinates associated to the various blow-ups in this region
  \begin{itemize}
  \item Away from the semiclassical face $\mathcal{A}$, on $\mathfrak{M}$, the projective coordinates are 
$$\tilde{\mu}' = \tfrac{z'}{h}, h  , \mathsf{Z}$$
Away from the b front face, the projective coordinates on $\mathfrak{M}$ are : 
$$ \tilde{\nu}' = \tfrac{h}{z'}, z' , \mathsf{Z}$$  
  \item On $\mathfrak{M}$, away from the semiclassical face $\mathcal{A}$, the projective coordinates are 
  $$\tilde{\mu}  = \tfrac{z}{h}, h  , \mathsf{Z}'$$
  while away from the b front face, the projective coordinates are 
  $$ \tilde{\nu}  = \tfrac{h}{z}, z , \mathsf{Z}'$$
\end{itemize}

\textbf{Away from the left boundary} $\mathcal{L}$ :
 The coordinates in this region are 
 $$ \tilde{\mu}' = \dfrac{z'}{h}, h , \mathsf{Z} $$ 
 in which  the phase function takes the form :
 \begin{align*}
 \varphi &= \dfrac{2}{3} h^{-1} \left[  - ( \mathsf{Z} \tilde{\mu}' h + 1)^{3/2} + (\tilde{\mu}' h + 1)^{3/2}\right]\\
 &= \dfrac{2}{3}h^{-1} \dfrac{\left[3\tilde{\mu}'h + 3(\tilde{\mu}')^2h^2 + (\tilde{\mu}')^3 h^3\right] - \left[3\mathsf{Z}\tilde{\mu}'h + 3\mathsf{Z}^3(\tilde{\mu}')^2h^2 + \mathsf{Z}^3 (\tilde{\mu}')^3 h^3 \right]}{  ( \mathsf{Z} \tilde{\mu}' h + 1)^{3/2} + (\tilde{\mu}' h + 1)^{3/2}}\\
 &= \dfrac{2}{3} \dfrac{\left[3\tilde{\mu}' + 3(\tilde{\mu}')^2h + (\tilde{\mu}')^3 h^2\right] - \left[3\mathsf{Z}\tilde{\mu}' + 3\mathsf{Z}^3(\tilde{\mu}')^2h + \mathsf{Z}^3 (\tilde{\mu}')^3 h^2 \right]}{  ( \mathsf{Z} \tilde{\mu}' h + 1)^{3/2} + (\tilde{\mu}' h + 1)^{3/2}}
 \end{align*}
 since $\mathsf{Z} , \tilde{\mu}', h \geq 0$  the above expression is smooth in $\tilde{\mu}' , h, \mathsf{Z}$.
 
\textbf{Away from the right boundary} $\mathcal{R}$ :
 The coordinates on this region are
 $$ \tilde{\mu}  = \dfrac{z }{h} , h , \mathsf{Z}' $$
in which the phase function takes the form:
 \begin{align*}
 \varphi &= \dfrac{2}{3}h^{-1} \left[(z' + 1)^{3/2} - (z+1)^{3/2} \right]= \dfrac{2}{3} h^{-1} \left[(z \mathsf{Z}' + 1)^{3/2} - (z + 1)^{3/2} \right]\\
 &= \dfrac{2}{3}h^{-1} \left[(\tilde{\mu} h \mathsf{Z}' + 1)^{3/2} - (\tilde{\mu}h + 1)^{3/2} \right]\\
 &= \dfrac{2}{3} h^{-1} \dfrac{(\tilde{\mu} h\mathsf{Z}')^3 + 3(\tilde{\mu} h\mathsf{Z}')^2 + 3  \tilde{\mu} h\mathsf{Z}' - \left[(\tilde{\mu} h)^3 + 3(\tilde{\mu} h)^2 + 3  \tilde{\mu} h\right] }{(\tilde{\mu} h \mathsf{Z}' + 1)^{3/2} + (\tilde{\mu} h + 1)^{3/2} }\\
 &= \dfrac{2}{3} \dfrac{(\tilde{\mu} \mathsf{Z}')^3 h^2+ 3(\tilde{\mu} \mathsf{Z}')^2h +   \tilde{\mu} \mathsf{Z}' - \left[(\tilde{\mu} )^3 h^2+ 3(\tilde{\mu} )^2h +   \tilde{\mu} \right] }{(\tilde{\mu} h \mathsf{Z}' + 1)^{3/2} + (\tilde{\mu} h + 1)^{3/2} }
 \end{align*}
since $\mathsf{Z}' , \tilde{\mu} , h \geq 0$  the above expression is smooth in $\tilde{\mu} , h, \mathsf{Z}'$.

\begin{remark}
The phase function is smooth up to $\mathfrak{B}_1$ but blows up at the semiclassical face. 
\end{remark}

\subsection{The blow-up space}\label{blowupSpace}\

  \begin{figure}
  \includegraphics[scale=0.80]{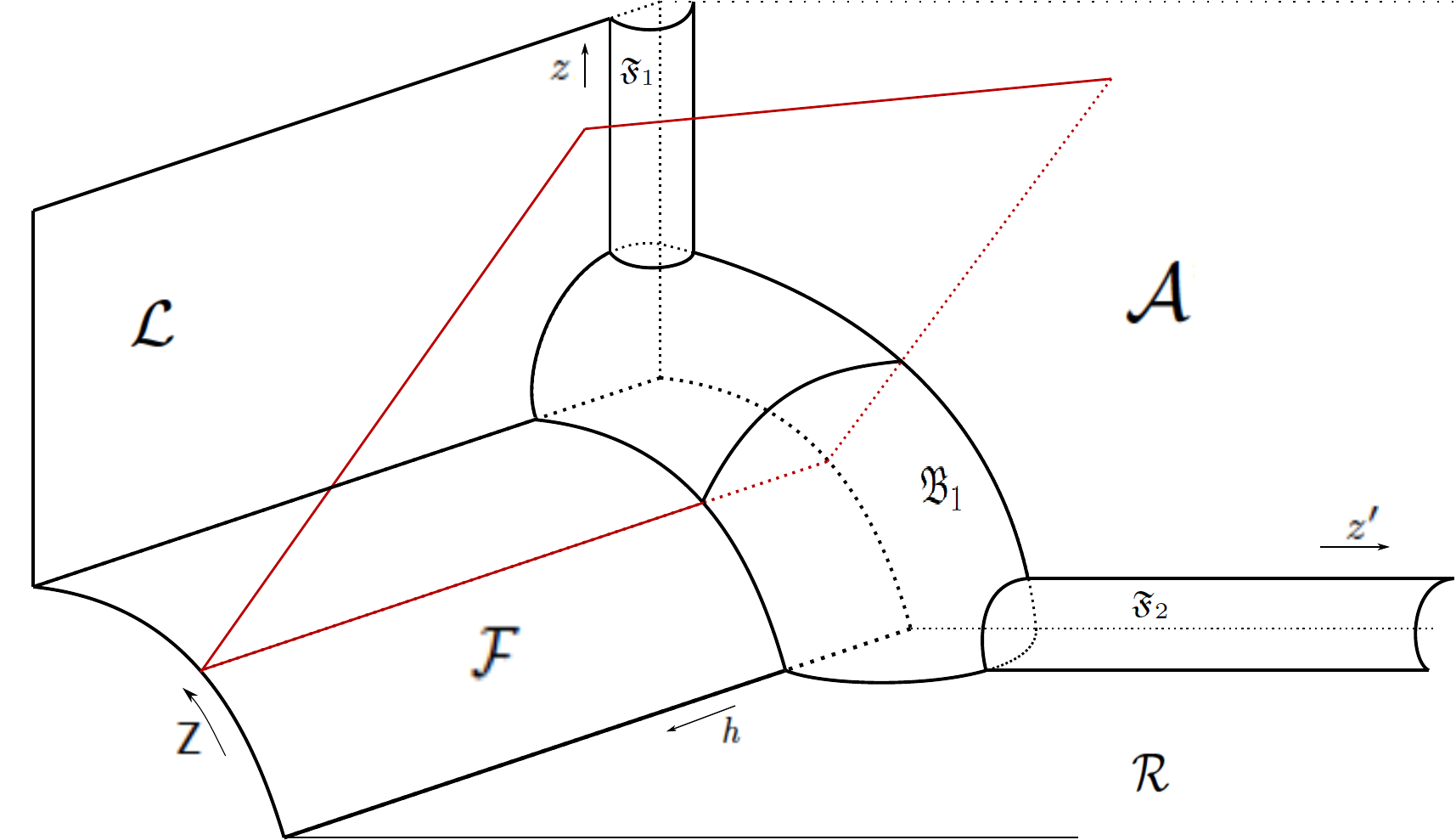}
  \caption{Blown-up space $\mathfrak{M}_{b,h}$ at the $Z_0 > 0$ end \label{frakM}}
  \end{figure}

  \begin{figure}
  \includegraphics[scale=0.80]{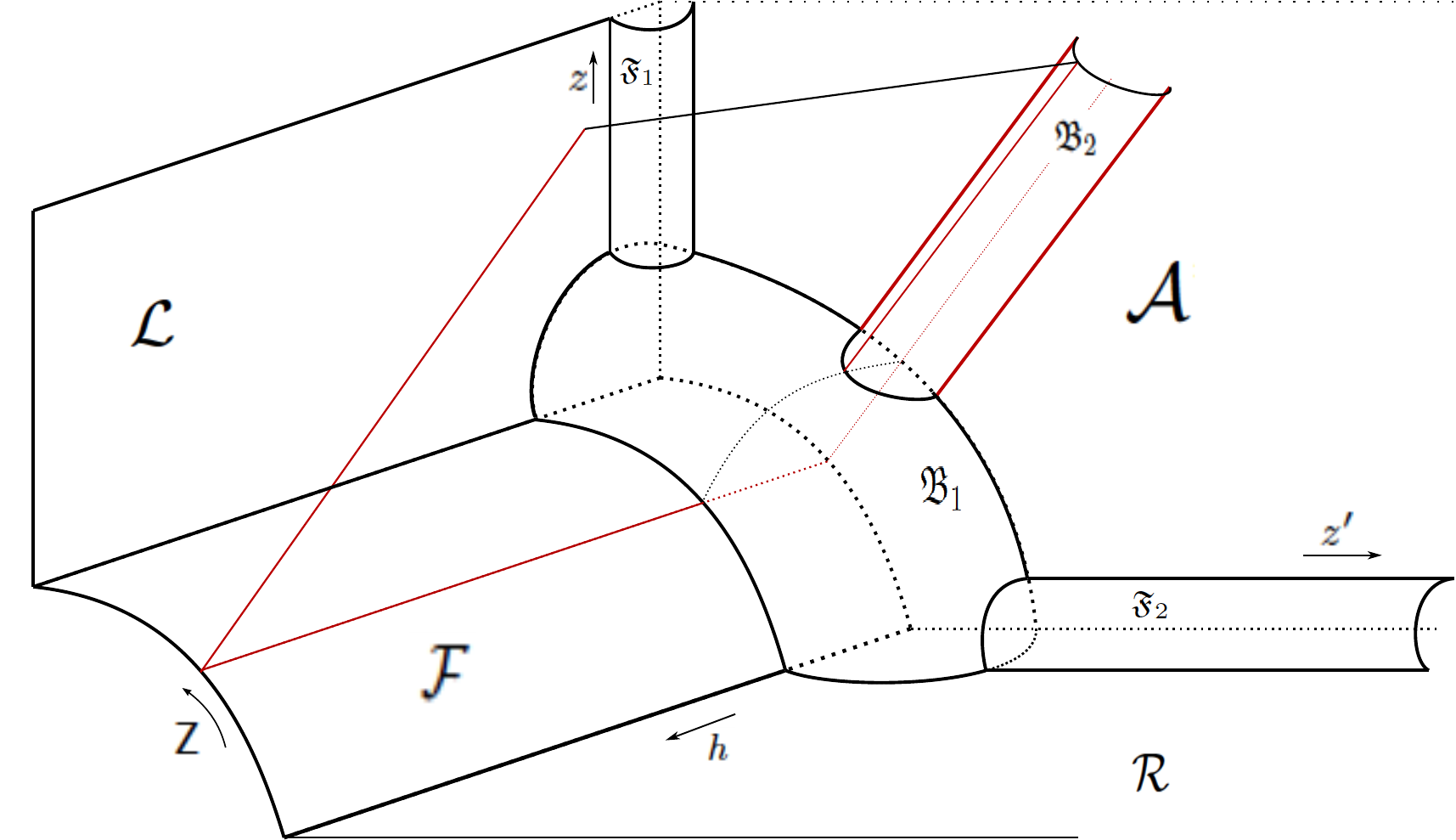}
  \caption{Blown-up space $\tilde{\mathfrak{M}}_{b,h}$ at the $Z_0 > 0$ end\label{tidfrakM}}
  \end{figure}
 
 Denote by $X = [0,1)$. We have two semiclassical problems when $Z_0 \rightarrow \infty$ and $Z_0 \rightarrow -\infty$ ie when $h\rightarrow 0$ (defined by (\ref{defOfh})). Start with the space $X\times_b X \times [-\infty, \infty]_{Z_0}$, where $X\times_b X$ is the b-blown up space which is obtained from $X\times X$ by blowing up the corner  $\partial_X \times \partial_X$ ( which in coordinates is $\{z=z'=0\}$) in $X\times X$. In this section, we will analyse separately each end but the parametrix will be built all at once on the whole blow-up space. Locally each end can be considered as $X\times_b X \times [0,\infty)_h$.
 

  \textbf{Notation} : Denote 
  \begin{itemize}
  \item by $\mathcal{L}$ the lift to $X\times_b X \times [-\infty, \infty]_{Z_0}$ of $\{z' = 0\}$, by $\mathcal{R}$ of $\{z = 0\}$, 
  \item by the semiclassical face $\mathcal{A}$ the lift of $\{ Z_0^{-1} = 0, Z_0 > 0\}$ ( = $\{h=0\}$ on the $Z_0 > 0$ end), 
     \item by the semiclassical face $\tilde{\mathcal{A}}$ the lift of $\{ Z_0^{-1} = 0, Z_0 < 0\}$ ( = $\{h=0\}$ on the $Z_0 < 0$ end), 
   \item  by $\Diag_{b,h}$ the lift of the diagonal $\{z = z'\}$, 
  \item and by $\mathcal{F}$ the resulting blown-up b - front face.
 \end{itemize}
  \textbf{The usual approach:} In order to resolve the singularity of the phase function at the semiclassical face we have to do these following two blow-ups: the intersection of the b front face with the semiclassical face, and the intersection of the lifted diagonal with the semiclassical face. The two operations however do not commute with each other; different spaces result, depending on the order in which the blow-ups is carried out. For example in \cite{M-SB-AV}, the authors blow-up the intersection of the lifted diagonal and the semiclassical face in $X\times_b X \times [0,\infty)_h$, whose resulting front face is called the semiclassical front face $\mathcal{S}$. The operator in \cite{M-SB-AV} has uniform property on $\mathcal{S}$ all the way down to the intersection of $\mathcal{S}$ and $\mathcal{F}$. However in our situation, if taking this approach, we have degeneracy as we approach $\mathcal{F}$ along $\mathcal{S}$. This  degeneracy requires further blow up at the intersection of $\mathcal{S}$ and $\mathcal{F}$.\newline
  In particular, the normal operator on the semiclassical front face has the form for $Z_0 > 0$, 
 $$ \partial_{\mathsf{Z}_h}^2 + (z')^3 + (z')^2  ;\ \  \mathsf{Z}_h = \tfrac{\mathsf{Z}-1}{h} ;\ \mathsf{Z}=\tfrac{z}{z'}$$
We have degeneracy of the potential and the behavior is not uniform as we approach the b front face (i.e. as $z' \rightarrow 0$) along the semiclassical  face. Denote by $\lambda = (z')^3 + (z')^2$. At arbitrary level $\lambda \neq 0$ solution has the asymptotic form:
$$\exp(-i \sqrt{\lambda} \lvert \mathsf{Z}_h\rvert ) \  \mathcal{C}^{\infty}.$$
This gives motivation to do a blow-up at $\mathsf{Z}_h \rightarrow \infty$ and $\lambda \rightarrow 0$ i.e. $z' \rightarrow 0$. We also have degeneracy when $Z_0 < 0$ with the  normal operator on the semiclassical front face of the form:
   $$ \partial_{\mathsf{Z}_h}^2 + (z')^3 - (z')^2  ;\ \  \mathsf{Z}_h = \tfrac{\mathsf{Z}-1}{h};\ \ \mathsf{Z}=\tfrac{z}{z'}$$

  \textbf{The blow-up space in our setting (see figure \ref{tidfrakM}):} The blow-up space we will use is obtained from these following blow-ups. For both ends $Z_0 > 0$ and $Z_0 < 0$, we will first blow up the intersection of $\mathcal{F}$ and $\mathcal{A}$, and the intersection of $\mathcal{F}$ and $\tilde{\mathcal{A}}$ in $X\times_b X\times [-\infty,\infty]_{Z_0}$, whose resulting new front face is denoted by $\mathfrak{B}_1$ and $\tilde{\mathfrak{B}}_1$ respectively. 
  Next on this new space, we will blow-up the intersection of the lifted diagonal and the semiclassical face $\mathcal{A}$ and similarly the intersection of the lifted diagonal and the other semiclassical face $\tilde{\mathcal{A}}$; denote respectively by $\mathfrak{B}_2$ and $\tilde{\mathfrak{B}}_2$ the resulting blown-up front faces.  Call our newly constructed blown-up space $\mathfrak{M}$. (See figure \ref{frakM}) \nl
  We will need to carry out two additional blow-ups on $\mathfrak{M}$. On $Z_0 < 0$, the problem is elliptic there (\emph{so singularities do not get propagated}), these following blow-ups are more of formal nature. However for $Z_0 > 0$, due to the hyperbolicity of problem (\emph{hence propagation of singlarities}), these procedures will be necessary. To desingularize the flow on $\mathfrak{B}_1$, we will blow-up the corner intersection $\mathcal{L}\cap \mathcal{A}$, denoted the resulting front face by $\mathfrak{F}_1$. On the other hand, the transport equation along semiclassical face $\mathcal{A}$ does not give rise to the boundary asymptotic at $\mathcal{R}$ dependent on the indicial roots. Thus, we will blow up the corner intersection at $\mathcal{R}\cap \mathcal{A}$ and denote by $\mathfrak{F}_2$ the resulting front face. $\mathfrak{F}_2$ serves as the transition region from the scattering behavior to the boundary behavior prescribed by indicial roots by one solving the Bessel equation there. See the computation at the end of the subsection for more details.
  \begin{figure}
\centering 
 \includegraphics[scale=0.90]{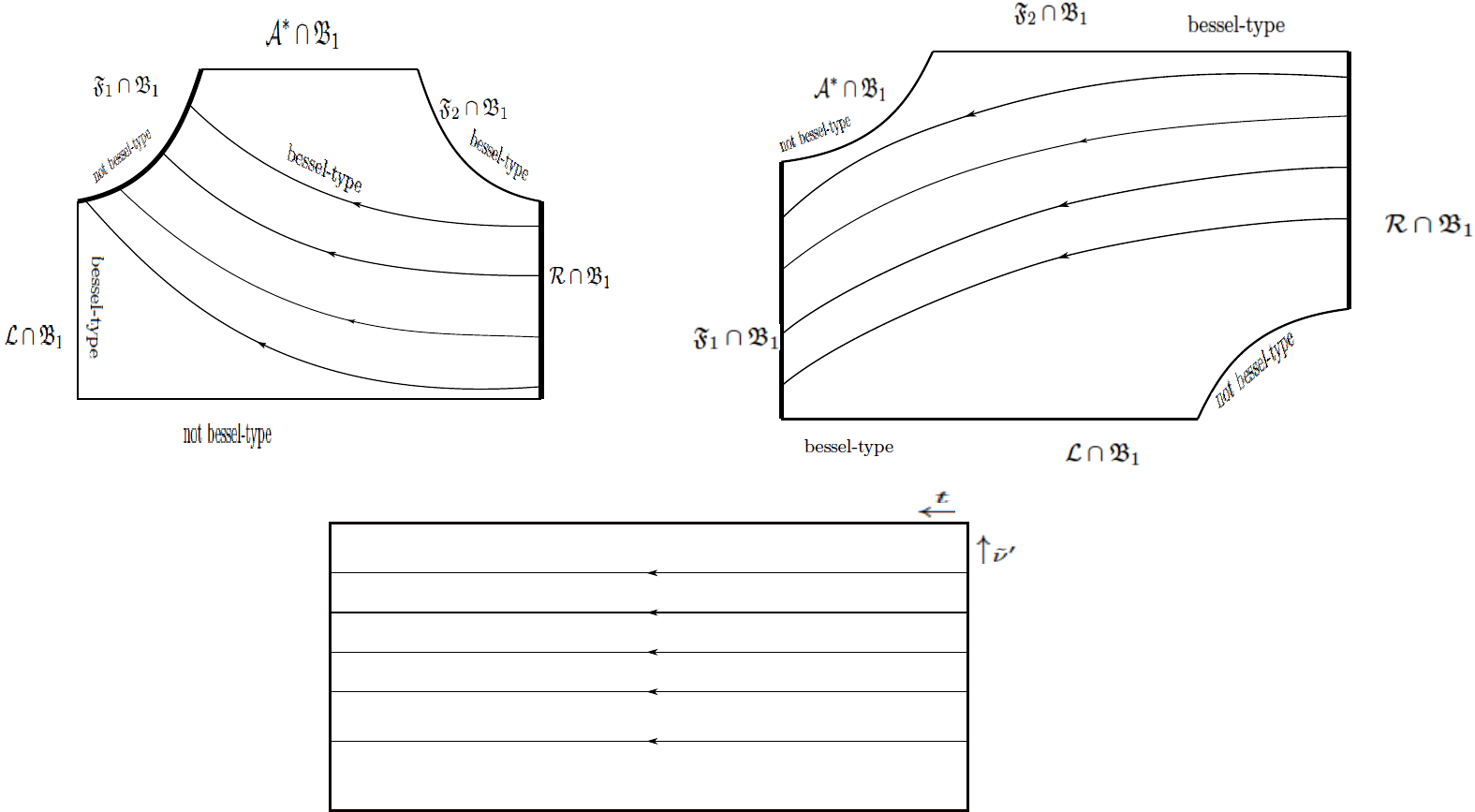}
 \caption{We have the Bessel type ODE in $\tilde{t}$ along the flow $\{\tilde{\nu}' =\text{const}\}$. $\mathfrak{B}_1$ is a blown-up version of the product structure shown on the right $\RR^+_{\tilde{t}} \times \RR^+_{\tilde{\nu}'}$. The label `bessel-type' and `not bessel-type' denote the behavior of the limiting operator on the corresponding boundary sides of $\mathfrak{B}_1$\label{ffb1} }
  \end{figure}
 For $Z_0 < 0$, denote by the $\tilde{\mathfrak{F}}_1$ and $\tilde{\mathfrak{F}}_2$ the front faces resulting from the same blow-up procedures. Denote by $\tilde{\mathfrak{M}}_{b,h}$ the resulting space (see figure \ref{tidfrakM}). Denote by $\mathfrak{M}_{b,h}$ the space obtained from $\tilde{\mathfrak{M}}_{b,h}$ by blowing down $\mathfrak{B}_2$ and $\tilde{\mathfrak{B}}_2$.  
  
  \textbf{Computation related to the blowup of $\mathfrak{F}_1$ and $\mathfrak{F}_2$:} \newline
  The equation in $z$ variable:
  $$ h^2(z\partial_z)^2 + h^2( \lambda - \dfrac{n^2}{4}) + z^3 + z^2$$
 In projective blow-up coordinates on $X^2_{b,h}$, $\mathsf{Z} , \tilde{\mu} = \tfrac{z'}{h} , h$ so $z' = \tilde{\mu} h$,  the operator has the form: 
    $$ h^2( \mathsf{Z} \partial_{\mathsf{Z}})^2 + h^2 (\lambda - \dfrac{n^2}{4}) + (\mathsf{Z} \rho)^3 + (\mathsf{Z}\rho)^2$$
 In coordinates $\tilde{\mu}', \mathsf{Z}, z'$ on $\tilde{M}_2$ (away from the left boundary $\mathfrak{L}$) the operator becomes:
    $$h^2( \mathsf{Z} \partial_{\mathsf{Z}})^2 + h^2 (\lambda - \dfrac{n^2}{4}) + (\mathsf{Z} \tilde{\mu}' h)^3 + (\mathsf{Z}\tilde{\mu}'h)^2 = h^2 \left[( \mathsf{Z} \partial_{\mathsf{Z}})^2 + (\lambda - \dfrac{n^2}{4}) + h (\mathsf{Z}\tilde{\mu}')^3 + (\mathsf{Z} \tilde{\mu}')^2  \right]$$
    This a family of ODE varying smoothly in $h$. In the neighborhood where these coordinates are valid, the boundary face $\mathfrak{B}_1$ can be described by $\{ h = 0\}$. Hence when restricted to $h=0$ i.e. the front face $\mathfrak{B}_1$, we get $$ (\mathsf{Z} \partial_{\mathsf{Z}})^2 + (\lambda - \dfrac{n^2}{4})  + (\mathsf{Z} \tilde{\mu}')^2$$
 The flow on which we have the ODE is along $\{\tilde{\mu}' = \text{constant}\}$. This flow becomes singular as one approaches the corner where $\mathfrak{B}_1$ meets  $\mathcal{L}$ and semiclassical face $\mathcal{A}$. To desingularize the flow, we are going to blow up the edge which is the intersection $\mathcal{L} \cap \mathcal{A}$ in $\mathfrak{M}$. In addition to this, we need to resolve the behavior at $\mathsf{Z} = 0 $ and $\tilde{\nu}' =(\tilde{\mu}')^{-1} = 0$. The projective coordinates to describe this blow-up:
   $$ \tilde{t} = \dfrac{\mathsf{Z}}{\tilde{\nu}'}, \tilde{\nu}' , z'$$ 
 For the computation for $Z_0$, we just need to change the sign in front of $z^2$ ie the operator on $\tilde{\mathfrak{B}}_1$ is 
 $$ (\mathsf{Z} \partial_{\mathsf{Z}})^2 + (\lambda - \dfrac{n^2}{4})  - (\mathsf{Z} \tilde{\mu}')^2.$$

\subsection{Operator space }
   We follow closely the definition used in \cite{M-SB-AV}, however the resulting operator spaces are not the same since we are not on the same blow-up spaces\footnote{See explanation in subsection (\ref{blowupSpace})}. Denote by $\Psi^m_{0,h}$ the space of operators $P$ whose kernel 
   $$K_P(z,z',h) \lvert dg (z')\rvert$$
   lifts to $\tilde{\mathfrak{M}}_{b,h}$ a cornomal distribution of order $m$ to the lifted diagonal and vanishes to infinite order at all faces, except the b front face $\mathcal{F}$, up to which it is smooth (with values in conormal distributions) and the semiclassical front faces $\mathfrak{B}_1$ and $\mathfrak{B}_2$, up to which it is $h^{-1} \mathcal{C}^{\infty}$ . Similarly behavior for $\tilde{\mathfrak{B}}_1$ and $\tilde{\mathfrak{B}}_2$.

\section{Semiclassical left parametrix construction for $\QQ(h,\lambda)$ for $0 \leq z ,z'  \leq 1$}\label{ParametrixConstruction}
As motivated in section \ref{mainIdeaApprox} (see in particular remark (\ref{motivationResolvent})), we need to build a local parametrix for $\tilde{\QQ}$. For the parametrix construction, we will work with $\QQ$ where $\QQ = h^2\tilde{\QQ}$, so $\QQ$ is of the form:
 $$\QQ_+ = h^2(z \partial_z)^2  + h^2(\lambda - \dfrac{n^2}{4})+ z^3+ z^2   $$
 $$ \QQ_- = h^2(z \partial_z)^2 + h^2(\lambda - \dfrac{n^2}{4})+ z^3-  z^2   $$
$\QQ_+$ corresponds to $Z_0 > 0$ and $\QQ_-$ to $Z_0 < 0$. For definition of $Z_0$, $h$ and $z$ see (\ref{defOfZ0}), (\ref{defOfh}) and  (\ref{definitionOfz}). This family of semiclassical operators $\tilde{\QQ}$ is obtained from our original operator (\ref{originalProblem}) by a fourier transform and several changes of variable (in section \ref{reductionSml} in particular equation (\ref{smlBnSc})). As motivated in section \ref{mainIdeaApprox}, we use the local parametrix for two purposes. First, we obtain from the parametrix an local (near zero) exact solution for $\QQ u = 0$. The second (in fact main) purpose of the parametrix is obtained a global resolvent in order to solve exactly semiclassically trivial errors.  \nl
\textbf{The main method of the construction is as follows:} We construct $U$ on $\tilde{\mathfrak{M}}_{b,h}$ such that  
     $$\QQ U  - \text{SK}_{\text{Id}} \in    h^{\infty} z^{\infty} (z')^{\sqrt{\tfrac{n^2}{4}-\lambda}  -\tilde{\gamma}_{\mathfrak{B}_1} +2}$$
     where $ \text{SK}_{\text{Id}} $ denotes the schwartz kernel of the identity operator $\Id$. We will be removing error  the boundary faces of $\tilde{\mathfrak{M}}_{b,h}$ in the following order: 
     \begin{enumerate}
\item \textbf{Step 1:} We first remove diagonal singularities by extending the lifted diagonal past the boundary faces $\tilde{\mathfrak{M}}_{b,h}$ at both ends $Z_0 > 0$ and $Z_0 < 0$ and use transversal elliptic construction.
 \item \textbf{Step 2:} next we remove error at the semiclassical front faces $\mathfrak{B}_2$ and $\tilde{\mathfrak{B}}_2$. The normal operator in local coordinates has the form
 $$    \mathcal{N}_{\mathfrak{B}_2}(\QQ_h)=  (z')^2 \left[ \partial_{\mathsf{Z}_h}^2+z'+1 \right],$$
 this is the Euclidean laplacian with spectral parameter $z' + 1$. Note the uniform behavior down to $\mathfrak{B}_1$\footnote{This is the reason why we choose to work with our current blowup space instead of one used in e.g \cite{M-SB-AV}. See subsection \ref{blowupSpace} for further discussion}. For $Z_0 < 0$
  $$    \mathcal{N}_{\tilde{\mathfrak{B}}_2}(\QQ_h)=  (z')^2 \left[ \partial_{\mathsf{Z}_h}^2+z'-1 \right]. $$
 After this step we can blow-down $\mathfrak{B}_2$ and $\tilde{\mathfrak{B}}_2$ and work on $\mathfrak{M}_{b,h}$.
 
 \begin{remark}
At the $Z_0 < 0$, we have an elliptic problem. Unlike when $Z_0 > 0$, for $Z_0< 0$, we will not have singularities propagated from $\tilde{\mathfrak{B}_2}$ onto $\tilde{\mathcal{A}}$, thus steps 3, 4, 5 and 6b are only needed for $Z_0 > 0$.
\end{remark}
\begin{remark}\label{prepForB1}[\textbf{\emph{Providing the motivation for the next steps from step 3 to 5}}]
  Step 3, 4 and 5 served to blow down $\mathfrak{B}_1$ and $\tilde{\mathfrak{B}}_1$ to a product structure. In more details, on $\mathfrak{B}_1$, $\QQ_b$ restricts to :  \emph{(see step 4 for computation)}
 $$(\tilde{t}\partial_{\tilde{t}})^2 + (\lambda - \tfrac{n^2}{4}) +\tilde{t}^2$$  which thus can be considered to live on the product space $B_1 = \RR^+_t \times \RR^+_{\tilde{\nu}'}$ (so $B_1$ is the 2-dimensional product structure \emph{represented } in figure (\ref{ffb1}) by the rectangle). On the other hand, the error we need to solve away (ie the inhomogeneous part) lives on \emph{(i.e is nice' function on)} $\mathfrak{B}_1$. We will replace this error by something that behaves trivially at $\mathcal{A}\cap \mathfrak{B}_1$ and $\mathcal{F}\cap \mathfrak{B}_1$ in order to be work with an error that resides $B_1$. As for the $Z_0 < 0$ end ie on $\tilde{\mathfrak{B}}_1$, the normal operator is
 $$(\tilde{t}\partial_{\tilde{t}})^2 + (\lambda - \tfrac{n^2}{4}) -\tilde{t}^2$$ 
 As noted by the previous remark due to the ellipticity of the problem we automatically have infinite vanishing at $\tilde{\mathcal{A}} \cap \mathfrak{B}_1$ after step 2.
  \end{remark}

    \item  \textbf{Step 3:} for this step we remove error at the semiclassical front face $\mathcal{A}$ and $\mathfrak{F}_2$. The process is carried out in 3 smaller steps:
\begin{itemize}
\item Step 3a: Singularities are propagated out from $\mathfrak{B}_2$ onto the semiclassical face $\mathcal{A}$. We remove the error at $\mathcal{A}$ along the Lagrangian $\Gamma_1$ associated to the phase function $-\varphi_{\text{in}}$ where
$$\varphi_{\text{in}} = \tfrac{2}{3} h^{-1} \big| (z'+1)^{3/2} - (z+1)^{3/2}\big|$$ 
 by  solving the transport equations with initial condition at the diagonal.     \begin{figure}
 \includegraphics{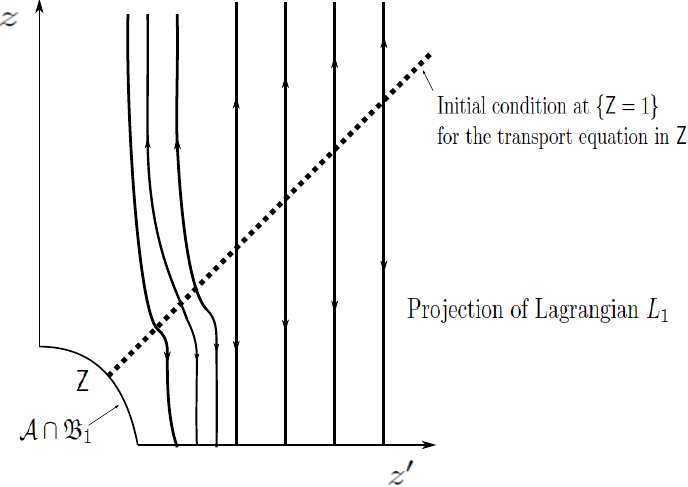}
 \caption{ Lagrangian $\Gamma_1$ associated with phase function $\varphi_{\text{in}}$. Note that the integral curves never hit the left boundary $\mathcal{L}$ \label{faceAL1}}
  \end{figure}
       Note that the integral curves never hit the left boundary $\mathcal{L}$, see figure (\ref{faceAL1}). In the region close to $\mathcal{A} \cap \mathfrak{F}_2$ we have $z < z'$  so $\varphi_{\text{in}}$ takes the form:
  $$\varphi_{\text{in}} = \tfrac{2}{3} h^{-1}\left[(z'+1)^{3/2} - (z+1)^{3/2}\right]$$
 
   \item  Step 3b: Next, we correct the error on $\mathfrak{F}_2$ by solving inhomogeneous Bessel equation.  Although this introduces error on $\mathcal{A}$, we solve the inhomogeneous Bessel for specific solutions that does not undo what we have achieved on $\mathcal{A}$, ie a solution that satisfies the boundary asymptotic at $\mathcal{R}\cap \mathfrak{F}_2$ (\emph{dependent on $\lambda$}) and certain prescribed asymptotic behavior at infinity. As a result, the singularities of the solution lie on a different Lagrangian which we denote by $\Gamma_2$, and not the incoming Lagrangian $\Gamma_1$. The phase function associated to $\Gamma_2$ is $-\varphi_{\text{out}}$ where
  $$\varphi_{\text{out}} =  \tfrac{2}{3} h^{-1} \left[(z'+1)^{3/2} + (z+1)^{3/2}-2\right]  $$

\item Step 3c: In the third step, we correct the resulting error from the above step on the semiclassical face $\mathcal{A}$ but this time along the reflected lagrangian $\Gamma_2$ in phase space over the corner $\mathcal{A} \cap \mathfrak{F}_2$. The initial condition is now at $\mathcal{R}\cap \mathcal{A}$, see figure \ref{faceAL2}.
  \begin{figure}
  \centering
  \includegraphics{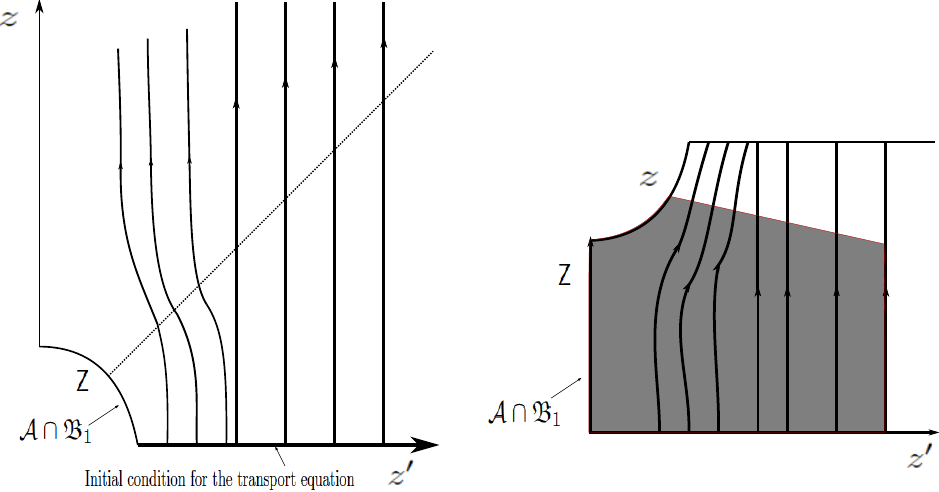}
  \caption{Figure showing lagrangian $\Gamma_2$ associated with face function $\varphi_{\text{out}}$; the shaded region represents $\{z<1, z' <1\}$}
  \label{faceAL2}
  \end{figure}

\end{itemize}

\item \textbf{Step 4:} We remove error at $\mathcal{F}$. The normal operator on $\mathcal{F}$ is : 
    $$ N_{\mathcal{F}}(\QQ) =  h^2 \big[  (\mathsf{Z}\partial_{\mathsf{Z}})^2 + (\lambda - \dfrac{n^2}{4})\big]$$
On $\mathcal{F}$, we use Mellin transform to construct solution with wanted asymptotic behavior at $\mathcal{R}\cap \mathcal{F}$ and $\mathcal{L} \cap \mathcal{F}$.

\item \textbf{Step 5:} By remark (\ref{prepForB1}), we are now ready to remove error on $\mathfrak{B}_1$ and $\tilde{\mathfrak{B}}_1$.
 $$N_{\mathfrak{B}_1}(\QQ) = (\tilde{t} \partial_{\tilde{t}})^2 +   (\lambda - \dfrac{n^2}{4})  + \tilde{t}^2 :\  \textbf{Bessel type ODE}$$
 $$N_{\tilde{\mathfrak{B}}_1}(\QQ) = (\tilde{t} \partial_{\tilde{t}})^2 +   (\lambda - \dfrac{n^2}{4})  - \tilde{t}^2 :\  \textbf{Bessel type ODE}$$
 $\mathfrak{B}_1$ and $\tilde{\mathfrak{B}}_1$ are blown-up versions of the product structure  $\RR^+_{\tilde{t}} \times \RR^+_{\tilde{\nu}'}$ shown on the right of figure \ref{ffb1}.

\item \textbf{Step 6}: At the $Z_0 > 0$ end, we remove error at at $\mathfrak{F}_1$ that is introduced there from the process of error removal at $\mathfrak{B}_1$ in the previous step.
\item \textbf{Step 7:} Finally, we remove error at $\mathcal{R}$.
 
\end{enumerate}

 In the following subsections, we expand in details the various steps of the construction.
\subsection{Step 1 : Removing error at the lifted diagonal}\

  In the first step we construct $U_0 \in \Psi^{-2}_{0,h}$ such that $ E_0 = \QQ(h,\lambda)  U_0(h,\lambda) - \delta(z-z')$ has no singularity at the lifted diagonal $\Diag_{b,h}$ with $U_0$ supported in a neighborhood of $\Diag_{b,h}$ in $\tilde{\mathfrak{M}}_{b,h}$ that only intersects the boundaries of $\tilde{\mathfrak{M}}_{b,h}$ at $\mathfrak{B}_1, \mathfrak{B}_2, \tilde{\mathfrak{B}}_1, \tilde{\mathfrak{B}}_2$ and $\mathcal{F}$.
 \begin{equation}
\QQ(h,\lambda) = h^2(z \partial_z)^2   + h^2(\lambda-\dfrac{n^2}{4}) + z^3+ z^2  
 \end{equation}
This is elliptic in the normal bundle of the interior of the lift $\Diag_{b,h}$ (of the diagonal) to $\tilde{\mathfrak{M}}_{b,h}$. We would like to carry this elliptic construction uniformly up to the b-front face $\mathcal{F}$, the blown-up front face $\mathfrak{B}_1, \mathfrak{B}_2, \tilde{\mathfrak{B}}_1$ and $\tilde{\mathfrak{B}}_2$. The diagonal $\Diag_{b,h}$ meets these five boundary faces transversally, and hence can be extended to a neighborhood of $\mathcal{F}, \mathfrak{B}_1, \mathfrak{B}_2, \tilde{\mathfrak{B}}_1$ and $\tilde{\mathfrak{B}}_2$ in the double space of $\tilde{\mathfrak{M}}_{b,h}$. We want to extend $\QQ_{b,h}$ to this neighborhood (in the double space) and apply the standard theory of \emph{conormal distributions to an embedded submanifold without boundary} (which in this case is the extension of the diagonal $\Diag_{b,h}$). In order to do this, we need to show that $\QQ$ is transversally elliptic or is at least a multiple of a transversally elliptic operator down to these five boundary faces. 
\begin{itemize}
\item \textbf{b-front face} $\mathcal{F}$ : Denote the coordinate on the left factor of $X$ by $z$ while the coordinate on the right factor will be $z'$ . Define the projective coordinates
$$\rho = z', \mathsf{Z} = \dfrac{z}{z'}, h $$
which hold away from the left face. The b - front face is given by $\mathcal{F} = \{\rho = 0\}$ and the lift of the diagonal is $\Diag_{b,h} =\{\mathsf{Z}=1 \}$. Since $z\partial_z $ lifts to $ \mathsf{Z}\partial_{\mathsf{Z}}$, the lift of $\QQ(h,\lambda)$ under the b blow-down map $\beta_b$ is equal to 
\begin{equation}
\QQ_b (h,\lambda) = \beta^* \QQ(h,\lambda) = h^2 (\mathsf{Z}\partial_{\mathsf{Z}})^2 + h^2 (\lambda-\dfrac{n^2}{4}) + (\mathsf{Z}\rho)^3+ (\mathsf{Z}\rho)^2
\end{equation}
The operator is transversally elliptic in a neighborhood of $\{\mathsf{Z}=1\}$ away from $h = 0$. The restriction of the lift of $\QQ(h,\lambda)$ to the front face $\mathcal{F} = 0 = \{\rho =0\}$ is given by 
$$ N_{\mathcal{F}} \left( \QQ_b(h,\lambda)\right) = h^2 \left[ (\mathsf{Z}\partial_{\mathsf{Z}})^2 + (\lambda-\dfrac{n^2}{4})  \right]$$

\item \textbf{The blow-up front face} $\mathfrak{B}_1$ : now we blow-up the intersection of b front face and semiclassical face $\{h=0\}$ in $X^2_b$. In coordinates
$$\tilde{\mu}' = \dfrac{z'}{h} , \mathsf{Z}, h$$
the operator takes the form
 \begin{equation}
 \begin{aligned}
 &h^2( \mathsf{Z} \partial_{\mathsf{Z}})^2 + h^2 (\lambda - \dfrac{n^2}{4}) + (\mathsf{Z} \tilde{\mu}' h)^3 + (\mathsf{Z}\tilde{\mu}'h)^2 \\
 = \ & h^2 \left[( \mathsf{Z} \partial_{\mathsf{Z}})^2 + (\lambda - \dfrac{n^2}{4}) + h (\mathsf{Z}\tilde{\mu}')^3 + (\mathsf{Z} \tilde{\mu}')^2  \right]
 \end{aligned}
 \end{equation}
The operator in the bracket restricts to $\mathfrak{B}_1$ face (given by by $\{h=0\}$) is of the form:
$$ \left[( \mathsf{Z} \partial_{\mathsf{Z}})^2 + (\lambda - \dfrac{n^2}{4}) + (\mathsf{Z} \tilde{\mu}')^2  \right]$$
Hence $\QQ_{b,h}$ is a factor of an operator which is elliptic in a neighborhood of $\{\mathsf{Z}=1\}$ down to $h=0$.
 
\item \textbf{The blown-up face} $\mathfrak{B}_2$  : the front face comes from  blowing up the intersection of the lifted $\Diag_{b,\mathfrak{B}} $ and semiclassical face $\{\tilde{\nu}' = 0\}$. Define the projective coordinates by
$$ \mathsf{Z}_h = \dfrac{\mathsf{Z} - 1}{\tilde{\nu}'}, \tilde{\nu}' , z'$$
The lift of $\QQ_{b,h}$ to $\mathfrak{M}_1$ in a neighborhood of $\mathfrak{B}_2$:
\begin{equation}
\begin{aligned}
 \rho^2 \left[\left((1 + \tilde{\nu}' \mathsf{Z}_h)\partial_{\mathsf{Z}_h} \right)^2+ (\tilde{\nu}')^2 (\lambda-\dfrac{n^2}{4}) + (\mathsf{Z}_h \tilde{\nu}' + 1)^3\rho+ (\mathsf{Z}_h \tilde{\nu}' +1)^2 \right]
 \end{aligned}
 \end{equation}
This operator when restricted to $\mathfrak{B}_2$ (defined by $\{\tilde{\nu}' = 0\} $) gives the normal operator 
$$\mathcal{N}_{\mathfrak{B}_2}(\QQ_h)=  \rho^2 \left[ \partial_{\mathsf{Z}_h}^2+\rho+1 \right]$$
which is transversally elliptic down to $\mathfrak{B}_2$ near $\tilde{\nu}' = 0$.

\item \textbf{The blow-up front face} $\tilde{\mathfrak{B}}_1$ : In the neighborhood of this face we use coordinates
$$\tilde{\mu}' = \dfrac{z'}{h} , \mathsf{Z}, h$$
in which, the operator takes the form
 \begin{equation}
 h^2 \left[( \mathsf{Z} \partial_{\mathsf{Z}})^2 + (\lambda - \dfrac{n^2}{4}) + h (\mathsf{Z}\tilde{\mu}')^3 - (\mathsf{Z} \tilde{\mu}')^2  \right]
 \end{equation}
The operator in the bracket restricts to $\tilde{\mathfrak{B}}_1$ face, which is  given by by $\{h=0\}$, is :
$$ \left[( \mathsf{Z} \partial_{\mathsf{Z}})^2 + (\lambda - \dfrac{n^2}{4}) - (\mathsf{Z} \tilde{\mu}')^2  \right]$$
Hence $\QQ_{b,h}$ is a factor of an operator which is elliptic in a neighborhood of $\{\mathsf{Z}=1\}$ down to $h=0$.

\item \textbf{The blown-up face} $\tilde{\mathfrak{B}}_2$: The front face comes from the blow-up of the intersection of the lifted $\Diag_{b,\mathfrak{B}} $ and semiclassical face $\tilde{\mathcal{A}}$. We use the projective coordinates: $ \mathsf{Z}_h = \tfrac{\mathsf{Z} - 1}{\tilde{\nu}'}, \tilde{\nu}' , z'$, in which the lift has the form:
\begin{equation}
\begin{aligned}
 \QQ_h &= \beta_h^* \QQ (h,\lambda)  \\
 &= \rho^2 \left[\left((1 + \tilde{\nu}' \mathsf{Z}_h)\partial_{\mathsf{Z}_h} \right)^2+ (\tilde{\nu}')^2 (\lambda-\dfrac{n^2}{4}) + (\mathsf{Z}_h \tilde{\nu}' + 1)^3\rho- (\mathsf{Z}_h \tilde{\nu}' +1)^2 \right]
 \end{aligned}
 \end{equation}
 $\beta_h^* \QQ (h,\lambda)$ restricted to $\mathfrak{B}_2$ ( defined by $\{\tilde{\nu}' = 0\} $) then gives the normal operator 
$$\mathcal{N}_{\tilde{\mathfrak{B}}_2}(\QQ_h)=  \rho^2 \left[ \partial_{\mathsf{Z}_h}^2+\rho-1 \right]$$
which is transversally elliptic down to $\tilde{\mathfrak{B}}_2$  near $\tilde{\nu}' = 0$.

\end{itemize}
 
\subsubsection{\textbf{Properties of newly constructed function and the resulting error}}
  From the verification above, one can find distribution $U_0$ with the following properties
  \begin{itemize}
  \item The support of $U_0$ is close to $\Diag_{b,h}$ and which intersects only the boundary faces $\mathcal{F},\mathfrak{B}_1, \mathfrak{B}_2, \tilde{\mathfrak{B}}_1$ and $\tilde{\mathfrak{B}}_2$ 
  \item $U_0$ is cornormal of degree $-2$ to $\Diag_{b,h}$ 
  \item $\QQ_{b,h} U_0 - \Id = E_0$, where $\supp E_0$ is close to $\Diag_{b,h}$ and intersects only the boundary faces $\mathcal{F},\mathfrak{B}_1$ and $\mathfrak{B}_2$. Moreover, $E_0$ is smooth across $\Diag_{b,h}$ and down to $\mathcal{F}$  (\emph{with values in conormal distribution to the lifted diagonal}).
  \end{itemize}
  $$ U_0 \in \rho_{\mathfrak{B}_1}^{-1} \rho_{\tilde{\mathfrak{B}}_1}^{-1} \rho_{\tilde{\mathfrak{B}}_2}^{-1}\Psi^{-2}_{b,h}$$
  $$E_0 \in \rho_{\mathfrak{B}_1}^0 \rho_{\mathfrak{B}_2}^{-1} \rho_{\tilde{\mathfrak{B}}_1}^0 \rho_{\tilde{\mathfrak{B}}_2}^{-1} \mathcal{C}^{\infty}(\tilde{\mathfrak{M}}_{b,h}).$$

\subsection{Step 2 : remove the error at the semiclassical front face $\mathfrak{B}_2$ and $\tilde{\mathfrak{B}}_2$.}\

 We will find a distribution $U_1, U_{1,\tilde{\mathfrak{B}}_2}$ such that $U_1$ is supported near $\mathfrak{B}_2$, and $U_{1,\tilde{\mathfrak{B}}_2}$ is supported near $\tilde{\mathfrak{B}}_2$ such that 
 $$  \QQ_h (U_0 + U_1 + U_{1,\tilde{\mathfrak{B}}_2} ) = E_1$$ where 
 $$U_1 = e^{-i\varphi_{\text{in}}} \tilde{U}_1 ,\ \  \tilde{U}_1 \in \rho_{\mathfrak{B}_1}^{\gamma_{\mathfrak{B}_1}} \rho_{\mathfrak{B}_2}^{-1}\rho_{\mathcal{A}}^{\gamma_{\mathcal{A}} } \mathcal{C}^{\infty}(\tilde{\mathfrak{M}}_{b,h});\ \  U_{1,\tilde{\mathfrak{B}}_2} \in \rho_{\tilde{\mathfrak{B}}_1}^{\gamma_{\tilde{\mathfrak{B}}_1}} \rho_{\tilde{\mathfrak{B}}_2}^{-1}\rho_{\tilde{\mathcal{A}}}^{\infty}\mathcal{C}^{\infty}(\tilde{\mathfrak{M}}_{b,h})$$
 $$E_1 =   e^{-i\varphi_{\text{in}}}  F_1 + F_{1,\tilde{\mathfrak{B}}_2}$$
 $$F_1 \in \rho_{\mathfrak{B}_1}^{\gamma_{\mathfrak{B}_1}+2} \rho_{\mathfrak{B}_2}^{\infty} \rho_{\mathcal{A}}^{\gamma_{\mathcal{A}}+1 } \mathcal{C}^{\infty}(\tilde{\mathfrak{M}}_{b,h}) ;\ \ F_{1,\tilde{\mathfrak{B}}_2} \in \rho_{\tilde{\mathfrak{B}}_1}^{\gamma_{\tilde{\mathfrak{B}}_1}+2} \rho_{\tilde{\mathfrak{B}}_2}^{\infty} \rho_{\tilde{\mathcal{A}}}^{\infty} \mathcal{C}^{\infty}(\tilde{\mathfrak{M}}_{b,h}) $$
 $$ \gamma_{\tilde{\mathfrak{B}}_1}= \gamma_{\mathfrak{B}_1} = -2 ; \ \gamma_{\mathcal{A}} = -1$$
  with $\beta_h^* K_{F_1}$ supported away from $\mathcal{L}$ and $\mathcal{R}$ and close to $\mathfrak{B}_2$, whereas  $F_{1,\tilde{\mathfrak{B}}_2}$ is supported away from $\mathfrak{B}_2, \mathcal{L}$ and $\mathcal{R}$. 
   
\subsubsection{Preparation ingredients for $\mathfrak{B}_2$}\

 \textbf{Fiber structure near the front face} $\mathfrak{B}_2$\nl
  In the current neighborhood we are away from $\mathcal{R}$ and $\mathcal{L}$. The coordinates in use
    $$\mathsf{Z}_h =\dfrac{\mathsf{Z}-1}{\tilde{\nu}'} , \tilde{\nu}', z'$$
    $$\rho_{\mathcal{A}} = \mathsf{Z}_h^{-1} ;\  \rho_{\mathfrak{B}_2} = \tilde{\nu}' $$
    
  \textbf{Expansion of a smooth function around} $\mathfrak{B}_2$: If $v \in \mathcal{C}^{\infty}(\mathfrak{M}_{b,h})$ expand $v$ in Taylor series around $\mathcal{S}$ up to order $N$
  \begin{equation}\label{smoothfnc}
  v = \sum_{k\leq N} \rho^k_{\mathfrak{B}_2} v_k' + v' , v_k' \in \mathcal{C}^{\infty}(\mathsf{Z}_h,z) , v' \in \rho_{\mathfrak{B}_2}^{N+1} \mathcal{C}^{\infty}(\mathfrak{M}_{b,h})
  \end{equation}
 
  \textbf{Corresponding change to the operator in order to act on the expression of the form (\ref{smoothfnc})} : We have $\QQ$ lifts 
\begin{align*}
 \QQ_h &= \beta_h^* \QQ (h,\lambda)  = \rho^2 \left[\left((1 + \tilde{\nu}' \mathsf{Z}_h)\partial_{\mathsf{Z}_h} \right)^2+ (\tilde{\nu}')^2 (\lambda-\dfrac{n^2}{4}) + (\mathsf{Z}_h \tilde{\nu}' + 1)^3\rho+ (\mathsf{Z}_h \tilde{\nu}' +1)^2 \right]\\
 &= \mathcal{N}_{\mathfrak{B}_2}(\QQ_h)
 + \rho^2 \Big[  2\tilde{\nu}'\mathsf{Z}_h \partial_{\mathsf{Z}_h}^2 + (\tilde{\nu}')^2 \mathsf{Z}_h^2 \partial_{\mathsf{Z}_h}^2
 + \tilde{\nu}'(1 + \tilde{\nu}'\mathsf{Z}_h) \partial_{\mathsf{Z}_h}  + (\tilde{\nu}')^2 (\lambda-\dfrac{n^2}{4})   \\
 &+(3\mathsf{Z}_h \tilde{\nu}' + 3 \mathsf{Z}_h^2(\tilde{\nu}')^2 + \mathsf{Z}_h^3(\tilde{\nu}')^3) \rho + (2\mathsf{Z}_h \tilde{\nu}' + \mathsf{Z}_h^2 (\tilde{\nu}')^2)\Big]
 \end{align*}
 where 
 $$\mathcal{N}_{\mathfrak{B}_2}(\QQ_h)=  \rho^2 \left[ \partial_{\mathsf{Z}_h}^2+\rho+1 \right]$$
   
\textbf{The resolvent for} $\partial_{\mathsf{Z}_h}^2+\rho+1$ \nl
 Since the semiclassical normal operator is the Euclidean laplacian on $\RR^{m+1}$ ($m=0$ \emph{in our current case}) with a family of spectral parameter $1 + \rho$ (in the case there is no $\omega$) uniformly down to $\rho =0$ . 
 $$N_{\mathfrak{B}_2}(\QQ_h) = \partial_{\mathsf{Z}_h}^2 + \rho+ 1$$
 The asymptotic expansion of the outgoing solution for $N(\QQ_h) u = f$ where $f$ is smooth of compact support 
 $$ u\sim e^{-i(\rho + 1)^{1/2} \lvert \mathsf{Z}_h\rvert}  c_0 + O(\lvert \mathsf{Z}_h\rvert^{\tfrac{1}{2}})$$
 
  In addition,  for $g$ smooth in $\mathsf{Z}_h$ and $z'$, $g$ of compact support in $\mathsf{Z}_h$ :
  $$\QQ(h,\lambda)(\tilde{\nu}')^kg = (\tilde{\nu}')^k N_{\mathfrak{B}_2}(\QQ_h) g + \tilde{\nu}^{k+1} \tilde{\QQ}_h g$$ 
  where $\tilde{\QQ}_h$ is a differential operator in $\partial_{Z_h}$ with polynomial coefficients of degree $\leq  3$ in $Z_h$. Denote by $R_0$ the resolvent to $ N_{\mathfrak{B}_2}(\QQ_h)$. We can apply Lemma 5.2 in the paper by S.B - M - V, to apply $R_0$ to an expression of the form $L_3 R_0 g$, where $L_3$ is a differential operator with polynomial coefficients of degree $\leq 3$. In our specific case with the dimension of the space is $1$ and the coefficient of the differential operators can be up to degree $2$, we have
 \begin{lemma}\label{lemma5.2}[Lemma 5.2 in S.B-M-V]
 Suppose $M_j, j = 1, \ldots, N$ are different operators with polynomial coefficients of degree $3$ on $\RR$. Then
 $$R_0 M_1 R_0 M_2 \ldots R_0 M_N R_0 : \mathcal{C}^{\infty}_c(\RR) \rightarrow \mathcal{C}^{-\infty}(\RR)$$
 has an analytic extension from $\Im( \sigma) < 0 $ to $\Re \sigma > 0, \Im \sigma \in \RR$ so with $\sigma = \sqrt{\rho+1}$ we have
 $$R_0 M_1 R_0 M_2 \ldots R_0 M_N R_0 : \mathcal{C}^{\infty}_c(\RR_{\mathsf{x}}) \rightarrow
 e^{-i\sqrt{\rho+1} \lvert \mathsf{x}\rvert} \lvert \mathsf{x}\rvert^{3N} $$
 
 \end{lemma}

\subsubsection{Construction for $\mathfrak{B}_2$}\

Now we can start the process of eliminating the error on the semiclassical front face,
    \begin{itemize}
    \item We expand $E_0$ as in (\ref{smoothfnc}) 
    $$E_0 \stackrel{\text{taylor expand}}{\sim} (\tilde{\nu}')^{-1} \sum (\tilde{\nu}')^k \rho^{-2} E_{0,k}; \ \ E_{0,k} \in \mathcal{C}^{\infty}(\mathsf{Z}_h, z)$$
    note that each $E_{0,k}$ is compactly supported in $\mathsf{Z}_h$
 \item Firstly we construct
 $$\tilde{U}_{1,0} = -\rho^{-4}R_0 E_{0,0}$$
 From the Lemma \ref{lemma5.2}, we have $\tilde{U}_{1,0}$ is of the asymptotic form : 
 $$e^{-i (\rho + 1)^{1/2} \lvert \mathsf{Z}_h\rvert}  c_0 + O(\lvert \mathsf{Z}_h\rvert^{\tfrac{1}{2}})$$
 The exponent part $e^{-i (\rho+ 1)^{1/2} \lvert \mathsf{Z}_h\rvert }$ is the restriction to $h = 0 $ of 
 $$\exp\Big(-i\dfrac{2}{3}(\tilde{\nu}' z')^{-1}\big|(\tilde{\nu}' \mathsf{Z}_hz' + 1+z' )^{3/2} - (z'+1)^{3/2}\big| \Big)$$ 
 Denote by $U_{1,0} = - e^{-i\varphi_{\text{in}}+i (\rho + 1)^{1/2} \lvert \mathsf{Z}_h\rvert} R_0 E_{0,0}$ where 
 $$\varphi_{\text{in}} = \dfrac{2}{3}(\tilde{\nu}' z')^{-1}\Big|(\tilde{\nu}' \mathsf{Z}_hz' + 1+z' )^{3/2} - (z' +1)^{3/2} \Big|$$
We consider  the difference between the phase function $\varphi_{\text{in}}$ and the pulled back of the $e^{-i (\rho+ 1)^{1/2} \lvert \mathsf{Z}_h\rvert }$ using the product structure near the semiclassical front face $\mathfrak{B}_2$.

 The prep computations \emph{included at the end of the subsection} gives: 
\begin{align*}
\rho^{4}\Big[\QQ_h  (\tilde{\nu}')^{-1} U_{1,0}  + \rho^{-4} E_{0,0} \Big] & = -N_{\mathfrak{B}_2}(\QQ) e^{-i\varphi_{\text{in}} + i (\rho + 1)^{1/2} \lvert \mathsf{Z}_h\rvert } R_0 E_{0,0} - \tilde{\nu}'  \tilde{\QQ}_{b,h} e^{-i\varphi_{\text{in}}+i (\rho + 1)^{1/2} \lvert\mathsf{Z}_h\rvert} R_0 E_{0,0}\\
&=\big[ 1- e^{-i\varphi_{\text{in}}+i (\rho+ 1)^{1/2} \lvert \mathsf{Z}_h\rvert} \big]E_{0,0} - 2 \left(\partial_{\mathsf{Z}_h} e^{-i\varphi_{\text{in}}+i(\rho + 1)^{1/2} \lvert \mathsf{Z}_h\rvert}\right)\partial_{\mathsf{Z}_h} R_0 E_{0,0}\\
& - \left( \partial_{\mathsf{Z}_h}^2e^{-i\varphi_{\text{in}} +i (\rho + 1)^{1/2} \lvert \mathsf{Z}_h\rvert} \right) R_0 E_0 - \tilde{\nu}'  \tilde{\QQ}_{b,h} \rho^{-2}e^{-i\varphi_{\text{in}}+i (\rho + 1)^{1/2} \lvert \mathsf{Z}_h\rvert} R_0 E_{0,0}
\end{align*}
 which means that we have replaced the error $z^0(\tilde{\nu}')^{-1} E_{0,0}$ by an error $E_{0,1}$ of the form $z^0\mathsf{O}((\tilde{\nu}')^0)$. 
\item Continue in this way to construct
$$\tilde{U}_{1,1} =  - R_0 \left(\rho^{-4} E_{0,1} +  \QQ_h  (\tilde{\nu}')^{-1} G_{1,0}  + E_{0,0}\right) $$
 Then define $U_{1,1} =  e^{-i\varphi_{\text{in}}+i (\rho + 1)^{1/2} \lvert \mathsf{Z}_h\rvert}  \tilde{U}_{1,1}$ so that  $\tilde{U}_{1,1} $ is equal to the restriction of $U_{1,1}$ on $\mathfrak{B}_2$.
And in the same manner we construct $U_{1,k}$. Note the procedure of applying $R_0$ to the new resulting error  is assured by Lemma \ref{lemma5.2} in (\cite{M-SB-AV}). \item Thus we have constructed :
$$U_{1,k} = e^{-i\varphi_{\text{in}}} \lvert \mathsf{Z}_h\rvert^{-1+3k} U'_{1,k} , U'_{1,k} \in \mathcal{C}^{\infty}(\ldots)$$
with $\varphi_{\text{in}} = \dfrac{2}{3}(\tilde{\nu}' z')^{-1} \Big|(\tilde{\nu}' \mathsf{Z}_h\rho + 1 + \rho)^{3/2} - (\rho + 1)^{3/2} \Big|$
$$\text{and} :\ \ \QQ_h \sum_{k=0}^{\infty} (\tilde{\nu}')^{k-1} U_{1,k} \sim \sum_{k=0}^{\infty} (\tilde{\nu}')^{k-1} E_{0,k}$$
where the statement means the equality of coefficients. 

\item Borel summation to obtain $U_1$
$$U_1 \stackrel{\text{borel sum}}{\sim} (\tilde{\nu}')^{-1} \sum_j (\tilde{\nu}')^j U_{1,j}$$
  to get $U_1$, and we can assume $U_1$ to be supported in a small neighborhood of $\rho_{\mathfrak{B}_2}$  is small. 
  \end{itemize}
  
\textbf{PREP COMPUTATIONS}
\begin{align*}
&-\rho^2 N_{\mathcal{S}} (\QQ_h) U_{1,0}\\
&= N_{\mathcal{S}} (\QQ_h) e^{-i\varphi_{\text{in}}+i (\rho + 1)^{1/2} \lvert \mathsf{Z}_h\rvert } R_0 E_{0,0}\\
 &=  e^{-i\varphi_{\text{in}}+i (\rho + 1)^{1/2}\lvert \mathsf{Z}_h\rvert}  N_{\mathcal{S}} (\QQ_h)R_0 E_{0,0} + \big[N_{\mathcal{S}} (\QQ_h) , e^{-i\varphi_{\text{in}}+i (\rho + 1)^{1/2} \lvert \mathsf{Z}_h\rvert}  \big]R_0 E_{0,0}\\
&= e^{-i\varphi_{\text{in}}+i (\rho + 1)^{1/2} \lvert \mathsf{Z}_h\rvert}E_0 +  [\partial_{\mathsf{Z}_h}^2 , e^{-i\varphi_{\text{in}}+i (\rho + 1)^{1/2} \lvert \mathsf{Z}_h\rvert}\big] R_0 E_{0,0}\\
&= e^{-i\varphi_{\text{in}}+i (\rho + 1)^{1/2} \lvert \mathsf{Z}_h\rvert} E_0 + 2 \left(\partial_{\mathsf{Z}_h} e^{-i\varphi_{\text{in}}+i(\rho + 1)^{1/2} \lvert \mathsf{Z}_h\rvert}\right)\partial_{\mathsf{Z}_h} R_0 E_0 + \left( \partial_{\mathsf{Z}_h}^2e^{-i \varphi_{\text{in}}+i (\rho + 1)^{1/2} \lvert \mathsf{Z}_h\rvert} \right) R_0 E_0
\end{align*}

Away from the singular set of the phase function, compute:
\begin{align*}
\lim_{\tilde{\nu}'\rightarrow 0} (\tilde{\nu}')^{-1}\left[ e^{-i\varphi_{\text{in}}+i (\rho + 1)^{1/2} \lvert \mathsf{Z}_h\rvert} - 1\right]
&= -\lim_{h\rightarrow 0} e^{-i\varphi_{\text{in}}+i (\rho + 1)^{1/2} \lvert \mathsf{Z}_h\rvert} \partial_{\tilde{\nu}'} \varphi_{\text{in}} \\
&\stackrel{\text{by several application of l'hopital's rule}}{=} -\dfrac{1}{2}(1 + z')^{-1/2} (\mathsf{Z}_h \rho)^2
\end{align*}

\begin{align*}
\partial_{\mathsf{Z}_h} e^{-i\varphi_{\text{in}} +i(\rho + 1)^{1/2} \lvert \mathsf{Z}_h\rvert}
&= - \sgn(\mathsf{Z}_h) e^{-i\varphi_{\text{in}} +i(\rho + 1)^{1/2} \lvert \mathsf{Z}_h\rvert} \left[i(\tilde{\nu}'\mathsf{Z}_h\rho + 1+\rho)^{1/2}\rho -i(\rho + 1)^{1/2}\right] \\
\partial_{\mathsf{Z}_h}^2 e^{-i \varphi_{\text{in}} + i(\rho + 1)^{1/2} \lvert \mathsf{Z}_h\rvert }
&= -\sgn(\mathsf{Z}_h) e^{i\varphi_{\text{in}} -i(\rho + 1)^{1/2} \lvert \mathsf{Z}_h\rvert} \left[i(\tilde{\nu}'\mathsf{Z}_h\rho + 1+\rho)^{1/2}\rho -i(\rho + 1)^{1/2}\right]^2\\
&- \sgn(\mathsf{Z}_h) e^{i\varphi_{\text{in}} -i(\rho + 1)^{1/2} \lvert \mathsf{Z}_h\rvert} i\tfrac{1}{2} (\tilde{\nu}'\mathsf{Z}_h\rho + 1+\rho)^{-1/2}\tilde{\nu}'\rho
\end{align*}

Since 
$$ \lim_{\tilde{\nu}'\rightarrow  0} \dfrac{i(\tilde{\nu}'\mathsf{Z}_h\rho + 1+\rho)^{1/2}\rho -i(\rho + 1)^{1/2}}{\tilde{\nu}'} = i\dfrac{1}{2} (1+\rho)^{-1/2} \mathsf{Z}_h \rho $$  
  we have 
  $$\partial_{\mathsf{Z}_h} e^{-i\varphi_{\text{in}}+i(\rho + 1)^{1/2} \lvert \mathsf{Z}_h\rvert} ,  \partial_{\mathsf{Z}_h}^2e^{-i\varphi_{\text{in}}+i (\rho + 1)^{1/2} \lvert \mathsf{Z}_h\rvert}  = \mathsf{O}(\tilde{\nu}')$$

\subsubsection{Construction for $\tilde{\mathfrak{B}}_2$}\

\textbf{Fiber structure near the front face} $\tilde{\mathfrak{B}}_2$\newline
  In the current neighborhood we are away from $\mathcal{R}$ and $\mathcal{L}$. The coordinates in use
    $$\mathsf{Z}_h =\dfrac{\mathsf{Z}-1}{\tilde{\nu}'} , \tilde{\nu}' = \dfrac{h}{z'}, z'$$
    $$\rho_{\tilde{\mathcal{A}}} = \mathsf{Z}_h^{-1} ;\  \rho_{\tilde{\mathfrak{B}}_2} = \tilde{\nu}' $$
The process is much simpler since $\mathcal{N}_{\tilde{\mathfrak{B}}_2}(\QQ_h)$ is now an elliptic operator
 $$\QQ(h,\lambda)(\tilde{\nu}')^kg = (\tilde{\nu}')^k \mathcal{N}_{\tilde{\mathfrak{B}}_2}(\QQ_h) g + \tilde{\nu}^{k+1} \tilde{\QQ}_{-,h} g$$ 
 $$\mathcal{N}_{\tilde{\mathfrak{B}}_2}(\QQ_h)=  \rho^2 \left[ \partial_{\mathsf{Z}_h}^2+\rho-1 \right]$$
 By Fourier transform we can construct the solution $u \in \mathcal{S}$ (rapidly decay at infinity) of $\mathcal{N}_{\tilde{\mathfrak{B}}_2}(\QQ_h) u = f$ where $f$ is smooth of compact support. Denote by $R_{0,\tilde{\mathfrak{B}}_2}$ the resolvent that produces such rapidly decaying solutions.

We work with the portion of $E_0$ close to $\tilde{\mathfrak{B}}_2$, denoted by $E_{0,\tilde{\mathfrak{B}}_2}$. We expand $E_{0,\tilde{\mathfrak{B}}_2}$ at $\tilde{\mathfrak{B}}_2$: 
    $$E_{0,\tilde{\mathfrak{B}}_2} \stackrel{\text{taylor expand}}{\sim} (\tilde{\nu}')^{-1} \sum (\tilde{\nu}')^k \rho^{-2} E_{0,\tilde{\mathfrak{B}}_2;k}; \ \ E_{0,\tilde{\mathfrak{B}}_2;k} \in \mathcal{C}^{\infty}(\mathsf{Z}_h, z)$$
    note that each $E_{0,\tilde{\mathfrak{B}}_2;k}$ is compactly supported in $\mathsf{Z}_h$. Let 
 $$U_{1,\tilde{\mathfrak{B}}_2; k} = -\rho^{-4}R_{0,\tilde{\mathfrak{B}}_2}\big( E_{0,\tilde{\mathfrak{B}}_2;k} +\tilde{Q}_{-,h} U_{1,\tilde{\mathfrak{B}}_2; k-1}  \big) \in \mathcal{S}(\RR_{\mathsf{Z}_h})$$
 Thus we have 
$$ \QQ_h \sum_{k=0}^{\infty} (\tilde{\nu}')^{k-1} U_{1,\tilde{\mathfrak{B}}_2; k} \sim \sum_{k=0}^{\infty} (\tilde{\nu}')^{k-1} E_{0,\tilde{\mathfrak{B}}_2;k} $$
where the statement means the equality of coefficients. 
 Borel summation to obtain $U_{1,\tilde{\mathfrak{B}}_2}$
$$U_{1,\tilde{\mathfrak{B}}_2} \stackrel{\text{borel sum}}{\sim} (\tilde{\nu}')^{-1} \sum_j (\tilde{\nu}')^j U_{1,\tilde{\mathfrak{B}}_2;j}$$
   we can assume $U_{1,\tilde{\mathfrak{B}}_2}$ to be supported in a small neighborhood of $\rho_{\tilde{\mathfrak{B}}_2}$  is small.

\

\subsection{Step 3 : removing error at $\mathcal{A}$ and $\mathfrak{F}_2$ }\

   Removing error at $\mathcal{A}$ and $\mathfrak{F}_2$ will be carried out in 3 smaller steps. Firstly, on $\mathfrak{M}_{b,h}$ we remove the error at the semiclassical face $\mathcal{A}$ along the Lagrangian $\Gamma_1$ associated to the phase function: 
$$\varphi_{\text{in}} = \tfrac{2}{3} h^{-1} \big| (z'+1)^{3/2} - (z+1)^{3/2}\big|$$ 
 by  solving the transport equations. See the following figure (\ref{faceAL1}).\nl Secondly, we correct the error on $\mathfrak{F}_2$ by solving inhomogeneous Bessels for specific solutions that have the required asymptotic type at the $\mathcal{R}\cap \mathfrak{F}_2$ while resulting in an error that does not undo what we have achieved on $\mathcal{A}$. Although this introduces error on $\mathcal{A}$, the resulting error will have singularities on a different Lagrangian which we denote by $\Gamma_2$, and not $\Gamma_1$. In order for this to happen, we require that the solution to the inhomogeneous Bessel equation satisfy a fixed type of oscillatory behavior at $\mathcal{A} \cap \mathfrak{F}_2$. In short, before correcting error on $\mathfrak{F}_2$, singularities lie on the lagrangian $\Gamma_1$ associated with the following phase function $-\varphi_{\text{in}}$:
   $$\varphi_{\text{in}} = \tfrac{2}{3} h^{-1} \Big| (z'+1)^{3/2} - (z+1)^{3/2}\Big|$$ 
  In the region close to $\mathcal{A} \cap \mathfrak{F}_2$ we have $z < z'$  so $\varphi_{\text{in}}$ takes the form:
  $$\varphi_{\text{in}} = \tfrac{2}{3} h^{-1}\left[(z'+1)^{3/2} - (z+1)^{3/2}\right]$$
  After we have constructed the solution that satisfies the above stated requirements, the error resulting will have singularities along Lagrangian $\Gamma_2$, which is associated with the following phase function $-\varphi_{\text{out}}$: (\emph{see the remark at the beginning of step 3b for reasons as to why $\varphi_{\text{out}}$ should have the following form})
  $$\varphi_{\text{out}} =  \tfrac{2}{3} h^{-1} \left[(z'+1)^{3/2} + (z+1)^{3/2}-2\right]  $$
  In the third step, we correct the resulting error from the above step on the semiclassical face $\mathcal{A}$ but this time along the reflected lagrangian $\Gamma_2$ in phase space over the corner $\mathcal{A} \cap \mathfrak{F}_2$. See figure \ref{faceAL2}.

\subsubsection{Some prep computations}

  For matter of convenience we write out the form the operator in several coordinates in the neighborhood of $\mathfrak{F}_1$:
  \begin{align*}
   &\text{In}\  \tilde{\nu}' = \tfrac{h}{z'}, z', \mathsf{Z} : &&  (z' \tilde{\nu}')^2 (\mathsf{Z} \partial_{\mathsf{Z}})^2 + (z' \tilde{\nu}')^2  (\lambda - \dfrac{n^2}{4}) + (\mathsf{Z}z')^3 + (\mathsf{Z} z')^2\\
 &  &&= (z')^2 \left[(\tilde{\nu}')^2 (\mathsf{Z} \partial_{\mathsf{Z}})^2 + (\tilde{\nu}')^2  (\lambda - \dfrac{n^2}{4}) + \mathsf{Z}^3 z' + \mathsf{Z}^2 \right]\\
 &\text{In}\  \tilde{t}, \tilde{\nu}' , z' : && (z')^2 \left[(\tilde{\nu}')^2 (\tilde{t} \partial_{\tilde{t}})^2 + (\tilde{\nu}')^2  (\lambda - \dfrac{n^2}{4}) + (\tilde{t} \tilde{\nu}')^3 z' + (\tilde{t} \tilde{\nu}')^2 \right]\\
 &   &&= (z' \tilde{\nu}')^2 \left[ (\tilde{t} \partial_{\tilde{t}})^2 +   (\lambda - \dfrac{n^2}{4}) + \tilde{t}^2  \tilde{\nu}' z' + \tilde{t}^2 \right]
   \end{align*}
In this coordinates, $\mathfrak{B}_1$ can locally be described as $\{ z' =0\}$ where as $ \mathcal{A}$ locally is  $\{\tilde{t}^{-1} =0\}$. On $\mathfrak{B}_1$, the operator in the bracket is a Bessel equation of order $(\tfrac{n^2}{4}-\lambda)^{1/2}$:
$$ (\tilde{t} \partial_{\tilde{t}})^2 +   (\lambda - \dfrac{n^2}{4})  + \tilde{t}^2$$
 The projective coordinates associated to the above blow-up ie close to $\mathfrak{F}_2$:
 $$\tilde{\tilde{\mu}} = \dfrac{\tilde{\nu}}{\mathsf{Z}'}, \mathsf{Z}' , z$$
  Note that away from $\mathcal{R}$ the flow $\tilde{\nu}' = $ constant looks like
$ \tilde{\nu}' = \dfrac{h}{z'} = \dfrac{h}{z\mathsf{Z}'} = \dfrac{\tilde{\nu}}{\mathsf{Z}'} = \tilde{\tilde{\mu}} = \text{constant}$. The forms of the ODE in various coordinates in this neighborhood:
  \begin{align*}
  &\text{In}\ \mathsf{Z}' , z, h : && h^2(\mathsf{Z}'\partial_{\tilde{\mathsf{Z}'}} - z\partial_z)^2 + h^2( \lambda - \dfrac{n^2}{4}) + z^3 + z^2\\
&\text{In}\ \tilde{\nu} = \dfrac{h}{z}, \mathsf{Z}', z : &&
   (z\tilde{\nu})^2(\mathsf{Z}'\partial_{\tilde{\mathsf{Z}'}} +\tilde{\nu}\partial_{\tilde{\nu}} - z\partial_z)^2 + (z\tilde{\nu})^2( \lambda - \dfrac{n^2}{4}) + z^3 + z^2 \\
&\text{In}\  \tilde{\tilde{\mu}} = \dfrac{\tilde{\nu}}{\mathsf{Z}'}, \mathsf{Z}' , z : &&
  (z\mathsf{Z}' \tilde{\tilde{\mu}})^2(\mathsf{Z}'\partial_{\mathsf{Z}'} - z\partial_z)^2 + (z\mathsf{Z}' \tilde{\tilde{\mu}})^2( \lambda - \dfrac{n^2}{4}) + z^3 + z^2\\
&\text{In}\ \tilde{\tilde{\nu}}  = \dfrac{\mathsf{Z}'}{\tilde{\nu}}, \tilde{\nu}, z : &&
 (z\tilde{\nu})^2(\tilde{\nu}\partial_{\tilde{\nu}} - z\partial_z)^2 + (z\tilde{\nu})^2( \lambda - \dfrac{n^2}{4}) + z^3 + z^2
 \end{align*}

\subsubsection{Step 3a}\
  We can assume that error has error close to $\mathcal{A}$ ie has supported in $\{\tilde{\tilde{\nu}}' = \tilde{t}^{-1} < \epsilon \}$. We use the product structure close to $\mathcal{A}$
  $$\RR^+_{\rho_{\mathcal{A}}} \times \mathcal{A} $$
  $$\rho_{\mathcal{A}} = \tilde{\tilde{\nu}}' \langle \mathsf{Z}\rangle$$

  On $\mathfrak{M}_{b,h}$ the coordinates close to $\mathcal{A}$ and closer to the corner $\mathfrak{F}_2 \cap \mathcal{A}$ :
  $$\tilde{\tilde{\nu}}' = \dfrac{\tilde{\nu}'}{\mathsf{Z}}  ( = \tilde{t}^{-1}), \mathsf{Z} , z'$$
We compute the form of the operator $\QQ_b$ when conjugated by 
$$\exp\big(-i\tfrac{2}{3} (\tilde{\tilde{\nu}}' \mathsf{Z} z')^{-1}\big|(\mathsf{Z}z'+1)^{3/2} - (z'+1)^{3/2} \big|\big)$$
\begin{itemize}
\item on $\mathsf{Z} < 1$
\begin{align*}
 &(\tilde{\tilde{\nu}}' \mathsf{Z} z' )^2 \left(\mathsf{Z} \partial_{\mathsf{Z}} - \tilde{\tilde{\nu}} \partial_{\tilde{\tilde{\nu}}}\right)^2   + (\tilde{\tilde{\nu}}' \mathsf{Z} z' )^2 (\lambda-\dfrac{n^2}{4}) - 2 i\tilde{\tilde{\nu}}' \mathsf{Z} z'(\mathsf{Z} z' +1)^{1/2} \mathsf{Z} z'  \left(\mathsf{Z} \partial_{\mathsf{Z}} - \tilde{\tilde{\nu}} \partial_{\tilde{\tilde{\nu}}}\right)\\
& -i\tilde{\tilde{\nu}}' \mathsf{Z} z'  (\mathsf{Z} z' +1)^{1/2} \mathsf{Z}z' 
 -i\tilde{\tilde{\nu}}' \mathsf{Z} z'  \dfrac{1}{2}(\mathsf{Z} z' +1)^{-1/2} (\mathsf{Z} z')^2  \\
 =\ &\tilde{\tilde{\nu}}'( \mathsf{Z} z' )^2 \Big[  \tilde{\tilde{\nu}}'   \left(\mathsf{Z} \partial_{\mathsf{Z}} - \tilde{\tilde{\nu}} \partial_{\tilde{\tilde{\nu}}}\right)^2   + \tilde{\tilde{\nu}}' (\lambda-\dfrac{n^2}{4}) - 2 i(\mathsf{Z} z' +1)^{1/2} \left(\mathsf{Z} \partial_{\mathsf{Z}} - \tilde{\tilde{\nu}} \partial_{\tilde{\tilde{\nu}}}\right) \\
 & -i (\mathsf{Z} z' +1)^{1/2} 
 -i\mathsf{Z} z'  \dfrac{1}{2}(\mathsf{Z} z' +1)^{-1/2} \Big]
 \end{align*}
 \item and on $\mathsf{Z} > 1$
  \begin{align*}
   &\ \tilde{\tilde{\nu}}'( \mathsf{Z} z' )^2 \Big[  \tilde{\tilde{\nu}}'   \left(\mathsf{Z} \partial_{\mathsf{Z}} - \tilde{\tilde{\nu}} \partial_{\tilde{\tilde{\nu}}}\right)^2   + \tilde{\tilde{\nu}}' (\lambda-\dfrac{n^2}{4}) + 2 i(\mathsf{Z} z' +1)^{1/2} \left(\mathsf{Z} \partial_{\mathsf{Z}} - \tilde{\tilde{\nu}} \partial_{\tilde{\tilde{\nu}}}\right) \\
 & +i (\mathsf{Z} z' +1)^{1/2} 
 +i\mathsf{Z} z'  \dfrac{1}{2}(\mathsf{Z} z' +1)^{-1/2} \Big]
 \end{align*}
 \end{itemize}
  To get the transport equations we conjugate the operator by $\exp(-i\varphi_{\text{in}})$. Denote by 
  $$ \QQ_b = \rho_{\mathcal{A}} (z')^2 \dfrac{\mathsf{Z}}{\langle \mathsf{Z}\rangle}\mathsf{Z} \mathbf{P}_{b,h}$$
Using the above computation of the conjugated form of $\QQ_b$ we have:
\begin{equation}
\begin{aligned}
\mathbf{P}_{b,h} \rho_{\mathcal{A}}^k g = & \ \ \mathbf{P}_{b,h} (\tilde{\tilde{\nu}}' \langle\mathsf{Z}\rangle)^k g \\
\stackrel{\mathsf{Z} < 1}{=} & \ \rho_{\mathcal{A}}^{k+1}\langle \mathsf{Z}\rangle^{-1} \Big[  \left( \mathsf{Z}\partial_{\mathsf{Z}}+ k (\tfrac{\mathsf{Z}}{\langle \mathsf{Z}\rangle})^2 - k \right)^2 +  (\lambda - \dfrac{n^2}{4} )  \Big] g \\
&\  + \rho_{\mathcal{A}}^k\Big[   - 2 i(\mathsf{Z} z' +1)^{1/2}  \left( \mathsf{Z}\partial_{\mathsf{Z}}+ k (\tfrac{\mathsf{Z}}{\langle \mathsf{Z}\rangle})^2 - k \right)
- i (\mathsf{Z} z' +1)^{1/2} - i\mathsf{Z} z'  \dfrac{1}{2}(\mathsf{Z} z' +1)^{-1/2}\Big] g
\end{aligned}
\end{equation}

\begin{equation}
\begin{aligned}
\mathbf{P}_{b,h} \rho_{\mathcal{A}}^k g \stackrel{\mathsf{Z} > 1}{=} & \ \rho_{\mathcal{A}}^{k+1}\langle \mathsf{Z}\rangle^{-1} \Big[  \left( \mathsf{Z}\partial_{\mathsf{Z}}+ k (\tfrac{\mathsf{Z}}{\langle \mathsf{Z}\rangle})^2 - k \right)^2 +  (\lambda - \dfrac{n^2}{4} )  \Big] g \\
&\  + \rho_{\mathcal{A}}^k\Big[    2 i(\mathsf{Z} z' +1)^{1/2}  \left( \mathsf{Z}\partial_{\mathsf{Z}}+ k (\tfrac{\mathsf{Z}}{\langle \mathsf{Z}\rangle})^2 - k \right)
+ i (\mathsf{Z} z' +1)^{1/2} + i\mathsf{Z} z'  \dfrac{1}{2}(\mathsf{Z} z' +1)^{-1/2}\Big] g
\end{aligned}
\end{equation}

In coordinates $\tilde{\tilde{\nu}}' = \tfrac{\tilde{\nu}'}{\mathsf{Z}} (= \tilde{t}^{-1}) , \mathsf{Z} , z'$ the semiclassical face $\mathcal{A}$ is defined by $\{\tilde{\tilde{\nu}}' = 0 \}$. $E_1$ can be considered as being of the form
$$E_1 = \exp(- i\varphi_{\text{in}}) F_1 ;\  F_1 \in \rho_{\mathcal{A}}^{\gamma_{\mathcal{A}}+1} \mathcal{C}^{\infty} (\mathfrak{M}_{b,h}); \ \gamma_{\mathcal{A}} = -1 $$
Consider the asymptotic expansion at the boundary face $\tilde{\tilde{\nu}}' = 0$ of the form
$$F_1 \sim \rho_{\mathcal{A}}^{\gamma_{\mathcal{A}}+1} \sum_{j=0}^{\infty} \rho_{\mathcal{A}}^j F_{1,j} , F_{1,j} \in \rho_{\mathfrak{B}_1}^{\gamma_{\mathfrak{B}_1}+2 } \mathcal{C}^{\infty}_{\mathsf{Z},  z'} ; \ \ \gamma_{\mathfrak{B}_1} = -2; \ \gamma_{\mathcal{A}} = -1 $$

The construction of an approximate solution $U_2$ so that 
$$U_2 = \exp(-i\varphi_{\text{in}}) \tilde{U}_2; \ \ U_2 \ \text{is cornormal and polyhomogeneous on}\  \mathcal{A}_2$$
$$\QQ_b(h,\lambda) e^{-i\varphi_{\text{in}}} U_2 - e^{-i\varphi_{\text{in}}} F_1 = e^{-i\varphi_{\text{in}}}R_3 ; \ \ R_3 \in  \rho_{\mathcal{A}}^{\infty} \mathcal{C}^{\infty}_{\mathsf{Z},z'} $$
On the region $\mathsf{Z} < 1$, the integral curves hit the boundary $\mathcal{R} \cap \mathcal{A}$ hence we can solve the regular singular ODE and solve each transport equation for $\tilde{U}_{2,j}$ uniformly down to the $\mathfrak{B}_1 \cap \mathcal{A}$. On the other hand, in the region $\mathsf{Z} > 1$, the integral curves do not hit $\mathcal{L}\cap \mathcal{A}$ and the fiber becomes singular. On this region, to find polyhomogeneous approximate solution we 
 follows exactly the same methods as in the step 3c. \textbf{In fact the transport equations in the two situations have the exactly same form; } so we refer the readers to step 3c for the motivations and details. Here we will still write out the transport equations. After constructing $\tilde{U}_{2,j}$, $\tilde{U}_2$ will then be constructed using Borel summation of these $\tilde{U}_{2,j}$.

\begin{enumerate}
\item \textbf{FIRST TRANSPORT EQUATION}: comes from coefficient on $\rho_{\mathcal{A}}^{ \gamma_{\mathcal{A}}+1}$, where $\gamma_{\mathcal{A}} = -1$. 
We solve for $\tilde{U}_{2,0}$ on $\mathsf{Z} < 1$ as the unique solution satisfying the Cauchy problem :
\begin{equation}\label{fitlag1a}
\begin{cases}
\mathsf{Z} < 1; \  \ - 2 i(\mathsf{Z} z' +1)^{1/2} \mathsf{Z} \dfrac{\mathsf{Z} }{\langle\mathsf{Z}\rangle}(z' )^2\left[  \mathsf{Z}\partial_{\mathsf{Z}}+ (  \gamma_{\mathcal{A}}+0) (\tfrac{\mathsf{Z}}{\langle \mathsf{Z}\rangle})^2 - ( \gamma_{\mathcal{A}}+0) +\dfrac{1}{2}
 + \dfrac{\mathsf{Z} z' }{4(\mathsf{Z} z' +1)}  \right]  \tilde{U}_{2,0} \\
 \ \ \ \ \ \ \ \ =  - F_{1,0}\\
U_{4,0} |_{\mathsf{Z} = 1} = -0 : \ \text{IC}  
\end{cases}
\end{equation}
Since $F_{1,j}$ are of compact support around $\mathsf{Z} = 1$, the ODE is regular singular, $\tilde{U}_{2,0}$ will behave like $\tilde{U}_{2,0} \sim  \gamma_{\mathcal{A}} - \tfrac{1}{2} $ at $\mathsf{Z} = 0$.\

For region $\mathsf{Z} > 1$ we use method of step 3c to obtain $\tilde{U}_{2,0} $  \emph{(see details of approximate solution of the first transport equation in step 3c)}
 \begin{equation}\label{fitlag1b}
\begin{cases}
\mathsf{Z} > 1; \ \  2 i(\mathsf{Z} z' +1)^{1/2} \mathsf{Z} \dfrac{\mathsf{Z} }{\langle\mathsf{Z}\rangle}(z' )^2\left[  \mathsf{Z}\partial_{\mathsf{Z}}+ ( \gamma_{\mathcal{A}}+0) (\tfrac{\mathsf{Z}}{\langle \mathsf{Z}\rangle})^2 - ( \gamma_{\mathcal{A}}+0) +\dfrac{1}{2}
 + \dfrac{\mathsf{Z} z' }{4(\mathsf{Z} z' +1)}  \right]  \tilde{U}_{2,0} \\
 \ \ \ \ \ \ \ \ =  - F_{1,0}\\
\tilde{U}_{2,0} |_{\mathsf{Z} = 1} = -0 : \ \text{IC}  
\end{cases}
\end{equation}
For this region we obtain $\tilde{U}_{2,0}$ polyhomogeneous conormal at the corner of $\mathfrak{B}_1\cap\mathfrak{F}_1$
$$ \tilde{U}_{2,0} \in  \rho_{\mathfrak{F}_2}^{\infty} \rho_{\mathfrak{B}_1}^{\alpha_{\mathfrak{B}_1}} \rho_{\mathfrak{F}_1}^{\alpha_{\mathfrak{F}_1}}  $$
 $$\text{where}\ \ \alpha_{\mathfrak{F}_1} =  \alpha_{\mathfrak{B}_1} +\tfrac{1}{2}+1 ; \  \alpha_{\mathfrak{B}_1 } = \gamma_{\mathfrak{B}_1} +2 - 2 ; \gamma_{\mathfrak{B}_1} = -2$$

\item \textbf{HIGHER TRANSPORT EQUATION} :   comes from coefficient on $\rho_{\mathcal{A}}^{1+\gamma_{\mathcal{A}}+k}$:

For region $\mathsf{Z}  < 1$
\begin{equation}\label{hitlag2a}
\begin{aligned}
 &- 2 i(\mathsf{Z} z' +1)^{1/2} \mathsf{Z} \dfrac{\mathsf{Z} }{\langle\mathsf{Z}\rangle}(z' )^2\left[  \mathsf{Z}\partial_{\mathsf{Z}}+ ( \gamma_{\mathcal{A}}+k) (\tfrac{\mathsf{Z}}{\langle \mathsf{Z}\rangle})^2 - ( \gamma_{\mathcal{A}}+k) +\dfrac{1}{2}
 + \dfrac{\mathsf{Z} z' }{4(\mathsf{Z} z' +1)}  \right]  \tilde{U}_{2,k}  \\
 &+ \big( \dfrac{\mathsf{Z} }{\langle\mathsf{Z}\rangle}z' \big)^2 \Big[  \left( \mathsf{Z}\partial_{\mathsf{Z}}+ (\gamma_{\mathcal{A}}+k-1) (\tfrac{\mathsf{Z}}{\langle \mathsf{Z}\rangle})^2 - (\gamma_{\mathcal{A}}+k-1) \right)^2 +  (\lambda - \dfrac{n^2}{4} )  \Big] \tilde{U}_{2,k-1} = - F_{1,k}
\end{aligned}
\end{equation}
The asymptotic behavior of $\tilde{U}_{2,k}$ at $ \mathsf{Z} = 0$ which is $\mathsf{Z}^{\gamma_{\mathcal{A}} - \tfrac{1}{2}}$. Hence on the region $\mathsf{Z} < 1$
$$\tilde{U}_{2,k} \in (z')^{\alpha_{\mathfrak{B}_1}} \rho_{\mathcal{A}}^{-2+k} \mathsf{Z}^{-2 - \tfrac{1}{2}}\mathcal{C}^{\infty}(\mathsf{Z})$$

For $\mathsf{Z}  >  1$
\begin{equation}\label{hitlag2b}
\begin{aligned}
 &2 i(\mathsf{Z} z' +1)^{1/2} \mathsf{Z} \dfrac{\mathsf{Z} }{\langle\mathsf{Z}\rangle}(z' )^2\left[  \mathsf{Z}\partial_{\mathsf{Z}}+ ( \gamma_{\mathcal{A}}+k) (\tfrac{\mathsf{Z}}{\langle \mathsf{Z}\rangle})^2 - ( \gamma_{\mathcal{A}}+k) +\dfrac{1}{2}
 + \dfrac{\mathsf{Z} z' }{4(\mathsf{Z} z' +1)}  \right]  \tilde{U}_{2,k}  \\
 &+\big( \dfrac{\mathsf{Z} }{\langle\mathsf{Z}\rangle}z' \big)^2 \Big[  \left( \mathsf{Z}\partial_{\mathsf{Z}}+ ( \gamma_{\mathcal{A}}+k-1) (\tfrac{\mathsf{Z}}{\langle \mathsf{Z}\rangle})^2 - ( \gamma_{\mathcal{A}}+k-1) \right)^2 +  (\lambda - \dfrac{n^2}{4} )  \Big] \tilde{U}_{2,k-1} = - F_{1,k}
\end{aligned}
\end{equation}
using method in step 3c we construct
$$ \tilde{U}_{4,k} \in  \rho_{\mathfrak{F}_2}^{\infty} \rho_{\mathfrak{B}_1}^{\alpha_{\mathfrak{B}_1}+k} \rho_{\mathfrak{F}_1}^{\alpha_{\mathfrak{F}_1}+k}  $$
 $$\text{where}\ \ \alpha_{\mathfrak{F}_1} =  \alpha_{\mathfrak{B}_1} +\tfrac{1}{2}+1 ; \  \alpha_{\mathfrak{B}_1 } = \gamma_{\mathfrak{B}_1} +2 - 2; \ \gamma_{\mathfrak{B}_1} = -2$$

\end{enumerate}

 Take borel summation to construct $\tilde{U}_2$
 $$\tilde{U}_2\stackrel{\text{borel sum}}{\sim} \rho_{\mathcal{A}}^{\gamma_{\mathcal{A}}}\sum \rho_{\mathcal{A}}^k U_{2,k}$$
  $$ \Rightarrow \tilde{U}_2 \in \rho_{\mathcal{A}}^{\gamma_{\mathcal{A}}} \rho_{\mathfrak{F}_2}^{\gamma_{\mathfrak{F}_2} } \rho_{\mathfrak{B}_1}^{\alpha_{\mathfrak{B}_1}}  \rho_{\mathfrak{F}_1}^{\alpha_{\mathfrak{F}_1}} \mathcal{C}^{\infty}(\mathfrak{M}_{b,h})$$
   $$\text{where}\ \ \alpha_{\mathfrak{F}_1} =  \alpha_{\mathfrak{B}_1} +\tfrac{1}{2}+1 ; \  \alpha_{\mathfrak{B}_1 } = \gamma_{\mathfrak{B}_1} +2 - 2; \ \gamma_{\mathfrak{B}_1} = -2;\ \gamma_{\mathcal{A}} = -1 ; \gamma_{\mathfrak{F}_2} = \gamma_{\mathcal{A}}-\tfrac{1}{2}$$

 Define
 \begin{equation} 
 U_2 = \exp(-i\varphi_{\text{in}}) \tilde{U}_2
 \end{equation}
we can take $U_2$ to be supported near $\mathcal{A}$ by taking small cut off in $\rho_{\mathcal{A}}$

\textbf{The error after this stage close to $\mathcal{A}$ :} Define 
\begin{equation}\label{errorstep3a}
\begin{aligned}
 &E_2  = \QQ U_2 + E_1\\
\Rightarrow &\ \ E_2 = \exp(-i\varphi_{\text{in}}) F_2;\ \ \ F_2 \in   \rho_{\mathcal{A}}^{\infty} \rho_{\mathfrak{F}_2}^{ \gamma_{\mathfrak{F}_2}+1+2} \rho_{\mathfrak{B}_1}^{\alpha_{\mathfrak{B}_1}+2}  \rho_{\mathfrak{F}_1}^{\infty} \\
   \text{where}&\ \  \  \alpha_{\mathfrak{B}_1 } = \gamma_{\mathfrak{B}_1} +2 - 2; \ \gamma_{\mathfrak{B}_1} = -2 ; \ \gamma_{\mathfrak{F}_2} = \gamma_{\mathcal{A}}-\tfrac{1}{2}
      \end{aligned}
   \end{equation}

 \subsubsection{Step 3b : remove error on $\mathfrak{F}_2$}\
 
 \textbf{PREP INGREDIENTS}
\begin{itemize}
\item \emph{Product structure}:  We use the following product  structure in a neighborhood of $\mathfrak{F}_2$
 $$\RR^+_{\rho_{\mathfrak{F}_2} }\times \mathfrak{F}_2 \ \  ; \   \mathfrak{F}_2 = \overline{\RR^+_{\tilde{t}}} \times [0,1]_z$$
  
 \item   \emph{A global defining function} for $\mathfrak{F}_2$  can be chosen as
$$\rho_{\mathfrak{F}_2} = \tilde{\nu}' \langle \tilde{t}\rangle$$
In the neigborhood close to $\mathfrak{F}_2$ away from $\mathcal{A}$ $\langle \tilde{t}\rangle \sim 1$ thus $\tilde{\nu}'$ acts as the defining function in this neigborhood.  Since $\tilde{t} = \tfrac{\mathsf{Z}}{\tilde{\nu}'}$, away from $\mathcal{R}$
$$\rho_{\mathfrak{F}_2} = \mathsf{Z}(1 + (\tilde{t}^{-1})^2)^{1/2}$$
thus $\mathsf{Z}$ acts as the defining function.
  
 \item \emph{Form of the ODE :} The ODE in coordinates $\tilde{t} = \tfrac{\mathsf{Z}}{\tilde{\nu}'} , \tilde{\nu}' , z'$ has the following form:
  $$\QQ_{b,h} = (z')^2\rho_{\mathfrak{F}_2}^2\langle \tilde{t}\rangle^{-2} \left[(\tilde{t} \partial_{\tilde{t}})^2 +   (\lambda - \dfrac{n^2}{4}) + \tilde{t} \dfrac{\tilde{t}}{\langle \tilde{t}\rangle}  \rho_{\mathfrak{F}_2} z' + \tilde{t}^2\right]
   = (z')^2\rho_{\mathfrak{F}_2}^2\langle \tilde{t}\rangle^{-2}  \mathbf{P}_{\mathfrak{F}_2} $$
  
  \item  \emph{PHASE FUNCTION}:  In the current region $\varphi_{\text{in}}$ has the form
  \begin{align*}
  \varphi_{\text{in}} &= \dfrac{2}{3} (z' \tilde{\nu}')^{-1} \left[(z'+1)^{3/2} - (\tilde{t} \tilde{\nu}'z' + 1)^{3/2}\right] \\
 & =  \dfrac{2}{3} (z' \tilde{\nu}')^{-1} \left[   \left( (z'+1)^{3/2} - 1 \right) -\left( (\tilde{t} \tilde{\nu}'z' + 1)^{3/2} - 1 \right) \right]\\
  &= \dfrac{2}{3} (z' \tilde{\nu}')^{-1}    \left[ (z'+1)^{3/2} - 1 \right] -  \dfrac{2}{3} (z' \tilde{\nu}')^{-1}\left[ (\tilde{t} \tilde{\nu}'z' + 1)^{3/2} - 1 \right]
  \end{align*}
  
  One can see that the second term behaves like $-\tilde{t}$ by writing out its taylor series (\emph{note we are working in} $\tilde{t} \tilde{\nu}'z' << 1$)
 \begin{equation}\label{behaviorint}
  -\tfrac{2}{3} (z' \tilde{\nu}')^{-1} \left[ 1+ \tfrac{3}{2} \tilde{t} \tilde{\nu}'z' + \ldots  -1    \right] = - \tilde{t} \left[ \ldots \right]
  \end{equation}
 \textbf{In order to isolate this behavior, we will work with the conjugated version of} $\QQ$ by 
 $$\exp\big(-i\tfrac{2}{3} (z' \tilde{\nu}')^{-1}    \left[ (z'+1)^{3/2} - 1 \right]\big)$$ \textbf{however, since this commutes with $\QQ_{b,h}$, the conjugation leaves the operator unchanged}.\nl
  As indicated at the beginning of step 3, in order not to undo the effect we have achieved so far in step 3a, we will construct solution whose error has singularities along lagrangian $\mathcal{L}_2$ and thus has to behave like $\exp(-i\tilde{t})$ close to $\mathcal{A}\cap \mathfrak{B}_2$, which is opposite of the behavior $\sim \exp(i\tilde{t})$ as shown in (\ref{behaviorint}) associated to langrangian $\mathcal{L}_1$. Thus in order to achieve this effect (\emph{note that any choice of phase function has to satisfy the eikonal equation}), the phase function associated the outgoing lagrangian $\mathcal{L}_2$ has to be $-\varphi_{\text{out}}$ :
 \begin{align*}
 \varphi_{\text{out}} &=  \dfrac{2}{3} (z' \tilde{\nu}')^{-1}    \left[ (z'+1)^{3/2} - 1 \right] +  \dfrac{2}{3} (z' \tilde{\nu}')^{-1}\left[ (\tilde{t} \tilde{\nu}'z' + 1)^{3/2} - 1 \right] \\
& =  \dfrac{2}{3} (z' \tilde{\nu}')^{-1}  \left[  (z'+1)^{3/2} + (\tilde{t} \tilde{\nu}'z' + 1)^{3/2}  - 2\right] 
\end{align*}
 \end{itemize}
   
\textbf{CONSTRUCTION}\

 \emph{Properties of the error we are working with} :  Write $E_2 = E_2^1 + E_2^2$, where $E_2^i$ is supported closer to $\mathfrak{F}_i$. We are working with the portion of $E_2$ that is close to $\mathfrak{F}_2$ ie $E_2^2$. For matter of convenience, we will restate the properties (\ref{errorstep3a}) of $E_2$ in terms of the coordinate in region close to $\mathfrak{F}_2$ ie in coordinates $\tilde{t}, \tilde{\nu}'$, and $z'$. 
\begin{itemize}
\item  $E_2^2$ is supported away from $\tilde{t} = 0$
\item $E_2^2 \in \rho_{\mathfrak{F}_2}^{\gamma_{\mathfrak{F}_2}+1+2} \rho_{\mathfrak{B}_1}^{\alpha_{\mathfrak{B}_1}+2}$.
\item $E_2^2$ vanishes rapidly as $\tilde{t} \rightarrow \infty$.
\end{itemize}

  Following to observation made on the choice of the phase function, to start step 3b, we first rewrite the error from step 3b as
  $$E_2^2 = \exp(-\tfrac{2}{3} (z'\tilde{\nu}')^{-1}\left[(z'+1)^{3/2}-1\right])\tilde{E}_2$$
   Expand $\tilde{E}_2$ around boundary face $\mathfrak{F}_2$ \emph{whose defining function is given by} using the global face function $\rho_{\mathfrak{F}_2}$ and take into account the stated properties of $E_2$ at the end of step 3a :
  $$\tilde{E}_2 \sim \rho_{\mathfrak{F}_2}^{\gamma_{\mathfrak{F}_2}+1+2} (z')^{\alpha_{\mathfrak{B}_1}+2}\sum_k \rho_{\mathfrak{F}_2}^k \tilde{E}_{2,k};\ \text{where} \  \tilde{E}_{2,k} \in \mathcal{C}^{\infty}(z, \mathcal{S}^+(\tilde{t}))$$
  where  $\mathcal{C}^{\infty}(z', \mathcal{S}^+(\tilde{t}))$ comprises of functions $g$ that are smooth in $z$ and $\tilde{t}$,  is supported away from $\tilde{t} = 0$ and rapidly decreases as $\tilde{t} \rightarrow \infty$, together with $\supp g \in \RR^+_{z'}\times \RR^+_{\tilde{t}}$.\nl
 
  \textbf{We solve the following ODEs for $\tilde{U}_{3,k}$ :}
 $$ \left[(\tilde{t} \partial_{\tilde{t}})^2 +   (\lambda - \dfrac{n^2}{4}) + \tilde{t}^2  \tilde{\nu}' z' + \tilde{t}^2\right](\rho_{\mathfrak{F}_2})^{ \alpha_{\mathfrak{F}_2}} (z')^{ \alpha_{\mathfrak{B}_1}}\sum_k (\rho_{\mathfrak{F}_2})^k \tilde{U}_{3,k} =  (\rho_{\mathfrak{F}_2})^{ \gamma_{\mathfrak{F}_2}+1}  (z')^{\alpha_{\mathfrak{B}_1}}\langle \tilde{t}\rangle^2\sum_k (\rho_{\mathfrak{F}_2})^k \tilde{E}_{2,k}$$
 $$ \alpha_{\mathfrak{F}_2} = \gamma_{\mathfrak{F}_2} +1 $$
  Solving in series for (by matching coefficients of $\rho_{\mathfrak{F}_2}$) $\tilde{U}_{3,k}$ amounts to conjugating the operator for $\rho_{\mathfrak{F}_2}^k$ \emph{(see step 5 for this approach)}. However the operator if conjugated only powers of $\tilde{\nu}'$ has a simpler form and in fact gives a simple operator of the type Bessel, thus we rewrite $\rho_{\mathfrak{F}_2}$ in the above expression with $\rho_{\mathfrak{F}_2} = \tilde{\nu}' \langle \tilde{t}\rangle$:
   \begin{equation}
   \begin{aligned}
  &  \left[(\tilde{t} \partial_{\tilde{t}})^2 +   (\lambda - \dfrac{n^2}{4}) + \tilde{t}^2  \tilde{\nu}' z' + \tilde{t}^2\right](\tilde{\nu}')^{ \alpha_{\mathfrak{F}_2}}  (z')^{\alpha_{\mathfrak{B}_1}}\sum_k (\tilde{\nu}')^k\langle \tilde{t}\rangle^{k + \alpha_{\mathfrak{F}_2}}  \tilde{U}_{3,k} \\
  &=  (\tilde{\nu}')^{ \alpha_{\mathfrak{F}_2}}  (z')^{\alpha_{\mathfrak{B}_1}}\sum_k (\tilde{\nu}')^k \langle \tilde{t}\rangle^{k+ \alpha_{\mathfrak{F}_2}+2}   \tilde{E}_{2,k}
  \end{aligned}
  \end{equation}
  Solving in series by \emph{equating the coefficients of} $\tilde{\nu}'$, we start with the following inhomogeneous Bessel of order $\alpha = \sqrt{-\lambda + \tfrac{n^2}{4}}$ 
  $$ \left[(\tilde{t} \partial_{\tilde{t}})^2 +   (\lambda - \dfrac{n^2}{4}) +  \tilde{t}^2\right]  \langle \tilde{t}\rangle^{ \alpha_{\mathfrak{F}_2}}   (z')^{\alpha_{\mathfrak{B}_1}} \tilde{U}_{3,0} =  (z')^{\alpha_{\mathfrak{B}_1}}\langle \tilde{t}\rangle^{\alpha_{\mathfrak{F}_2}}  \langle \tilde{t}\rangle^2 \tilde{E}_{2,0}$$
     for which we construct a solution that satisfies : 
  \begin{align*}
 & \text{At}\ \tilde{t} \sim 0 : && \langle \tilde{t}\rangle^{ \alpha_{\mathfrak{F}_2}}   (z')^{\alpha_{\mathfrak{B}_1}} \tilde{U}_{3,0} \in \tilde{t}^{\alpha} (z')^{\alpha_{\mathfrak{B}_1}}\mathcal{C}^{\infty}(\tilde{t}, z');\ \ \alpha = \sqrt{-\lambda + \tfrac{n^2}{4}}\\
&\text{As} \ \tilde{t} \rightarrow\infty : && \langle \tilde{t}\rangle^{ \alpha_{\mathfrak{F}_2}}  (z')^{\alpha_{\mathfrak{B}_1}} \tilde{U}_{3,0}\in \exp(-i\tilde{t}) \tilde{t}^{-1/2} (z')^{\alpha_{\mathfrak{B}_1}} S^0_{\tilde{t}} ;\\
& &&\text{where}\  S^0_{\tilde{t}}\  \text{comprises of smooth functions behaving as a symbol of order 0 at infinity}  
  \end{align*}
by choosing
  $$\langle \tilde{t}\rangle^{\alpha_{\mathfrak{F}_2}}   (z')^{\alpha_{\mathfrak{B}_1}}\tilde{U}_{3,0} :=  H^{(2)}_{\alpha}\int_t^c \dfrac{H^{(1)}  E_{2,0}}{W} \,ds + H^{(1)}_{\alpha}\int_t^{\infty} \dfrac{H^{(2)}  E_{2,0}}{W} \,ds$$
 where $H^{(i)}_{\alpha}$ being the Hankel functions (\emph{they are two linearly independent solutions to the homomgeneous Bessel equation of order} $\alpha = \sqrt{-\lambda + \tfrac{n^2}{4}}$)  and $W$ being the Wronskian:
  $$W\{H^{(1)}_{\alpha}, H^{(2)}_{\alpha}\} = -\dfrac{4i}{\pi t}$$
 $c$ is chosen so that $ \tilde{t}^{\alpha} U_{3,0} |_{\tilde{t} = 0} = 0$ \emph{(This is possible, since at} $\tilde{t}\sim 0$, $H^{(i)} \sim \tilde{t}^{\alpha} \mathcal{C}^{\infty}(\tilde{t}) +  \tilde{t}^{-\alpha} \mathcal{C}^{\infty}(\tilde{t})$). Since the second terms rapidly decreasing as $\tilde{t}\rightarrow \infty$, $\tilde{U}_{3,0}$ satisfies the required oscillatory behavior .
  
Define 
\begin{equation}
\bar{U}_{3,0} = e^{-i \phi+ i\tilde{t}} \tilde{U}_{3,0};\ \ \phi = i \tfrac{2}{3}\dfrac{1}{z' \tilde{\nu}'} \left((\tilde{t}\tilde{\nu}' z' + 1)^{3/2}  - 1\right)\end{equation}
then by the prep calculation \emph{placed at the end of the subsection} we have $\bar{U}_{3,0} |_{z'=0} = \tilde{U}_{3,0}$. In addition we also have : 
$$ \left[(\tilde{t} \partial_{\tilde{t}})^2 +   (\lambda - \dfrac{n^2}{4}) + \tilde{t}^2  \tilde{\nu}' z' + \tilde{t}^2\right] (\tilde{\nu}')^{\alpha_{\mathfrak{F}_2}}\langle \tilde{t}\rangle^{\alpha_{\mathfrak{F}_2}}   (z')^{\alpha_{\mathfrak{B}_1}}\bar{U}_{3,0} $$
$$+ (\tilde{\nu}')^{\alpha_{\mathfrak{F}_2}} \langle \tilde{t}\rangle^{ \alpha_{\mathfrak{F}_2}}   (z')^{\alpha_{\mathfrak{B}_1}}\langle \tilde{t}\rangle^2 \tilde{E}_{2,0} =  \mathsf{O}\big(\rho_{\mathfrak{F}_2}^{\alpha_{\mathfrak{F}_2}+1}\big)$$
 
  We next solve for $\bar{U}_{3,k}$ for $k > 0$ in a similar manner, ie solve for $\tilde{U}_{3,k}$
 \begin{equation}
 \begin{aligned}
 &\left[(\tilde{t} \partial_{\tilde{t}})^2 +   (\lambda - \dfrac{n^2}{4}) + \tilde{t}^2  \tilde{\nu}' z' + \tilde{t}^2\right] \langle \tilde{t}\rangle^{ k +\alpha_{\mathfrak{F}_2}}   (z')^{\alpha_{\mathfrak{B}_1}}\tilde{U}_{3,k}\\
 & = -  \langle \tilde{t}\rangle^{k +\alpha_{\mathfrak{F}_2}}   (z')^{\alpha_{\mathfrak{B}_1}}\tilde{E}_{2,k} - \rho_{\mathfrak{F}_2}^{-1} \Big(\mathbf{P}_{\mathfrak{F}_2} \langle \tilde{t}\rangle^{ k -1+\alpha_{\mathfrak{F}_2}}  (z')^{\alpha_{\mathfrak{B}_1}}\bar{U}_{3,k-1}  + \langle\tilde{t}\rangle^{ k -1+\alpha_{\mathfrak{F}_2}} \langle \tilde{t}\rangle^2 \tilde{E}_{2,k-1}     \Big)\Big|_{\rho_{\mathfrak{F}_2 } = 0}
 \end{aligned}
 \end{equation}
 Properties of $\tilde{U}_{3,k}$
\begin{align*}
&\tilde{t} = 0 :&&  \bar{U}_{3,k} \in \tilde{t}^{\sqrt{\tfrac{n^2}{4} -\lambda}} \mathcal{C}^{\infty}(\tilde{t})\\
&\tilde{t} \sim \infty : && \tilde{U}_{3,k} \in \exp(-i\tilde{t}) \tilde{t}^{-1/2} \langle \tilde{t}\rangle^{-( k+\alpha_{\mathfrak{F}_2})} \mathcal{C}^{\infty}(\tilde{t})
\end{align*}
 Define 
 \begin{equation}
\bar{U}_{3,k} = e^{-i \phi+ i\tilde{t}} \tilde{U}_{3,k} \\
\end{equation}
Finally we construct $\bar{U}_3$ by borel sum in $\rho_{\mathfrak{F}_2}$ 
 \begin{equation}
 \bar{U}_3 \stackrel{\text{borel sum in }}{\sim}  \rho_{\mathfrak{F}_2}^{\alpha_{\mathfrak{F}_2}}   (z')^{\alpha_{\mathfrak{B}_1}} \sum_k \rho_{\mathfrak{F}_2}^k   \bar{U}_{3,k} 
 \end{equation}
 
 Multiply back by the factor $\tfrac{2}{3} (z' \tilde{\nu}')^{-1}    \left[ (z'+1)^{3/2} - 1 \right]$ \emph{that we have implicitly factored out at the beginning of the construction (view the remarks in the prep ingredients subsection)}, we obtain 
 \begin{equation}
 \begin{aligned}
 &U_3 =  \exp\big(-i\dfrac{2}{3} (z' \tilde{\nu}')^{-1}    \left[ (z'+1)^{3/2} - 1 \right]\big)\bar{U}_3 = \exp(-i\varphi_{\text{out}})\tilde{U}_3\\
 &\varphi_{\text{out}} = \dfrac{2}{3} (z' \tilde{\nu}')^{-1}  \left[  (z'+1)^{3/2} + (\tilde{t} \tilde{\nu}'z' + 1)^{3/2}  - 2\right] \\
 &\tilde{U}_3 \in \rho_{\mathfrak{B}_1}^{ \alpha_{\mathfrak{B}_1}} \rho_{\mathcal{R}}^{\sqrt{\tfrac{n^2}{4}-\lambda}} \rho_{\mathcal{A}}^{\tfrac{1}{2}+ \alpha_{\mathfrak{F}_2}}\rho_{\mathfrak{F}_2}^{\alpha_{\mathfrak{F}_2}} \mathcal{C}^{\infty}(\mathfrak{M}_{b,h})
 \end{aligned}
 \end{equation}

 \textbf{Prep computations used in the construction}
 $$\lim_{\tilde{\nu}'\rightarrow 0} \dfrac{2}{3} \dfrac{(\tilde{t} \tilde{\nu}' z'+1)^{3/2} - 1}{\tilde{\nu}' z'} = \tilde{t}$$
 $$\Rightarrow \lim_{\tilde{\nu}' \rightarrow 0} \exp\big( \dfrac{2}{3} \dfrac{(\tilde{t} \tilde{\nu}' z'+1)^{3/2} - 1}{\tilde{\nu}' z'} - t\big) - 1  = 0$$
 
 \begin{align*}
 \tilde{t}\partial_{t} \left[  \dfrac{2}{3}\dfrac{1}{z' \tilde{\nu}'} \left((\tilde{t}\tilde{\nu}' z' + 1)^{3/2}  - 1\right)- \tilde{t}\right]
&= \tilde{t}\left[ (\tilde{t} \tilde{\nu}' z' +1)^{1/2} - 1 \right]\\
(\tilde{t}\partial_{t})^2 \left[  \dfrac{2}{3}\dfrac{1}{z' \tilde{\nu}'}  \left((\tilde{t}\tilde{\nu}' z' + 1)^{3/2} -1 \right) - \tilde{t}\right] &= \tilde{t}\left[(\tilde{t} \tilde{\nu}' z' +1)^{1/2} - 1 \right] +\dfrac{1}{2} \dfrac{\tilde{t^2}\tilde{\nu}' z'} {(\tilde{t} \tilde{\nu}' z' +1)^{1/2}}\\
&= \dfrac{\tilde{t}^2 \tilde{\nu}' z' }{ 1 +  (\tilde{t} \tilde{\nu}' z' +1)^{1/2} } +\dfrac{1}{2} \dfrac{\tilde{t^2}\tilde{\nu}' z'} {(\tilde{t} \tilde{\nu}' z' +1)^{1/2}}
 \end{align*}
 \begin{align*}
&\big[ (\tilde{t} \partial_{\tilde{t}})^2 +   (\lambda - \dfrac{n^2}{4})  + \tilde{t}^2 + \tilde{t}^2\tilde{\nu}' z' \big] \langle \tilde{t}\rangle^{ \alpha_{\mathfrak{F}_2}}   (z')^{\alpha_{\mathfrak{B}_1}}\bar{U}_{3,0}\\
=&\big[ (\tilde{t} \partial_{\tilde{t}})^2 +   (\lambda - \dfrac{n^2}{4})  + \tilde{t}^2 + \tilde{t}^2\tilde{\nu}' z' \big] e^{-i\phi + i\tilde{t}} \langle \tilde{t}\rangle^{ \alpha_{\mathfrak{F}_2}}   (z')^{\alpha_{\mathfrak{B}_1}} \tilde{U}_{3,0}\\
=& e^{-i\phi + i\tilde{t}} \big[ (\tilde{t} \partial_{\tilde{t}})^2 +   (\lambda - \dfrac{n^2}{4})  + \tilde{t}^2 \big] \langle \tilde{t}\rangle^{ \alpha_{\mathfrak{F}_2}}   (z')^{\alpha_{\mathfrak{B}_1}} \tilde{U}_{3,0}\\
 &+[(\tilde{t} \partial_{\tilde{t}})^2, e^{-i\phi + i\tilde{t}} ]\langle \tilde{t}\rangle^{ \alpha_{\mathfrak{F}_2}}   (z')^{\alpha_{\mathfrak{B}_1}} \tilde{U}_{3,0} + \tilde{t}^2\tilde{\nu}' z'  \langle \tilde{t}\rangle^{ \alpha_{\mathfrak{F}_2}}   (z')^{\alpha_{\mathfrak{B}_1}}\bar{U}_{3,0}  
 \end{align*}
 
 \textbf{Properties of constructed distrubtion and the resulting error}
  Denote by 
  $$E_3 = \QQ_b U_3 + E_2 $$
$$  \varphi_{\text{out}} = \dfrac{2}{3} (z' \tilde{\nu}')^{-1}  \left[  (z'+1)^{3/2} + (\tilde{t} \tilde{\nu}'z' + 1)^{3/2}  - 2\right]$$
  \begin{align*}
  &U_3 = \exp(-i\varphi_{\text{out}} ) \tilde{U}_3 \ \ 
 &\tilde{U}_3 \in \rho_{\mathfrak{B}_1}^{ \alpha_{\mathfrak{B}_1}} \rho_{\mathcal{R}}^{\sqrt{\tfrac{n^2}{4}-\lambda}} \rho_{\mathcal{A}}^{\tfrac{1}{2}+ \alpha_{\mathfrak{F}_2}}\rho_{\mathfrak{F}_2}^{\alpha_{\mathfrak{F}_2}} \mathcal{C}^{\infty}(\mathfrak{M}_{b,h})\\
 &E_3 = \exp(-i\varphi_{\text{out}}) F_3\ \ 
 &F_3 \in \rho_{\mathfrak{B}_1}^{ \alpha_{\mathfrak{B}_1}+2} \rho_{\mathcal{R}}^{\sqrt{\tfrac{n^2}{4}-\lambda}+1} \rho_{\mathcal{A}}^{\tfrac{1}{2}+ \alpha_{\mathfrak{F}_2}+2} \rho_{\mathfrak{F}_2}^{\infty}S^1_0(z' , \tilde{t}) \end{align*}

\subsubsection{STEP 3C}\

    Step 3b introduces error back on $\mathcal{A}$ however this time along a different Lagrangian associated to the $\varphi_{\text{out}}$, which we will proceed to remove in this step. \nl 
   \textbf{The product structure} We use the product structure close to $\mathcal{A}$
  $$\RR^+_{\rho_{\mathcal{A}}} \times \mathcal{A} $$
  \textbf{The `global' defining function for face} $\mathcal{A}$
  $$\rho_{\mathcal{A}} = \tilde{\tilde{\nu}}' \langle \mathsf{Z}\rangle$$
\textbf{Reminder of the phase functions we are using:} 
$$\varphi_{\text{in}} = \dfrac{2}{3}h^{-1} \Big| (z'+1)^{3/2} - (z+1)^{3/2}\Big|$$
 $$\text{In the current region} \ z' > z, \ \varphi_{\text{in}} =  \dfrac{2}{3}h^{-1} \left[ (z'+1)^{3/2} - (z+1)^{3/2}\right]$$
  $$\varphi_{\text{out}} =  \dfrac{2}{3}h^{-1} \left[ (z'+1)^{3/2} +(z+1)^{3/2}-2\right]$$
  \textbf{Prep computations:} 
  $$\mathsf{Z}\partial_{\mathsf{Z}} \varphi_{out}
= (\mathsf{Z} z' +1)^{1/2} \mathsf{Z} z'   
$$
$$(\mathsf{Z}\partial_{\mathsf{Z}})^2 \varphi_{out} 
=  (\mathsf{Z} z' +1)^{1/2} \mathsf{Z}z' 
 + \tfrac{1}{2}(\mathsf{Z} z' +1)^{-1/2} (\mathsf{Z} z')^2  $$
 
\textbf{The conjugated operator by the outgoing oscillatory part}. To get the transport equations to solve for the 'symbols' we conjugate the operator $\QQ_b$ by 
$$ \exp( -i\varphi_{\text{out}}) = \exp\big(-i\tfrac{2}{3} (\tilde{\tilde{\nu}}' \mathsf{Z} z')^{-1}\big[(\mathsf{Z}z'+1)^{3/2} + (z'+1)^{3/2} -2 \big]\big)$$
\begin{equation}\label{conjugateop3c}
\begin{aligned}
 &(\tilde{\tilde{\nu}}' \mathsf{Z} z' )^2 \left(\mathsf{Z} \partial_{\mathsf{Z}} - \tilde{\tilde{\nu}} \partial_{\tilde{\tilde{\nu}}}\right)^2   + (\tilde{\tilde{\nu}}' \mathsf{Z} z' )^2 (\lambda-\dfrac{n^2}{4}) - 2 i\tilde{\tilde{\nu}}' \mathsf{Z} z'(\mathsf{Z} z' +1)^{1/2} \mathsf{Z} z'  \left(\mathsf{Z} \partial_{\mathsf{Z}} - \tilde{\tilde{\nu}} \partial_{\tilde{\tilde{\nu}}}\right)\\
& -i\tilde{\tilde{\nu}}' \mathsf{Z} z'  (\mathsf{Z} z' +1)^{1/2} \mathsf{Z}z' 
 -i\tilde{\tilde{\nu}}' \mathsf{Z} z'  \dfrac{1}{2}(\mathsf{Z} z' +1)^{-1/2} (\mathsf{Z} z')^2  \\
 =\ &\tilde{\tilde{\nu}}'( \mathsf{Z} z' )^2 \Big[  \tilde{\tilde{\nu}}'   \left(\mathsf{Z} \partial_{\mathsf{Z}} - \tilde{\tilde{\nu}} \partial_{\tilde{\tilde{\nu}}}\right)^2   + \tilde{\tilde{\nu}}' (\lambda-\dfrac{n^2}{4}) - 2 i(\mathsf{Z} z' +1)^{1/2} \left(\mathsf{Z} \partial_{\mathsf{Z}} - \tilde{\tilde{\nu}} \partial_{\tilde{\tilde{\nu}}}\right) \\
 & -i (\mathsf{Z} z' +1)^{1/2} 
 - i\mathsf{Z} z'  \dfrac{1}{2}(\mathsf{Z} z' +1)^{-1/2} \Big]
 \end{aligned}
 \end{equation}

Solving in series for the coefficients of $\rho_{\mathcal{A}}$ amounts to further conjugation of the above operator (\ref{conjugateop3c}) by powers of $\rho_{\mathcal{A}}$.\nl
%
Denote by 
\begin{equation}
\mathbf{P}_{b,h} =  \tilde{\tilde{\nu}}'   \left(\mathsf{Z} \partial_{\mathsf{Z}} - \tilde{\tilde{\nu}} \partial_{\tilde{\tilde{\nu}}}\right)^2   + \tilde{\tilde{\nu}}' (\lambda-\dfrac{n^2}{4}) - 2 i(\mathsf{Z} z' +1)^{1/2} \left(\mathsf{Z} \partial_{\mathsf{Z}} - \tilde{\tilde{\nu}} \partial_{\tilde{\tilde{\nu}}}\right) 
 -i (\mathsf{Z} z' +1)^{1/2} 
 - i\mathsf{Z} z'  \dfrac{1}{2}(\mathsf{Z} z' +1)^{-1/2}
 \end{equation}
\begin{equation}\label{tocomputetransport}
\begin{aligned}
\mathbf{P}_{b,h} (\rho_{\mathcal{A}})^k g = & \ \mathbf{P}_{b,h} (\tilde{\tilde{\nu}}' \langle\mathsf{Z}\rangle)^k g \\
= & \ \rho_{\mathcal{A}}^{k+1}\langle \mathsf{Z}\rangle^{-1} \Big[  \left( \mathsf{Z}\partial_{\mathsf{Z}}+ k (\tfrac{\mathsf{Z}}{\langle \mathsf{Z}\rangle})^2 - k \right)^2 +  (\lambda - \dfrac{n^2}{4} )  \Big] g \\
&\  + \rho_{\mathcal{A}}^k\Big[  - 2 i(\mathsf{Z} z' +1)^{1/2}  \left( \mathsf{Z}\partial_{\mathsf{Z}}+ k (\tfrac{\mathsf{Z}}{\langle \mathsf{Z}\rangle})^2 - k \right)
-i (\mathsf{Z} z' +1)^{1/2} - i\mathsf{Z} z'  \dfrac{1}{2}(\mathsf{Z} z' +1)^{-1/2}\Big] g
\end{aligned}
\end{equation}
\textbf{The error from step 3b} : Write $E_3 = E_3^{\mathcal{A}} + E_3^{\mathcal{R}}$ with $E_3^{\mathcal{A}}, E_3^{\mathcal{R}}$ supported closer to $\mathcal{A}$ and $\mathcal{R}$ correspondingly. For this step, we will only need to work with $E_3^{\mathcal{A}}$. Write
 $$E^{\mathcal{A}}_3 = \exp( -i\varphi_{\text{out}}) F_3 ;\ \ 
F_3 \in \rho_{\mathfrak{B}_1}^{\alpha_{\mathfrak{B}_1}+2} \rho_{\mathcal{R}}^{\infty} \rho_{\mathcal{A}}^{\tfrac{1}{2}+\alpha_{\mathfrak{F}_2}+2} \rho_{\mathfrak{F}_2}^{\infty}S^1_0(z' , \tilde{t}) $$
 Expand $F_3$ at $\mathcal{A}$.
$$F_3 \stackrel{\text{taylor expand}}{\sim}  \rho_{\mathcal{A}}^{\tfrac{1}{2}+\alpha_{\mathfrak{F}_2}+2} \rho_{\mathfrak{B}_1}^{\alpha_{\mathfrak{B}_1}+2} \sum_k  \rho_{\mathcal{A}}^k F_{3,k} ;\ \ F_{3,k} \in \rho_{\mathfrak{F}_2}^{\infty} \mathcal{C}^{\infty}(\mathcal{A})$$

From the above computation (\ref{tocomputetransport}) we can write the transport equations satisfied by the `symbols' $U_{4,k}$ :


\begin{enumerate}
\item \textbf{FIRST TRANSPORT EQUATION} comes from coefficient on $\rho_{\mathcal{A}}^{1+ \beta_{\mathcal{A}}}$:
\begin{equation}\label{fitlag2}
\begin{cases}
 -2 i(\mathsf{Z} z' +1)^{1/2} \dfrac{( \mathsf{Z} z' )^2}{\langle\mathsf{Z}\rangle}\left[  \mathsf{Z}\partial_{\mathsf{Z}}+ (\beta_{\mathcal{A}}+0) (\tfrac{\mathsf{Z}}{\langle \mathsf{Z}\rangle})^2 - (\beta_{\mathcal{A}}+0) +\dfrac{1}{2}
 + \dfrac{\mathsf{Z} z' }{4(\mathsf{Z} z' +1)}  \right]  U^0_{4,0} =  -F_{3,0}\\
U_{4,0} |_{\mathsf{Z} = 0} = -0 : \ \text{IC}  \ \  \textbf{double check the IC}
\end{cases}
\end{equation}
We have denoted by $\beta_{\mathcal{A}}$ the leading power of $\rho_{\mathcal{A}}$, ie $\beta_{\mathcal{A}} = \tfrac{1}{2 }+ \alpha_{\mathfrak{F}_2} + 1$ 

Denote by $\mathsf{P}_0$ the operator 
$$\mathsf{P}_0 = -2 i(\mathsf{Z} z' +1)^{1/2}  \dfrac{( \mathsf{Z} z' )^2}{\langle\mathsf{Z}\rangle}\left[  \mathsf{Z}\partial_{\mathsf{Z}}+ (\beta_{\mathcal{A}}+0) (\dfrac{\mathsf{Z}}{\langle \mathsf{Z}\rangle})^2 - (\beta_{\mathcal{A}}+0) +\dfrac{1}{2}
 + \dfrac{\mathsf{Z} z' }{4(\mathsf{Z} z' +1)}  \right] $$
 $P U^0_{4,0} + F_{3,0} = f_0$ which is nontrivial at $\mathfrak{B}_1$ and $\mathcal{L}$.

\begin{remark} \ $\tfrac{\mathsf{Z}}{\mathsf{Z}} = \tfrac{z}{(z^2 + (z')^2)^{1/2}}$. Hence $\mathsf{P}_0$ does not live on the blow down of $\mathcal{A}$, which is the product space $\RR^+_z \times \RR^+_{z'}$.\nl \textbf{The short coming of just solving along each regular fiber:  }
If one were to carry out solving just along each regular fiber, we will get a solution that does not live on the product space $\overline{\RR^+}_{\mathsf{Z}} \times [0,1]_{z'}$ \emph{; ( the solution lives on $\RR^+_{\mathsf{Z}} \times [0,1]_{z'}$ but does not have uniform behavior as $\mathsf{Z}\rightarrow\infty$ and $z' \rightarrow 0$)}. When considered on $\mathcal{A}$, we have the existence of the solution constructed this way on every compact subset of $\mathcal{A}$ away from $\mathcal{A} \setminus ( \mathfrak{F}_1 \cup \mathfrak{B}_1)$ and we do not know the asymptotic behavior of the solution at the corner $\mathfrak{B}_1 \cap \mathfrak{F}_1$ in $\mathcal{A}$ and along $\mathfrak{F}_1$\nl
   \textbf{New method to construct polyhomogeneous solution to (\ref{fitlag2}): } 
 Because of the singular fiber $( \mathfrak{B}_1\cap \mathcal{A} ) \cup (\mathfrak{F}_1\cap \mathcal{A})$, it will involve more work than just solving along each fiber. What we set out to do is : \emph{(the following procedure is  completely in face $\mathcal{A}$)}
\begin{enumerate}
\item Find a polyhomongeous function $U_{4,0}^{\mathfrak{F}_1} + U_{4,0}^{\mathfrak{B}_1}$ on $\mathcal{A}$ so that 
$$f_2 = \mathsf{P}_0( U_{4,0}^{\mathfrak{F}_1} + U_{4,0}^{\mathfrak{B}_1}) \in \rho_{\mathcal{L}\cap \mathcal{A}}^{\infty} \rho_{\mathcal{A}\cap \mathfrak{B}_1}^{\infty}  $$
\emph{(we will further break this into two steps : first remove the error at $\mathfrak{B}_1$ then at $\mathcal{L}$ then)}
\item Now we can blowdown $\mathfrak{B}_1\cap \mathcal{A}$ in  $\mathcal{A}$ and consider $f_3$ to live on the compactified product space $[0,1]_z \times [0,1]_{z'}$. 
\end{enumerate}
  \end{remark}

\
\

\textbf{DETAILS OF THE CONSTRUCTION :}
\begin{itemize}
  \item \textbf{Step 1a : }    We will use as  the 'global' defining function for $\mathfrak{B}_1$ in $\mathcal{A}$
 $$\rho_{\mathfrak{B}_1\cap \mathcal{A}} = z' \langle \mathsf{Z}\rangle$$
 Since we are working close to $\mathfrak{B}_1$ we will only need to consider $F_{3,0}$ multiplied with a cut-off close $\mathfrak{B}_1$ (in $\mathcal{A}$), denoted by $F_{3,0}^1$. Write $F_{3,0} = F_{3,0}^1 + F_{3,0}^2$. Expand $F_{3,0}^1$ at $\mathcal{B}_1 \cap \mathcal{A}$ 
 $$F_{3,0}^1 \sim \rho_{\mathfrak{B}_1\cap \mathcal{A}}^{\alpha_{\mathfrak{B}_1}+2} \sum_j \rho_{\mathfrak{B}_1\cap \mathcal{A}} ^j F^1_{3,0; j}$$
 Denote by $\beta_{\mathfrak{B}_1} $ the leading power of $\rho_{\mathfrak{B}_1}$ in the constructed solution $U_{4,0}$.\nl We want to solve for 
 $$U_{4,0}^{\mathfrak{B}_1} \sim \rho_{\mathcal{A}\cap \mathfrak{B}_1}^{\beta_{\mathfrak{B}_1}} \sum_j  \rho_{\mathcal{A}\cap \mathfrak{B}_1}^j U^1_{4,0j} \ \ \ \text{such that}\ \ \mathsf{P} U_{4,0}^{\mathfrak{B}_1} \in \rho_{\mathfrak{B}_1}^{\infty} ; \ \ \beta_{\mathfrak{B}_1} =    \alpha_{\mathfrak{B}_1} +2 - 2 $$
 Solving in series amounts to conjugating $\mathsf{P}_0$ by powers of $\rho_{\mathfrak{B}_1 \cap \mathcal{A}}$ which is the same as conjugating by powers of $\langle \mathsf{Z}\rangle$ since power of $z'$ commutes with the operator. By direct computation we get :
\begin{align*}
\rho_{\mathfrak{B}_1 \cap \mathcal{A}}^{-j-\beta_{\mathfrak{B}_1}} \mathsf{P}_0  \rho_{\mathfrak{B}_1 \cap \mathcal{A}}^{j+\beta_{\mathfrak{B}_1}}
 = &\ -2 i(\dfrac{\mathsf{Z}}{\langle \mathsf{Z}\rangle} \rho_{\mathfrak{B}_1} +1)^{1/2}  \dfrac{\mathsf{Z}^2}{\langle \mathsf{Z}\rangle^2}\langle \mathsf{Z}\rangle^{-1}\rho_{\mathfrak{B}_1}^2 \Big[\mathsf{Z}\partial_{Z} + (\beta_{\mathfrak{B}_1}+j) (\dfrac{\mathsf{Z}}{\langle\mathsf{Z}\rangle})^2\\
 &\ \ + (\beta_{\mathcal{A}}+0) (\dfrac{\mathsf{Z}}{\langle \mathsf{Z}\rangle})^2 - (\beta_{\mathcal{A}}+0) +\dfrac{1}{2} +\dfrac{1}{4}
 - \dfrac{1 }{4(\tfrac{\mathsf{Z}}{\langle\mathsf{Z}\rangle} \rho_{\mathfrak{B}_1\cap \mathcal{A}}+1)}  \Big] 
 \end{align*}
Since $\tfrac{\mathsf{Z}}{\langle\mathsf{Z}\rangle}$ is a bounded smooth global function and have value less than 1, we can expand in series $\tfrac{1 }{(\tfrac{\mathsf{Z}}{\langle\mathsf{Z}\rangle} \rho_{\mathfrak{B}_1\cap \mathcal{A}}+1)}
= 1 - \tfrac{\mathsf{Z}}{\langle\mathsf{Z}\rangle} \rho_{\mathfrak{B}_1\cap \mathcal{A}} + \ldots $
Thus $U_{4,0; 0}^{\mathfrak{B}_1}$ and $U_{4,0;j}^{\mathfrak{B}_1}$ have to satisfy
\begin{equation}
\begin{aligned}
&  2 i \left[\mathsf{Z}\partial_{Z} + (0+\beta_{\mathfrak{B}_1}) (\dfrac{\mathsf{Z}}{\langle\mathsf{Z}\rangle})^2+ (\beta_{\mathcal{A}}+0) (\dfrac{\mathsf{Z}}{\langle \mathsf{Z}\rangle})^2 - (\beta_{\mathcal{A}}+0) +\dfrac{1}{2}   \right] U^{\mathfrak{B}_1}_{4,0} = Z^{-2}F^1_{3,0;0} \\
&  2 i \left[\mathsf{Z}\partial_{Z} + (j+\beta_{\mathfrak{B}_1}) (\dfrac{\mathsf{Z}}{\langle\mathsf{Z}\rangle})^2+ (\beta_{\mathcal{A}}+0) (\dfrac{\mathsf{Z}}{\langle \mathsf{Z}\rangle})^2 - (\beta_{\mathcal{A}}+0) +\dfrac{1}{2}   \right] U^{\mathfrak{B}_1}_{4,j} \\
&\                       + \dfrac{1}{4} \sum_{l=0}^{j-1} c_j U^1_{4,l} = -\sum_{l=0}^j c_l Z^{-2} F^1_{3,0,l}\  + ;\  c_l\ \text{known constants}
\end{aligned}
\end{equation}
Use Borel sum to define from $U_{4,0}^{\mathfrak{B}_1}$
$$U_{4,0}^{\mathfrak{B}_1} (\rho_{\mathcal{A}\cap \mathfrak{B}_1}, \mathsf{Z}) \stackrel{\text{borel sum}}{\sim} \rho_{\mathcal{A}\cap \mathfrak{B}_1}^{\beta_{\mathfrak{B}_1}} \sum_j (\rho_{\mathcal{A}\cap \mathfrak{B}_1})^j U_{4,0; j}^{\mathfrak{B}_1}(\mathsf{Z})$$
Denote $ f_1 = \mathsf{P}_0  U_{4,0}^{\mathfrak{B}_1}  +   F^1_{3,0}$

\emph{\textbf{Behaviors of constructed function and resulting error :}}
The behavior of $F^1_{3,0;j}$ at $\mathsf{Z} = 0$ (ie at $\mathfrak{F}_2$) dictates that of $U^{\mathfrak{B}_1}_{4,j}$. Note $F^1_{3,0;j}$ has support in $\mathsf{Z}  < \epsilon'$ for some $\epsilon'$ and vanishes to infinite order at $\mathfrak{F}_2$. Thus 
\begin{align*}
U_{4,0}^{\mathfrak{B}_1} \in \rho_{\mathfrak{B}_1}^{ \beta_{\mathfrak{B}_1}}  \rho_{\mathfrak{F}_2}^{\infty} \rho_{\mathfrak{F}_1}^{\beta_{\mathfrak{B}_1} +\tfrac{1}{2}} \\
f_1 \in \rho_{\mathcal{A}\cap \mathfrak{B}_1}^{\infty} \rho_{\mathfrak{F}_2}^{\infty}     \rho_{\mathfrak{F}_1}^{\beta_{\mathfrak{B}_1} +\tfrac{1}{2}+1+1}   
\end{align*}

\item \textbf{Step 1b :} Now we remove the error at $\mathfrak{F}_1$ in $\mathcal{A}$. Write $f_1 = f^1_1 + f_1^2$ where $f_1^1$ is supported away from $\mathfrak{F}_2$. Expand $f_1$ at $\mathfrak{F}_1\cap \mathcal{A}$.
$$f_1^1 \sim \rho_{\mathfrak{F}_1}^{\beta_{\mathfrak{B}_1} +\tfrac{1}{2}+1+1} \sum_p \rho_{\mathfrak{F}_1}^p f^1_{1,p}(z) $$
We would like to find 
$$U_{4,0}^{\mathfrak{F}_1} \sim\rho_{\mathfrak{F}_1}^{\beta{\mathfrak{F}_1}} \sum_m \rho_{\mathfrak{F}_1}^m U^2_{4,m} \ \ \text{so that } \mathsf{P}_0 U_{4,0}^{\mathfrak{F}_1} + f^1_1 \in \rho_{\mathcal{A}\cap \mathfrak{F}_1}^{\infty} \rho_{\mathcal{A}\cap \mathfrak{B}_1}^{\infty} $$
In this case $\beta_{\mathfrak{F}_1} =  \beta_{\mathfrak{B}_1} +\tfrac{1}{2}+1$

In coordinates $z, \mathsf{Z}'$ , $\mathsf{P}_0$ is equal to
$$\mathsf{P}_0 = 2 i(z+1)^{1/2} z^2\dfrac{\mathsf{Z}'}{\langle\mathsf{Z}'\rangle}\left[  z\partial_z - \mathsf{Z}'\partial_{\mathsf{Z}'}+ (\beta_{\mathcal{A}}+0) \langle \mathsf{Z}'\rangle^{-2} - (\beta_{\mathcal{A}}+0) +\dfrac{1}{2}
 + \dfrac{z }{4(z+1)}  \right]  $$

\textbf{Prep computation}
$$\mathsf{P}_0  (\mathsf{Z}')^m g(z)=
z\partial_z - m+ (\beta_{\mathcal{A}}+0) \langle \mathsf{Z}'\rangle^{-2} - (\beta_{\mathcal{A}}+0) +\dfrac{1}{2}
 + \dfrac{z }{4(z+1)} $$
 so $U_{4,0;0}^{\mathfrak{F}_1}$ and $U_{4,0;p}^{\mathfrak{F}_1}$,  have to satisfy correspondingly :
$$ 2 i(z+1)^{1/2} z^2\left[  z\partial_z - (\beta_{\mathfrak{F}_1}+0) + (\beta_{\mathcal{A}}+0)  - (\beta_{\mathcal{A}}+0) +\dfrac{1}{2}+ \dfrac{z }{4(z+1)}  \right]  U^{\mathfrak{F}_1}_{4,0;0} = - f^1_{1,0}(z)$$
$$ 2 i(z+1)^{1/2} z^2\left[  z\partial_z - (\beta_{\mathfrak{F}_1}+m) + (\beta_{\mathcal{A}}+0)  - (\beta_{\mathcal{A}}+0) +\dfrac{1}{2}+ \dfrac{z }{4(z+1)}  \right]  U^2_{4,0;m} $$
$$+ 2 i(z+1)^{1/2} z^2\sum_{p=0}^{j-1} d_p U^2_{4,0;p} = -f^1_{1,j}$$
Use Borel sum to  define $U_{4,0;0}^{\mathfrak{F}_1}$
$$U_{4,0;0}^{\mathfrak{F}_1} \stackrel{\text{borel sum}}{\sim}  (\mathsf{Z}')^{\beta_{\mathfrak{F}_1}} \sum_m  (\mathsf{Z}')^m U^{\mathfrak{F}_1}_{4,0;m}$$
Denote $f_2 =  \mathsf{P}_0  U_{4,0;0}^{\mathfrak{F}_1} + f_1^1 $ 

\emph{\textbf{Properties of constructed function and resulting error} :}
Since we are dealing with regular singular ODE, we can choose solution $U_{4,0;m}^{\mathfrak{F}_1}$ whose behavior $z = 0$ is determined by that of $f_{1,m}$, which vanishes to infinite order in $\mathfrak{F}_1\cap\mathcal{A}$ at $z = 0$. Also we can assume $U^{\mathfrak{F}_1}_{4,m}$ are supported away from $\mathfrak{F}_2$.
$$U^{\mathfrak{F}_1}_{4,0;m} \in z^{\infty} (\mathsf{Z}')^{\beta_{\mathfrak{F}_1}}$$
$$\Rightarrow U^{\mathfrak{F}_1}_{4,0} \in \rho_{\mathfrak{B}_1}^{\infty} \rho_{\mathfrak{F}_1}^{\beta_{\mathfrak{F}_1} }   ; \ \ \text{supported away from $\mathfrak{F}_2$}$$
$$f_2 \in  \rho_{\mathfrak{B}_1}^{\infty} \rho_{\mathfrak{F}_1}^{\infty}   ; \ \ \text{supported away from $\mathfrak{F}_2$}$$
From the properties inherited by $f_2$, we can consider $f_2$ to live on the compactified product space $[0,1]_z \times [0,1]_{z'}$. In addition $f_2\in z^{\infty} (z')^{\infty}$ vanishes rapidly as $z' \rightarrow 0$ uniformly as  $z \rightarrow 0$ since
we are away from $\mathsf{Z} = 0$ and $   ( z' \langle \mathsf{Z}\rangle )^{\infty} (\mathsf{Z}^{-1} )^{\infty}$ .

\item  \textbf{Step 2} : Since we have been able to blow down and consider the error on the product space, now we solve for $\mathsf{P}_0$ for $U_{4,0}^{\mathfrak{F}_2}$ along each regular fiber \emph{(parametrized by $z'$)}. The ODE in $z,z'$ is
$$2i (z + 1)^{1/2} \dfrac{z^2 (z')^2}{(z^2+(z')^2)^{1/2}} \Big[ z\partial_z + (\beta_{\mathcal{A}}+0) \dfrac{z}{(z^2+(z')^2)^{1/2}} + \dfrac{1}{2} + \dfrac{z}{4(z+1)}\Big]$$
the second coefficient is not quite 'nice'. However this is the conjugated version of a nicer operator, ie we will instead work with:
$$2i (z + 1)^{1/2}\dfrac{z^2 (z')^2}{(z^2+(z')^2)^{1/2}} \Big[ z\partial_z  + \dfrac{1}{2} + \dfrac{z}{4(z+1)}\Big]  
\big(1 + (\dfrac{z}{z'})^2 \big)^{\tfrac{\beta_{\mathcal{A}}}{2}} U_{4,0}^{\mathfrak{F}_2} $$
$$=\big(1 + (\dfrac{z}{z'})^2 \big)^{\tfrac{\beta_{\mathcal{A}}}{2}} \big( f_2 + f_1^2+  F^2_{3,0}\big)$$
$F^2_{3,0}$ is supported away from $\mathfrak{B}_1$ and has the following property at $\mathfrak{F}_2$ : $F^2_{3,0} \in \rho_{\mathfrak{F}_2}^{\infty} $. $f_2$ also vanishes to infinite order at $z = 0$. Same reasoning as for $f_2$, we have $f_1^2$ supported away from $\mathfrak{F}_1$ and vanishes to infinite order at $\mathfrak{B}_1$ and $\mathfrak{F}_2$ hence on the space $[0,1]_z\times [0,1]_{z'}$ it is in $z^{\infty}(z')^{\infty}$. In short, the RHS has the property : $z^{\infty} (z')^{\infty}$. Thus we can solve for a unique $(1 + (\dfrac{z}{z'})^2)^{\beta_{\mathcal{A}}/2} U_{4,0}^{\mathfrak{F}_2} $ solution which satisfies
 $$    (1 + (\dfrac{z}{z'})^2)^{\beta_{\mathcal{A}}/2} U_{4,0}^{\mathfrak{F}_2}  \sim z^{\infty} (z')^{\infty}$$

\end{itemize}
\textbf{Conclusion : }  We have constructed a polyhomogeneous function on $\mathcal{A}$ 
$$U_{4,0} := U^{\mathfrak{B}_1}_{4,0} + U^{\mathfrak{F}_1}_{4,0} + U^{\mathfrak{F}_2}_{4,0}$$
$$\text{with}\  \mathsf{P}_0 U_{4,0} = -F_{3,0}$$
The behavior of $U_{4,0}$ on $\mathcal{A}$ is determined by the first two terms, more specifically:
$$ U_{4,0} \in    \rho_{\mathfrak{B}_1}^{ \beta_{\mathfrak{B}_1}}  \rho_{\mathfrak{F}_2}^{\infty} \rho_{\mathfrak{F}_1}^{\beta_{\mathfrak{B}_1} +\tfrac{1}{2}} + \rho_{\mathfrak{B}_1}^{\infty} \rho_{\mathfrak{F}_1}^{\beta_{\mathfrak{F}_1} } \rho_{\mathfrak{F}_2}^{\infty}
 $$
 $$\Rightarrow U_{4,0} \in  \rho_{\mathfrak{F}_2}^{\infty} \rho_{\mathfrak{B}_1}^{\beta_{\mathfrak{B}_1}} \rho_{\mathfrak{F}_1}^{\beta_{\mathfrak{F}_1}+\tfrac{1}{2}}  $$
 $$\text{where}\ \ \beta_{\mathfrak{F}_1} =  \beta_{\mathfrak{B}_1} +\tfrac{1}{2}+1 ; \beta_{\mathfrak{B}_1}  = \alpha_{\mathfrak{B}_1} +2  - 2 ; \alpha_{\mathfrak{B}_1} = \gamma_{\mathfrak{B}_1} +2 - 2 ; \gamma_{\mathfrak{B}_1} = -2$$

\item \textbf{HIGHER TRANSPORT EQUATIONS} :   come from coefficients of $\rho_{\mathcal{A}}^{\beta_{\mathcal{A}}+k}$:
\begin{equation}\label{hitlag2}
\begin{aligned}
 &2 i(\mathsf{Z} z' +1)^{1/2} \dfrac{( \mathsf{Z} z' )^2}{\langle \mathsf{Z}\rangle}\left[  \mathsf{Z}\partial_{\mathsf{Z}}+ (\beta_{\mathcal{A}}+k) (\tfrac{\mathsf{Z}}{\langle \mathsf{Z}\rangle})^2 - (\beta_{\mathcal{A}}+k) +\dfrac{1}{2}
 + \dfrac{\mathsf{Z} z' }{4(\mathsf{Z} z' +1)}  \right]  U_{4,k}  \\
 &+ \dfrac{\mathsf{Z}^2}{\langle \mathsf{Z}\rangle^2}(z')^2 \Big[  \left( \mathsf{Z}\partial_{\mathsf{Z}}+ (\beta_{\mathcal{A}}+k-1) (\tfrac{\mathsf{Z}}{\langle \mathsf{Z}\rangle})^2 - (\beta_{\mathcal{A}}+k-1)\right)^2 +  (\lambda - \dfrac{n^2}{4} )  \Big] U_{4,k-1} = -F_{3,k}
\end{aligned}
\end{equation}
Using the exactly same method we can construct a polyhomogeneous conormal solution $U_{4,k-1}$ on $\mathcal{A}$ with the property
$$ U_{4,k-1} \in \rho_{\mathfrak{F}_2}^{\infty} \rho_{\mathfrak{B}_1}^{\beta_{\mathfrak{B}_1} + k-1} \rho_{\mathfrak{F}_1}^{\beta_{\mathfrak{F}_1}+k-1}$$
 $$\text{where}\ \ \beta_{\mathfrak{F}_1} =  \beta_{\mathfrak{B}_1} +\tfrac{1}{2}+1 ; \beta_{\mathfrak{B}_1}  = \alpha_{\mathfrak{B}_1} +2 - 2 ; \alpha_{\mathfrak{B}_1} = \gamma_{\mathfrak{B}_1} +2 - 2 ; \gamma_{\mathfrak{B}_1} = -2$$

\end{enumerate}

 Now we can construct $\tilde{U}_4$ on $\mathfrak{M}_{b,h}$ by taking the Borel sum
$$\tilde{U}_4 \stackrel{\text{borel sum}}{\sim} \rho_{\mathcal{A}}^{\beta_{\mathcal{A}}}\sum_k (\rho_{\mathcal{A}})^k U_{4,k} \Rightarrow \tilde{U}_4 \in \rho_{\mathcal{A}}^{\beta_{\mathcal{A}}} \rho_{\mathfrak{F}_2}^{\infty} \rho_{\mathfrak{B}_1}^{\beta_{\mathfrak{B}_1} } \rho_{\mathfrak{F}_1}^{\beta_{\mathfrak{F}_1}}$$
$$\beta_{\mathcal{A}} = \tfrac{1}{2}+\alpha_{\mathfrak{F}_2} +1$$
Next define 
$$U_4 = \exp(i\varphi_{\text{out}}) \tilde{U}_4; \  E_4 = \QQ U_4 + E_3 $$ 
 $$\Rightarrow E_4 = \exp(i\varphi_{\text{out}})F_4; \ \ F_4 \in \rho_{\mathcal{A}}^{\infty} \rho_{\mathfrak{F}_2}^{\infty} \rho_{\mathfrak{B}_1}^{\beta_{\mathfrak{B}_1}+2 } \rho_{\mathfrak{F}_1}^{\infty}$$
 $$\text{where}\ \ \beta_{\mathfrak{F}_1} =  \beta_{\mathfrak{B}_1} +\tfrac{1}{2}+1 ; \beta_{\mathfrak{B}_1}  = \alpha_{\mathfrak{B}_1}+2 - 2 ; \alpha_{\mathfrak{B}_1} = \gamma_{\mathfrak{B}_1} +2 - 2 ; \gamma_{\mathfrak{B}_1} = -2$$

\subsubsection{\textbf{Summary of properties of resulting error close to} $\mathcal{A}$ \textbf{from step 3a - 3c}}\
 
    The error produced from the above steps : 
   $$   E_3^{\mathcal{R}}  \in \rho_{\mathfrak{B}_1}^{\alpha_{\mathfrak{B}_1}+2} \rho_{\mathcal{R}}^{\sqrt{\tfrac{n^2}{4}-\lambda}+1} \rho_{\mathcal{A}}^{\infty} \rho_{\mathfrak{F}_2}^{\infty}S^1_0(z' , \tilde{t}) $$
   $$   \text{where} \ \   \alpha_{\mathfrak{B}_1 } = \gamma_{\mathfrak{B}_1} +2 - 2; \ \gamma_{\mathfrak{B}_1} = -2  $$

      \begin{remark}
    Note that all of the above-mentioned errors vanish to infinite order at $\mathcal{A}$, hence can be considered to live on the blown down space, ie their portion close to $\mathfrak{B}_1$ can be consider to live on the product space
    $$[0,\delta)_{\rho_{\mathfrak{B}_1}} \times \mathsf{B}_1$$
    where $\mathsf{B}_1$ is the blowndown version of $\mathfrak{B}_1$ where we blow-down $\mathcal{A}\cap \mathfrak{B}_1$ in $\mathfrak{B}_1$. In fact we will proceed to remove the error at b- front face $\mathcal{F}$ to be able to consider the resulting error to live on $B_1 = \RR^+_{\tilde{t}} \times \RR^+_{\tilde{\nu}'}$.
       \end{remark}


\subsection{Step 4 : remove the error at the b front face $\mathcal{F}$. }\

   We will only work with the portion of the error supported close the $\mathcal{F}$, which in this coordinate in supported away from $\mathsf{Z} = 0$ and $\mathsf{Z}^{-1} = 0$. This error comes from after step 1. Write the error as 
   $$\tilde{E}_1 = \tilde{E}_1^{\mathfrak{B}_1} + \tilde{E}_1^{\tilde{\mathfrak{B}}_1} + \tilde{E}_1^{\text{interior}}$$
   where the superscripts indicate the neighborhood of support. Each of the terms is dealt with in the same way. We write out the details in dealing with $\tilde{E}_{1,\mathfrak{B}_1}$ i.e the term which is supported close to $\mathfrak{B}_1$.
   
   \textbf{Product structure near} $\mathcal{F}$ : $$\mathcal{F}_{h,\mathsf{Z}} \times \RR^+_{\rho_{\mathcal{F}}},$$ where $\rho_{\mathcal{F}}$ is a choice of the global defining function for b blown-up front face $\mathcal{F}$ (\emph{global in the sense the function is valid at both $\mathcal{R}$ and $\mathcal{L}$ end}).
   For a choice of global defining function, we use a function such that in neighborhood away from $\mathcal{L}$, $\rho_{\mathcal{F}}$ has the form
 $\rho_{\mathcal{F}} = \tilde{\mu}' \langle \mathsf{Z}\rangle $
   and in coordinates  valid away from right boundary $\mathcal{R}$, $\rho_{\mathcal{F}} = \tilde{\mu} \langle \mathsf{Z}'\rangle$.\newline
   
  \textbf{ A choice of  coordinate away from} $\mathcal{L}$ : 
   $$ \tilde{\mu}'  ( = \dfrac{z'}{h}) , \mathsf{Z}, h$$
   while away from $\mathcal{R}$ we use
    $$ \tilde{\mu}\  ( =  \dfrac{z}{h}),  \mathsf{Z}' = \mathsf{Z}^{-1}, h$$
    The form of the operator the coordinates in these coordinates
   $$ h^2 \big[  (\mathsf{Z}\partial_{\mathsf{Z}})^2 + (\lambda - \dfrac{n^2}{4}) + h(\mathsf{Z}\tilde{\mu}')^3 + (\mathsf{Z}\tilde{\mu}')^2\big]$$
   
    \textbf{Prep Computation :} Since any power of $\tilde{\mu}'$ commutes with $\QQ_b$, conjugation $\QQ_b$ by $\rho_{\mathcal{F}}^j$ amounts to conjuation of $\QQ_b$ by $\langle \mathsf{Z}\rangle$ :
  \begin{align*}
 &(\rho_{\mathcal{F}})^{-j} \QQ_b (\rho_{\mathcal{F}})^j\\
 = &\  h^2 \langle \mathsf{Z}\rangle^{-j} 
\big[  (\mathsf{Z}\partial_{\mathsf{Z}})^2 + (\lambda - \dfrac{n^2}{4}) + h(\mathsf{Z}\tilde{\mu}')^3 + (\mathsf{Z}\tilde{\mu}')^2\big] \langle \mathsf{Z}\rangle^j \\
   = &\ h^2 \big[(\mathsf{Z}\partial_{\mathsf{Z}})^2 
    + 2 j\langle \mathsf{Z}\rangle^{-2} \mathsf{Z}^2 (\mathsf{Z}\partial_{\mathsf{Z}}) + 2j \mathsf{Z}^2 \langle \mathsf{Z}\rangle^{-2} + j (j-2) \langle \mathsf{Z}\rangle^{-4} \mathsf{Z}^4 +  (\lambda - \dfrac{n^2}{4}) + h(\mathsf{Z}\tilde{\mu}')^3 + (\mathsf{Z}\tilde{\mu}')^2 \big] \\
    = &\ h^2 \Big[(\mathsf{Z}\partial_{\mathsf{Z}})^2 
    + 2 j(\dfrac{\mathsf{Z}}{\langle \mathsf{Z}\rangle}\big)^2  (\mathsf{Z}\partial_{\mathsf{Z}}) + 2j(\dfrac{\mathsf{Z}}{\langle \mathsf{Z}\rangle}\big)^2  + j (j-2) (\dfrac{\mathsf{Z}}{\langle \mathsf{Z}\rangle}\big)^4  +  (\lambda - \dfrac{n^2}{4}) + h(\dfrac{\mathsf{Z}}{\langle \mathsf{Z}\rangle}\big)^3 \rho_{\mathcal{F}}^3 + (\dfrac{\mathsf{Z}}{\langle \mathsf{Z}\rangle}\big)^2 \rho_{\mathcal{F}}^2 \Big] 
\end{align*}
   Note that $\tfrac{\mathsf{Z}}{\langle \mathsf{Z}\rangle }$ is a global function on  each compactified fiber of $\mathcal{F}$.
   
   \textbf{Construction:} We looked for $U_5^{\mathfrak{B}_1}$ of the form $U_5^{\mathfrak{B}_1} \sim  \rho_{\mathcal{F}}^0 \sum_j \rho_{\mathcal{F}}^j U_{5j}^{\mathfrak{B}_1}$ and supported close to $\mathcal{F}$ such that
$$\QQ_b U_5^{\mathfrak{B}_1} + \tilde{E}_1^{\mathfrak{B}_1} \in \rho_{\mathcal{F}}^{\infty}\rho_{\mathfrak{B}_1}^0 \rho_{\mathcal{R}}^{\sqrt{\tfrac{n^2}{4} - \lambda}+1} \rho_{\mathcal{L}}^{\sqrt{\tfrac{n^2}{4} - \lambda}+1}  \mathcal{C}^{\infty}(\mathfrak{M}_{b,h})$$
         Expand $ \tilde{E}_1^{\mathfrak{B}_1}$ at the b front face $\mathcal{F}$:
   $$ \tilde{E}_1^{\mathfrak{B}_1} \stackrel{\text{taylor expand}}{\sim}  \rho_{\mathcal{F}}^0 \sum_j \rho_{\mathcal{F}}^j \tilde{E}_{1,j}^{\mathfrak{B}_1} ;\ \ \tilde{E}_{1,j}^{\mathfrak{B}_1} \in h^0 \mathcal{C}^{\infty}(\mathsf{Z}, z')$$
 Solving in series and using the prep computation, we obtain the ODE-s which $U_{5,j}^{\mathfrak{B}_1}$ need to satisfy. Note that these $U_{5,j}^{\mathfrak{B}_1}$ are global solutions living on each fiber $\{h = \text{constant}\}$ on $\mathcal{F}$ as follows : 
  \begin{equation}\label{Mellin1}
    h^2 \Big[(\mathsf{Z}\partial_{\mathsf{Z}})^2 
    + 2(\gamma_{\mathcal{F}}+0)(\dfrac{\mathsf{Z}}{\langle \mathsf{Z}\rangle}\big)^2  (\mathsf{Z}\partial_{\mathsf{Z}}) + 2(\gamma_{\mathcal{F}}+0)(\dfrac{\mathsf{Z}}{\langle \mathsf{Z}\rangle}\big)^2  + (\gamma_{\mathcal{F}}+0)(\gamma_{\mathcal{F}}+0-2) (\dfrac{\mathsf{Z}}{\langle \mathsf{Z}\rangle}\big)^4  +  (\lambda - \dfrac{n^2}{4})  \Big] U_{5,0}^{\mathfrak{B}_1} 
   = -\tilde{E}_{1,0}^{\mathfrak{B}_1}  ;\ \   \gamma_{\mathcal{F}} = 0
   \end{equation}
   \begin{equation}
   \begin{aligned}
    &h^2 \Big[(\mathsf{Z}\partial_{\mathsf{Z}})^2 
    + 2 (\gamma_{\mathcal{F}}+j)(\dfrac{\mathsf{Z}}{\langle \mathsf{Z}\rangle}\big)^2  (\mathsf{Z}\partial_{\mathsf{Z}}) + 2(\gamma_{\mathcal{F}}+j)(\dfrac{\mathsf{Z}}{\langle \mathsf{Z}\rangle}\big)^2  + (\gamma_{\mathcal{F}}+j)(\gamma_{\mathcal{F}}+j-2) (\dfrac{\mathsf{Z}}{\langle \mathsf{Z}\rangle}\big)^4  +  (\lambda - \dfrac{n^2}{4})  \Big] U_{5,j}^{\mathfrak{B}_1} \\
    &= -\tilde{E}_{1,j}^{\mathfrak{B}_1} -  h^3(\dfrac{\mathsf{Z}}{\langle \mathsf{Z}\rangle}\big)^3 U_{5,j-3}^{\mathfrak{B}_1} - h^2(\dfrac{\mathsf{Z}}{\langle \mathsf{Z}\rangle}\big)^2 U_{5,j-2}^{\mathfrak{B}_1}
    \end{aligned}
    \end{equation}
    
   However for simpler computation, instead of solving directly for $U_{5,j}^{\mathfrak{B}_1}$ we will look solve for $ \langle \mathsf{Z}\rangle^{\gamma_{\mathcal{F}}+j} U_{5,j}^{\mathfrak{B}_1}$ which in turn satisfy simpler ODE in $\mathsf{Z}$.
    $$h^2 \Big[(\mathsf{Z}\partial_{\mathsf{Z}})^2 
   +  (\lambda - \dfrac{n^2}{4})  \Big] \langle \mathsf{Z}\rangle^{\gamma_{\mathcal{F}}+0} U_{5,0}^{\mathfrak{B}_1} 
    = -\langle \mathsf{Z}\rangle^{\gamma_{\mathcal{F}}+0} \tilde{E}_{1,0}^{\mathfrak{B}_1} $$
  $$h^2 \Big[(\mathsf{Z}\partial_{\mathsf{Z}})^2 
 +  (\lambda - \dfrac{n^2}{4})  \Big] U_{5,j}^{\mathfrak{B}_1} = -\langle \mathsf{Z}\rangle^{\gamma_{\mathcal{F}}+j}\tilde{E}_{1,j}^{\mathfrak{B}_1} -  h^3\mathsf{Z}^3   \langle \mathsf{Z}\rangle^{\gamma_{\mathcal{F}}+j-3} U_{5,j-3}^{\mathfrak{B}_1} - h^2\mathsf{Z}^2 \langle \mathsf{Z}\rangle^{\gamma_{\mathcal{F}}+j-2}U_{5,j-2}^{\mathfrak{B}_1}$$
  \begin{itemize}
  \item   We will revoke the facts about Mellin transform from remark (\ref{Mellin}) (\emph{at the end of the section}) to obtain specific global solution for $\langle \mathsf{Z}\rangle^{\gamma_{\mathcal{F}}+0}U_{5,0}^{\mathfrak{B}_1}$ that solves
   $$\begin{cases}  h^2 \Big[(\mathsf{Z}\partial_{\mathsf{Z}})^2 
     +  (\lambda - \dfrac{n^2}{4})  \Big] \langle \mathsf{Z}\rangle^{\gamma_{\mathcal{F}}+0}U_{5,0}^{\mathfrak{B}_1} = -\langle \mathsf{Z}\rangle^{\gamma_{\mathcal{F}}+0}\tilde{E}_{1,0}^{\mathfrak{B}_1} ;\ \ \gamma_{\mathcal{F}} = 0\\
     \mathsf{Z}\sim 0 :\ \  \langle \mathsf{Z}\rangle^{\gamma_{\mathcal{F}}+0} U_{5,0}^{\mathfrak{B}_1} =  \mathsf{Z}^{\sqrt{\tfrac{n^2}{4}-\lambda}} h^{-2} \mathcal{C}^{\infty}(\mathcal{F})  \\
      \mathsf{Z}\rightarrow\infty:   \langle \mathsf{Z}\rangle^{\gamma_{\mathcal{F}}+0} U_{5,0}^{\mathfrak{B}_1} = \mathsf{Z}^{-\sqrt{\tfrac{n^2}{4}-\lambda}} h^{-2} \mathcal{C}^{\infty}(\mathcal{F}) \end{cases}$$
   \emph{In this case we do not have poles contribution from the inhomongeous term } since $\tilde{E}_{1,0}$ has trivial behavior at both ends.  
   
 \item   $\langle \mathsf{Z}\rangle^{\gamma_{\mathcal{F}}+1} U_{5,1}^{\mathfrak{B}_1}$ satisfies the same equation with $\langle \mathsf{Z}\rangle^{\gamma_{\mathcal{F}}+0}\tilde{E}_{1,0}$ switched to  $\langle \mathsf{Z}\rangle^{\gamma_{\mathcal{F}}+1}\tilde{E}_{1,1}$
  
\item   To deal with higher order equation for $U_{5,j}$, we consider the ODE that $U_{5,2}$ which has to satisfy. 
   \begin{equation}\label{higherMellineqn}
   \begin{cases}
   h^2 \Big[(\mathsf{Z}\partial_{\mathsf{Z}})^2 
 +  (\lambda - \dfrac{n^2}{4})  \Big] \langle \mathsf{Z}\rangle^{\gamma_{\mathcal{F}}+2}U^{\mathfrak{B}_1}_{5,2} = -\langle \mathsf{Z}\rangle^{\gamma_{\mathcal{F}}+2}\tilde{E}^{\mathfrak{B}_1}_{1,2} - h^2\mathsf{Z}^2 \langle \mathsf{Z}\rangle^{\gamma_{\mathcal{F}}+2-2}U^{\mathfrak{B}_1}_{5,0}\\
    \mathsf{Z}\sim 0 :\ \  \langle \mathsf{Z}\rangle^{\gamma_{\mathcal{F}}+0} U_{5,0}^{\mathfrak{B}_1} =  \mathsf{Z}^{\sqrt{\tfrac{n^2}{4}-\lambda}} h^{-2} \mathcal{C}^{\infty}(\mathcal{F})  \\
      \mathsf{Z}\rightarrow\infty:   \langle \mathsf{Z}\rangle^{\gamma_{\mathcal{F}}+0} U_{5,0}^{\mathfrak{B}_1} = \mathsf{Z}^{-\sqrt{\tfrac{n^2}{4}-\lambda}} h^{-2} \mathcal{C}^{\infty}(\mathcal{F})  \ (**)
    \end{cases}
    \end{equation}
   Also use remark (\ref{Mellin}) but now keep in mind that the right hand side has nontrivial behavior at both ends. With the same notation used in the remark (\ref{Mellin}) we have
 $\beta = \sqrt{\alpha}+2$ while $\gamma = \sqrt{\alpha} - 2 $ so using the Mellin transform we can get a global solution
   $$ \langle \mathsf{Z}\rangle^{?+2}U^{\mathfrak{B}_1}_{5,2}\sim \mathsf{Z}^{\sqrt{\alpha}} + \mathsf{Z}^{ \sqrt{\alpha}+2}      \sim    \mathsf{Z}^{\sqrt{\alpha}}             $$
   $$\langle \mathsf{Z}\rangle^{?+2}U^{\mathfrak{B}_1}_{5,2} \sim \mathsf{Z}^{-\sqrt{\alpha}} + \mathsf{Z}^{-\sqrt{\alpha} + 2}  \sim \mathsf{Z}^{-\sqrt{\alpha} + 2} $$

 More generally for $U_{5,j}$
 $$ \begin{cases}
    &h^2 \Big[(\mathsf{Z}\partial_{\mathsf{Z}})^2  +  (\lambda - \dfrac{n^2}{4})  \Big] \mathsf{Z}^{j+\gamma_{\mathcal{F}}}U^{\mathfrak{B}_1}_{5,j} = -\langle \mathsf{Z}\rangle^{\gamma_{\mathcal{F}}+j}\tilde{E}_{1,j} -  h^3\mathsf{Z}^3   \langle \mathsf{Z}\rangle^{\gamma_{\mathcal{F}}+j-3} U^{\mathfrak{B}_1}_{5,j-3} - h^2\mathsf{Z}^2 \langle \mathsf{Z}\rangle^{\gamma_{\mathcal{F}}+j-2}U^{\mathfrak{B}_1}_{5,j-2}\\
          &\mathsf{Z} \sim 0  :\ \ \ \mathsf{Z}^{j+\gamma_{\mathcal{F}}}U^{\mathfrak{B}_1}_{5,j} = \mathsf{Z}^{\sqrt{\tfrac{n^2}{4}-\lambda} +j} h^{-2} \mathcal{C}^{\infty}(\mathcal{F})   (*)   \\  &\mathsf{Z} \rightarrow \infty :\ \ \ \mathsf{Z}^{j+\gamma_{\mathcal{F}}} U^{\mathfrak{B}_1}_{5,j}  = \mathsf{Z}^{-\sqrt{\tfrac{n^2}{4}-\lambda} + j}  h^{-2} \mathcal{C}^{\infty}(\mathcal{F})  (**) \end{cases}$$
 With the same notation used in the remark (\ref{Mellin}), for this case, $\beta = \sqrt{\alpha} +2 $ while $\gamma = \sqrt{\alpha} - j $ so  using the Mellin transform we can get a global solution $U^{\mathfrak{B}_1}_{5,2}$ with 
 \begin{align*}
  \mathsf{Z} \sim 0 \ :&\ \ \langle \mathsf{Z}\rangle^{\gamma_{\mathcal{F}}+j}U^{\mathfrak{B}_1}_{5,j} \sim    h^{-2}\mathsf{Z}^{\sqrt{\alpha}}             \mathcal{C}^{\infty}(\mathcal{F}) \\
  \mathsf{Z} \rightarrow \infty \ &:\ \ \langle \mathsf{Z}\rangle^{\gamma_{\mathcal{F}}+j}U^{\mathfrak{B}_1}_{5,j}  \sim h^{-2}\mathsf{Z}^{-\sqrt{\alpha} + j}  \mathcal{C}^{\infty}(\mathcal{F}) 
  \end{align*}
  \end{itemize}

  Now take   $$  U_5^{\mathfrak{B}_1} \stackrel{\text{borel sum}}{\sim} \sum_j (\rho_{\mathcal{F}})^j U_{5,j}^{\mathfrak{B}_1}$$
   and thus $U_5^{\mathfrak{B}_1}$ inherits the asymptotic properties :
    \begin{align*}
   & \text{At}\  \mathcal{R} \cap \mathcal{F}\  \text{ie where}\ \mathsf{Z} = 0 
  :  U_5^{\mathfrak{B}_1} \in h^{-2}  \mathsf{Z}^{\sqrt{\tfrac{n^2}{4} - \lambda}} \mathcal{C}^{\infty}(\rho_{\mathcal{F}}, \mathsf{Z} , h) \\
   & \text{At}\  \mathcal{L} \cap \mathcal{F}\  \text{ie where}\ \mathsf{Z}^{-1} = 0 
  :  U_5^{\mathfrak{B}_1} \in h^{-2}  \mathsf{Z}^{-\sqrt{\tfrac{n^2}{4} - \lambda}} \mathcal{C}^{\infty}(\rho_{\mathcal{F}}, \mathsf{Z} , h) \\
   \end{align*}    
   With superscript indicating the support, the error resulting from this step has the properties
   \begin{align*}
   & \text{At}\  \mathcal{R} \cap \mathcal{F}\  \text{ie where}\ \mathsf{Z} = 0 
  :  E_6^{\mathfrak{B}_1} \in h^0  \mathsf{Z}^{\sqrt{\tfrac{n^2}{4} - \lambda}+1} \mathcal{C}^{\infty}(\rho_{\mathcal{F}}, \mathsf{Z} , h) \\
   & \text{At}\  \mathcal{L} \cap \mathcal{F}\  \text{ie where}\ \mathsf{Z}^{-1} = 0 
  :  E_6^{\mathfrak{B}_1} \in h^0  \mathsf{Z}^{-\sqrt{\tfrac{n^2}{4} - \lambda}-1} \mathcal{C}^{\infty}(\rho_{\mathcal{F}}, \mathsf{Z} , h) \\
   \end{align*}    

 \textbf{Similarly, to take care of $\tilde{E}_1^{\tilde{\mathfrak{B}}_1}$}  we use $\rho_{\mathcal{F}} = \tilde{\mu}' \langle \mathsf{Z}\rangle $
   and in coordinates  valid away from right boundary $\mathcal{R}$, $\rho_{\mathcal{F}} = \tilde{\mu} \langle \mathsf{Z}'\rangle$.\newline
   A choice of  coordinate away from $\mathcal{L}$ is $ \tilde{\mu}'  ( = \dfrac{z'}{h}) , \mathsf{Z}, h$ while away from $\mathcal{R}$ we use $ \tilde{\mu}\  ( =  \dfrac{z}{h}),  \mathsf{Z}' = \mathsf{Z}^{-1}, h$. The form of the operator the coordinates in these coordinates
   $$ h^2 \big[  (\mathsf{Z}\partial_{\mathsf{Z}})^2 + (\lambda - \dfrac{n^2}{4}) + h(\mathsf{Z}\tilde{\mu}')^3 - (\mathsf{Z}\tilde{\mu}')^2\big]$$
As for $\tilde{E}_1^{\text{interior}}$, we use coordinates . (Note $S = Z - Z_0$) 
$$ \mathsf{Z} = \dfrac{S}{S'} , S' , Z_0$$
the defining function for $\mathcal{F}$ is $\rho_{\mathcal{F}} = S' \langle \mathsf{Z}\rangle $ while the operator has the form
$$ (\mathsf{Z}\partial_{\mathsf{Z}})^2 + \lambda - \dfrac{n^2}{4} + (S'\mathsf{Z})^3 + Z_0 (\mathsf{Z} S')^2$$
By the same method, we construct $U_5^{\tilde{\mathfrak{B}}_1}$ and  $U_5^{\text{interior}}$. See below for their properties.

\subsubsection{\textbf{Summarizing the properties of newly constructed function and the resulting error}}
With superscript indicating the support
$$U_5 = U_5^{\mathfrak{B}_1} + U_5^{\tilde{\mathfrak{B}}_1} + U_5^{\text{interior}}$$
$$ E_6 = E_5^{\mathfrak{B}_1} + E_5^{\tilde{\mathfrak{B}}_1}  = \QQ \big( U_5^{\mathfrak{B}_1} + U_5^{\tilde{\mathfrak{B}}_1} + U_5^{\text{interior}}\big) $$
 \begin{equation}
 \begin{aligned}
&  U_5^{\mathfrak{B}_1} \in \rho_{\mathcal{F}}^{\gamma_{\mathcal{F}}}\rho_{\mathfrak{B}_1}^{-2} \rho_{\mathcal{R}}^{\sqrt{\tfrac{n^2}{4} - \lambda}} \rho_{\mathcal{L}}^{\sqrt{\tfrac{n^2}{4} - \lambda}}  \mathcal{C}^{\infty}(\mathfrak{M}_{b,h}); \ \ \gamma_{\mathcal{F}}  = 0\\
 &E_6^{\mathfrak{B}_1} \in \rho_{\mathcal{F}}^{\infty}\rho_{\mathfrak{B}_1}^0 \rho_{\mathcal{R}}^{\sqrt{\tfrac{n^2}{4} - \lambda}} \rho_{\mathcal{L}}^{\sqrt{\tfrac{n^2}{4} - \lambda}}  \mathcal{C}^{\infty}(\mathfrak{M}_{b,h}); \\
 &  U_5^{\tilde{\mathfrak{B}}_1} \in \rho_{\mathcal{F}}^{\gamma_{\mathcal{F}}}\rho_{\tilde{\mathfrak{B}}_1}^{-2} \rho_{\mathcal{R}}^{\sqrt{\tfrac{n^2}{4} - \lambda}} \rho_{\mathcal{L}}^{\sqrt{\tfrac{n^2}{4} - \lambda}}  \mathcal{C}^{\infty}(\mathfrak{M}_{b,h}); \\
 &E_6^{\tilde{\mathfrak{B}}_1} \in \rho_{\mathcal{F}}^{\infty}\rho_{\tilde{\mathfrak{B}}_1}^0 \rho_{\mathcal{R}}^{\sqrt{\tfrac{n^2}{4} - \lambda}} \rho_{\mathcal{L}}^{\sqrt{\tfrac{n^2}{4} - \lambda}}  \mathcal{C}^{\infty}(\mathfrak{M}_{b,h}); \\
& U_5^{\text{interior}} \in \rho_{\mathcal{F}}^{\gamma_{\mathcal{F}}} \rho_{\mathcal{R}}^{\sqrt{\tfrac{n^2}{4} - \lambda}} \rho_{\mathcal{L}}^{\sqrt{\tfrac{n^2}{4} - \lambda}}  \mathcal{C}^{\infty}(\mathfrak{M}_{b,h}) ; \ \ \supp U_5^{\text{interior}} \ \text{only intersects} \ \mathcal{F}, \mathcal{R}, \mathcal{L}
  \end{aligned}
 \end{equation}

   \
   \

 
 \begin{remark}\label{Mellin}[on solving regular singular ODE using Mellin transform]\
 
 Firstly we state the definition and some properties of the Mellin transform. For $g \in \mathcal{C}^{\infty}_0(\RR^+)$, the Mellin transform of $g$ is defined as
 $$Mg (s) = \int_0^{\infty} g(\mathsf{x})  \mathsf{x}^{is} \dfrac{d\mathsf{x}}{\mathsf{x}} $$
 $$s = \Rea s + i \Img s; \  M : \mathsf{x}^{\delta} L^2 (\RR^+; d\mathsf{x}) \stackrel{\text{isomorphism}}{\rightarrow} L^2 \big( \{ \Img s = \delta -\tfrac{1}{2}\} ; d\Rea s  \big)$$
 $$M(\mathsf{x}\partial_{\mathsf{x}} g) = -i s Mg $$
 The inverse Mellin transform is given by integrating along a horizontal line, provided this integral exists
 $$s = \Rea s + i \Img s; \ M^{-1} ( \mathsf{g} (s)) = \int_{\Img s = \text{constant}} \mathsf{g} (s) \mathsf{x}^{-is} \, d\Rea s $$
 So in particular 
 $$M^{-1} : L^2 \big( \{ \Img s = 0\} ; d\Rea s  \big) \rightarrow \mathsf{x}^{\tfrac{1}{2}} L^2 (\RR^+; dx)$$
 For our purpose, we will apply the Mellin transform to solve the following ODE on the half line $[0,\infty)_{\mathsf{x}}$ with $f \in \mathcal{S}_{[0,\infty)}$, where $\mathcal{S}_{[0,\infty)}$ has support away from $0$ and vanishes rapidly at infinity.
 \begin{equation}\label{eqMellin}
  \left[(\mathsf{x} \partial_{\mathsf{x}})^2 - \alpha  \right] u = f ;\ \alpha > 0 
  \end{equation}
 Take the Mellin transform on both sides
 $$ \big((-is)^2 -\alpha \big)Mu = Mf \Rightarrow Mu  = \dfrac{-Mf}{s^2 + \alpha} $$
 and we obtain a specific solution for the ODE (\ref{eqMellin})
 $$ \Rightarrow u(s) = M^{-1}_{\Img s = 0} \Big(\dfrac{-Mf}{s^2 + \alpha} \Big)
  = \int_{\Img s = 0} \mathsf{x}^{-is}  \dfrac{-Mf}{s^2 + \alpha}\, d \Rea s $$
 Note that the RHS has poles at $\pm i \sqrt{\alpha}$, $\alpha > 0$. In order to know what the behavior of $u$ is at $\mathsf{x} = 0$ and $\mathsf{x} \rightarrow \infty$, we deform the contour of integration. In this case since we assume $f$ is `nice', only the poles $\pm\sqrt{\lvert \alpha\rvert}$ will affect the asymptotic behavior of $u$. In greater details, to read off the asymptotic property when $\mathsf{x} \rightarrow 0$, deform the contour of integration above $\Img s  = 0$, ie denote by $\tilde{u}$ by
 $$\tilde{u} = M^{-1}_{\Img s = c, c >0 } \Big(\dfrac{-Mf}{s^2 + \alpha} \Big)  = \int_{\Img s = c > 0} \mathsf{x}^{-is}  \dfrac{-Mf}{s^2 + \alpha}\, d \Rea s $$
 then at $\mathsf{x} \sim 0 $, $\tilde{u} \sim \mathsf{x}^{c}$ where $c  > 0$ since $\mathsf{x}^{-i s } = \mathsf{x}^{-i \Rea s} \mathsf{x}^{\Img s}$. Also we have 
 $$ u = \tilde{u} + \Res|_{s = i\sqrt{\alpha}} \Big(\dfrac{\mathsf{x}^{-is}(-M f)} {s^2 + \alpha} \Big); \dfrac{\mathsf{x}^{-is} (-M f)} {s^2 + \alpha} = \dfrac{\mathsf{x}^{-is} (-M f)} {(s-i\sqrt{\alpha})(s+i\sqrt{\alpha})}$$
 For $c >\sqrt{\alpha}$ 
 $$u = \tilde{u} +  i\dfrac{1}{2\sqrt{\alpha}}  \mathsf{x}^{\sqrt{\alpha}} (M f)(i\sqrt{\alpha}) $$
   $\tilde{u} \sim \mathsf{x}^c, c > 0$. Since the above is true for all $c >0$, take $c$ arbitrarily large so we can make $\tilde{u}$ vanishes to arbitrary order at $\mathsf{x} = 0$. Hence at $\mathsf{x} \sim 0$ the asymptotic behavior of $u$ is determined by the second term ie 
   $$\text{At}\ \mathsf{x} \sim 0 :\ u \sim \mathsf{x}^{\sqrt{\alpha}}$$
   Similarly, to determine the behavior of $u$ as $\mathsf{x}\rightarrow \infty$, we deform the contour to arbitrarily negative end of the imaginary axis ie denote by $\tilde{\tilde{u}}$
   $$\tilde{\tilde{u}} = M^{-1}_{\Img s = -\tilde{c} , \tilde{c} > 0} \Big(\dfrac{(-M f)}{s^2 + \alpha} \Big)  = \int_{\Img s = -\tilde{c}, \tilde{c} > 0} \mathsf{x}^{-is}  \dfrac{(-M f)}{s^2 + \alpha}\, d \Rea s $$ 
   $\tilde{\tilde{u}} \sim \mathsf{x}^{-\tilde{c}} , \tilde{c} > 0$ . Since the above is true for all $\tilde{c} >0$, take $\tilde{c}$ arbitrarily large ie we can make $\tilde{u}$ vanishes to arbitrary order at $\mathsf{x}^{-1} = 0$. Hence at $\mathsf{x} \rightarrow  \infty $ the asymptotic behavior of $u$ is determined by the second term ie 
    $$\text{At}\ \mathsf{x} \rightarrow  \infty :\ u \sim \mathsf{x}^{-\sqrt{\alpha}}$$

 Thus we have found a specific solution that solves 
 \begin{equation}\label{slnMellin}
 \begin{cases}
  \Big((\mathsf{x} \partial_{\mathsf{x}})^2 - \alpha  \Big) u = f ;\ \alpha > 0 , f \in \mathcal{S}_{[0,\infty)}\\
\mathsf{x} \sim 0 : u  = \mathsf{x}^{\sqrt{\alpha}}  \mathcal{C}^{\infty}\ \
\mathsf{x} \rightarrow \infty : u =   \mathsf{x}^{-\sqrt{\alpha}} \mathcal{C}^{\infty}
 \end{cases}
 \end{equation}
\textbf{This presents another way} to view  $(\mathsf{x} \partial_{\mathsf{x}})^2 - \alpha$ which as a hyperbolic problem is a second order ODE. Instead we can view it as an global elliptic PDE and to obtain a specific solution, we impose boundary condition, which in this case comprises of conditions at $\{\mathsf{Z}=0\}$ and at $\{\mathsf{Z}^{-1} = 0\}$ (\emph{in fact, due to the kernel being symmetric we will impose the same Dirichlet boundary at both ends}).
  
  \textbf{In the case where the inhomogeneous part $f$ has asymptotic at $\mathsf{x} \sim 0$ and $\mathsf{x}^{-1} \sim 0$} . These exponent will contribute to the poles. In the case there are just simple poles then we use the same exact technique to read off the behavior of the $u$ at $\mathsf{x}\sim 0$ and $\mathsf{x}^{-1} \sim 0$, but now the difference of $u$ and $\tilde{u}$ will be a sum of residues. Heuristic computation to see the contribution of the inhomogeneous part to the poles : if $f = f_1 + f_2 $ where $f_1\sim \mathsf{x}^{\beta}$ with $\beta \neq \pm{\sqrt{\alpha}}$  and $f_1$ decays rapidly at $\infty$, while  $f_2\sim \mathsf{x}^{-\gamma}$ with $\gamma \neq \pm{\sqrt{\alpha}}$  and $f_1$ vanishes to infinite order at $\mathsf{x} =0$, $\gamma  >0$
  $$M f \sim \int_0^c \mathsf{x}^{\beta} \mathsf{x}^{is -1 } \, d\mathsf{x} + \int_c^{\infty} \mathsf{x}^{-\gamma} \mathsf{x}^{is -1 } \, d\mathsf{x} 
   \sim \dfrac{c^{\beta +is}}{\beta +is}  +  \dfrac{c^{-\gamma +is}}{-\gamma +is}  $$
   For $ Mu = \dfrac{-Mf}{s^2+\alpha}$, the RHS will now have three poles : $\pm{\sqrt{\alpha}} , i\beta, -i\gamma$. The poles contributing to the behavior of $u$ at $\mathsf{x} \sim 0$ : $i\sqrt{\gamma}, i\beta$ ie $u \sim \mathsf{x}^{\sqrt{\gamma}} + \mathsf{x}^{\beta}$ . On the other hand, the poles contributing to the behavior of $u$ at $\mathsf{x}^{-1}  \sim 0$ : $-i\sqrt{\alpha}, -i\gamma$ ie $u \sim \mathsf{x}^{-\sqrt{\alpha}} + \mathsf{x}^{-\gamma}$ .

  \end{remark}

\subsection{Step 5 :  remove the error at $\mathfrak{B}_1$ and $\tilde{\mathfrak{B}}_1$ }


\subsubsection{Prep computation}\

  The projective coordinates away from $\mathcal{A}$  are
   $$ \tilde{t} = \dfrac{\mathsf{Z}}{\tilde{\nu}'}, \tilde{\nu}' , z'$$
   Defining function for $\mathfrak{B}_1$ : 
   $$\rho_{\mathfrak{B}_1} =  z' g_{\mathfrak{B}_1}$$
   where $g_{\mathfrak{B}_1}$ has the property
   \begin{itemize}
   \item Close $\mathfrak{F}_2$ then $g_{\mathfrak{B}_1}\in  \mathcal{C}^{\infty}(\mathfrak{M}_{b,h})$.
 \item Close to $\mathfrak{F}_1$ and $\mathcal{L}$ : $g_{\mathfrak{B}_1} \in \mathsf{Z} \mathcal{C}^{\infty}(\mathfrak{M}_{b,h})$ ; $\mathsf{Z} = \tfrac{z}{z'}$.
   \item  Close to $\mathcal{R}$ and $\mathcal{F}$ : $g_{\mathfrak{B}_1} \in \tilde{\nu}' (\mathfrak{M}_{b,h})$ ; $\tilde{\nu}' = \tfrac{h}{z'}$
   \item Close to $\mathcal{L}$ and $\mathcal{F}$ : $g_{\mathfrak{B}_1} \in \mathsf{Z} \tilde{\nu} (\mathfrak{M}_{b,h})$ ; $\mathsf{Z} = \tfrac{z}{z'} ; \tilde{\nu} = \tfrac{h}{z}$.
   \end{itemize}
    
  In these coordinates $\tilde{t}, \tilde{\nu}'$ and $z'$  $\QQ_b$ has the form : 
     \begin{align*}
  & (z')^2 \left[(\tilde{\nu}')^2 (\tilde{t} \partial_{\tilde{t}})^2 + (\tilde{\nu}')^2  (\lambda - \dfrac{n^2}{4}) + (\tilde{t} \tilde{\nu}')^3 z' + (\tilde{t} \tilde{\nu}')^2 \right]\\
 & = (z' \tilde{\nu}')^2 \left[ (\tilde{t} \partial_{\tilde{t}})^2 +   (\lambda - \dfrac{n^2}{4}) + \tilde{t}^3  \tilde{\nu}' z' + \tilde{t}^2 \right]
     \end{align*}
    the operator in the bracket on $\mathfrak{B}_1$ is a Bessel equation of order $(\dfrac{n^2}{4}-\lambda)^{1/2}$.

\subsubsection{\textbf{CONSTRUCTION}}
The error results from step 3, step 4 and step 5, which we will denote by $E_{\mathfrak{B}_1}$. 
 \begin{equation}\label{erroronB1}
 \begin{aligned}
 E_{\mathfrak{B}_1} =   E^{\mathcal{R}}_3  +   E_6^{\mathfrak{B}_1} \\
 E_3^{\mathcal{R}}  \in  \rho_{\mathfrak{B}_1}^{\alpha_{\mathfrak{B}_1}+2} \rho_{\mathcal{R}}^{\sqrt{\tfrac{n^2}{4}-\lambda}+1} \rho_{\mathcal{A}}^{\infty} \rho_{\mathfrak{F}_2}^{\infty}S^1_0(z' , \tilde{t}) \\
 E_6^{\mathfrak{B}_1} \in \rho_{\mathcal{F}}^{\infty}\rho_{\mathfrak{B}_1}^0 \rho_{\mathcal{R}}^{\sqrt{\tfrac{n^2}{4} - \lambda}+1} \rho_{\mathcal{L}}^{\sqrt{\tfrac{n^2}{4} - \lambda}+1}  \mathcal{C}^{\infty}(\mathfrak{M}_{b,h})
   \end{aligned}
    \end{equation}
 \textbf{Support:}  $E_6^{\mathfrak{B}_1}$ is supported away from $\mathcal{A}$, while $E_3^{\mathcal{R}}$ is supported near $\mathcal{R}$ and away from $\mathfrak{F}_1$ and $\mathcal{A}$. Since $E_{\mathfrak{B}_1} \in \rho_{\mathcal{A}}^{\infty}\rho_{\mathcal{F}}^{\infty} \rho_{\mathfrak{F}_1}^{\infty} \mathcal{C}^{\infty}(\mathfrak{M}_{b,h})$, we can consider $E_{\mathfrak{B}_1}$ to live on $B_1$,  the blown-down space from $\mathfrak{B}_1$ obtained by blowing down $\mathcal{A}\cap \mathfrak{B}_1$ and $\mathcal{F}\cap \mathfrak{B}_1$ (  where $B_1$ is the 2-dimensional product structure \emph{represented } in figure (\ref{ffb1}) by the rectangle). See figure (\ref{ffb1}). \newline 
   Expand $E_{\mathfrak{B}_1} $ at $\mathfrak{B}_1$
 $$E_{\mathfrak{B}_1} \stackrel{\text{taylor expand}}{\sim} \rho_{\mathfrak{B}_1}^{\tilde{\gamma}_{\mathfrak{B}_1}} \sum_j \rho_{\mathfrak{B}_1}^j E_{\mathfrak{B}_1; j};\ \ \tilde{\gamma}_{\mathfrak{B}_1} = \min ( \alpha_{\mathfrak{B}_1} +2, 0 ) = 0 $$
 \textbf{Properties of} $E_{\mathfrak{B}_1,j}$ \textbf{on} $\mathfrak{B}_1$ from (\ref{erroronB1}): they vanish to infinite order at all boundary faces except at $\mathcal{L} \cap \mathfrak{B}_1$ and $\mathcal{R}\cap \mathfrak{B}_1$ where they have exponent $\sqrt{\tfrac{n^2}{4} - \lambda}$.\newline
  We solve for $U_6^{\mathfrak{B}_1}$ of the form
   $$U_6^{\mathfrak{B}_1} \sim  \rho_{\mathfrak{B}_1}^{\tilde{\gamma}_{\mathfrak{B}_1}-2} \sum_j \rho_{\mathfrak{B}_1}^j U_{6; j};\ \ \tilde{\gamma}_{\mathfrak{B}_1} = 0  $$
    However for simpler computation, in stead of conjugating by power of $\rho_{\mathfrak{B}_1}$, we conjugate by power of $z'\langle \tilde{\nu}'\rangle$, which commutes with the operator. \nl
   
  $U_{6,0} $ \textbf{has to satisfy:}   
$$ \left[  (\tilde{t} \partial_{\tilde{t}})^2 +   (\lambda - \dfrac{n^2}{4})  + \tilde{t}^2\right]g_{\mathfrak{B}_1}^{\tilde{\gamma}_{\mathfrak{B}_1} -2} U_{6,0}=  - g_{\mathfrak{B}_1}^{\tilde{\gamma}_{\mathfrak{B}_1} } ( \tilde{\nu}')^{-2} \mathsf{E}_{5,0}$$
    On the product structure $B_1 = \RR^+_{\tilde{t}} \times \RR^+_{\tilde{\nu}'}$, properties of $g_{\mathfrak{B}_1}^{\tilde{\gamma}_{\mathfrak{B}_1} -2} ( \tilde{\nu}')^{-2} \mathsf{E}_{5,0}$ at the more interesting corners:
    \begin{itemize}
    \item Away from $\mathcal{F}, \mathcal{R}$ and $\mathfrak{F}_2$, $g_{\mathfrak{B}_1} \sim \mathsf{Z} = \tilde{t} \tilde{\nu}'$. However since $\mathsf{E}_{5,0}$ vanishes rapidly at $\tilde{t}$, in the corner of $\tilde{t}^{-1} = 0$ and $\tilde{\nu}^{-1} = 0$  $g_{\mathfrak{B}_1}^{\tilde{\gamma}_{\mathfrak{B}_1} } ( \tilde{\nu}')^{-2} \mathsf{E}_{5,0} \in (\tilde{\nu}')^{-\sqrt{\tfrac{n^2}{4} -\lambda}+\tilde{\gamma}_{\mathfrak{B}_1}-2}\mathcal{S}_{\tilde{t}}$, where $\mathcal{S}_{\tilde{t}}$ denotes functions vanishing rapidly in $\tilde{t}^{-1}= 0 $ and $\tilde{t} = 0$.
    \item Away from $\mathfrak{F}_1, \mathcal{A}, \mathfrak{F}_2, \mathcal{R}$, $g \sim Z\tilde{\nu} = \tilde{\nu}'$. 
   In the corner of $\tilde{t} = 0 , (\tilde{\nu}')^{-1} = 0$, $g_{\mathfrak{B}_1}^{\tilde{\gamma}_{\mathfrak{B}_1} } ( \tilde{\nu}')^{-2} \mathsf{E}_{5,0} \in (\tilde{\nu}')^{-\sqrt{\tfrac{n^2}{4} -\lambda} +\tilde{\gamma}_{\mathfrak{B}_1}-2}\mathcal{S}_{\tilde{t}}$.
       \end{itemize}

    along each flow we can use the Hankel functions (\emph{solutions to the Bessel equation}) to obtain the specific solution to the inhomogeneous equation.

     Furthermore, since the inhomogeneous part behaves in a specific way \emph{(using the first two listed properties of $E_{\mathfrak{B}_1;j}$)} we can uniquely solve for $\big(\langle \tilde{t}\tilde{\nu}'\rangle\langle\tilde{\nu}'\rangle\big)^{\tilde{\gamma}_{\mathfrak{B}_1} -2} U_{6,0}(\tilde{t},\tilde{\nu})$ with the following properties: 
\begin{itemize}
\item For $\lambda - \tfrac{n^2}{4} \notin \ZZ$ : 
at $\tilde{t}\sim 0$, $ g_{\mathfrak{B}_1}^{\tilde{\gamma}_{\mathfrak{B}_1} -2} U_{6,0} \stackrel{\text{asymptotic}}{\sim} \tilde{t}^{\sqrt{\tfrac{n^2}{4} - \lambda}}  (\tilde{\nu}')^{-\sqrt{\tfrac{n^2}{4} -\lambda} +\tilde{\gamma}_{\mathfrak{B}_1}-2}L^{\infty}$
 \item As $\tilde{t} \rightarrow \infty$,  $g_{\mathfrak{B}_1}^{\tilde{\gamma}_{\mathfrak{B}_1} -2} U_{6,0}\stackrel{\text{asymp}}{\sim}  \exp^{-i\tilde{t}} \tilde{t}^{-1/2} (\tilde{\nu}')^{-\sqrt{\tfrac{n^2}{4} -\lambda} +\tilde{\gamma}_{\mathfrak{B}_1}-2} S^0(\tilde{t})_{\tilde{\nu}'}$; 
  where $S^0(\tilde{t})_{\tilde{\nu}'}$ is symbol in $\tilde{t}$ with smooth parameter in $\tilde{\nu'}$.
 \end{itemize}
 
Let $\chi$ be a smooth compactly supported function with $\supp \chi(\mathsf{t}) \subset \{\lvert \mathsf{t} \rvert \leq 1\}$ and $\chi(t) \equiv 1 $ for $\lvert \mathsf{t} \rvert\leq \tfrac{1}{2}$. Define $U_{6,0}$ by
\begin{equation}\label{pullback0}
U_{6,0} = \big(1 -  \chi(\tilde{t}\tilde{\nu}') \big) \exp^{i\phi - i\tilde{t}+i\tfrac{1}{\tilde{\nu}'}} \tilde{U}_{2,0} + \chi(\tilde{t}\tilde{\nu}')   \tilde{U}_{2,0}
\end{equation}
ie we modify $\tilde{U}_{2,0}$ (in fact the pull back of $\tilde{U}_{2,0}$ from $\mathcal{B}_1$ by the projection $\pi : \mathcal{B}_1 \times \RR^+_{\rho_{\mathcal{B}_1}} \rightarrow  \mathcal{B}_1$) on the region $\{\tilde{t}\tilde{\nu}' > 1\}$ to have oscillatory part given by the phase function we have chosen ie $\varphi_{\text{in}}$.

Consider the resulting error at this stage
\begin{align*}
&\big[ (\tilde{t} \partial_{\tilde{t}})^2 +   (\lambda - \dfrac{n^2}{4})  + \tilde{t}^2 + \tilde{t}^2\tilde{\nu}' z' \big] \rho_{\mathfrak{B}_1}^{\tilde{\gamma}_{\mathfrak{B}_1} -2}U_{6,0}\\
= &\big[ (\tilde{t} \partial_{\tilde{t}})^2 +   (\lambda - \dfrac{n^2}{4})  + \tilde{t}^2 + \tilde{t}^2\tilde{\nu}' z' \big] \rho_{\mathfrak{B}_1}^{\tilde{\gamma}_{\mathfrak{B}_1} -2}\Big(\big(1 -  \chi(\tilde{t}\tilde{\nu}') ) e^{i\phi - i\tilde{t}+i\tfrac{1}{\tilde{\nu}'}} \tilde{U}_{6,0} + \chi(\tilde{t}\tilde{\nu}')   \tilde{U}_{6,0}\Big)\\
= & \ (1 -  \chi(\tilde{t}\tilde{\nu}') ) \big( e^{i\phi - i\tilde{t}+i\tfrac{1}{\tilde{\nu}'}}  - 1\big) \big[ (\tilde{t} \partial_{\tilde{t}})^2 +   (\lambda - \dfrac{n^2}{4})  + \tilde{t}^2 \big]  \rho_{\mathfrak{B}_1}^{\tilde{\gamma}_{\mathfrak{B}_1} -2}\tilde{U}_{6,0}\\
& + (1 -  \chi(\tilde{t}\tilde{\nu}') )[(\tilde{t} \partial_{\tilde{t}})^2, e^{i\phi - i\tilde{t}+i\tfrac{1}{\tilde{\nu}'}}  ] \rho_{\mathfrak{B}_1}^{\tilde{\gamma}_{\mathfrak{B}_1} -2}\tilde{U}_{6,0}  -   [(\tilde{t} \partial_{\tilde{t}})^2, \chi(\tilde{t}\tilde{\nu}')  ] e^{i\phi - i\tilde{t}+i\tfrac{1}{\tilde{\nu}'}}  \rho_{\mathfrak{B}_1}^{\tilde{\gamma}_{\mathfrak{B}_1} -2} \tilde{U}_{6,0}   \\
&+   [(\tilde{t} \partial_{\tilde{t}})^2, \chi(\tilde{t}\tilde{\nu}')  ] \rho_{\mathfrak{B}_1}^{\tilde{\gamma}_{\mathfrak{B}_1} -2} \tilde{U}_{6,0}
 +   \big[ (\tilde{t} \partial_{\tilde{t}})^2 +   (\lambda - \dfrac{n^2}{4})  + \tilde{t}^2 \big] \rho_{\mathfrak{B}_1}^{\tilde{\gamma}_{\mathfrak{B}_1} -2} \tilde{U}_{6,0} + \tilde{t}^2\tilde{\nu}' z'  \rho_{\mathfrak{B}_1}^{\tilde{\gamma}_{\mathfrak{B}_1} -2}U_{6,0}
 \end{align*}
And from the prep calculations \emph{placed at the end of the subsubsection}, we have :
$$\QQ_b \rho_{\mathfrak{B}_1}^{\tilde{\gamma}_{\mathfrak{B}_1} -2}  U_{6,0} + \rho_{\mathfrak{B}_1}^{\tilde{\gamma}_{\mathfrak{B}_1} } \mathsf{E}_{5,0} = \mathsf{O}(\rho_{\mathfrak{B}_1}^{\tilde{\gamma}_{\mathfrak{B}_1} -2+1}) $$

  \textbf{Equations satisfied by $U_{2,j}$ for $j > 0$} : we are dealing with the same type of inhomongeous ODE:
\begin{equation}
\begin{aligned}
&(\tilde{\nu}')^2 \left[(\tilde{t} \partial_{\tilde{t}})^2 +   (\lambda - \dfrac{n^2}{4}) + \tilde{t}^2\right] g_{\mathfrak{B}_1}^{\tilde{\gamma}_{\mathfrak{B}_1} -2+j}  U_{2,j}  =  - g_{\mathfrak{B}_1}^{\tilde{\gamma}_{\mathfrak{B}_1} -2+j}\mathsf{E}_{5,j} \\
&- \lim_{\rho_{\mathfrak{B}_1}\rightarrow 0} \rho_{\mathfrak{B}_1}^{-1}\Big[g_{\mathfrak{B}_1}^{\tilde{\gamma}_{\mathfrak{B}_1} -2+j-1} \mathsf{E}_{5,j-1} +  Q_b g_{\mathfrak{B}_1}^{\tilde{\gamma}_{\mathfrak{B}_1} -2+j-1}  U_{2,j-1}   \Big]  
\end{aligned}
\end{equation}
Since the the inhomogeneous part has the asymptotic we want (\emph{by the way we constructed $U_{2,0}$ and then by induction $U_{2,j-1}$ inherits the same property}), $\tilde{U}_{2,j}$ also satisfies the following properties:
\begin{itemize}
\item For $2\sqrt{\tfrac{n^2}{4}-\lambda} \notin \ZZ$ : 
at $\tilde{t}\sim 0$, $ g_{\mathfrak{B}_1}^{\tilde{\gamma}_{\mathfrak{B}_1} -2+j} U_{6,j} \stackrel{\text{asymptotic}}{\sim} \tilde{t}^{\sqrt{\tfrac{n^2}{4} - \lambda}}  (\tilde{\nu}')^{-\sqrt{\tfrac{n^2}{4} -\lambda} +\tilde{\gamma}_{\mathfrak{B}_1}-2+j}L^{\infty}$
 \item As $\tilde{t} \rightarrow \infty$,  $g_{\mathfrak{B}_1}^{\tilde{\gamma}_{\mathfrak{B}_1} -2+j} U_{6,j}\stackrel{\text{asymp}}{\sim}  \exp^{-i\tilde{t}} \tilde{t}^{-1/2} (\tilde{\nu}')^{-\sqrt{\tfrac{n^2}{4} -\lambda} +\tilde{\gamma}_{\mathfrak{B}_1}-2+j} S^0(\tilde{t})_{\tilde{\nu}'}$; 
  where $S^0(\tilde{t})_{\tilde{\nu}'}$ is symbol in $\tilde{t}$ with smooth parameter in $\tilde{\nu'}$.
 \end{itemize}
 Then we define $U_{6,j}$ as  \begin{equation}\label{functionOnB1}
 U_{6,j} = (1 -  \chi(\tilde{t}\tilde{\nu}') ) \exp(-i\varphi_{\text{lim}} + i\tilde{t}-i\tfrac{1}{\tilde{\nu}'}) \tilde{U}_{6,j} + \chi(\tilde{t}\tilde{\nu}')   \tilde{U}_{6,j}
 \end{equation}
 where in the local coordinates on the region $\tilde{t}\tilde{\nu}' > 1$
 $$\phi_{\text{lim}} =\phi_{\text{in}} \big|_{\tilde{t}\tilde{\nu}' > 1} = \tfrac{2}{3}(z'\tilde{\nu}')^{-1} \big[ (\tilde{t}\tilde{\nu}' z' + 1)^{3/2} - (z'+1)^{3/2}\big]$$
 thus the $U_{6,j}$ has only oscillatory behavior at $\mathfrak{F}_{1}\cap \mathfrak{B}_1$. Take asymptotic summation to construct $U_6$:
$$U_6^{\mathfrak{B}_1}  \stackrel{\text{borel summation}}{\sim} \rho_{\mathfrak{B}-1}^{\tilde{\gamma}_{\mathfrak{B}_1} -2} \sum_j \rho_{\mathfrak{B}-1}^j U_{6,j} ;\ \ \tilde{\gamma}_{\mathfrak{B}_1} = 0$$
$$\text{so}\ \  U_6^{\mathfrak{B}_1} = \exp(- i\varphi_{\text{lim}} )V_6^{\mathfrak{B}_1}+ \tilde{V}_6^{\mathfrak{B}_1}  $$
$$\text{where}\ \ 
 V_6^{\mathfrak{B}_1}  \in  \rho_{\mathcal{A}}^{\infty} \rho_{\mathfrak{F}_2}^{\infty} \rho_{\mathcal{F}}^{\infty} \rho_{\mathfrak{B}_1}^{\tilde{\gamma}_{\mathfrak{B}_1} -2} \rho_{\mathcal{R}}^{\infty}  \rho_{\mathcal{L}}^{\sqrt{\tfrac{n^2}{4}-\lambda}  }  \rho_{\mathfrak{F}_1}^{ \tilde{\gamma}_{\mathfrak{B}_1} -2 + \tfrac{1}{2}} $$
 $$  \tilde{V}_6^{\mathfrak{B}_1}\in  \rho_{\mathcal{A}}^{\infty} \rho_{\mathfrak{F}_2}^{\infty} \rho_{\mathcal{F}}^{\infty} \rho_{\mathfrak{B}_1}^{\tilde{\gamma}_{\mathfrak{B}_1} -2} \rho_{\mathcal{L}}^{\sqrt{\tfrac{n^2}{4}-\lambda}}  \rho_{\mathcal{R}}^{\sqrt{\tfrac{n^2}{4}-\lambda}  }  \rho_{\mathfrak{F}_1}^{\infty}   $$
\begin{remark}
To determine the asymptotics  of $V_6^{\mathfrak{B}_1}$ and $\tilde{V}_6^{\mathfrak{B}_1}$ :
\begin{enumerate}
\item At the corner of $\mathfrak{F}_1$ and $\mathcal{L}$, the local coordinates are
$$\tilde{\nu} = \tfrac{h}{z}, \tilde{\tilde{\nu}} = ( \mathsf{Z} \tilde{\nu})^{-1} , z$$
$$\rho_{\mathcal{L}} \sim  \tilde{\tilde{\nu}};\ \rho_{\mathfrak{F}_1} \sim \tilde{\nu} ; \tilde{\mathfrak{B}_1} = z$$
Away from $\mathcal{F}, \mathcal{R}$ and $\mathfrak{F}_2$, $g_{\mathfrak{B}_1} \sim  \mathsf{Z} = \tilde{\nu}^{-1} \tilde{\tilde{\nu}}^{-1}$. Also expressing $\tilde{\nu}'$ and $\tilde{t}$ in terms of the local coordinates: $\tilde{t} = \tilde{\nu}^{-1}$, $\tilde{\nu}' = \tilde{\tilde{\nu}}^{-1}$. Thus 
$$\exp(i\varphi_{\text{lim}}) \tilde{U}_{6,j}
 \sim \tilde{\nu}^{\tfrac{1}{2}} \tilde{\tilde{\nu}}^{\sqrt{\tfrac{n^2}{4} -\lambda} -\tilde{\gamma}_{\mathfrak{B}_1}+2}
 \tilde{\nu}^{\tilde{\gamma}_{\mathfrak{B}_1}-2} \tilde{\tilde{\nu}}^{\tilde{\gamma}_{\mathfrak{B}_1}-2 }
 = \tilde{\nu}^{\tfrac{1}{2} +\tilde{\gamma}_{\mathfrak{B}_1}-2  }\tilde{\tilde{\nu}}^{\sqrt{\tfrac{n^2}{4} -\lambda}}$$

\item At the corner of $\mathcal{L} $ and $\mathcal{F}$ the coordinates : $\mathsf{Z}^{-1}, h, \tilde{\mu} = \tfrac{z}{h}$. Away from $\mathfrak{F}_1, \mathcal{A}, \mathfrak{F}_2, \mathcal{R}$, $g \sim \mathsf{Z} \tilde{\nu} =\mathsf{Z} \tilde{\mu}^{-1}$ . Expressing $\tilde{\nu}'$ and $\tilde{t}$ in terms of :
$$\tilde{t} = \tilde{\mu} ; \tilde{\nu}' = \mathsf{Z} \tilde{\mu}^{-1}$$
$$  \tilde{\mu}^{\sqrt{\tfrac{n^2}{4} - \lambda}}     \mathsf{Z}^{-\tilde{\gamma}_{\mathfrak{B}_1}+2 }  \tilde{\mu}^{\tilde{\gamma}_{\mathfrak{B}_1}-2}  \mathsf{Z}^{-\sqrt{\tfrac{n^2}{4} -\lambda} +\tilde{\gamma}_{\mathfrak{B}_1}-2} \tilde{\mu}^{\sqrt{\tfrac{n^2}{4} -\lambda} -\tilde{\gamma}_{\mathfrak{B}_1}+2}
$$
Here the exponent in $\tilde{\mu}$ does not matter since we already have infinite order vanishing at $\mathcal{F}$ whose boundary defining function is $\tilde{\mu}$. Thus at the corner of $\mathcal{L}$ and $\mathcal{F}$, $$\tilde{U}_{6,j} \sim \rho_{\mathcal{L} }^{\sqrt{\tfrac{n^2}{4}-\lambda}} \rho_{\mathcal{F}}^{\infty} $$
\end{enumerate}
\end{remark}
Denote by $E_{6b}$ the resulting error, since $\tilde{\nu}' = h z^{-1} \mathsf{Z}^{-1}$ and account for a factor of $(\tilde{\nu}')^2$ we have
$$E_{6b} = \QQ_b U_6 + E_{\mathfrak{B}_1} $$
then 
$$E_{6b} = \exp(-i\varphi_{\text{in}}) F_{6b} $$
$$  F_{6b} \in \rho_{\mathcal{A}}^{\infty} \rho_{\mathfrak{F}_2}^{\infty} \rho_{\mathcal{F}}^{\infty} \rho_{\mathfrak{B}_1}^{\infty} \rho_{\mathcal{L}}^{\sqrt{\tfrac{n^2}{4}-\lambda}+ \tilde{\gamma}_{\mathfrak{B}_1} }  \rho_{\mathcal{R}}^{\sqrt{\tfrac{n^2}{4}-\lambda} +1 }  \rho_{\mathfrak{F}_1}^{ \sqrt{\tfrac{n^2}{4}-\lambda}  + \tilde{\gamma}_{\mathfrak{B}_1} -2 + \tfrac{1}{2}+1} ;  \tilde{\gamma}_{\mathfrak{B}_1} = 0$$
Note that the above process introduces error with non trivial behavior back at $\mathfrak{F}_1$ so we will have eliminate it in the following subsection 6b, using the same argument as in step 4. However this is time is easier since we will be only dealing with one conjugated operator since there is one phase function $\varphi_{\text{in}}$ that we are dealing with.

\subsubsection{Dealing with the error at $\tilde{\mathfrak{B}}_1$}

$$E_{\tilde{\mathfrak{B}}_1} =  E_6^{\tilde{\mathfrak{B}}_1} \in \rho_{\mathcal{F}}^{\infty}\rho_{\tilde{\mathfrak{B}}_1}^0 \rho_{\mathcal{R}}^{\sqrt{\tfrac{n^2}{4} - \lambda}} \rho_{\mathcal{L}}^{\sqrt{\tfrac{n^2}{4} - \lambda}}  \mathcal{C}^{\infty}(\mathfrak{M}_{b,h}) ; $$
 The projective coordinates away from $\mathcal{A}$  are
   $$ \tilde{t} = \dfrac{\mathsf{Z}}{\tilde{\nu}'}, \tilde{\nu}' , z'$$
  In these coordinates $\QQ_b$ has the form : 
     \begin{align*}
  &(z' \tilde{\nu}')^2 \left[ (\tilde{t} \partial_{\tilde{t}})^2 +   (\lambda - \dfrac{n^2}{4}) + \tilde{t}^3  \tilde{\nu}' z' - \tilde{t}^2 \right]
     \end{align*}
     The operator in the bracket  on $\tilde{\mathfrak{B}}_1$, $\mathsf{P}_{\tilde{\mathfrak{B}}_1}$ restricts to a Bessel equation of imaginary argument of order $(\tfrac{n^2}{4}-\lambda)^{1/2}$:
 $$(\tilde{t} \partial_{\tilde{t}})^2 +   (\lambda - \dfrac{n^2}{4})  - \tilde{t}^2    $$
 On the product structure $\tilde{B}_1 = \RR^+_{\tilde{t}} \times \RR^+_{\tilde{\nu}'}$, along each flow we can use the Macdonal functions to obtain the specific solution to the inhomogeneous equation with the following properties: 
\begin{itemize}
\item For $2\sqrt{\tfrac{n^2}{4}-\lambda} \notin \mathbb{Z}$ : 
at $\tilde{t}\sim 0$, $  \stackrel{\text{asymptotic}}{\sim} \tilde{t}^{\sqrt{\tfrac{n^2}{4} - \lambda}} L^{\infty}$
 \item  at $\tilde{t} \rightarrow \infty$,  $\stackrel{\text{asymp}}{\sim}  \exp(-\tilde{t}) \ \tilde{t}^{-1/2}\ S^0(\tilde{t})_{\tilde{\nu}'}$; ie rapidly deceasing, 
  where $S^0(\tilde{t})_{\tilde{\nu}'}$ is symbol in $\tilde{t}$ with parameter in $\tilde{\nu'}$.
 \end{itemize}
 Expand in taylor series around $\tilde{\mathfrak{B}}_1$.
  $$E_{\tilde{\mathfrak{B}}_1} \stackrel{\text{taylor expand}}{\sim} \rho_{\mathfrak{B}_1}^{\tilde{\gamma}_{\mathfrak{B}_1}} \sum_j \rho_{\mathfrak{B}_1}^j E_{\tilde{\mathfrak{B}}; j}  ; \ \ \tilde{\gamma}_{\tilde{\mathfrak{B}}_1} = 0 $$
   We solve for $U_6^{\tilde{\mathfrak{B}}}$ of the form
   $$U_6^{\tilde{\mathfrak{B}}} \sim  \rho_{\tilde{\mathfrak{B}}_1}^{-2} \sum_j \rho_{\tilde{\mathfrak{B}}_1}^j U_{6; j};\ \ \tilde{\gamma}_{\tilde{\mathfrak{B}}_1} = 0  $$

  Note that line (\ref{functionOnB1}) would be replaced by:
\begin{align*}
 U_{6,j}^{\tilde{\mathfrak{B}}_1} = (1 -  \chi(\tilde{t}\tilde{\nu}') ) \exp^{-\varphi_{\tilde{\mathfrak{B}}_1} + \tilde{t}-\tfrac{1}{\tilde{\nu}'}} \tilde{U}_{6,j}^{\tilde{\mathfrak{B}}_1} + \chi(\tilde{t}\tilde{\nu}')   \tilde{U}_{6,j}^{\tilde{\mathfrak{B}}_1}\\
 \varphi_{\tilde{\mathfrak{B}}_1} = \tfrac{2}{3} \big[- (1 - z)^{3/2} + (1 - z')^{3/2}  \big]
 \end{align*}
Then take asymptotic summation to construct $U_6$:
$$ U_6 = \exp(- \varphi_{\tilde{\mathfrak{B}}_1} ) \tilde{U}_6^{\tilde{\mathfrak{B}}_1}$$
$$\tilde{U}_6^{\tilde{\mathfrak{B}}_1} \in  \rho_{\mathcal{A}}^{\infty} \rho_{\tilde{\mathfrak{F}}_2}^{\infty}\rho_{\tilde{\mathfrak{F}}_1}^{\infty}  \rho_{\mathcal{F}}^{\infty} \rho_{\tilde{\mathfrak{B}}_1}^{\tilde{\gamma}_{\tilde{\mathfrak{B}}_1} -2} \rho_{\mathcal{L}}^{\sqrt{\tfrac{n^2}{4}-\lambda} }  \rho_{\mathcal{R}}^{\sqrt{\tfrac{n^2}{4}-\lambda} }  ; \  \tilde{\gamma}_{\tilde{\mathfrak{B}}_1} = 0 $$

\subsubsection{Summarizing}

$$ U_6 = U_6^{\mathfrak{B}_1}  + U_6^{\tilde{\mathfrak{B}}_1}$$
$$ U_6^{\tilde{\mathfrak{B}}_1} = \exp(- \varphi_{\tilde{\mathfrak{B}}_1} ) V_6^{\tilde{\mathfrak{B}}_1}+\tilde{V}_6^{\tilde{\mathfrak{B}}_1} $$
$$\tilde{U}_6^{\tilde{\mathfrak{B}}_1} \in  \rho_{\mathcal{A}}^{\infty} \rho_{\tilde{\mathfrak{F}}_2}^{\infty}\rho_{\tilde{\mathfrak{F}}_1}^{\infty}  \rho_{\mathcal{F}}^{\infty} \rho_{\tilde{\mathfrak{B}}_1}^{\tilde{\gamma}_{\tilde{\mathfrak{B}}_1} -2} \rho_{\mathcal{L}}^{\sqrt{\tfrac{n^2}{4}-\lambda}}  \rho_{\mathcal{R}}^{\sqrt{\tfrac{n^2}{4}-\lambda}  }  ; \  \tilde{\gamma}_{\tilde{\mathfrak{B}}_1} = 0 $$
$$\QQ_b U_6^{\tilde{\mathfrak{B}}_1}\in  \rho_{\mathcal{A}}^{\infty} \rho_{\tilde{\mathfrak{F}}_2}^{\infty}\rho_{\tilde{\mathfrak{F}}_1}^{\infty} \rho_{\mathcal{F}}^{\infty} \rho_{\tilde{\mathfrak{B}}_1}^{\infty} \rho_{\mathcal{L}}^{\sqrt{\tfrac{n^2}{4}-\lambda}}  \rho_{\mathcal{R}}^{\sqrt{\tfrac{n^2}{4}-\lambda}  +1 }   ;  \tilde{\gamma}_{\tilde{\mathfrak{B}}_1} = 0 $$

$$U_6^{\mathfrak{B}_1} = \exp(- i\varphi_{\text{lim}} ) V_6^{\mathfrak{B}_1} + \tilde{V}_6^{\mathfrak{B}_1}$$
$$V_6^{\mathfrak{B}_1}  \in  \rho_{\mathcal{A}}^{\infty} \rho_{\mathfrak{F}_2}^{\infty} \rho_{\mathcal{F}}^{\infty} \rho_{\mathfrak{B}_1}^{\tilde{\gamma}_{\mathfrak{B}_1} -2} \rho_{\mathcal{R}}^{\infty}  \rho_{\mathcal{L}}^{\sqrt{\tfrac{n^2}{4}-\lambda} }  \rho_{\mathfrak{F}_1}^{\tilde{\gamma}_{\mathfrak{B}_1} -2 + \tfrac{1}{2}} $$
$$  \tilde{V}_6^{\mathfrak{B}_1}\in  \rho_{\mathcal{A}}^{\infty} \rho_{\mathfrak{F}_2}^{\infty} \rho_{\mathcal{F}}^{\infty} \rho_{\mathfrak{B}_1}^{\tilde{\gamma}_{\mathfrak{B}_1} -2} \rho_{\mathcal{L}}^{\sqrt{\tfrac{n^2}{4}-\lambda} }  \rho_{\mathcal{R}}^{\sqrt{\tfrac{n^2}{4}-\lambda}  }  \rho_{\mathfrak{F}_1}^{\tilde{\gamma}_{\mathfrak{B}_1} -2 + \tfrac{1}{2}}   $$

Write $E_{6b} = \QQ_b U_6^{\mathfrak{B}_1} + E_{\mathfrak{B}_1} $ then
$$E_{6b}  \in \rho_{\mathcal{A}}^{\infty} \rho_{\mathfrak{F}_2}^{\infty} \rho_{\mathcal{F}}^{\infty} \rho_{\mathfrak{B}_1}^{\infty} \rho_{\mathcal{R}}^{\sqrt{\tfrac{n^2}{4}-\lambda}+1}  \rho_{\mathcal{L}}^{\sqrt{\tfrac{n^2}{4}-\lambda} + \tilde{\gamma}_{\mathfrak{B}_1} }  \rho_{\mathfrak{F}_1}^{ \sqrt{\tfrac{n^2}{4}-\lambda}  + \tilde{\gamma}_{\mathfrak{B}_1} -2 + \tfrac{1}{2}+1} ; \ \  \tilde{\gamma}_{\mathfrak{B}_1} = 0 $$

\textbf{PREP COMPUTATIONS used in calculating error}  
   The chosen phase function in current coordinates and on $\tilde{t}\tilde{\nu}'  >1$ 
$$\varphi_{\text{in}} = \dfrac{2}{3}\dfrac{1}{z' \tilde{\nu}'} \big[(\tilde{t}\tilde{\nu}' z' + 1)^{3/2} -(z'+1)^{3/2} \big]$$
  $$\lim_{z' \rightarrow 0} -\dfrac{2}{3} \dfrac{\big[(\tilde{t} \tilde{\nu}' z' + 1)^{3/2} - (z'+1)^{3/2} \big]}{z' \tilde{\nu}'}
   = -\dfrac{(\tilde{t} \tilde{\nu}' - 1)}{\tilde{\nu}'}$$
  
 \text{Prep calculations} for $\tilde{t}\rightarrow \infty$ so for $\tilde{\nu}' \neq 0$ , $\tilde{t} \tilde{\nu}' > 0$
 \begin{align*}
\tilde{t}\partial_{t} \left[  \dfrac{2}{3}\dfrac{1}{z' \tilde{\nu}'} \Big[(\tilde{t}\tilde{\nu}' z' + 1)^{3/2} -(z'+1)^{3/2} \Big] -  (\tilde{t} - \dfrac{1}{\tilde{\nu}'})\right]
&= \tilde{t}\left[ (\tilde{t} \tilde{\nu}' z' +1)^{1/2} + 1 \right]\\
(\tilde{t}\partial_{t})^2 \left[  \dfrac{2}{3}\dfrac{1}{z' \tilde{\nu}'} \Big[(\tilde{t}\tilde{\nu}' z' + 1)^{3/2} -(z'+1)^{3/2} \Big] - ( \tilde{t} - \dfrac{1}{\tilde{\nu}'}\right] &= \tilde{t}\left[ (\tilde{t} \tilde{\nu}' z' +1)^{1/2} - 1 \right] +\dfrac{1}{2} \dfrac{\tilde{t^2}\tilde{\nu}' z'} {(\tilde{t} \tilde{\nu}' z' +1)^{1/2}}\\
&= \dfrac{\tilde{t}^2 \tilde{\nu}' z' }{ 1 +  (\tilde{t} \tilde{\nu}' z' +1)^{1/2} } +\dfrac{1}{2} \dfrac{\tilde{t^2}\tilde{\nu}' z'} {(\tilde{t} \tilde{\nu}' z' +1)^{1/2}}
 \end{align*}
 Also on this region on which $\{\tilde{t}\tilde{\nu}' > 1 \}$ we have
 $$ \lim_{z' \rightarrow 0} \dfrac{2}{3} \dfrac{(\tilde{t} \tilde{\nu}' z' + 1)^{3/2} - (z+1)^{3/2} }{z' \tilde{\nu}'}
   = \dfrac{\tilde{t} \tilde{\nu}' - 1}{\tilde{\nu}'}
   \Rightarrow  \lim_{z' \rightarrow 0} \exp \Big(-i\varphi_{\text{lim}} + ( i\tilde{t} -\dfrac{1}{\tilde{\nu}'}) \Big) - 1 = 0$$
   $$\Rightarrow  \exp \Big(-i\varphi_{\text{lim}} + ( i\tilde{t} -\dfrac{1}{\tilde{\nu}'} )\Big) - 1 = \mathsf{O}(z')$$

We have similar computation for $\varphi_{\tilde{\mathfrak{B}}_1}$.

\subsection{Step 6 :  removing error at $\mathfrak{F}_1$}
  The construction from step 6 introduces error on $\mathfrak{F}_1$, which will be removed in this step.  
    We will need to work with the conjugated operator, conjugated by $\varphi_{\text{lim}}$ (which is $\varphi_{\text{in}}$ on this region).  As a reminder we write out the form of the phase function near $\mathfrak{F}_1$ 
    $$\varphi_{\text{lim}} = \tfrac{2}{3} (z\mathsf{Z}' \tilde{\tilde{\mu}})^{-1}\big[ (z+1)^{3/2} - (z\mathsf{Z}'+1)^{3/2}\big]$$
    
    \textbf{The coordinates in this neighborhood and product structrure} :Product structure 
    $$[0,\delta)_{\rho_{\mathfrak{F}_1}} \times ( \mathfrak{F}_1)_{\tilde{\tilde{\mu}}, z}$$
    $$\text{Coordinates}:\ \ \tilde{\tilde{\mu}} = \dfrac{\tilde{\nu}}{\mathsf{Z}'} , \mathsf{Z}', z $$
    $$\rho_{\mathfrak{F}_1} = \mathsf{Z}' \langle \tilde{\tilde{\mu}}\rangle$$
    By direct computation, the operator conjugated by $\exp(-\varphi_{\text{lim}})$ has the form :
    $$ z^2\dfrac{ \tilde{\tilde{\mu}}}{\langle \tilde{\tilde{\mu}}\rangle} \rho_{\mathfrak{F}_1}
    \Big[ \dfrac{ \tilde{\tilde{\mu}}}{\langle \tilde{\tilde{\mu}}\rangle} \rho_{\mathfrak{F}_1}(- \mathsf{Z}' \partial_{\mathsf{Z}'} + z\partial_z)^2 +\dfrac{ \tilde{\tilde{\mu}}}{\langle \tilde{\tilde{\mu}}\rangle} \rho_{\mathfrak{F}_1} (\lambda - \dfrac{n^2}{4}) + 2i (z+1)^{1/2}   (- \mathsf{Z}' \partial_{\mathsf{Z}'} + z\partial_z)    $$
    $$ +i(z+1)^{1/2} +\dfrac{1}{2}  i\dfrac{z}{(z+1)^{1/2}} \Big]$$
    
    We are removing  $E_{6b}^{\mathfrak{F}_1} = \exp(-i\varphi_{\text{lim}}) F_{6b}^{\mathfrak{F}_1}$, which comes from $ E_{6b} = E_{6b}^{\mathfrak{F}_1} + E_{6b}^{\mathcal{R}}$ with the index indicates the support of the function. 
$$  F_{6b} \in \rho_{\mathcal{A}}^{\infty} \rho_{\mathfrak{F}_2}^{\infty} \rho_{\mathcal{F}}^{\infty} \rho_{\mathfrak{B}_1}^{\infty} \rho_{\mathcal{R}}^{\infty}  \rho_{\mathcal{L}}^{\sqrt{\tfrac{n^2}{4}-\lambda}  }  \rho_{\mathfrak{F}_1}^{  \tilde{\gamma}_{\mathfrak{B}_1} -2 + \tfrac{1}{2}+1} ; \tilde{\gamma}_{\mathfrak{B}_1} = 0
$$
   We will construct $U_{6b}$ (supported close to $\mathfrak{F}_1$) with
 $$U_{6b} = \exp(-i\varphi_{\text{lim}}) \tilde{U}_{6b}$$
 $$\tilde{U}_{6b} \in \rho_{\mathfrak{F}_1}^{ \tilde{\gamma}_{\mathfrak{B}_1} -2 + \tfrac{1}{2} } \rho_{\mathfrak{B}_1}^{\infty} \rho_{\mathcal{A}}^{\infty} \rho_{\mathcal{L}}^{ \sqrt{\tfrac{n^2}{4}-\lambda}} ; \  \tilde{\gamma}_{\mathfrak{B}_1} = 0 $$ 
 Expand $F_{6b}$ at $\mathfrak{F}_1$
   $$F_{6b} \stackrel{\text{taylor expansion}}{\sim}  \rho_{\mathfrak{F}_1}^{\tilde{\gamma}_{\mathfrak{B}_1} -2 + \tfrac{1}{2}+1}
   \sum_k  \rho_{\mathfrak{F}_1}^k \tilde{F}_{2,k} $$

        \textbf{First transport equation}
   $$  \begin{cases}
     2i(z+1)^{1/2}  z^2\dfrac{\tilde{\tilde{\mu}}}{\langle \tilde{\tilde{\mu}}\rangle} \big[ z\partial_z - (\tilde{\gamma}_{\mathfrak{B}_1} -2 + \tfrac{1}{2}  +0) + \dfrac{1}{2} + \dfrac{1}{4} \dfrac{z}{z+1}  \big]  U_{6b;0} = - F_{6b;0}\\
     \text{IC} : z^4 \mathsf{U}_{5,0} |_{z=0} = 0  
     \end{cases}$$

     \textbf{Higher transport equation}
   $$   \begin{cases}
       2i(z+1)^{1/2}  z^2\dfrac{\tilde{\tilde{\mu}}}{\langle \tilde{\tilde{\mu}}\rangle}\big[ z\partial_z - (\tilde{\gamma}_{\mathfrak{B}_1} -2 + \tfrac{1}{2}  +j)+ \dfrac{1}{2} + \dfrac{1}{2} \dfrac{z}{z+1}  \big]  \tilde{U}_{6b;j} = - F_{5,0}\\
    \ \ \ \ \ \ \ \  \ \ \ \ \    - z^2\dfrac{\tilde{\tilde{\mu}}^2}{\langle \tilde{\tilde{\mu}}\rangle^2}\big[ \big(z\partial_z - (\tilde{\gamma}_{\mathfrak{B}_1} -2 + \tfrac{1}{2}  +j-1)\big)^2 + \lambda - \dfrac{n^2}{4}\big]\tilde{U}_{6b;j-1}  \\
     \text{IC} : z^4 \mathsf{U}_{5,0} = 0  
    \end{cases} $$
    
Take Borel summation to construct $\tilde{U}_{6b}$
    $$\tilde{U}_{6b} \stackrel{\text{borel sum}}{\sim} \rho_{\mathfrak{F}_1}^{\alpha_{\mathfrak{F}_1} -1} \sum_k \rho_{\mathfrak{F}_1}^j \tilde{U}_{6b;j}  \Rightarrow \tilde{U}_{6b}  \in \rho_{\mathfrak{F}_1}^{\alpha_{\mathfrak{F}_1} -1} \rho_{\mathfrak{B}_1}^{\alpha_{\mathfrak{B}_1}-2} \rho_{\mathcal{A}}^{\infty} \rho_{\mathcal{L}}^{\infty} $$

    Define
    $$ U_{6b}  = \exp(-i\varphi_{\text{lim}}) \tilde{U}_{6b} $$
  note this is supported close to $\mathfrak{F}_1$.

 The resulting error denote by $E_5$ supported near $\mathfrak{F}_1$ has the properties
  $$ E_5 \in \rho_{\mathcal{A}}^{\infty} \rho_{\mathfrak{F}_2}^{\infty} \rho_{\mathcal{F}}^{\infty} \rho_{\mathfrak{B}_1}^{\infty} \rho_{\mathcal{R}}^{\infty}  \rho_{\mathcal{L}}^{\sqrt{\tfrac{n^2}{4}-\lambda} }  \rho_{\mathfrak{F}_1}^{ \infty} ; \tilde{\gamma}_{\mathfrak{B}_1} = 0$$

  \subsection{Step 7 : remove the error at $\mathcal{R}$}
 
  We proceed to remove the error at $\mathcal{R}$, which we denote by $E_{\mathcal{R}}$ .
  $$E_{\mathcal{R}} =   E_{6b}^{\mathcal{R}}    + \QQ_b U_6^{\tilde{\mathfrak{B}}_1} 
   \in  \rho_{\mathcal{A}}^{\infty} \rho_{\mathfrak{F}_2}^{\infty}  \rho_{\mathfrak{B}_1}^{\infty} \rho_{\mathfrak{F}_1}^{ \infty}  \rho_{\tilde{\mathcal{A}}}^{\infty} \rho_{\tilde{\mathfrak{F}}_2}^{\infty} \rho_{\tilde{\mathfrak{B}}_1}^{\infty} \rho_{\tilde{\mathfrak{F}}_1}^{ \infty} \rho_{\mathcal{F}}^{\infty} \rho_{\mathcal{R}}^{\sqrt{\tfrac{n^2}{4}-\lambda}+1}  \rho_{\mathcal{L}}^{\sqrt{\tfrac{n^2}{4}-\lambda}  }  
$$
\emph{(Our convention : $\mathcal{R} = \{\mathsf{Z} = 0\}$)}. Since we have infinite order vanishing at other the other faces except $\mathcal{R}$ and $\mathcal{L}$, we can all blow them down and consider the error to live on $[0,1)_z \times [0,1)_{z'} \times \RR^+_h$. The coordinates in use are then :
 $$ h, z, z'$$
 In these coordinates property of $E_{\mathcal{R}}$ :
 $$E_{\mathcal{R}} \in h^{\infty} z^{\sqrt{\tfrac{n^2}{4}-\lambda}+1}  (z')^{\sqrt{\tfrac{n^2}{4}-\lambda}  }  \mathcal{C}^{\infty}(z,z')$$

 The operator in these coordinates
 $$\QQ (h,z)= h^2 (z\partial_z)^2 + h^2 (\lambda - \dfrac{n^2}{4})  + z^3 + z^2 $$
 One use power series argument to find $U_7$ of the form $U_7 \sim z^{\sqrt{\tfrac{n^2}{4}-\lambda}+1}  \sum_k z^k U_{7,k}$ such that 
 $$ \QQ U_7  - E_{\mathcal{R}} \in h^{\infty} z^{\infty}  (z')^{\sqrt{\tfrac{n^2}{4}-\lambda}  }  \mathcal{C}^{\infty}(z,z')$$
specifically we solve in series for the $U_{7;j}$ . 

Hence we can construct $U_7$ with
$$U_7 \in h^{\infty} z^{\sqrt{\tfrac{n^2}{4}-\lambda}+1} (z')^{\sqrt{\tfrac{n^2}{4}-\lambda}  } $$
with resulting error $E_8$
$$E_8 \in h^{\infty} z^{\infty} (z')^{\sqrt{\tfrac{n^2}{4}-\lambda}  }$$

\subsection{Concluding the construction}
  We have constructed $U$ such that 
  $$\tilde{U} = \tilde{U}_0 + U_1 + U_2 + U_3 + U_4+U_5 + U_6^{\mathfrak{B}_1} + U_6^{\tilde{\mathfrak{B}}_1}+ U_{6b}+ U_7 $$
  $$\QQ \tilde{U}  - \text{SK}_{\Id} \in    h^{\infty} z^{\infty} (z')^{\sqrt{\tfrac{n^2}{4}-\lambda}  }$$
 
   To compute the overall asymptotic property of $U$, we first relist the property of each of the component :
  \begin{align*}
   \alpha_{\mathfrak{F}_1} =  \alpha_{\mathfrak{B}_1} +\tfrac{1}{2}+1 ; \  \alpha_{\mathfrak{B}_1 } = \gamma_{\mathfrak{B}_1} +2 - 2; \ \gamma_{\mathfrak{B}_1} = -2 ; \  \gamma_{\mathcal{A}} = -1;\  \gamma_{\mathfrak{F}_2} = \gamma_{\mathcal{A}}-\tfrac{1}{2}\\
    \alpha_{\mathfrak{F}_2} = \gamma_{\mathfrak{F}_2} +1 ; \ \ 
\beta_{\mathcal{A}}  = \tfrac{1}{2} + \alpha_{\mathfrak{F}_2} + 1 ;\ \ 
  \beta_{\mathfrak{F}_1} =  \beta_{\mathfrak{B}_1} +\tfrac{1}{2}+1 ;\  \beta_{\mathfrak{B}_1}  = \alpha_{\mathfrak{B}_1}+2 - 2 \\
\tilde{\gamma}_{\mathfrak{B}_1} = 0 
 \end{align*}
   in a simplified version
    \begin{equation}\label{exponendef}
    \begin{aligned}
    \alpha_{\mathfrak{F}_1} =  \beta_{\mathfrak{F}_1}  = -2 +\tfrac{1}{2} + 1 = -\tfrac{1}{2}\\
     \beta_{\mathfrak{B}_1} = \alpha_{\mathfrak{B}_1} = \gamma_{\mathfrak{B}_1} = -2 \\
    \gamma_{\mathcal{A}} = -1;\ \  \beta_{\mathcal{A}} = 1\\ 
    \gamma_{\mathfrak{F}_2} = -1 -\tfrac{1}{2} = -\tfrac{3}{2} ; \ \alpha_{\mathfrak{F}_2} = -\tfrac{1}{2} \\
    \tilde{\gamma}_{\mathfrak{B}_1} = 0
     \end{aligned}
    \end{equation}

$$ U_0 \in \rho_{\mathfrak{B}_1}^{-1} \rho_{\tilde{\mathfrak{B}}_1}^{-1} \Psi^{-2}_{b,h}$$
 $$U_1 = \exp(-i\varphi_{\text{in}}) \tilde{U}_1 ,\ \  \tilde{U}_1 \in \rho_{\mathfrak{B}_1}^{\gamma_{\mathfrak{B}_1}} \rho_{\mathfrak{B}_2}^{-1}\rho_{\mathcal{A}}^{\gamma_{\mathcal{A}} } \mathcal{C}^{\infty}(\tilde{\mathfrak{M}}_{b,h});\ \  U_{1,\tilde{\mathfrak{B}}_2} \in \rho_{\tilde{\mathfrak{B}}_1}^{\gamma_{\tilde{\mathfrak{B}}_1}} \rho_{\tilde{\mathfrak{B}}_2}^{-1}\rho_{\tilde{\mathcal{A}}}^{\infty}\mathcal{C}^{\infty}(\tilde{\mathfrak{M}}_{b,h})$$
  $$U_2 = \exp(-i\varphi_{\text{in}}) \tilde{U}_2; \ \  \tilde{U}_2 \in \rho_{\mathcal{A}}^{\gamma_{\mathcal{A}}} \rho_{\mathfrak{F}_2}^{\gamma_{\mathfrak{F}_2} } \rho_{\mathfrak{B}_1}^{\alpha_{\mathfrak{B}_1}}  \rho_{\mathfrak{F}_1}^{\alpha_{\mathfrak{F}_1}} \mathcal{C}^{\infty}(\mathfrak{M}_{b,h})$$
   $$\varphi_{\text{in}} =   h^{-1}\lvert \varphi (z) - \varphi(z')\rvert; \ \ \varphi(z) = \tfrac{2}{3}\big[ (z+1)^{3/2} -1\big] \sgn\theta_n$$
  
  $$U_3 = \exp(-i\varphi_{\text{out}} ) \tilde{U}_3 + V_3  $$
  $$\varphi_{\text{out}} = h^{-1} \big[ \varphi(z) + \varphi(z') \big] \sgn \theta_n; \ \ \varphi(z) = \tfrac{2}{3}\big[ (z+1)^{3/2} -1\big] \sgn\theta_n$$
  $$\tilde{U}_3 \in \rho_{\mathfrak{B}_1}^{ \alpha_{\mathfrak{B}_1}} \rho_{\mathcal{A}}^{\tfrac{1}{2}+ \alpha_{\mathfrak{F}_2}}\rho_{\mathfrak{F}_2}^{\alpha_{\mathfrak{F}_2}} \mathcal{C}^{\infty}(\mathfrak{M}_{b,h});\ \ \supp\tilde{U}_3 \ \text{is close to}\  \mathcal{A} \ \text{and away from}\ \mathcal{R}$$
  $$\ V_3 \in  \rho_{\mathcal{R}}^{\sqrt{\tfrac{n^2}{4}-\lambda}}\rho_{\mathfrak{F}_2}^{\alpha_{\mathfrak{F}_2}}  \rho_{\mathfrak{B}_1}^{ \alpha_{\mathfrak{B}_1}};\ \ \supp V_3\  \text{is away from}\ \mathcal{A}$$
$$U_4 = \exp(i\varphi_{\text{out}}) \tilde{U}_4;\  \ \tilde{U}_4 \in \rho_{\mathcal{A}}^{\beta_{\mathcal{A}}} \rho_{\mathfrak{F}_2}^{\infty} \rho_{\mathfrak{B}_1}^{\beta_{\mathfrak{B}_1} } \rho_{\mathfrak{F}_1}^{\beta_{\mathfrak{F}_1}}$$

$$U_5 = U_5^{\mathfrak{B}_1} + U_5^{\tilde{\mathfrak{B}}_1} + U_5^{\text{interior}}$$
$$  U_5^{\mathfrak{B}_1} \in \rho_{\mathcal{F}}^{\gamma_{\mathcal{F}}}\rho_{\mathfrak{B}_1}^{-2} \rho_{\mathcal{R}}^{\sqrt{\tfrac{n^2}{4} - \lambda}} \rho_{\mathcal{L}}^{\sqrt{\tfrac{n^2}{4} - \lambda}}  \mathcal{C}^{\infty}(\mathfrak{M}_{b,h}); \ \ \gamma_{\mathcal{F}}  = 0$$
 $$U_5^{\tilde{\mathfrak{B}}_1} \in \rho_{\mathcal{F}}^{\gamma_{\mathcal{F}}}\rho_{\tilde{\mathfrak{B}}_1}^{-2} \rho_{\mathcal{R}}^{\sqrt{\tfrac{n^2}{4} - \lambda}} \rho_{\mathcal{L}}^{\sqrt{\tfrac{n^2}{4} - \lambda}}  \mathcal{C}^{\infty}(\mathfrak{M}_{b,h})$$
$$U_5^{\text{interior}} \in \rho_{\mathcal{F}}^{\gamma_{\mathcal{F}}} \rho_{\mathcal{R}}^{\sqrt{\tfrac{n^2}{4} - \lambda}} \rho_{\mathcal{L}}^{\sqrt{\tfrac{n^2}{4} - \lambda}}  \mathcal{C}^{\infty}(\mathfrak{M}_{b,h}) ; \ \ \supp U_5^{\text{interior}} \ \text{only intersects} \ \mathcal{F}, \mathcal{R}, \mathcal{L}$$

$$ U_6 = U_6^{\mathfrak{B}_1}  + U_6^{\tilde{\mathfrak{B}}_1}$$
$$ U_6^{\tilde{\mathfrak{B}}_1} = \exp(- \varphi_{\tilde{\mathfrak{B}}_1} ) V_6^{\tilde{\mathfrak{B}}_1}+\tilde{V}_6^{\tilde{\mathfrak{B}}_1} $$
  $$ \text{In coordinates close to $\tilde{\mathfrak{F}}_1$ and $\mathcal{L}$} : \varphi_{\tilde{\mathfrak{B}}_1} = \tfrac{2}{3} \big[- (1 - z)^{3/2} + (1 - z')^{3/2}  \big]$$
$$V_6^{\tilde{\mathfrak{B}}_1} \in  \rho_{\mathcal{A}}^{\infty} \rho_{\tilde{\mathfrak{F}}_2}^{\infty}\rho_{\tilde{\mathfrak{F}}_1}^{\infty}  \rho_{\mathcal{F}}^{\infty} \rho_{\tilde{\mathfrak{B}}_1}^{\tilde{\gamma}_{\tilde{\mathfrak{B}}_1} -2} \rho_{\mathcal{L}}^{\sqrt{\tfrac{n^2}{4}-\lambda}}  \rho_{\mathcal{R}}^{\sqrt{\tfrac{n^2}{4}-\lambda}  }  ; \  \tilde{\gamma}_{\tilde{\mathfrak{B}}_1} = 0 $$
$$ U_6 = U_6^{\mathfrak{B}_1}  + U_6^{\tilde{\mathfrak{B}}_1}$$
$$U_6^{\mathfrak{B}_1} = \exp(- i\varphi_{\text{lim}} ) V_6^{\mathfrak{B}_1} + \tilde{V}_6^{\mathfrak{B}_1}$$
$$ \text{Close to $\tilde{\mathfrak{F}}_1$ and $\mathcal{L}$} : \varphi_{\text{lim}} = \varphi_{\text{in}}$$
$$V_6^{\mathfrak{B}_1}  \in  \rho_{\mathcal{A}}^{\infty} \rho_{\mathfrak{F}_2}^{\infty} \rho_{\mathcal{F}}^{\infty} \rho_{\mathfrak{B}_1}^{\tilde{\gamma}_{\mathfrak{B}_1} -2} \rho_{\mathcal{R}}^{\infty}  \rho_{\mathcal{L}}^{\sqrt{\tfrac{n^2}{4}-\lambda}}  \rho_{\mathfrak{F}_1}^{\tilde{\gamma}_{\mathfrak{B}_1} - 2 + \tfrac{1}{2}} $$
$$  \tilde{V}_6^{\mathfrak{B}_1}\in  \rho_{\mathcal{A}}^{\infty} \rho_{\mathfrak{F}_2}^{\infty} \rho_{\mathcal{F}}^{\infty} \rho_{\mathfrak{B}_1}^{\tilde{\gamma}_{\mathfrak{B}_1} -2} \rho_{\mathcal{L}}^{\sqrt{\tfrac{n^2}{4}-\lambda} }  \rho_{\mathcal{R}}^{\sqrt{\tfrac{n^2}{4}-\lambda}  }  \rho_{\mathfrak{F}_1}^{\infty}   $$
$$U_{6b}  = \exp(-i\varphi_{\text{lim}}) \tilde{U}_{6b};\  \tilde{U}_{6b}  \in \rho_{\mathfrak{F}_1}^{\alpha_{\mathfrak{F}_1} -1} \rho_{\mathfrak{B}_1}^{\alpha_{\mathfrak{B}_1}-2} \rho_{\mathcal{A}}^{\infty} \rho_{\mathcal{L}}^{\infty} ;\ \supp U_{6b} \ \text{is close to} \ \mathfrak{F}_1$$
 $$U_7 \in h^{\infty} z^{\sqrt{\tfrac{n^2}{4}-\lambda}+1} (z')^{\sqrt{\tfrac{n^2}{4}-\lambda}  } $$

 $\varphi_{\text{in}}$ and $\varphi_{\text{out}}$ have support close to $\mathfrak{B}_1$ and $\mathcal{A}$, ie we only have oscillatory behavior at the $Z_0 > 0$ 
The term containing $\varphi_{\text{lim}}$ is supposed close to $\mathfrak{F}_1$.
  
\subsection{Summary of parametrix construction for $\tilde{\QQ}$}
  Reminder $\tilde{\QQ} = h^{-2}\QQ$ 
    $$\tilde{\QQ} U  - \text{SK}_{\Id} \in    h^{\infty} z^{\infty} (z')^{\sqrt{\tfrac{n^2}{4}-\lambda}  }$$
  \begin{equation}\label{parametrix}
U = U_0+\exp(-i\varphi_{\text{in}}) U_{\text{in}} +  \exp(-i\varphi_{\text{out}}) U_{\text{out}}  +  \exp(-i\varphi_{\text{lim}}) U_{\varphi_{\text{lim}}}  +U_{\textbf{bdy}} + U_{\tilde{\mathfrak{B}}_1}
\end{equation}
where
$$ U_0 \in \rho_{\mathfrak{B}_1}^1 \rho_{\tilde{\mathfrak{B}}_1}^1 \Psi^{-2}_{b,h}$$
$$U_{\text{in}} \in   
\rho_{\mathfrak{B}_1}^{\gamma_{\mathfrak{B}_1} +2} \rho_{\mathfrak{B}_2}^{-1+2} \rho_{\mathcal{A}}^{\gamma_{\mathcal{A}}+2}  \rho_{\mathfrak{F}_1}^{\alpha_{\mathfrak{F}_1}+2} \rho_{\mathfrak{F}_2}^{\gamma_{\mathfrak{F}_2}+2} \mathcal{C}^{\infty}(\mathfrak{M}_{b,h})  ;\ \supp U_{\text{in}}\  \text{is close to}\  \mathcal{A}  $$
$$ \gamma_{\mathfrak{B}_1}= -2; \gamma_{\mathcal{A}} = -1; \alpha_{\mathfrak{F}_1} = -\tfrac{1}{2} ; \gamma_{\mathfrak{F}_2} =  -\tfrac{3}{2}$$

 $$U_{\text{out}} \in \rho_{\mathfrak{B}_1}^{\alpha_{\mathfrak{B}_1}+2}  \rho_{\mathcal{A}}^2\rho_{\mathfrak{F}_2}^{\alpha_{\mathfrak{F}_2}+2}  \rho_{\mathfrak{F}_1}^{\beta_{\mathfrak{F}_1}+2}  \mathcal{C}^{\infty}(\mathfrak{M}_{b,h})$$
 $$\beta_{\mathfrak{F}_1} =  \alpha_{\mathfrak{F}_2}= -\tfrac{1}{2} ; \ \alpha_{\mathfrak{B}_1} = -2$$

 $$U_{\varphi_{\text{lim}}} \in \rho_{\mathcal{A}}^{\infty}\rho_{\mathcal{F}}^{\infty} \rho_{\mathfrak{B}_1}^{\tilde{\gamma}_{\mathfrak{B}_1} } \rho_{\mathcal{L}}^{\sqrt{\tfrac{n^2}{4}-\lambda}  }  \rho_{\mathfrak{F}_1}^{\tilde{\gamma}_{\mathfrak{B}_1}  + \tfrac{1}{2}}\mathcal{C}^{\infty}(\mathfrak{M}_{b,h}) ;\ \supp U_{\varphi_{\text{lim}}} \ \text{is close to}\ \mathfrak{F}_1 ; \ \tilde{\gamma}_{\mathfrak{B}_1} = 0$$

$$U_{\textbf{bdy}} \in \rho_{\mathcal{R}}^{\sqrt{\tfrac{n^2}{4}-\lambda}}\rho_{\mathcal{L}}^{\sqrt{\tfrac{n^2}{4}-\lambda}} \rho_{\mathfrak{F}_2}^{\infty} \rho_{\mathfrak{B}_1}^0 \rho_{\tilde{\mathfrak{B}}_1}^0 \rho_{\mathcal{F}}^0 \mathcal{C}^{\infty}(\mathfrak{M}_{b,h})\mathcal{C}^{\infty}(\mathfrak{M}_{b,h}); \ \supp U_{\mathfrak{B}_1} \text{is away from }\ \mathcal{A}, \mathfrak{F}_1$$
$$U_{\tilde{\mathfrak{B}}_1} \in \exp(- \varphi_{\tilde{\mathfrak{B}}_1} )  \rho_{\tilde{\mathcal{A}}}^{\infty} \rho_{\mathcal{R}}^{\sqrt{\tfrac{n^2}{4}-\lambda}}\rho_{\mathcal{L}}^{\sqrt{\tfrac{n^2}{4}-\lambda}} \rho_{\tilde{\mathfrak{F}}_2}^{\infty} \rho_{\tilde{\mathfrak{F}}_1}^{\infty}\rho_{\tilde{\mathfrak{B}}_1}^0 \rho_{\mathcal{F}}^0 \mathcal{C}^{\infty}(\mathfrak{M}_{b,h})$$

\section{The global resolvent fiberwise on $\mathbb{F}_C$}\label{globalResEst}
  As motivated in section \ref{mainIdeaApprox}, we want an exact solution to $\tilde{\QQ} u = 0$ on $\mathbb{F}_C$ as well as information on its asymptotic behavior. In order to obtain this, we first construct an approximate solution with an error that is trivial at every boundary face of $\mathbb{F}_C$ (see figure \ref{ffFC}) and in particular is semiclassically trivial i.e $\mathsf{O}(h^{\infty})$. We will need a global resolvent on $\mathbb{F}_C$ to solve completely this resulting error. In this section we show how such a global resolvent is constructed. \nl
  We will construct the global resolvent fiberwise. On the fibers where $\lvert Z_0\rvert$ is bounded, the problem is elliptic and is not semiclassical, so we can patch together by partition of unity the resolvent at $z = 0$ and that at infinity ie at $z^{-1}  = 0$. However, since the operator is not semiclassically elliptic, this method will not give uniform bound as $Z_0 \rightarrow \pm\infty$ and will result in error that is not semiclasslically trivial. When $Z_0 \rightarrow  \infty$, we use the gluing technique by Datchev-Vasy in \cite{Andras-Kiril} to combine the readily available resolvent estimate for the infinite region given by Vasy - Zworski in \cite{Andras-Zworksi} with an estimate we will obtain from the local parametrix constructed for the region close to $0$ in section \ref{ParametrixConstruction}. 
   The situation for $Z_0 \rightarrow -\infty$ is much simpler since the gluing is taken over an elliptic region, i.e. there will be no semiclassical singularity propagation.
   See remark (\ref{motivationResolvent}) for further discussion on the difference in methods between the two regions.\nl  The proof for the main lemma will contain 3 ingredients: estimate for the infinite region, estimate for the local region near $0$ and the gluing process.
    
  \begin{lemma}\label{globalresolvent}[Mapping property of the global resolvent on $\mathbb{F}_C$]\
  
  Partition $\mathbb{F}_0$ by $X_0$ and $X_1$ where $X_0 = \{ z > 1/4\} , X_1 = \{ z < 1 \}$, $X_1 $ (trapping region)
$\chi_1, \chi_0$ partition of unity, $\psi_j \sim 1$ near $\supp \chi_j$. 
 $$ \tilde{\QQ} G_1 = \Id + E_1 ,  \lVert E_1\rVert \in \mathsf{O}(h^{\infty}) , \lVert G_1\rVert \leq a_1 = c_1h^{\alpha_{\text{sml}}}$$
 $$ \tilde{\QQ}G_0 = \Id , \lVert G_0\rVert \leq a_0 = c_0h^{-1} $$
where $\alpha_{\text{sml}}$ is given by (\ref{weL2X1}). Here the norms are taken in some weighted\footnote{see estimate (\ref{weL2inX0}) for weighted norm for $G_0$ and see (\ref{weL2X1}) for $G_1$} $L^2$ spaces.\newline
Denote by $\text{Osc}_{\mathbb{F}_C}$ the set of oscillatory functions on $\mathbb{F}_C$ (see figure \ref{ffFC}) with the behavior:
\begin{equation}\label{defOsc}
z^{\sqrt{\tfrac{n^2}{4} - \lambda}} h^{\infty}  \mathcal{C}^{\infty}( [0,1)_z) 
+ \exp(\phi_{\text{osc}}) h^{\infty}  \mathcal{C}^{\infty}( [1,\infty)_z) \end{equation}
i.e. having boundary asymptotic prescribed by $\lambda$ at $z \sim 0$ and oscillatory behavior at $z \rightarrow \infty$, where in the corresponding coordinates used on each region we have 
$$\phi_{\text{osc}} = \begin{cases}  -i\tfrac{2}{3}h^{-1} \big[(z+1)^{3/2} - 1\big]\sgn\theta_n & Z_0 > 0 \\ -\tfrac{2}{3}h^{-1} \big[ i(z - 1)^{3/2} + 1\big] & Z_0 < 0\\
-i\tfrac{2}{3}Z^{3/2} \sgn\theta_n & \lvert Z_0\rvert \, \text{bounded} \end{cases}$$
 We obtain the mapping property of the global resolvent $G_{\tilde{\QQ}} $ on $\mathbb{F}_C$
 $$ G_{\tilde{\QQ}} :   \dot{\mathcal{C}}^{\infty}(\mathbb{F}_C)  \rightarrow \text{Osc}_{\mathbb{F}_C}; \ \ \lVert G\rVert \leq C a_1 $$
close to $z = 0$ the weighted $L^2$ norm is that same as that for $G_1$, while close to infinity we use that for $G_0$. \end{lemma}

\subsection{$Z_0 > 0$}
\subsubsection{Ingredient 1: Semiclassical estimate at infinite}

$$\QQ_+  \stackrel{ \mathsf{t} = z^{-3/2}}{=} -\mathsf{t}^{-2} \Big[ -h^2(\mathsf{t}^2 \partial_{\mathsf{t}})^2 + \mathsf{t}^{2/3}V^+ - \tfrac{4}{9}    \Big]; 
  \ V^+ =   h^2 \mathsf{t}^{7/3} \partial_{\mathsf{t}} -\tfrac{4}{9}(\lambda - \tfrac{n^2}{4}) h^2\mathsf{t}^{4/3} - \tfrac{4}{9}    $$
\textbf{Notation: } In this subsection for matter of convenience we will work with variable $\sigma = z^{3/2}$ and $\mathsf{t} = \sigma^{-1}$ instead of $z$.\newline
 Vasy and Zworski in (\cite{Andras-Zworksi}) prove a semiclassical estimate for asymptotically euclidean scattering with long-range perturbation.
   \begin{theorem}[Vasy - Zworski (2000)]
  Let $X$ be a manifold with boundary. Let $\Delta$ be the Laplacian of a scattering metric on $X$. If $P = h^2 \Delta + V$ is a semi-classical long-range perturbation of $h^2 \Delta$ and $R(\lambda) = (P-\lambda)^{-1}$ its resolvent then for all $m\in \RR$
  $$\lVert R(\lambda + it) f\rVert_{H^{m, -\tfrac{1}{2}-\epsilon}_{\text{sc}}(X)}
  \leq C_0 h^{-1} \lVert f\rVert_{H^{m-2, \tfrac{1}{2}+\epsilon}_{\text{sc}}(X)};\ \epsilon > 0$$
  With $C_0$ independent of $t \neq 0$ real  $\lambda \in I, I\subset (0,+\infty)$ a compact interval in the set of non-trapping energies for $P$.\newline Here $H^{m,k}_{\text{sc}} (X)$ denotes Sobolev spaces adapted to the scattering calculus. The index $m$ indicates smoothness and $k$ the rate of decay at infinity.
  \end{theorem}
  \begin{remark}[Long range perturbation]
   Long range potential $V$ are second order semiclassical operator which in local coordinates of the form $V = \sum_{\lvert \alpha\rvert \leq 2} v_{\alpha}(z,h) (hD_z)^{\alpha}$ and near the boundary $\partial Z$, in local coordinates $y\in \partial X$, 
  $$V = x^{\gamma} V_0 ,\ V_0 = \sum_{\lvert\alpha\rvert+k\leq 2} v_{k\alpha}(x,y,h) (hx^2 D_x)^k (h D_y)^{\alpha},$$
  $$ v_{k\alpha} - v^0_{k\alpha} \in h S^{0,0,0} (X) , v^0_{k\alpha} \in S^{0,0}(X) ; \ \gamma > 0.$$
In Euclidean setting the condition on the coefficients corresponds to assuming that the coefficients are symbols in the Euclidean base variables. The constraint on $\gamma$ is just $\gamma > 0$. For application in our case, $\gamma = \tfrac{2}{3}$, ie the pertubation goes to zero more slowly than Coulomb potential. 
 \end{remark}
 
 Denote by $$v^+ = V^+|_{h=0}  \Rightarrow v^+ =  - \tfrac{4}{9} \mathsf{t}^{2/3} = -\tfrac{4}{9} \sigma^{-2/3}.$$ 
 
 In order to apply the result we need find a global operator on $\RR_{\sigma}$ whose behavior coincides with that of $\QQ_+$ on $\sigma > 1$. $\QQ_+$ is extended to $\RR^+$ in a way so that: there is no extra boundary at $\sigma = 0$, the problem should be elliptic on $\sigma < 0 $ ie the negative infinity end, and most importantly the operator is globally nontrapping. The following extension satisfies these criteria:
 $$\QQ_{\text{global},+}  = \chi(\sigma) \bar{\QQ}_+ ; \ \bar{\QQ}_+ = \begin{cases}  - h^2\partial_{\sigma}^2 +V^+ - \tfrac{4}{9}  & \sigma > 1 \\   -h^2\partial_{\sigma}^2 + \tilde{V}^+ - \tfrac{4}{9} & \sigma \leq 1  \end{cases} $$
 $$\chi \in \mathcal{C}^{\infty} (\RR),\ \ \phi (\sigma) = \begin{cases}  \sigma^2  &\sigma \geq 1\\ \tfrac{1}{2} & \sigma  \leq \tfrac{1}{2} \end{cases}$$
  $\tilde{V}^+$ is such that  $\tilde{V}^+|_{\sigma  > 1} = V^+$ and on $h=0$, $\tilde{V}^+$ has the profile of $\tilde{v}^+$ in figure \ref{nontrappingPo1}. 
  \begin{figure}
  \centering
  \includegraphics[scale=0.85]{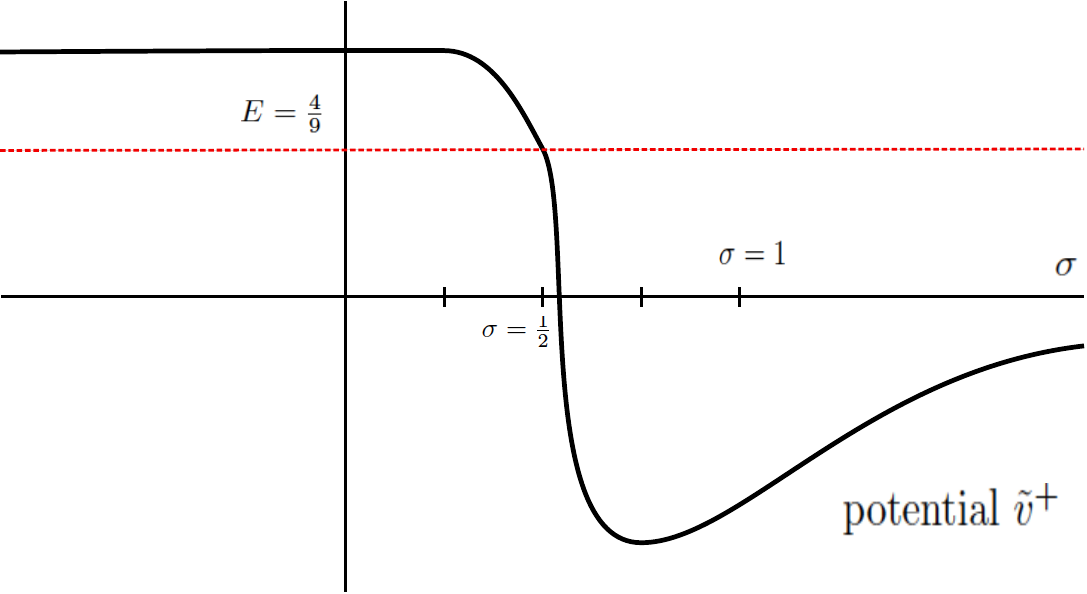}
  \caption{Extending the local problem of $\QQ_+$ at infinity to a global problem on $\sigma = z^{3/2}$ }
  \label{nontrappingPo1}
  \end{figure}
  In particular $\tilde{v}^+= c > \tfrac{4}{9}$ for $ \sigma < \tfrac{1}{4}$, $\tilde{v}^+ > \tfrac{4}{9}$ for $\sigma< \tfrac{1}{2}$ and $ < \tfrac{4}{9}$ for $\sigma > \tfrac{1}{2}$ and on $\sigma > 1$ we have $\tilde{v}^+ = v^+$. Next we need to verify that the semiclassical bicharacteristic of $\QQ_{\text{global},+}$ is globally nontrapping.\newline
   \textbf{Non-trapping condition verification:}\newline
   The semiclassical symbol of $\QQ_{\text{global},+}$
   $$\tilde{p}^+ = \chi(\sigma) \bar{p}^+ ;\ \ \bar{p}^+ = \xi^2 + \tilde{v} - \tfrac{4}{9} $$
   Since $\phi > 0$, on  the characteristic set $\tilde{p}^+$ we have $ \xi = \pm \sqrt{  \tfrac{4}{9} - \tilde{v}}$. $\chi \in\mathcal{C}^{\infty}$ so the bicharacteristic of $H_{\tilde{p}^+}$ is the same as that for $H_{\bar{p}^+}$. 
 $$\begin{cases} \dfrac{dx}{ds}  =  2\xi ;\ \dfrac{d\xi}{ds} = - \dfrac{\partial \tilde{v}}{\partial \sigma} ; \  \xi = \pm \sqrt{  \tfrac{4}{9} - \tilde{v}} \\ \sigma(0) = \tfrac{1}{2} \Rightarrow \xi(0) = 0\end{cases} $$
 When $\xi > 0$, $\sigma(t) $ increases. When $\xi < 0$, $\sigma(t)$ decreases.  Also the bicharacteristic set is a one dimensional submanifold, so away from $\xi = 0$ the bicharacteristic is non-trapping. 
 It remains to verify non-trapping condition for $\sigma < \tfrac{3}{4}$. \nl In fact, close to time $s =0$, we can work out directly the profile of the bicharacteristic set. $\sigma$ can be considered as a smooth function parametrized by $\xi$. This is seen as follows: $\tilde{v} = \tfrac{4}{9} - \xi^2$ and on $0 < \sigma < \tfrac{3}{4}$ so $\tilde{v}$ is a smooth strictly monotone function of $\sigma$ so $\sigma = \mathsf{f}(\tilde{v}) $ for some smooth strictly monotone function $f$ and hence $\sigma =  \mathsf{f}(\tfrac{4}{9} - \xi^2)$ is a smooth function of $\xi$. The profile of the characteristic set is as shown in figure (\ref{nontrappingBi1}). \nl
  \begin{figure}
  \centering
  \includegraphics{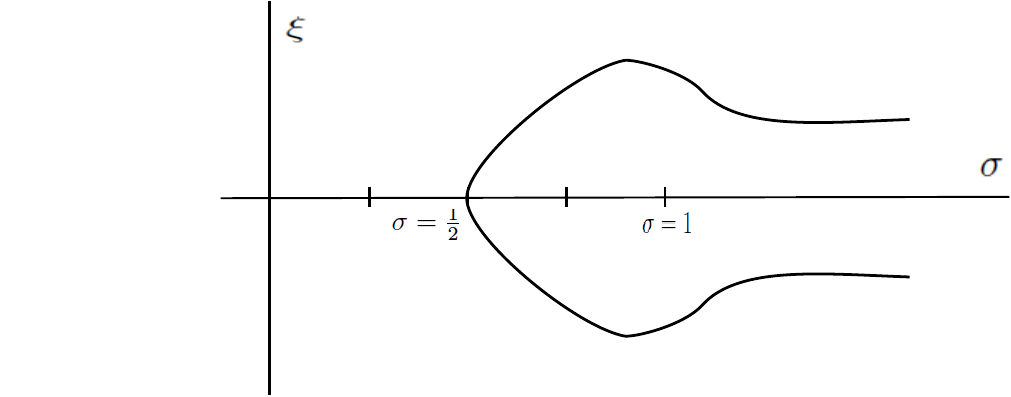}
  \caption{The bicharacteristic set of the new global operator $\QQ_{\text{global};+}$}
  \label{nontrappingBi1}
  \end{figure}
  Now we are ready to apply the result by Vasy-Zworski to $\QQ_{\text{global},+}$, denote by $ R(\lambda + it) = \QQ_{\text{global},+}^{-1}$ we have:
   $$\lVert R(\lambda + it) f\rVert_{H^{m, -\tfrac{1}{2}-\epsilon}_{\text{sc}} ( \RR_{\sigma})}
  \leq \tilde{C}_0 h^{-1} \lVert f\rVert_{H^{m-2, \tfrac{1}{2}+\epsilon}( \RR_{\sigma})};\ \epsilon > 0 ; \ \lambda = \tfrac{4}{9}$$
  Following the method in Melrose's paper \cite{Melrose03}, using these estimates, one can show that for fixed $h > 0$ the limits $R(\lambda \pm i0) f$ exist in $H^{m, -\tfrac{1}{2} - \epsilon}_{\text{sc}}( \RR_{\sigma})$ for $f \in H^{m-2, \tfrac{1}{2} + \epsilon}_{\text{sc}}(X_0)$, 
  $$\lVert R(\lambda + i0) f\rVert_{H^{m, -\tfrac{1}{2}-\epsilon}_{\text{sc}} ( \RR_{\sigma})}
  \leq \tilde{C}_0 h^{-1} \lVert f\rVert_{H^{m-2, \tfrac{1}{2}+\epsilon}_{\text{sc}} ( I_{\mathsf{t}})};\ \epsilon > 0 ;\  \lambda = \tfrac{4}{9} ;\  h \in (0,h_0)$$
  In terms of $z$ and $m = 2$
  \begin{equation}\label{weL2inX0}
  \lVert R(\lambda + i0)  f  \rVert_{\sigma^{(\tfrac{1}{2}+\epsilon )} L^2( \RR_{\sigma}) }
  \leq C_0 h^{-1} \lVert   f  \rVert_{\sigma^{-(\tfrac{1}{2}+\epsilon)}L^2( \RR_{\sigma})};\ \epsilon > 0 ;\  \lambda = \tfrac{4}{9} ;\  h \in (0,h_0)
  \end{equation}
$\tilde{\QQ} = h^{-2} \QQ$. Denote by 
$$G_0 = \psi R(\tfrac{4}{9} + i0) h^2\chi$$
where $\chi \equiv 1$ on $\sigma > 1$, $\psi \equiv 1$ on  $\supp \chi$, and $\supp \chi \subset \supp \psi \subset \{\sigma > \tfrac{3}{4}\}$.  Then, with the norms taken in corresponding weighted $L^2$ space, for some $a_0 (h) \geq h^1$ we have:
 \begin{equation}\label{estimateR_0}
 \lVert G_0 \rVert \leq c_0 h^1 = a_0
 \end{equation} 
 

\textbf{Mapping properties of} $G_0$ 
$$G_0 :    \dot{\mathcal{C}}^{\infty}(\mathbb{F}_C) \rightarrow h^{\infty}   \text{Osc}_{\infty}$$
where $h^{\infty}  {Osc}_{\infty}$ is the set of functions that are semiclassically trivial with trivial behavior at finite end (at $\sigma \sim 0$) and oscillatory behavior at the infinite end (as $\sigma\rightarrow \infty$). The oscillatory part is given by:
$$\exp(-i\phi)  \mathcal{C}^{\infty} (\mathbb{F}_{C,+}^{\infty});\ \  \phi (z) = \tfrac{2}{3}h^{-1} \big[(z+1)^{3/2} - 1\big]\sgn\theta_n.$$

\subsubsection{Ingredient 2: Estimate around $0$ from local parametrix of $\tilde{\QQ}$}\

 From the construction of the local parametrix near $z\sim 0$ in section (\ref{ParametrixConstruction}) we get: 
    $$\tilde{\QQ} U  - \text{SK}_{\Id} = E \in    h^{\infty} z^{\infty} (z')^{\sqrt{\tfrac{n^2}{4}-\lambda}  +\tilde{\gamma}_{\mathfrak{B}_1} -2}$$
 $$U = U_0+\exp(-i\varphi_{\text{in}}) U_{\text{in}} +  \exp(-i\varphi_{\text{out}}) U_{\text{out}}  +  \exp(-i\varphi_{\text{lim}}) U_{\varphi_{\text{lim}}}  +U_{\textbf{bdy}} + U_{\tilde{\mathfrak{B}}_1}$$
 where $U_{\tilde{\mathfrak{B}}_1}$ is the portion of the approximate solution near $\tilde{\mathfrak{B}}_1$ (where  $Z_0 < 0$). $U_0$ and $\text{SK}_{\Id}$ have support close to the lifted diagonal, away from $\mathcal{L}$ and $\mathcal{R}$. (See (\ref{parametrix}) for more details).  \newline
 For the purpose of weighted $L^2$ estimate computation, we write $U$ as
 $$U = U_{\Psi\text{do} } + U_{\text{osc}} $$
$U_{\Psi\text{do} }$ comes from the term $U_0$ which is supported close the diagonal and is conormal of order $-2$ to the lifted diagonal, while $U_{\text{osc}}$ come from: $\exp(-i\varphi_{\text{in}}) U_{\text{in}} +  \exp(-i\varphi_{\text{out}}) U_{\text{out}} + \exp(-i\varphi_{\text{lim}}) U_{\varphi_{\text{lim}}}$. Since we will only need a weak estimate and deal with $L^{\infty}$ norm, we can ignore the oscillatory terms.

 \begin{enumerate}
  \item  First we deal with the term $U_{\Psi\text{do} }$. Due to the cornormality of negative order at the diagonal we have
  $$\lvert U_{\Psi\text{do} }\rvert   \leq \lvert \dfrac{z-z'}{h}\rvert^{-1+2} \Rightarrow U_{\Psi\text{do} } \in L^1_{\text{loc}} (\RR)$$
  We rewrite $U_{\Psi\text{do} }$ as
   $$U_{\Psi\text{do} } = h\chi(z)\chi(z') f\big( \dfrac{z-z'}{h}\big) ,\ f\in L^{\infty}_c ;\ \supp \chi \subset [0,1]$$
 $$\sup_{z\in \RR} \int_{\RR} \lvert U_{\Psi\text{do} } \rvert \, dz' \ \ ; 
 \sup_{z'\in \RR} \int_{\RR} \lvert U_{\Psi\text{do} } \rvert \, dz  < \infty$$
 this is because
 $$ \int_{\RR} h \chi(z)\chi(z') \big| f\big( \dfrac{z-z'}{h}\big) \big| \, dz' 
 \stackrel{y = \tfrac{z-z'}{h}}{=} \int_{\RR} \chi(z)\chi(z-yh) \lvert f(y)\rvert \, dy  < \infty $$
 Similarly
 $$ \int_{\RR} h \chi(z)\chi(z') \big| f\big( \dfrac{z-z'}{h}\big) \big| \, dz
 \stackrel{y = \tfrac{z-z'}{h}}{=} \int_{\RR} \chi(z)\chi(z'+yh) \lvert f(y)\rvert \, dy  < \infty' $$

  \item We are going to compute which power of $h$ and $z'$ that will make $U_{\text{in}}, U_{\text{out}}$ and $U_{\varphi_{\text{lim}}}$ $L^{\infty}$ on $\mathfrak{M}_{b,h}$. In fact, to make sure the operator whose S.K of $U$ maps into $ z^{\sqrt{\tfrac{n^2}{4} -\lambda}} \mathcal{C}^{\infty}([0,1)_z)$, we find $c$ and $a$ such that:
  $$ h^c (z')^a z^{-\sqrt{\tfrac{n^2}{4}-\lambda}} U_{\text{in}}, \ h^c (z')^a z^{-\sqrt{\tfrac{n^2}{4}-\lambda}} U_{\text{out}} ,  h^c (z')^a z^{-\sqrt{\tfrac{n^2}{4}-\lambda}} U_{\varphi_{\text{lim}}} \in L^1_{\text{loc}}$$
 $(z')^a$ lifts to
$$\big(\rho_{\mathcal{L}} \rho_{\mathcal{F}} \rho_{\mathfrak{F}_1} \rho_{\mathfrak{B}_1} \big)^a\mathcal{C}^{\infty}(\mathfrak{M}_{b,h}) $$ 
$h^c$ lifts to
$$\big( \rho_{\mathfrak{B}_1} \rho_{\mathfrak{B}_2} \rho_{\mathfrak{F}_2} \rho_{\mathfrak{F}_1}  \rho_{\mathcal{A}}\big)^c\mathcal{C}^{\infty}(\mathfrak{M}_{b,h})$$
$z$ lifts to 
$$z = \rho_{\mathcal{R}} \rho_{\mathcal{F}} \rho_{\mathfrak{F}_2} \rho_{\mathfrak{B}_1} \mathcal{C}^{\infty}(\mathfrak{M}_{b,h})$$
From (\ref{parametrix}), $U_{\text{in}}$ has the following property:
$$U_{\text{in}} \in   
\rho_{\mathfrak{B}_1}^{\gamma_{\mathfrak{B}_1} +2} \rho_{\mathfrak{B}_2}^{-1+2} \rho_{\mathcal{A}}^{\gamma_{\mathcal{A}}+2}  \rho_{\mathfrak{F}_1}^{\alpha_{\mathfrak{F}_1}+2} \rho_{\mathfrak{F}_2}^{\gamma_{\mathfrak{F}_2}+2} \mathcal{C}^{\infty}(\mathfrak{M}_{b,h})  ;\ \supp U_{\text{in}}\  \text{is close to}\  \mathcal{A}  $$
$$ \gamma_{\mathfrak{B}_1}= -2; \gamma_{\mathcal{A}} = -1; \alpha_{\mathfrak{F}_1} = -\tfrac{1}{2} ; \gamma_{\mathfrak{F}_2} =  -\tfrac{3}{2}$$
To take care of the index at 
\begin{itemize}
\item $\mathfrak{B}_1$ : we require $a+c \geq  \ \sqrt{\tfrac{n^2}{4} -\lambda} $
\item $\mathfrak{F}_2$ : $ c  \geq  \sqrt{\tfrac{n^2}{4} -\lambda} -\tfrac{1}{2} $.
\end{itemize}
Similarly for $U_{\text{out}}$
$$U_{\text{out}} \in \rho_{\mathfrak{B}_1}^{\alpha_{\mathfrak{B}_1}+2}  \rho_{\mathcal{A}}^2\rho_{\mathfrak{F}_2}^{\alpha_{\mathfrak{F}_2}+2}  \rho_{\mathfrak{F}_1}^{\beta_{\mathfrak{F}_1}+2}  \mathcal{C}^{\infty}(\mathfrak{M}_{b,h});\ \ \beta_{\mathfrak{F}_1} =  \alpha_{\mathfrak{F}_2}= -\tfrac{1}{2} ; \ \alpha_{\mathfrak{B}_1} = -2$$
and for $U_{\varphi_{\text{lim}}}$ (note that the index of $\rho_{\mathcal{L}}$ is already positive)
 $$U_{\varphi_{\text{lim}}} \in \rho_{\mathcal{A}}^{\infty}\rho_{\mathcal{F}}^{\infty} \rho_{\mathfrak{B}_1}^{\tilde{\gamma}_{\mathfrak{B}_1} } \rho_{\mathcal{L}}^{\sqrt{\tfrac{n^2}{4}-\lambda}  }  \rho_{\mathfrak{F}_1}^{\tilde{\gamma}_{\mathfrak{B}_1}  + \tfrac{1}{2}}\mathcal{C}^{\infty}(\mathfrak{M}_{b,h}) ;\ \supp U_{\varphi_{\text{lim}}} \ \text{is close to}\ \mathfrak{F}_1 ; \ \tilde{\gamma}_{\mathfrak{B}_1} = 0$$
Choose
 $$a = \sqrt{\tfrac{n^2}{4} - \lambda}  ,\  c =  \sqrt{\tfrac{n^2}{4} - \lambda}$$
  Denote by $\dot{\mathcal{C}}^{\infty} (\mathbb{F}_C)$ the set of functions which vanish to infinite order at both the finite end (where $z=0$) and the infinite one ( where $z^{-1} = 0$) and is $\mathsf{O}(h^{\infty})$ at both end when $Z_0 \rightarrow \pm \infty$. Such a function has the form for any $N$ and $N'$
 $$e = h^{N'} (z')^N  \tilde{e}_{N', N}(h,z') ;\ \tilde{e}(z') \mathcal{C}^{\infty}(\RR^+_{z'} \times [0,1)_h)$$
$$\langle U , e(z') \rangle = \langle U h^c (z')^a , \tilde{e}_{c,a}(z') \rangle $$
with $a = \sqrt{\tfrac{n^2}{4} - \lambda} ,\  c = \sqrt{\tfrac{n^2}{4} - \lambda} $ we have $$ h^c (z')^a U   \in L^{\infty}( \mathfrak{M}_{b,h}) $$
The blown down of $ h^a  U   \in z^{\sqrt{\tfrac{n^2}{4} - \lambda}} L^{\infty}([0,1)_h \times [0,1)_z \times [0,1)_{z'}  )$ thus $L^1_{\text{loc}}$ and hence satisfies the requirement of Schur's lemma.  For $e \in \dot{\mathcal{C}}^{\infty} (\mathbb{F}_C)$ we have:
 $$\lVert  \langle  \exp(i\varphi_{\text{in}})U_{\text{in}},  \, e \rangle    \rVert_{z^{\sqrt{\tfrac{n^2}{4} - \lambda}}L^2([0,1]_z)} \leq  C h^{-c} \lVert  e \rVert_{z^a L^2} $$
 The same computation applies for $U_{\text{out}}$ and $U_{\varphi_{\text{lim}}}$. Thus the operator is bounded between the following weighted $L^2$ spaces :  
\begin{equation}\label{weL2X1}
\begin{aligned}
\big( \dot{\mathcal{C}}^{\infty} (\mathbb{F}_C), \lVert \cdot \rVert_{z^a L^2} \big)\rightarrow  z^{\sqrt{\tfrac{n^2}{4} - \lambda}} L^2([0,1]_z) \\
\lVert \cdot \rVert \leq h^{\alpha_{\text{sml}}} ;  \ \ a = \sqrt{\tfrac{n^2}{4} - \lambda}  ,\  -\alpha_{\text{sml}} =  c =  \sqrt{\tfrac{n^2}{4} - \lambda}
\end{aligned}
\end{equation}

 \end{enumerate}

Similarly we have the bound for the error term $E \in h^{\infty} z^{\infty} (z')^{\sqrt{\tfrac{n^2}{4} - \lambda} }$ with the same weighted space but with
    $$\lVert E\rVert = \mathsf{O}(h^{\infty})$$

 \subsubsection{Global resolvent estimate }
 Now we are ready to put together the two above estimates using the method of Datchev-Vasy in \cite{Andras-Kiril}. \nl Denote by $X_0 = \{ \sigma > 3/4\} , X_1 = \{ z < 1 \}$, $X_1 $ (trapping region). Let $\chi_1, \chi_0$ be a partition of unity, with $\psi_j \sim 1$ near $\supp \chi_j$. \nl Ingredient 1 (\ref{estimateR_0}) gives the estimate in the weighted $L^2$ bound $\mathsf{z}^{-\tfrac{3}{2}(\tfrac{1}{2}+\epsilon)}L^2\rightarrow z^{\tfrac{3}{2}(\tfrac{1}{2}+\epsilon )} L^2 $ and
 $$ \tilde{\QQ} G_0 = \Id ,\ \ \lVert G_0 z^{3/2} \rVert \leq c_0 h^{-1} = a_0$$
Ingredient 2 (\ref{weL2X1}) gives the weighted $L^2$ estimate from $z^NL^2  \rightarrow   x^{\alpha} L^2$ with
$$\tilde{\QQ} G_1 = \Id + E; \ \lVert G_1  \rVert \leq c_1 h^{\alpha_{\text{sml}}} = a_1 ;\    \lVert E \rVert = \mathsf{O}(h^{\infty}) $$
 Define
  $$ F := \psi_0 G_0 \chi_0 + \psi_1 G_1 \chi_1 $$
$$
 \tilde{\QQ} F = \Id + \psi_1 E \chi_1 + [\tilde{\QQ},\psi_0] G_0 \chi_0 + [\tilde{\QQ} , \psi_1] G_1 \chi_1= \Id + A_3 + A_0 + A_1$$
We have $$A_0^2 = A_1^1 =0$$
$$ \tilde{\QQ} ( F - FA_0) 
= \Id + A_3 - A_3A_0  + A_1 - A_1 A_0 $$
$$\tilde{\QQ}( F - FA_0 - FA_1) 
= \Id + A_3 - A_3(A_0 + A_1) - A_1 A_0 - A_0 A_1$$
$$
\tilde{\QQ}(F - FA_0 - FA_1 + F A_1 A_0) 
= \Id + A_3 - A_3 (A_0 + A_1 - A_1A_0) -A_0 A_1 + A_0A_1 A_0 \\
= \Id + \mathsf{E} $$
$ A_3 - A_3 (A_0 + A_1 - A_1A_0)$ is $\mathsf{O}(h^{\infty})$ and by the argument in the paper (\cite{Andras-Kiril}) $A_0 A_1$ is $\mathsf{O}(h^{\infty})$, so is $A_0A_1A_0$. Hence $\mathsf{E}\in \mathsf{O}(h^{\infty})$. For small enough $h$ the inverse for the term exists. The global resolvent is then
$$R = \big(F - FA_0 - FA_1 + F A_1 A_0\big)  ( \Id + \mathsf{E})^{-1} $$
Consider the bound for 
$$F - FA_0 - FA_1 + F A_1 A_0
 = F - \psi_1 G_1 \chi_1 A_0 + \psi_0 G_0 \chi_0 (-A_1 + A_1 A_0) $$
to get
$$\lVert R \rVert \leq C ( a_0 + a_1 + 2 h a_0 a_1 + h^2 a_0^2 a_1 )  \leq C h^2 a_0^2 a_1= Ca_1$$
with the norm near infinity is taken in the same weighted spaces as for those for $G_0$, while that near zero we use the weighted space for $G_1$. 

$\textbf{Mapping properties}$ : The mapping property of $R$ is given by that of $F=\psi_0 G_0 \chi_0 + \psi_1 G_1 \chi_1$. The oscillatory behavior at infinity of the image is given by $G_0$ while the boundary behavior prescribed by $\lambda$ at $z\sim 0$ is given by $G_1$. Hence $R$ maps to the space $\text{Osc}_{\mathbb{F}_C}$ as defined in (\ref{defOsc}).

\subsection{$Z_0 < 0$}

$$Q_-  \stackrel{ \mathsf{t} = \sigma^{-1}}{=} -\mathsf{t}^{-4} \Big[ -h^2(\mathsf{t}^2 \partial_{\mathsf{t}})^2 + \mathsf{t}^{2/3}\tilde{V}^- - \tfrac{4}{9}    \Big]; 
  \ \tilde{V}^- =   h^2 a\mathsf{t}^{7/3} \partial_{\mathsf{t}} - \tfrac{4}{9} \lambda h^2\mathsf{t}^{4/3} + \tfrac{4}{9}  $$
  The semiclassical scattering principal symbol is
  $$\xi^2 +   \tfrac{4}{9} \mathsf{t}^{2/3} - \tfrac{4}{9}$$
   
 \subsubsection{Ingredient 1}
  Denote by $$v^- = V^-|_{h=0}  \Rightarrow v^- =   \tfrac{4}{9} \mathsf{t}^{2/3} = \tfrac{4}{9} \sigma^{-2/3}.$$ 
  We find a global operator on $\RR_{\sigma}$ whose behavior on $\sigma > 1$ coincides with those of $\QQ_-$. Define
 $$\QQ_{\text{global};-}  = \chi(\sigma) \bar{\QQ}_- ; \ \bar{\QQ}_- = \begin{cases}  -h^2\partial_{\sigma}^2 + V^- - \tfrac{4}{9}  & \sigma > 1 \\   -h^2\partial_{\sigma}^2 + \tilde{V}^- - \tfrac{4}{9}   \end{cases} $$
 $$\chi \in \mathcal{C}^{\infty} (\RR),\ \ \phi (\sigma) = \begin{cases}  \sigma^4 & \sigma \geq 1\\ \tfrac{1}{2} & \sigma  \leq \tfrac{1}{2} \end{cases}$$
  $\tilde{V}^-$ is such that  $\tilde{V}^-|_{\sigma  > 1} = V^-$ and on $h=0$, $\tilde{V}^-$ has the profile of $\tilde{v}^-$ in figure (\ref{nontrappingPo2}). 
  \begin{figure}
  \centering
  \includegraphics[scale=0.85]{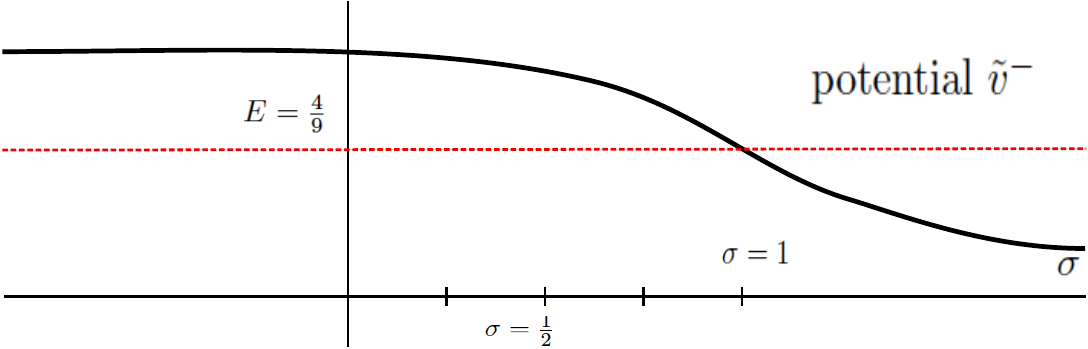}
  \caption{Extending the local problem of $\QQ_-$ at infinity to a global problem on $\sigma$}
  \label{nontrappingPo2}
  \end{figure}
  In particular the problem is semiclassical elliptic for $\sigma < 1$. $\QQ_{\text{global};-}$ satisfies global non-trapping with the profile of the characteristic set given by figure \ref{nontrappingBi2}. After this step we follow as in ingredient 1 for $Z_0 >0$. 
 \begin{figure}
  \centering
  \includegraphics{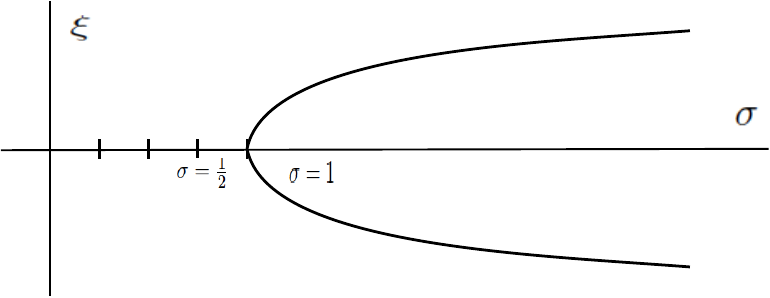}
  \caption{The bicharacteristic set of the global operator $\QQ_{\text{global};-}$}
  \label{nontrappingBi2}
  \end{figure}
\subsubsection{Ingredient 2}
 Similarly as in $Z_0>0$ but in this case it is even simpler since we have exponential decay at the semiclassical face. 

\subsubsection{ Gluing the two above estimates to get global resolvent estimate:} Since the gluing region is in $[\tfrac{1}{2}, \tfrac{3}{4}]$ where the problem is elliptic, we do not have propagation of singularities.\nl
Denote by $X_0 = \{ \sigma > 1/2\} , X_1 = \{ z < 3/4 \}$, $X_1 $ (trapping region). Let $\chi_1, \chi_0$ be a partition of unity, $\psi_j \sim 1$ near $\supp \chi_j$. Ingredient one gives the estimate in the weighted $L^2$ bound  
$$\mathsf{z}^{-\tfrac{3}{2}(\tfrac{1}{2}+\epsilon)}L^2(X_0) \rightarrow z^{\tfrac{3}{2}(\tfrac{1}{2}+\epsilon )} L^2(X_0) $$
  Define $$ F := \psi_0 G_0 \chi_0 + \psi_1 G_1 \chi_1 $$
   $$ \tilde{\QQ} F = \Id + \psi_1 E \chi_1 + [\tilde{\QQ} ,\psi_0] G_0 \chi_0 + [\tilde{\QQ}  , \psi_1] G_1 \chi_1= \Id + A_3 + A_0 + A_1$$
Since the problem is semiclassical elliptic, with $[\tilde{\QQ} ,\psi_0]$ and $\chi_0$ having disjoint supports,  $A_0 = \mathsf{O}(h^{\infty})$. Similarly for $A_1$. From here we procceed as in $Z_0 > 0$.

\subsection{$Z_0$ bounded}
Similar to $Z_0 < 0$.

\section{Construction of approximate solution on the blown-up space $C$: details}\label{detailofCon}

    In the following subsections, we fill out the details of section \ref{mainIdeaApprox} and construct the approximate solution to (\ref{originalProAfterF}) on the blown-up space $C$. We refer the readers to section \ref{mainIdeaApprox} for motivation and the main idea of the construction. See figure \ref{blowupspaceCnB} for blown-up space C.
 
\subsection{Step 1 - Approximate solution on $\mathbb{F}_C$ }
   
   \subsubsection{Step 1a : Getting the local approximate solution from the local parametrix}
 
  For convenience, we list the properties of interest for the local semiclassical parametrix constructed in (\ref{parametrix}) 
    $$\tilde{\QQ} U  - \text{SK}_{\Id} = E \in    h^{\infty} z^{\infty} (z')^{\sqrt{\tfrac{n^2}{4}-\lambda}  }$$
 $$U = U_0+\exp(-i\varphi_{\text{in}}) U_{\text{in}} +  \exp(-i\varphi_{\text{out}}) U_{\text{out}}  +  \exp(-i\varphi_{\text{lim}}) U_{\varphi_{\text{lim}}}  +U_{\textbf{bdy}} + U_{\tilde{\mathfrak{B}}_1}$$
  where $U_0, U_{\text{in}}$ and $U_{\text{out}}$ are supported away from $\mathcal{L}$, hence only these following terms are of interest for the purpose of (\ref{localApSo})
 $$U_{\varphi_{\text{lim}}} \in \rho_{\mathcal{A}}^{\infty} \rho_{\mathcal{F}}^{\infty}\rho_{\mathfrak{B}_1}^{\tilde{\gamma}_{\mathfrak{B}_1} } \rho_{\mathcal{L}}^{\sqrt{\tfrac{n^2}{4}-\lambda}}  \rho_{\mathfrak{F}_1}^{\tilde{\gamma}_{\mathfrak{B}_1}  + \tfrac{1}{2}} \mathcal{C}^{\infty}(\mathfrak{M}_{b,h});\ \supp U_{\varphi_{\text{lim}}} \ \text{is close to}\ \mathfrak{F}_1 ; \ \tilde{\gamma}_{\mathfrak{B}_1} = 0$$
 where close to $\mathfrak{F}_1$ and $\mathcal{L}$ : $\varphi_{\text{lim}} = \varphi_{\text{in}}$ with $\varphi_{\text{in}} =   h^{-1}\lvert \varphi (z) - \varphi(z')\rvert$,  $\varphi(z) = \tfrac{2}{3}\big[ (z+1)^{3/2} -1\big] \sgn\theta_n$.
%
    $$U_{\textbf{bdy}} \in \rho_{\mathcal{R}}^{\sqrt{\tfrac{n^2}{4}-\lambda}}\rho_{\mathcal{L}}^{\sqrt{\tfrac{n^2}{4}-\lambda}} \rho_{\mathfrak{F}_2}^{\infty} \rho_{\mathfrak{B}_1}^0 \rho_{\tilde{\mathfrak{B}}_1}^0\rho_{\mathcal{F}}^0\,\mathcal{C}^{\infty}(\mathfrak{M}_{b,h}) ; \ \supp U_{\mathfrak{B}_1} \text{is away from}\ \mathcal{A}, \mathfrak{F}_i, \tilde{\mathcal{A}}, \tilde{\mathfrak{F}}_i ; i =1,2$$
$$U_{\tilde{\mathfrak{B}}_1} \in \exp(- \varphi_{\tilde{\mathfrak{B}}_1} )  \rho_{\tilde{\mathcal{A}}}^{\infty} \rho_{\mathcal{R}}^{\sqrt{\tfrac{n^2}{4}-\lambda}}\rho_{\mathcal{L}}^{\sqrt{\tfrac{n^2}{4}-\lambda}} \rho_{\tilde{\mathfrak{F}}_2}^{\infty} \rho_{\tilde{\mathfrak{F}}_1}^{\infty}\rho_{\tilde{\mathfrak{B}}_1}^0 \rho_{\mathcal{F}}^0\, \mathcal{C}^{\infty}(\mathfrak{M}_{b,h})$$
    $$\varphi_{\tilde{\mathfrak{B}}_1} = h^{-1} \lvert  \tilde{\varphi}(z) - \tilde{\varphi}(z') \rvert ; \ \ \tilde{\varphi} = \tfrac{2}{3}\big[ 1 - (1 - z)^{3/2} \big]$$

We want to get an approximate solution $\tilde{u}_{\mathbb{F}_C}$ on $[0,1)_z\times \RR_{Z_0}$  such that 
\begin{equation}\label{localApSo}
 \begin{cases} \tilde{\QQ} \tilde{u}_{\mathbb{F}^0_C} = \tilde{e}_{\mathbb{F}^0_C}  \in h^{\infty} z^{\infty} \mathcal{C}^{\infty}(\mathbb{F}^0_C)\\
  \tilde{u}_{\mathbb{F}^0_C} = \rho_{\mathbb{F}_C^{(1)}}^{\sqrt{\tfrac{n^2}{4} - \lambda} }  \tilde{u}_{\mathbb{F}^0_C}^+ + \rho_{\mathbb{F}_C^{(1)}}^{-\sqrt{\tfrac{n^2}{4} - \lambda} }  \tilde{u}_{\mathbb{F}^0_C}^- ; \ \ \text{with}\  \tilde{u}_{\mathbb{F}^0_C}^-|_{\rho_{\mathbb{F}_C^{(1)}} = 0} = 1 ; \  \tilde{u}_{\mathbb{F}^0_C}^{\pm} \in \mathcal{C}^{\infty}(\mathbb{F}^0_C)  (*)\end{cases}
  \end{equation}
  See (\ref{defOft}) for the definition of $\rho_{\mathbb{F}_C^{(1)}}$ and (\ref{definitionOfz}) for the definition of $z$. 
 
  \begin{remark}
  If we modify $U$ by factor dependent only on $h$ and $z'$ then the resulting term still satisfies $\tilde{\QQ}u = 0$ up to a trivial smooth term.
  \end{remark}
 In order to to this, partition the front face $\mathbb{F}_C$ into $\mathbb{F}_{C,0}, \mathbb{F}_{C,+} =  \mathbb{F}_{C,+}^{\infty} \cup \mathbb{F}_{C,+}^0, \mathbb{F}_{C,-} =  \mathbb{F}_{C,-}^0 \cup  \mathbb{F}_{C,-}^{\infty} $ as pictured in figure (\ref{ffFC}). \begin{figure}
 \centering
  \includegraphics[scale=0.7]{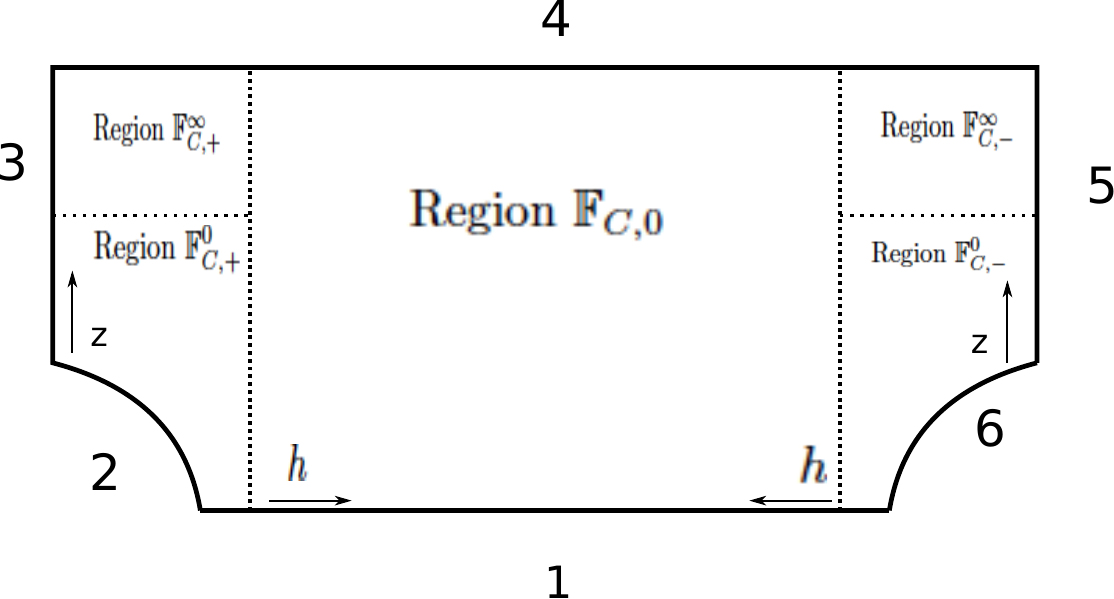}
  \caption{$\mathbb{F}_{C,+}^{0,0}$ denotes the neighborhood in region $\mathbb{F}_{C,+}^0$ close to the corner of $t = 0, h=0$ where $t$ is the projective blow-up with $t = \tfrac{z}{h}$, while $\mathbb{F}_{C,+}^{0,\infty}$ is the neighborhood close to $z = 0, t^{-1} = 0$. Similar notation for $\mathbb{F}_{C,-}^{0,\infty}$ and $\mathbb{F}_{C,-}^{0,0}$ for the other side where $Z_0 < 0$ of $\mathbb{F}_C$. For $i = 1,\cdots, 6$ the numbers indicate the label of the corresponding boundary face $\mathbb{F}_C^{(i)}$. Denote by $\mathbb{F}_C^0$ the region of $\mathbb{F}_C$ away from boundary face $\mathbb{F}_C^{(4)}$.\newline See (\ref{defOft}) for the definition of $\rho_{\mathbb{F}_C^{(1)}}$ and (\ref{definitionOfz}) for the definition of $z$. $\rho_{\mathbb{F}_C^{(4)}} = \sigma^{-1}$. See (\ref{definitionSig}) for definition of $\sigma$.}
  \label{ffFC}
\end{figure}

\begin{itemize}
\item Consider the portion of  $U_{\textbf{bdy}} = $  close to $\mathcal{L}$, denoted by $U_{\textbf{bdy}}^{\mathcal{L}}$. Close to $\mathcal{L} $
$$ \rho_{\mathcal{L}} \sim \mathsf{Z}; \ \mathsf{Z}^{-\sqrt{\tfrac{n^2}{4}-\lambda}} = z^{-\sqrt{\tfrac{n^2}{4}-\lambda}} (z')^{\sqrt{\tfrac{n^2}{4}-\lambda}}$$
$h$ lifts to $\rho_{\mathfrak{B}_1} \rho_{\mathfrak{F}_1} \rho_{\mathcal{A}} \rho_{\mathfrak{F}_2}\rho_{\tilde{\mathfrak{B}}_1}$.
$z'$ lifts to $\rho_{\mathcal{F}} \rho_{\mathfrak{B}_1} \rho_{\mathcal{L}} \rho_{\mathfrak{F}_1}\rho_{\tilde{\mathfrak{B}}_1}$. 
$z$ lifts to $\rho_{\mathcal{F}} \rho_{\mathfrak{B}_1} \rho_{\mathcal{R}} \rho_{\mathfrak{F}_2}\rho_{\tilde{\mathfrak{B}}_1}$. So we have:
$$h^m  (z')^{-\sqrt{\tfrac{n^2}{4}-\lambda}} U_{\textbf{bdy}} \in \rho_{\mathcal{F}}^{-\sqrt{\tfrac{n^2}{4}-\lambda}}\rho_{\mathfrak{B}_1}^{m-\sqrt{\tfrac{n^2}{4}-\lambda}}\rho_{\tilde{\mathfrak{B}}_1}^{m-\sqrt{\tfrac{n^2}{4}-\lambda}} \rho_{\mathcal{R}}^{\sqrt{\tfrac{n^2}{4}-\lambda}} \rho_{\mathcal{L}}^0\, \mathcal{C}^{\infty}(\mathfrak{M}_{b,h})$$
$$h^m  (z')^{-\sqrt{\tfrac{n^2}{4}-\lambda}} U_{\textbf{bdy}} =  \rho_{\mathcal{F}}^{-\sqrt{\tfrac{n^2}{4}-\lambda}} u_{\textbf{bdy}} $$
here we list only the exponents of faces that do not have infinite order vanishing. \newline If we choose $m =  \sqrt{\tfrac{n^2}{4}-\lambda}$, then $u_{\textbf{bdy}}$ is smooth down to $\mathcal{L}$ and $\mathcal{R}$ with value in distribution on $\mathfrak{M}_{b,h}$. Note $u_{\textbf{bdy}}|_{\mathcal{F}\cap \mathcal{L}}$ is smooth in $h$. 
Denote by $c_{\textbf{bdy}}$ a smooth function on $\mathfrak{M}_{b,h}$ such that $c_{\textbf{bdy}}$ is constant for $h$ constant and its restriction to $\mathcal{F}\cap \mathcal{L}$ equal $u_{\textbf{bdy}}|_{\mathcal{F}\cap \mathcal{L}}$. Constructed this way, $c_{\textbf{bdy}}$ only depends on $h$. The following term
$$\lim_{z' \rightarrow 0 }  h^m (z')^{-\sqrt{\tfrac{n^2}{4}-\lambda}} c_{\textbf{bdy}}^{-1} U_{\textbf{bdy}}^{\mathcal{L}}     $$
satisfies the boundary condition (*) in $(\ref{localApSo})$ ie at $\mathbb{F}_C^{(1)}$, smooth down to $\mathbb{F}_C^{(2)}$ and $\mathbb{F}_C^{(6)}$, and is supported away from $\mathbb{F}_C^{(3)}$ .\newline

\item With this choice of $m$, consider the asymptotic properties of 
$$\lim_{z'\rightarrow 0}  h^m (z')^{-\sqrt{\tfrac{n^2}{4}-\lambda}} c_{\textbf{bdy}}^{-1} U_{\varphi_{\text{lim}}} $$
We have:
$$ h^m (z')^{-\sqrt{\tfrac{n^2}{4}-\lambda}} U_{\varphi_{\text{lim}}} \in\rho_{\mathcal{A}}^{\infty} \rho_{\mathcal{F}}^{\infty}\rho_{\mathfrak{B}_1}^{m- \sqrt{\tfrac{n^2}{4}-\lambda}+ \tilde{\gamma}_{\mathfrak{B}_1} } \rho_{\mathcal{L}}^0 \rho_{\mathfrak{F}_1}^{m- \sqrt{\tfrac{n^2}{4}-\lambda}+ \tilde{\gamma}_{\mathfrak{B}_1}  + \tfrac{1}{2}} \mathcal{C}^{\infty}(\mathfrak{M}_{b,h}) $$
With the choice of $m = \sqrt{\tfrac{n^2}{4}-\lambda}$ we have
$$h^m (z')^{-\sqrt{\tfrac{n^2}{4}-\lambda}} U_{\varphi_{\text{lim}}} \in \rho_{\mathfrak{F}_1}^{\tilde{\gamma}_{\mathfrak{B}_1} + \tfrac{1}{2}} \rho_{\mathfrak{B}_1}^{\tilde{\gamma}_{\mathfrak{B}_1} }  \rho_{\mathcal{L}}^0 \rho_{\mathcal{A}}^{\infty} \rho_{\mathcal{F}}^{\infty}  ; \ \tilde{\gamma}_{\mathfrak{B}_1} = 0 $$ 
$$\Rightarrow \lim_{z'\rightarrow 0}  h^m (z')^{-\sqrt{\tfrac{n^2}{4}-\lambda}} c_{\textbf{bdy}}^{-1} U_{\varphi_{\text{lim}}}  $$
 is supported in region $\mathbb{F}^0_{C,+}$, away from $\mathbb{F}_C^{(1)}$ and $\mathbb{F}_C^{(4)}$ and at boundary faces $\mathbb{F}_C^{(3)}$ and $\mathbb{F}_C^{(2)}$ it has exponent $\tilde{\gamma}_{\mathfrak{B}_1} + \tfrac{1}{2} $ , $ \tilde{\gamma}_{\mathfrak{B}_1} =0 $ respectively.
\item Similarly, 
$$\lim_{z'\rightarrow 0}  h^m (z')^{-\sqrt{\tfrac{n^2}{4}-\lambda}} c_{\textbf{bdy}}^{-1} U_{\tilde{\mathfrak{B}}_1}$$
 is supported in region $\mathbb{F}^0_{C,-}$, away from $\mathbb{F}_C^{(1)}$ and $\mathbb{F}_C^{(4)}$, vanishes rapidly at $\mathbb{F}_C^{(5)}$ and has exponent $\tilde{\gamma}_{\mathfrak{B}_1} $ at $\mathbb{F}_C^{(6)}$.
 \end{itemize}
 
 \textbf{Summary :}
  $$ \lim_{z'\rightarrow 0}  h^{\sqrt{\tfrac{n^2}{4}-\lambda}}  (z')^{-\sqrt{\tfrac{n^2}{4}-\lambda}} c_{\textbf{bdy}}^{-1}U  = \tilde{u}_{\mathbb{F}_C^0}$$
 $ \tilde{u}_{\mathbb{F}_C^0} $ satisfies (\ref{localApSo}) and
  $$ \tilde{u}_{\mathbb{F}_C^0} =  u_{\mathbb{F}_{C, 0}^0} + u_{\mathbb{F}_{C, +}^0 } +  u_{\mathbb{F}_{C, -}^0 } $$
 with
 \begin{itemize}
 \item $u_{\mathbb{F}_{C,0}^0}$ lives on $[0,1)_z  \times \RR^+_{Z_0}$ 
 $$u_{\mathbb{F}_C}^{(2)} \in \rho_{\mathbb{F}_C^{(1)}}^{\sqrt{\tfrac{n^2}{4}-\lambda}}  \mathcal{C}^{\infty} (\mathbb{F}_{C,0}^0)  +   \rho_{\mathbb{F}_C^{(1)}}^{-\sqrt{\tfrac{n^2}{4}-\lambda}}  \mathcal{C}^{\infty} (\mathbb{F}_{C,0}^0) $$ 
  \item $u_{\mathbb{F}_{C,+}^0} $ is supported in $\mathbb{F}_{C,+}^0$, in more details 
 $$u_{\mathbb{F}_{C,+}^0 } \in  \exp(-i \phi )  \rho_{\mathbb{F}_C^{(2)}}^{\tilde{\gamma}_{\mathfrak{B}_1} } \rho_{\mathbb{F}_C^{(3)}}^{\tilde{\gamma}_{\mathfrak{B}_1}+\tfrac{1}{2}}  \mathcal{C}^{\infty} (\mathbb{F}_{C,+}^{0,\infty}) 
 \ +\  \rho_{\mathbb{F}_C^{(2)}}^0 \rho_{\mathbb{F}_C^{(1)}}^{\sqrt{\tfrac{n^2}{4}-\lambda}}  \mathcal{C}^{\infty} (\mathbb{F}_{C,+}^{0,0})  + \rho_{\mathbb{F}_C^{(2)}}^0 \rho_{\mathbb{F}_C^{(1)}}^{-\sqrt{\tfrac{n^2}{4}-\lambda}}  \mathcal{C}^{\infty} (\mathbb{F}_{C,+}^{0,0})  $$
where the phase function comes from $\lim_{z'\rightarrow 0} \varphi_{\text{lim}}$ and in coordinates $\sigma, h$ and $\nu, \sigma$ has the form:
 $$\phi = \dfrac{\varphi}{h} =   \dfrac{\varphi}{\sigma^{2/3}\nu^{2/3}} ; \ \ \varphi = \tfrac{2}{3} \big[ (\sigma^{2/3} + 1)^{3/2} - 1 \big]  $$
\item $u_{\mathbb{F}_{C,-}^0} $ is supported in $\mathbb{F}_{C,-}^0$, in more details:
 $$u_{\mathbb{F}_{C, - }^0 }\in  \exp(- \tfrac{ \tilde{\varphi}}{h}) \rho_{\mathbb{F}_C^{(6)}}^{\tilde{\gamma}_{\mathfrak{B}_1}}  \mathcal{C}^{\infty} (\mathbb{F}_{C,-}^{0,\infty})\   +\  \rho_{\mathbb{F}_C^{(6)}}^0 \rho_{\mathbb{F}_C^{(1)}}^{\sqrt{\tfrac{n^2}{4}-\lambda}}  \mathcal{C}^{\infty} (\mathbb{F}_{C,-}^{0,0})  + \rho_{\mathbb{F}_C^{(6)}}^0 \rho_{\mathbb{F}_C^{(1)}}^{-\sqrt{\tfrac{n^2}{4}-\lambda}}  \mathcal{C}^{\infty} (\mathbb{F}_{C,-}^{0,0})   $$
 $$\tilde{\varphi} = \tfrac{2}{3} \big[ 1 - (1 - \sigma^{2/3})^{3/2}\big];\  \tilde{\gamma}_{\tilde{\mathfrak{B}}_1} = 0$$

 \end{itemize}
 
 We have a local approximate solution $u_{\mathbb{F}_C^0} $ on this region with error $e_{\mathbb{F}_C^0}$ :  
\begin{equation}\label{localApSo2}
 \begin{cases} \tilde{\QQ} u_{\mathbb{F}^0_C} = e_{\mathbb{F}^0_C} ;  \ \ e_{\mathbb{F}^0_C}|_{\mathbb{F}_C^0} \in h^{\infty} z^{\infty} \mathcal{C}^{\infty}(\mathbb{F}^0_C)\\
  \tilde{u}_{\mathbb{F}^0_C} = \rho_{\mathbb{F}_C^{(1)}}^{\sqrt{\tfrac{n^2}{4} - \lambda} }  \tilde{u}_{\mathbb{F}^0_C}^+ + \rho_{\mathbb{F}_C^{(1)}}^{-\sqrt{\tfrac{n^2}{4} - \lambda} }  \tilde{u}_{\mathbb{F}^0_C}^- ; \ \ \text{with}\  \tilde{u}_{\mathbb{F}^0_C}^-|_{\rho_{\mathbb{F}_C^{(1)}} = 0} = 1 ; \  \tilde{u}_{\mathbb{F}^0_C}^{\pm} \in \mathcal{C}^{\infty}(\mathbb{F}^0_C)  (*)\end{cases}
  \end{equation}

 \subsubsection{Step 1b : solve away $e_{\mathbb{F}_C^0}$ along the global transport equation in $\sigma$ away to infinity}\
 
 
 \begin{remark}
 The computation in this section is carried out in $\sigma,h$ where $z = \sigma^{2/3}$ but could have been done in $z,h$.
 \end{remark}
\textbf{For the region} $Z_0 > 0$ : We deal with the portion of $e_{\mathbb{F}_C^0}$ that is supported close the corner of $\mathbb{F}_C^{(3)}$ and $\mathbb{F}_C^{(4)}$. To get the eikonal equation and the transport equations, we conjugate our operator by the oscillatory $\exp(-i\tfrac{\varphi}{h} )$:
\begin{align*}
e^{i\tfrac{\varphi}{h} } \tilde{\QQ} e^{-i\tfrac{\varphi}{h} }= &\  e^{i\tfrac{\varphi}{h} }
h^{-2}\sigma^2\left[ h^2\partial_{\sigma}^2 + b \dfrac{\partial_{\sigma}}{\sigma} h^2  + \dfrac{4}{9} \sigma^{-2/3} + \dfrac{4}{9}+\dfrac{4}{9}(\lambda-\tfrac{n^2}{4}) h^2\sigma^{-2}\right] e^{-i\tfrac{\varphi}{h} }u ; \ b = 1\\
= &\ \sigma^2\partial_{\sigma}^2 u + \left( 2 h^{-1}i ( \sigma\partial_{\sigma}\varphi) + b\right)  \sigma\partial_{\sigma} u\\
 &\ + \left[ih^{-1} \sigma^2\partial_{\sigma}^2 \varphi - h^{-2} \sigma^2(\partial_{\sigma}\varphi)^2 + \dfrac{4}{9}h^{-2}\sigma^{4/3} + \dfrac{4}{9}h^{-2} \sigma^2 +i b h^{-1} \sigma (\partial_{\sigma}\varphi) + \dfrac{4}{9}(\lambda-\tfrac{n^2}{4}) \right] u
 \end{align*}
 We chose
 $$\varphi = \tfrac{2}{3}\left[(\sigma^{2/3}+1)^{3/2} -1\right]\sgn\theta_n $$
 These computation will be useful in determining the behavior of the regular singular ODE at $\sigma^{-1} = 0$
\begin{align*}
\partial_{\sigma} \varphi &= (\sigma^{2/3}+1)^{1/2} \tfrac{2}{3}\sigma^{-1/3} \sgn\theta_n 
 \stackrel{\tilde{\sigma}=\sigma}{=}  (\tilde{\sigma}^{2/3}+1)^{1/2} \tfrac{2}{3}\tilde{\sigma}^{-1/3} \sgn\theta_n \\
\partial_{\sigma}^2 \varphi & 
=\tfrac{2}{9} (\sigma^{2/3}+1)^{-1/2} \sigma^{-2/3} \sgn\theta_n + \tfrac{2}{9} (\sigma^{2/3} + 1)^{1/2} \sigma^{-4/3} \sgn\theta_n\\
\Rightarrow \dfrac{\sigma\partial_{\sigma}^2 \varphi}{2\partial_{\sigma}\varphi} &= \dfrac{1}{6} (\sigma^{2/3}+1)^{-1} 
 \end{align*}
 We will look for a correction term $u_{\mathbb{F}_{C,+}^{\infty}}$ with
  $$u_{\mathbb{F}_{C,+}^{\infty}} = \exp(-i \phi) h^{n_3}\sum   h^j u_{\mathbb{F}_{C,+}^{\infty};j} ;\  n_3 = \tilde{\gamma}_{\mathfrak{B}_1}+\tfrac{1}{2},\ \tilde{\gamma}_{\mathfrak{B}_1} = 0 \ \text{and}$$
$$P u_{\mathbb{F}_{C,+}^{\infty}}  + e_{\mathbb{F}_C^0}\in \dot{\mathcal{C}}^{\infty} ( \mathbb{F}_{C,+})$$
Expand $e_{\mathbb{F}_C^0}$ at $\mathbb{F}_C^{(3)}$
$$e_{\mathbb{F}_C^0} \sim  \exp(-i \phi) h^{n_3}\sum   h^j e_{\mathbb{F}_C^0;j}$$
Apply the conjugated operator to $h^{n_3}\sum   h^j u_{\mathbb{F}_{C,+}^{\infty};j}$ and gather the terms according to power of $h$ (starting from lowest) and set them to zero :
\begin{itemize}
\item \textbf{Eikonal equation} is satisfied by our choice of $\varphi$ - coefficient of $h^{n_3}$, ie the term in the bracket vanishes
$$ u_{\mathbb{F}_{C,+}^{\infty};0} \left[ -(\partial_{\sigma}\varphi)^2 + \dfrac{4}{9}  +\dfrac{4}{9} \sigma^{-2/3}  \right] = 0$$
since
$$ -(\sigma^{2/3} +1) \dfrac{4}{9} \sigma^{-2/3} + \dfrac{4}{9}   +\dfrac{4}{9} \sigma^{-2/3} 
= -\dfrac{4}{9} - \dfrac{4}{9}\sigma^{-2/3} + \dfrac{4}{9}   +\dfrac{4}{9} \sigma^{-2/3}  =0$$

\item \textbf{First transport equation} - coefficient of $h^{n_3+1}$ : 
$$\begin{cases} \left[2i(\partial_{\sigma}\varphi) \partial_{\sigma}  +i \partial^2_{\sigma}\varphi  + ib \sigma^{-1} \partial_{\sigma} \varphi\right] u_{\mathbb{F}_{C,+}^{\infty};0} 
+ (-\partial_{\sigma}^2\varphi + \dfrac{4}{9}  +\dfrac{4}{9} \sigma^{-2/3})u_{\mathbb{F}_{C,+}^{\infty};1} = e_{\mathbb{F}_C^0;0}  \\ u_{\mathbb{F}_{C,+}^{\infty};0} |_{\sigma = 0}  = 0 \end{cases} $$
where $e_{\mathbb{F}_C^0;0}$ vanishes rapidly at both $\sigma = 0$ and $\sigma\rightarrow \infty$. Since $\varphi$ satisfies the eikonal equation, we obtain the first transport equation for $u_{\mathbb{F}_{C,+}^{\infty};0}$:
\begin{align*}
 &\left[2i(\sigma\partial_{\sigma}\varphi)\sigma \partial_{\sigma}  + i\sigma^2\partial^2_{\sigma}\varphi + ib \sigma \partial_{\sigma} \varphi\right] u_{\mathbb{F}_{C,+}^{\infty};0}  = e_{\mathbb{F}_C^0;0}\\
& \Leftrightarrow 2i\sigma\partial_{\sigma}\varphi\left[ \sigma \partial_{\sigma}  -\dfrac{1}{6}     (\sigma^{2/3} +1)^{-1}  + \dfrac{1}{2}b  \right] u_{\mathbb{F}_{C,+}^{\infty};0} = e_{\mathbb{F}_C^0;0}
\end{align*}
 the solution has the asymptotic property $u_{\mathbb{F}_{C,+}^{\infty};0} \sim \sigma^{-\tfrac{1}{2} b}$ with $b = 1$ as $\sigma\rightarrow +\infty$.
 
\item \textbf{Second transport equation} - coefficient of $h^{n_3+2}$:
$$\begin{cases} 
\left[2i ( \partial_{\sigma}\varphi ) \partial_{\sigma} + i\partial_{\sigma}^2 \varphi + ib \sigma^{-1} (\partial_{\sigma} \varphi) \right] u_{\mathbb{F}_{C,+}^{\infty};1}  +  ( \partial_{\sigma}^2 + b \sigma^{-1} \partial_{\sigma} + \tfrac{4}{9}(\lambda-\tfrac{n^2}{4})\sigma^{-2} ) u_{\mathbb{F}_{C,+}^{\infty};1} = e_{\mathbb{F}_C^0;1}\\ u_{\mathbb{F}_{C,+}^{\infty};1} |_{\sigma = 0}  = 0 \end{cases}$$
where $e_{\mathbb{F}_C^0;1}$ vanishes rapidly at both $\sigma = 0$ and $\sigma\rightarrow \infty$. 
Substituting in the expression for $\sigma\partial_{\sigma}\varphi$ we obtain
 $$\dfrac{2i\partial_{\sigma}\varphi}{\sigma}\left[\left( \sigma \partial_{\sigma}  -\dfrac{1}{6}     (\sigma^{2/3} +1)^{-1} + \dfrac{1}{2}b  \right) u_{\mathbb{F}_{C,+}^{\infty};1}  -i  \dfrac{ \left[(\sigma\partial_{\sigma})^2 + (b-1)  \sigma\partial_{\sigma}+\tfrac{4}{9}(\lambda-\tfrac{n^2}{4})\right] u_{\mathbb{F}_{C,+}^{\infty};0} }{2\sigma\partial_{\sigma} \varphi} \right]= e_{\mathbb{F}_C^0;1} $$
Since $\sigma\partial_{\sigma}\varphi \sim \sigma$ and $ \tfrac{ \left[(\sigma\partial_{\sigma})^2 + (b-1)  \sigma\partial_{\sigma}+\tfrac{4}{9}(\lambda-\tfrac{n^2}{4})\right] }{2\sigma\partial_{\sigma} \varphi}\sim \sigma^{-\tfrac{1}{2}b - 1}$
the asymptotic behavior for $u_{\mathbb{F}_{C,+}^{\infty};1}$ at $\sigma^{-1} = 0$ is 
$$u_{\mathbb{F}_{C,+}^{\infty};1}   \sim \sigma^{-\tfrac{1}{2}b-1} ; b = 1\ \ \sigma\rightarrow \infty $$

\item Similarly, we obtain higher transport equations :
$$\begin{cases} \dfrac{2i\partial_{\sigma}\varphi}{\sigma}\left[\left( \sigma \partial_{\sigma}  -\dfrac{1}{6}     (\sigma^{2/3} +1)^{-1} + \dfrac{1}{2}b  \right) u_{\mathbb{F}_{C,+}^{\infty};j+1}  -i  \dfrac{ \left[(\sigma\partial_{\sigma})^2 + (b-1)  \sigma\partial_{\sigma}+\tfrac{4}{9}(\lambda-\tfrac{n^2}{4})\right] u_{\mathbb{F}_{C,+}^{\infty};j} }{2\sigma\partial_{\sigma} \varphi} \right]= 0 \\ u_{\mathbb{F}_{C,+}^{\infty};j} |_{\sigma = 0}  = e_{\mathbb{F}_C^0;j} \end{cases}$$
where $e_{\mathbb{F}_C^0;j}$ vanishes rapidly at both $\sigma = 0$ and $\sigma\rightarrow \infty$. Similar computation to get 
$$u_{\mathbb{F}_{C,+}^{\infty};j}   \sim \sigma^{-\tfrac{1}{2}b-j}; b = 1 \ \ \text{as}\  \sigma\rightarrow \infty$$
\end{itemize}
   $u_{\mathbb{F}_{C,+}^{\infty}}$ is now obtained by borel summation and has the following asymptotic property:
  $$u_{\mathbb{F}_{C,+}^{\infty}} = \exp(-i \phi) h^{\tilde{\gamma}_{\mathfrak{B}_1} -2}\sum   h^j u_{\mathbb{F}_{C,+}^{\infty};j} ; \ \  u_{\mathbb{F}_{C,+};j} \sim \sigma^{-\tfrac{1}{2}b-j} \ \text{as}\ \sigma \rightarrow \infty.$$

\textbf{For the region} $Z_0 < 0$ : We deal with the portion of $e_{\mathbb{F}_C^0}$ that is supported close the corner of $\mathbb{F}_C^{(4)}$ and $\mathbb{F}_C^{(5)}$. With $\bar{\varphi} = \tfrac{2}{3} \left[ i(\sigma^{2/3}-1)^{3/2}+1 \right]$ we compute the conjugated operator 
\begin{align*}
e^{\tfrac{\bar{\varphi}}{h}} \tilde{\QQ}_- e^{-\tfrac{\bar{\varphi}}{h}} &e^{\tfrac{\bar{\varphi}}{h}}\left[ h^2\partial_{\sigma}^2 + b h^2\dfrac{\partial_{\sigma}}{\sigma} - \dfrac{4}{9}\sigma^{-2/3} + \dfrac{4}{9}+ \dfrac{4}{9} (\lambda-\tfrac{n^2}{4}) h^2\sigma^{-2}\right]( e^{-\tfrac{\bar{\varphi}}{h}} u ); \ b=1\\
&=\sigma^2 \partial_{\sigma}^2 u + \left(2h^{-1} \sigma\partial_{\sigma}\bar{\varphi} + b \right) \sigma\partial_{\sigma} u \\
&+ \left[ h^{-1}\sigma^2\partial_{\sigma}^2\bar{\varphi} +h^{-2}\sigma^2(\partial_{\sigma}\bar{\varphi})^2 +bh^{-1} \sigma\partial_{\sigma}\bar{\varphi} - \dfrac{4}{9}\sigma^{4/3}h^{-2} + \dfrac{4}{9}h^{-2}\sigma^2  + \dfrac{4}{9}(\lambda-\tfrac{n^2}{4})\right] u \\
 \end{align*}
  We will look for a correction term $u_{\mathbb{F}_{C,-}^{\infty}}$ with
  $$u_{\mathbb{F}_{C,-}^{\infty}} = \exp(- \tilde{\phi}) h^s \sum   h^j u_{\mathbb{F}_{C,-}^{\infty};j} ; \ \ \tilde{\phi} = \tfrac{2}{3} h^{-1} \left[ i(\sigma^{2/3}-1)^{3/2}+1 \right] $$
$$\text{and}\ \ P u_{\mathbb{F}_{C,-} ^{\infty}}  + e_{\mathbb{F}_C^0} \in \dot{\mathcal{C}}^{\infty} ( \mathbb{F}_{C,-})$$
note here since the previous error term is $h^{\infty}$, $s$ is arbitrarily large.

We obtain the eikonal equation and the transport equations:
\begin{itemize}
\item \textbf{Eikonal equation} : Our choice of $\bar{\varphi} = \tfrac{2}{3} \left[(1- \sigma^{2/3})^{3/2} - 1\right]$ satisfies the equation
\begin{equation}
  \sigma^{-4/3} \left[- (\partial_{\sigma}\bar{\varphi})^2 - \dfrac{4}{9}\sigma^{-2/3} + \dfrac{4}{9} \right]A_0 =0
 \end{equation}
 since
 $$-\partial_{\sigma} \tfrac{2}{3}\left[i(\sigma^{2/3}-1)^{3/2}+1\right]
  = -\tfrac{2}{3}i(\sigma^{2/3}-1)^{1/2} \sigma^{-1/3}$$
 $$\Big( \tfrac{2}{3}i(\sigma^{2/3}-1)^{1/2} \sigma^{-1/3} \Big)^2 - \dfrac{4}{9}\sigma^{-2/3} + \dfrac{4}{9}
 = -\tfrac{4}{9} (\sigma^{2/3}-1) \sigma^{2/3}  - \dfrac{4}{9}\sigma^{-2/3} + \dfrac{4}{9} = 0  $$
  
  \item \textbf{First order transport equation}
  $$\begin{cases} \left[\sigma\partial_{\sigma} +\tfrac{1}{6} (\sigma^{2/3}-1)^{-1} +\dfrac{1}{2}b \right] u_{\mathbb{F}_{C,-}^{\infty};0} = 0  \\   u_{\mathbb{F}_{C,-}^{\infty};0}|_{\sigma = 0} = 0\end{cases} $$
   the solution as the asymptotic property $u_{\mathbb{F}_{C,-}^{\infty};0}  \sim \sigma^{-\tfrac{1}{2} b}$ as $\sigma\rightarrow +\infty$.

 \item \textbf{Second order transport equation}
 $$\begin{cases} \left[ \sigma \partial_{\sigma}  +\dfrac{1}{6}     (\sigma^{2/3} -1)^{-1} + \dfrac{1}{2}b  \right] u_{\mathbb{F}_{C,-}^{\infty};1}  
+  \dfrac{( \sigma^2\partial_{\sigma}^2 -b  \sigma\partial_{\sigma} +\tfrac{4}{9}\lambda) u_{\mathbb{F}_{C,-}^{\infty};0} }{2\sigma\partial_{\sigma} \varphi} = 0   \\ u_{\mathbb{F}_{C,-}^{\infty};1}|_{\sigma = 0} = 0\end{cases}$$
  we have 
  $$\sigma\partial_{\sigma}\varphi \sim \sigma;\ \  \tfrac{( \sigma\partial_{\sigma}^2 -b  \partial_{\sigma} ) a_0}{2\partial_{\sigma} \varphi} = \tfrac{( (\sigma\partial_{\sigma})^2 (1-b)  \sigma\partial_{\sigma} ) a_0}{2\sigma\partial_{\sigma} \varphi} \sim \sigma^{-\tfrac{1}{2}b - 1}$$
  so $u_{\mathbb{F}_{C,-}^{\infty};1}   \sim \sigma^{-\tfrac{1}{2}b-1} $ as $\sigma\rightarrow \infty$. 

\item \textbf{Higher order the transport equation for} $s+2+j$
$$\begin{cases} \left[ \sigma \partial_{\sigma}  +\dfrac{1}{6}     (\sigma^{2/3} -1)^{-1} + \dfrac{1}{2}b  \right] u_{\mathbb{F}_{C,-}^{\infty};j+1}  
+  \dfrac{( \sigma^2\partial_{\sigma}^2 +b \sigma \partial_{\sigma} +\tfrac{4}{9}\lambda) u_{\mathbb{F}_{C,-}^{\infty};j} }{2\sigma\partial_{\sigma} \varphi} = 0 \\ u_{\mathbb{F}_{C,-}^{\infty};j}|_{\sigma = 0} = 0  \end{cases} $$
 Thus  $u_{\mathbb{F}_{C,-}^{\infty};j}   \sim \sigma^{-\tfrac{1}{2}b-j} $ as $\sigma\rightarrow \infty$.

\end{itemize}

Similarly for when $Z_0$ is bounded, this we work with 
$$ (S\partial_S)^2 + (\tilde{b} -1 ) S\partial_S + \lambda + S^3 + Z_0 S^2; \tilde{b} = 1$$
and conjugate the operator by $\exp( -i\tfrac{2}{3} (S+Z_0 )^{3/2} \sgn\theta_n)$ to obtain a regular singular ODE and whose solution has the asymptotic property as $S \rightarrow +\infty$
$$  \sim e^{-i\tfrac{2}{3}(S + Z_0)^{3/2}\sgn\theta_n}(S + Z_0)^{-\tfrac{3}{4} } \mathcal{C}^{\infty}(\mathbb{F}_C) ; \ $$

 \subsubsection{Summary of step 1}
 We have constructed an approximate solution supported in $\mathbb{F}_{C,+}$ that is polyhomongeous conormal to all boundaries of this face, with an error vanishing to infinite order at all boundaries of $\mathbb{F}_C$. The polyhomogeneous behavior of $u_{\mathbb{F}_{C,+}}$ in more details:
 \begin{equation}\tilde{\QQ} \hat{U}_{\mathbb{F}_C}  \in \dot{\mathcal{C}}^{\infty} (\mathbb{F}_C) ;\ \ \hat{U}_{\mathbb{F}_C} = \hat{U}^{\mathbb{F}_{C,+}}  + \hat{U}^{\mathbb{F}_{C,0}} +  \hat{U}^{\mathbb{F}_{C,-}} \end{equation}
 \begin{equation}\label{asymptotics}
 \begin{aligned}
  \hat{U}^{\mathbb{F}_{C,+} } \in \exp(-i \phi )  \rho_{\mathbb{F}_C^{(2)}}^{\tilde{\gamma}_{\mathfrak{B}_1} } \rho_{\mathbb{F}_C^{(3)}}^{\tilde{\gamma}_{\mathfrak{B}_1}+\tfrac{1}{2}}  \mathcal{C}^{\infty} (\mathbb{F}_{C,+}^{0,\infty}) 
 \ &+\  \rho_{\mathbb{F}_C^{(2)}}^0 \rho_{\mathbb{F}_C^{(1)}}^{\sqrt{\tfrac{n^2}{4}-\lambda}}  \mathcal{C}^{\infty} (\mathbb{F}_{C,+}^{0,0})  + \rho_{\mathbb{F}_C^{(2)}}^0 \rho_{\mathbb{F}_C^{(1)}}^{-\sqrt{\tfrac{n^2}{4}-\lambda}}  \mathcal{C}^{\infty} (\mathbb{F}_{C,+}^{0,0}) \\
 &+ \exp(-i \phi )  \rho_{\mathbb{F}_C^{(3)}}^{ \tilde{\gamma}_{\mathfrak{B}_1} +\tfrac{1}{2}}   \rho_{\mathbb{F}_C^{(4)}}^{\tfrac{1}{2}}  \mathcal{C}^{\infty} (\mathbb{F}_{C,+}^{\infty})\\
 \hat{U}^{\mathbb{F}_{C, - } }\in  \exp(- \tfrac{ \tilde{\varphi}}{h}) \rho_{\mathbb{F}_C^{(6)}}^{\tilde{\gamma}_{\mathfrak{B}_1}}  \mathcal{C}^{\infty} (\mathbb{F}_{C,-}^{0,\infty})\   &+\  \rho_{\mathbb{F}_C^{(6)}}^0 \rho_{\mathbb{F}_C^{(1)}}^{\sqrt{\tfrac{n^2}{4}-\lambda}}  \mathcal{C}^{\infty} (\mathbb{F}_{C,-}^{0,0})  + \rho_{\mathbb{F}_C^{(6)}}^0 \rho_{\mathbb{F}_C^{(1)}}^{-\sqrt{\tfrac{n^2}{4}-\lambda}}  \mathcal{C}^{\infty} (\mathbb{F}_{C,-}^{0,0})\ \\
 & +   \exp(- \tilde{\phi} )  \rho_{\mathbb{F}_C^{(5)}}^{\infty}   \rho_{\mathbb{F}_C^{(4)}}^{\tfrac{1}{2}}  \mathcal{C}^{\infty} (\mathbb{F}_{C,-}^{\infty})\\
  \hat{U}^{\mathbb{F}_{C, 0 } } \in \rho_{\mathbb{F}_C^{(1)}}^{\sqrt{\tfrac{n^2}{4}-\lambda}}  \mathcal{C}^{\infty} (\mathbb{F}_{C,0}^0)  &+   \rho_{\mathbb{F}_C^{(1)}}^{-\sqrt{\tfrac{n^2}{4}-\lambda}}  \mathcal{C}^{\infty} (\mathbb{F}_{C,0}^0)+\ e^{-i\tfrac{2}{3}(S + Z_0)^{3/2}\sgn\theta_n} \rho_{\mathbb{F}_C^{(4)}}^{\tfrac{1}{2}} \mathcal{C}^{\infty}(\mathbb{F}_{C,0}^{\infty}) 
  \end{aligned}
  \end{equation}
 
 where 
 $$\tilde{\varphi} = \tfrac{2}{3} \big[ 1 - (1 - z)^{3/2}\big]; \ \phi =  \tfrac{2}{3} h^{-1} \big[ (z + 1)^{3/2} - 1 \big] ;  \tilde{\phi} =  \tfrac{2}{3}h^{-1} \big[i (z -1)^{3/2} +1 \big]$$
\begin{remark}
The oscillatory behaviors agree on overlapping regions. The index at $\mathbb{F}_C^{(4)}$ corresponds to the choice $\rho_{\mathbb{F}_C^{(4)}} = \sigma^{-1}$. See (\ref{definitionSig}) for definition of $\sigma$.
\end{remark}

\subsection{Step 2 - Justification of the asymptotic properties of the exact solution on $\mathbb{F}_C$}
With $\hat{U}_{\mathbb{F}_C}$ being the approximate solution on $\mathbb{F}_C$ that we constructed in step 1, $\hat{\mathsf{U}}_{\mathbf{F}_C}$ the exact solution on $\mathbb{F}_C$, from Lemma \ref{globalresolvent} we have
 $$ \hat{U}_{\mathbb{F}_C} - \hat{\mathsf{U}}_{\mathbb{F}_C} =  G E_h$$
 where $GE_h$ lies in the space of oscillatory functions which have the same asymptotic behavior as $\hat{U}_{\mathbb{F}_C}$ and $GE_h = \mathsf{O}(h^{\infty})$.  
 
\subsection{Step 3 - Approximate solution to (\ref{originalProAfterF}) on blown-up space C }
\begin{enumerate}
\item We extend the exact solution $\hat{\mathsf{U}}_{\mathbb{F}_C}$ into the interior of the space for $\lvert\theta\rvert\geq 1$, using the fibration structure over front face $\mathbb{F}_B$ (see figure \ref{blowupspaceCnB}). 

\item We need to modify the extended function so that satisfy the boundary condition in  (\ref{originalProAfterF}) in terms of $x$. As defined in (\ref{definitionOfz}),
 $$  z = \chi_0 (Z_0) ( Z-Z_0) + \chi_+(Z_0)  \big( \dfrac{Z-Z_0}{Z_0}\big) + \chi_-(Z_0)  \big( \dfrac{Z-Z_0}{-Z_0}\big)$$$$ \rho_{ \mathbb{F}_C^{(1)}} =  \chi_0 (Z_0) ( Z-Z_0) + \chi_+(Z_0)  \big( \dfrac{Z-Z_0}{Z_0}\big)Z_0^{3/2} + \chi_-(Z_0)  \big( \dfrac{Z-Z_0}{-Z_0}\big)(-Z_0)^{3/2}$$
 $$Z_0 = \lvert\theta_n\rvert^{2/3} (1 - \lvert\hat{\theta}'\rvert^2) ; Z - Z_0 = x\lvert\theta_n\rvert^{2/3} ; \hat{\theta}' = \dfrac{\theta'}{\theta_n}$$
We have $\rho_{ \mathbb{F}_C^{(1)}} (x = 0, \theta)  = 0$ hence $\tilde{u}^-_{\mathbb{F}_C}|_{x=0} = 1$. See (\ref{localApSo2}) for definition of $\tilde{u}^-_{\mathbb{F}_C}$. \nl
 Define
 \begin{equation}\label{phiCorr}
 \begin{aligned}
   \phi_{\text{correction}} &:= c_0 \chi_0 (Z_0) \lvert\theta_n\rvert^{\tfrac{2}{3}\sqrt{\tfrac{n^2}{4}-\lambda}} + c_+\chi_+(Z_0)  \Big[\lvert\theta_n\rvert (1 - \lvert\hat{\theta}'\rvert^2)^{1/2}\Big]^{\sqrt{\tfrac{n^2}{4}-\lambda}} \\
  & + c_-\chi_-(Z_0) \Big[\lvert\theta_n\rvert (\lvert\hat{\theta}'\rvert^2-1)^{1/2}\Big]^{\sqrt{\tfrac{n^2}{4}-\lambda}}   
   \end{aligned}
   \end{equation}
   $c_0,c_-$ and $c_+$ are smooth bounded function of $Z_0$ (\emph{used to make up for the `piling up'  of the values of $\chi+,\chi_-, \chi_0$}). 
   \begin{remark}
   For singularity computation purpose, we will be concerned with the growth in terms of $\lvert\theta\rvert$ of the various terms in the definition of $\phi_{\text{correction}}$. On the the overlapping region between $\chi_0$ and $\chi_+$ both terms are both growing by $\lvert\theta_n\rvert^{-\tfrac{2}{3}(\tfrac{n}{2}-\sqrt{\tfrac{n^2}{4}-\lambda})}$. This is because on overlapping support of $\chi_0$ and $\chi_+$ $Z_0 > 0$ is bounded ie 
   $$  1 - \lvert\hat{\theta}'\rvert^2  \lesssim \lvert\theta_n\rvert^{-2/3}
   \Rightarrow \big(\lvert\theta_n\rvert (1 - \lvert\hat{\theta}'\rvert^2)^{1/2}\big)^{\sqrt{\tfrac{n^2}{4}-\lambda}}   \lesssim \lvert\theta_n\rvert^{\tfrac{2}{3}\sqrt{\tfrac{n^2}{4}-\lambda}}$$
   Similarly for the overlapping region of $\chi_-$ and $\chi_0$. 
   \end{remark}
   Define
\begin{equation}\label{slnInsing}
\hat{u}_x = \tilde{\chi}(\lvert\theta\rvert)\  \phi_{\text{correction}}\, x^{\tfrac{n}{2}}  \hat{U}_{\mathbb{F}_C} 
\end{equation}
 $\tilde{\chi}(\lvert\theta\rvert)\in \mathcal{C}^{\infty}$ with $\supp\tilde{\chi} \subset \{\lvert\theta\rvert\geq 1\}$. Hence $\tilde{\chi}$ is to cut off the exact solution away from finite region ($\lvert\theta\rvert\leq1$). So in fact $\hat{u}_x$ is an exact local solution near infinity $\lvert\theta\rvert\geq 1$ (however an approximate solution on the whole blown up space B)\footnote{Since we only care about the singularities behavior at infinity $\theta$, an approximate solution in this sense (cut off away from finite $\lvert\theta\rvert$) suffices}. We have $\hat{u}_x$ satisfies:
 $$ \begin{cases}
\hat{L} \hat{u} \in \dot{\mathcal{C}}^{\infty}(\mathbb{R}^{n+1}_+) \\
x^{-s_-} \hat{u} \ |_{x=0} = 1   & s_{\pm}(\lambda) = \dfrac{n}{2} \pm \sqrt{\dfrac{n^2}{4} - \lambda} \\
\hat{u} \in \exp(i\phi_{\text{in}})  \mathcal{L}_{ph}(C) & \phi_{\text{in}} = -\tfrac{2}{3} \theta_n^{-1} \big[ (1 +x - \lvert\hat{\theta}'\rvert^2)^{3/2} - (1 - \lvert\hat{\theta}'\rvert^2)^{3/2}\big] \text{sgn} \theta_n
\end{cases} $$

\end{enumerate}

\begin{remark}
 To get an approximate solution to the original problem (\ref{originalProblem}) we will take the inverse fourier transform of $\hat{\mathsf{U}}$. In the singularity computation section we will work with $\mathcal{F}^{-1} (\hat{\mathsf{U}})$.
\end{remark}

%% file: SingularitiesComputation_4_-noheader.tex
\section{Singularities computation}\label{SingComp}

In section \ref{detailofCon}, we constructed $\hat{u}_x$ that lives on the blownup space C (see figure \ref{blowupspaceCnB}) and satisfies$$
\begin{cases}
\hat{L} \hat{u}_x \in \dot{\mathcal{C}}^{\infty}(\mathbb{R}^{n+1}_+) \\
x^{-s_-} \hat{u}_x \ |_{x=0} = 1   & s_{\pm}(\lambda) = \dfrac{n}{2} \pm \sqrt{\dfrac{n^2}{4} - \lambda} \\
\hat{u}_x \in \exp(i\phi_{\text{in}})  \mathcal{L}_{ph}(C) & \phi_{\text{in}} = -\tfrac{2}{3} \theta_n^{-1} \big[ (1 +x - \lvert\hat{\theta}'\rvert^2)^{3/2} - (1 - \lvert\hat{\theta}'\rvert^2)^{3/2}\big] \text{sgn} \theta_n
\end{cases} $$
The function $\hat{u}_x$ lives on blownup space $C$ but due the reduction process in section \ref{reductionSml} and the construction in section \ref{detailofCon}, it has a complete description as a polyhomogeneous conormal function on the front face $\mathbb{F}_C$ of blowup space C (see figure \ref{blowupspaceCnB} and \ref{ffFC}). 
In this section, we will study the regularity / singularity properties in terms of conormality and wavefront set of the approximate solution $\mathsf{U} = \mathcal{F}^{-1}(\hat{u}_x)$ on each of the region which after Fourier transform corresponds to the partition of $\mathbb{F}_C$ (see figure \ref{ffFC}). Using these results, in section \ref{mainresultcomp}, we will draw conclusion on the singularity structure of the approximate solution $\mathsf{U}$ and then the exact solution $U$ to (\ref{originalProblem}) in section \ref{mainresultcomp}, which is the original goal of the paper. \nl
 \textbf{The main observation from the singularity computation}: The portion of of the solution in the interior of $\mathbb{F}_C$ under the fourier transform is `trivial' i.e. does not yield singularity. The structure of the solution away from $x=0$ is due to the asymptotic behavior at scattering end for different value of $h$ (near boundary face $\mathbb{F}_C^{(4)}$). In our case, due to the fact that we have a reduction to $\mathbb{F}_C$, singularities at the glancing point propagate into the interior (i.e away from $x=0$) and contributes to singularities at the hyperbolic points.
 \begin{center}
 \includegraphics[scale=0.5]{frontfaceFC-2dcase}
 \end{center}
For convenience we list the polyhomogeneity behavior of the constructed function $\hat{u}_x$ defined in (\ref{slnInsing}) at each boundary face of $\mathbb{F}_C$ (see figure \ref{ffFC}):
$$\hat{u}_x = \tilde{\chi}(\lvert\theta\rvert)\,  \phi_{\text{correction}} \, x^{\tfrac{n}{2}} \hat{\mathsf{U}}_{\mathbb{F}_C} $$
where
\begin{align*}
   \phi_{\text{correction}} &= c_0 \chi_0 (Z_0) \lvert\theta_n\rvert^{\tfrac{2}{3}\sqrt{\tfrac{n^2}{4}-\lambda}} + c_+\chi_+(Z_0) \big(\lvert\theta_n\rvert (1 - \lvert\hat{\theta}'\rvert^2)^{1/2}\big)^{\sqrt{\tfrac{n^2}{4}-\lambda}} \\
  & + c_-\chi_-(Z_0)  \big(\lvert\theta_n\rvert (\lvert\hat{\theta}'\rvert^2-1)^{1/2}\big)^{\sqrt{\tfrac{n^2}{4}-\lambda}}   
\end{align*}
We relist the asymptotic properties of $\hat{\mathsf{U}}_{\mathbb{F}_C}$ at each boundary faces of $\mathbb{F}_C$ (see also (\ref{asymptotics})) :
 $$\hat{U}^{\mathbb{F}_C} = \hat{U}^{\mathbb{F}_{C,+}}  + \hat{U}^{\mathbb{F}_{C,0}} +  \hat{U}^{\mathbb{F}_{C,-}} $$ 
 $$ \hat{U}^{\mathbb{F}_{C,+} } \in \exp(-i \phi )  \rho_{\mathbb{F}_C^{(2)}}^{\tilde{\gamma}_{\mathfrak{B}_1} } \rho_{\mathbb{F}_C^{(3)}}^{\tilde{\gamma}_{\mathfrak{B}_1}+\tfrac{1}{2}}  \mathcal{C}^{\infty} (\mathbb{F}_{C,+}^{0,\infty}) 
 \ +\  \rho_{\mathbb{F}_C^{(2)}}^0 \rho_{\mathbb{F}_C^{(1)}}^{\sqrt{\tfrac{n^2}{4}-\lambda}}  \mathcal{C}^{\infty} (\mathbb{F}_{C,+}^{0,0})  + \rho_{\mathbb{F}_C^{(2)}}^0 \rho_{\mathbb{F}_C^{(1)}}^{-\sqrt{\tfrac{n^2}{4}-\lambda}}  \mathcal{C}^{\infty} (\mathbb{F}_{C,+}^{0,0}) $$
 $$+ \exp(-i \phi )  \rho_{\mathbb{F}_C^{(3)}}^{ \tilde{\gamma}_{\mathfrak{B}_1} +\tfrac{1}{2}}   \rho_{\mathbb{F}_C^{(4)}}^{\tfrac{1}{2}}  \mathcal{C}^{\infty} (\mathbb{F}_{C,+}^{\infty})$$
 $$\hat{U}_{\mathbb{F}_{C, - } }\in  \exp(- \tfrac{ \tilde{\varphi}}{h}) \rho_{\mathbb{F}_C^{(6)}}^{\tilde{\gamma}_{\mathfrak{B}_1}}  \mathcal{C}^{\infty} (\mathbb{F}_{C,-}^{0,\infty})\   +\  \rho_{\mathbb{F}_C^{(6)}}^0 \rho_{\mathbb{F}_C^{(1)}}^{\sqrt{\tfrac{n^2}{4}-\lambda}}  \mathcal{C}^{\infty} (\mathbb{F}_{C,-}^{0,0})  + \rho_{\mathbb{F}_C^{(6)}}^0 \rho_{\mathbb{F}_C^{(1)}}^{-\sqrt{\tfrac{n^2}{4}-\lambda}}  \mathcal{C}^{\infty} (\mathbb{F}_{C,-}^{0,0})\ $$
 $$ +   \exp(- \tilde{\phi} )  \rho_{\mathbb{F}_C^{(5)}}^{\infty}   \rho_{\mathbb{F}_C^{(4)}}^{\tfrac{1}{2}}  \mathcal{C}^{\infty} (\mathbb{F}_{C,-}^{\infty})$$
  $$\hat{U}^{\mathbb{F}_{C, 0 } } \in \rho_{\mathbb{F}_C^{(1)}}^{\sqrt{\tfrac{n^2}{4}-\lambda}}  \mathcal{C}^{\infty} (\mathbb{F}_{C,0}^0)  +   \rho_{\mathbb{F}_C^{(1)}}^{-\sqrt{\tfrac{n^2}{4}-\lambda}}  \mathcal{C}^{\infty} (\mathbb{F}_{C,0}^0)+\ e^{-i\tfrac{2}{3}(S + Z_0)^{3/2}\sgn\theta_n} \rho_{\mathbb{F}_C^{(4)}}^{\tfrac{1}{2}} \mathcal{C}^{\infty}(\mathbb{F}_{C,0}^{\infty})  $$
 where 
 $$\tilde{\varphi} = \tfrac{2}{3} \big[ 1 - (1 - z)^{3/2}\big]; \ \phi =  \tfrac{2}{3} h^{-1} \big[ (z + 1)^{3/2} - 1 \big] ;  \tilde{\phi} =  \tfrac{2}{3}h^{-1} \big[i (z -1)^{3/2} +1 \big]$$
$$\rho_{\mathbb{F}_C^{(4)}} = \sigma^{-1}  ; z = \sigma^{2/3}$$

We will be consider the properties of the approximate solution $\hat{\mathsf{U}}$ on each of the following region: for $\delta_2 > \tilde{\delta}_2$ (see (\ref{definitionOfz}) for where $\tilde{\delta}_2$ appears):
   \begin{enumerate}
   \item $\lvert Z_0 \rvert < 2\delta_2, \lvert\theta\rvert \geq 1$ : denoted by region $\mathbb{F}_{C,0}$.
   \item $Z_0 > \delta_2 , \lvert\theta\rvert \geq 1$ : denoted by region $\mathbb{F}_{C,+}$ which will be subpartitioned into 
   \begin{itemize}
   \item $\mathbb{F}_{C,+}^0 = \mathbb{F}_{C,+}^{0,0} \cup \mathbb{F}_{C,+}^{0,\infty}$ localizing the two corners of the blown-up front face.
   \item  $\mathbb{F}_{C,+}^{\infty}$. 
   \end{itemize}
   
   \item $Z_0 < -\delta_2, \lvert\theta\rvert \geq 1$ : here we will need to consider
   $ \begin{cases} Z_0 < -\delta_2   \\ 
   \lvert \theta_n\rvert \geq (1 - 2\delta_1) \lvert\theta'\rvert\end{cases}
   $ denoted by $\mathbb{F}_{C,-,1}$ which is subpartitioned as above into 
   $$\mathbb{F}_{C,-,1}^{\infty}, \mathbb{F}_{C,-,1}^{0,\infty},\  \text{and}\   \mathbb{F}_{C,-,1}^{0,0}$$
   and
   $ \begin{cases} Z_0 < -\delta_2   \\ 
   \lvert \theta_n\rvert \leq (1 - \delta_1) \lvert\theta'\rvert\end{cases}$ denoted by $\mathbb{F}_{C,-,2}$ which is subpartitioned into 
   $$ \mathbb{F}_{C,-,2}^{0,\infty},\  \text{and}\   \mathbb{F}_{C,-,2}^{0,0}.$$
   
       \end{enumerate}
  To localize to each region we will use smooth compact function $\psi_{\mathbb{F}_{C,0}^0}$, $\psi_{\mathbb{F}_{C,0}^{\infty}}$, $\psi_{\mathbb{F}_{C,+}^{0,0}}$, $\psi_{\mathbb{F}_{C,+}^{0,\infty}}$ , $ \psi_{\mathbb{F}_{C,+}^{\infty}} $, $\psi_{\mathbb{F}_{C,-,1}^{0,0}}$, $\psi_{\mathbb{F}_{C,-,1}^{0,\infty}}$ , $\psi_{\mathbb{F}_{C,-,1}^{\infty}}$, $\psi_{\mathbb{F}_{C,-,2}^{0,0}}$ and $\psi_{\mathbb{F}_{C,-,2}^{0,\infty}}$.

\textbf{Notation} : \newline
Denote by $s_- = \tfrac{n}{2} - \sqrt{\tfrac{n^2}{4} - \lambda}$, $\tilde{a} = 1 - n$, $a = \tfrac{2}{3}\tilde{a} + \tfrac{1}{3} = -\tfrac{2}{3}n + 1$.\newline
We will use the projective blowup variable $s = t^{-1}, z$ near the corner of $\mathbb{F}_C^{(2)}$ and $\mathbb{F}_C^{(3)}$ and $t = \tfrac{z}{h}, h$ near the corner of $\mathbb{F}_C^{(1)}$ and $\mathbb{F}_C^{(2)}$ . Similarly for the other end where $Z_0 < 0$.
\subsection{On region $\mathbb{F}_{C,0}$}
In this region we have $\{ \lvert Z_0 \rvert \leq 2\delta_2 ;\  \lvert \theta\rvert\geq 1\} $ which lies in the conic set 
$ \{\lvert\theta\rvert \geq 1,   1 - \kappa_0(\delta_0) \leq \lvert\hat{\theta}'\rvert \leq 1 + \kappa_1(\delta_2)    \}$. 
 Hence $\lvert\theta_n\rvert$ is equivalent to $\lvert\theta'\rvert$ and thus to $\lvert\theta\rvert$. Thus also implies that $\lvert\theta_n\rvert, \lvert\theta'\rvert > 0$
 
  For the subregion $\mathbb{F}_{C,0}^0$ ie close to $\sigma = 0$, we will show
  \begin{enumerate}
  \item $\mathsf{U}^{\mathbb{F}_{C,0}^0}\in x^{s_-} \mathcal{C}^{\infty}_{x,y} + x^{s_+} \mathcal{C}^{\infty}_{x,y} $ on $y_n \neq 0$.
  \item $\mathsf{U}^{\mathbb{F}_{C,0}^0}$ is smooth on $x > \delta$ for any $\delta >0$
  \end{enumerate}

 For the subregion $\mathbb{F}_{C,0}^{\infty}$ ie close to $\sigma^{-1} = \infty$, we will show
 \begin{enumerate}
 \item $\mathsf{U}^{\mathbb{F}_{C,0}^{\infty}}$ is smooth for $y_n < 0$ and vanishes to infinite order at $x = 0$
 \item smoothness and and vanishes to infinite order at $x = 0$ for $y_n > 0$ if we are close enough to the boundary.
 \item Wavefront set statement away from $x = 0$ for $\mathsf{U}^{\mathbb{F}_{C,0}^{\infty}}$  : follows bicharacteristic strips

 \end{enumerate}

\subsubsection{On region $ \mathbb{F}_{C,0}^0$ }\
\begin{remark}
As defined in (\ref{definitionOfz}), and since $\delta_2 > \tilde{\delta}_2$, $z = \text{smooth function}(Z_0) S$ where this smooth function is $\equiv 1$ on $Z_0 < \tilde{\delta}_2$, $S = Z- Z_0 = x\lvert\theta_n\rvert^{2/3}$. On the current support $Z_0 < 2\delta_2$ so the properties of the solution in $S$ are the exactly the same in terms of $z$. 
\end{remark}
For $\chi(x,y) \in C^{\infty}_0(x,y)$
  \begin{align}\label{indicialF0}
  \langle \mathsf{U}^{\mathbb{F}_{C,0}^0}) , \chi(x,y) \rangle
   &= \lim_{\epsilon \rightarrow 0} (2\pi)^{-n} \int\int\int  \psi_{\mathbb{F}_{C,0}^0} \hat{u}_x \chi(x,y) \varpi(\epsilon\theta)  e^{iy\cdot \theta} dx dy d\theta
   \end{align}
   where we recall the form of $\hat{u}_x$ on the region currently considered.   $$\psi_{\mathbb{F}_{C,0}^0} \hat{u}_x =\lvert\theta_n\rvert^{\tfrac{2}{3}\sqrt{\tfrac{n^2}{4}-\lambda}}\hat{\mathsf{U}}^{\mathbb{F}_{C,0}^0} $$
   where  $ \hat{\mathsf{U}}^{\mathbb{F}_{C,0}^0}$ is conormal in $S$  . In more details:
\begin{align*}
 2\sqrt{\tfrac{n^2}{4}-\lambda}\notin \mathbb{Z} :&\  \hat{\mathsf{U}}^{\mathbb{F}_{C,+}^{0,0}} \in x^{\tfrac{n}{2}} S^{\sqrt{\tfrac{n^2}{4}-\lambda}} \mathcal{C}^{\infty}_{S,Z_0} + x^{\tfrac{n}{2}} S^{-\sqrt{\tfrac{n^2}{4}-\lambda}} \mathcal{C}^{\infty}_{S,Z_0}
  \end{align*}
  $$\lvert\theta_n\rvert^{\tfrac{2}{3}\sqrt{\tfrac{n^2}{4}-\lambda}} \hat{\mathsf{U}}^{\mathbb{F}_{C,+}^{0,0}} \in  \lvert\theta_n\rvert^{\tfrac{4}{3}\sqrt{\tfrac{n^2}{4}-\lambda}} x^{s_+} \mathcal{C}^{\infty}_{S,Z_0} + x^{s_-} \mathcal{C}^{\infty}_{S,Z_0}$$
   Using this form of $\hat{u}_x$ we rewrite (\ref{indicialF0}) as
  $$
  \langle \mathsf{U}^{\mathbb{F}_{C,0}^0} , \chi(x,y) \rangle
   = \lim_{\epsilon \rightarrow 0}(2\pi)^{-n}  \sum_{\pm}  x^{s_{\pm}}\int\int\int a_{\mathbb{F}_{C,0}^0;\pm} \chi(x,y) \varpi(\epsilon\theta)  e^{iy\cdot \theta} \,dx\, dy \,d\theta$$ 
 \begin{lemma}[Properties of $a_{\mathbb{F}_{C,0}^0;\pm} $]\label{FC00}\
 \begin{enumerate}
 \item For $L$ such that  $^t L = \dfrac{1}{\lvert y\rvert^2} \sum y_j \dfrac{\partial}{i\partial \theta_j}$, we have for $C_k$ is independent of $x$:
 $$\lvert D_x^l  L^k a_{\mathbb{F}_{C,0}^0;\pm}\rvert \leq C_k \lvert\theta\rvert^{-\tfrac{1}{3}k + {\tfrac{4}{3}\sqrt{\tfrac{n^2}{4}-\lambda}}+\tfrac{2}{3}l}.$$
 \item On $x \neq 0$, 
 $$a_{\mathbb{F}_{C,0}^0} = \sum_{\pm} x^{s_{\pm}} a_{\mathbb{F}_{C,0}^0;\pm} \in S^{-\infty}_{\tfrac{1}{3},0;\theta}.$$
\end{enumerate}
 \end{lemma}
 \begin{proof}\indent\par\noindent
 \begin{enumerate}
 \item First we consider the bound on $a_{\mathbb{F}_{C,0}^0;\pm}$. Note on the current region $\lvert\theta_n\rvert$ is equivalent to $\lvert\theta\rvert$.
   $$\psi_{\mathbb{F}_{C,0}^0} \hat{u}_x =\lvert\theta_n\rvert^{\tfrac{2}{3}\sqrt{\tfrac{n^2}{4}-\lambda}}\hat{\mathsf{U}}^{\mathbb{F}_{C,0}^0} $$
  $$\lvert\theta_n\rvert^{\tfrac{2}{3}\sqrt{\tfrac{n^2}{4}-\lambda}} \hat{\mathsf{U}}^{\mathbb{F}_{C,+}^{0,0}} \in  \lvert\theta_n\rvert^{\tfrac{4}{3}\sqrt{\tfrac{n^2}{4}-\lambda}} x^{s_+} \mathcal{C}^{\infty}_{S,Z_0} + x^{s_-} \mathcal{C}^{\infty}_{S,Z_0}$$
  $$\Rightarrow \lvert a_{\mathbb{F}_{C,0}^0;\pm} \rvert \leq \lvert \theta_n\rvert^{\tfrac{4}{3}\sqrt{\tfrac{n^2}{4}-\lambda}}$$
  To compute the bound for higher derivatives, we use :
\begin{equation}\label{derivativeFC00}
\begin{aligned}
\partial_{\theta_n}  = \lvert\theta_n\rvert^{-1} S \partial_S
  +   ( Z_0 \lvert\theta_n\rvert^{-1} - 2\lvert\theta_n\rvert^{-1/3} \lvert\hat{\theta}'\rvert^2   )\partial_{Z_0} \\
j  > n ;\ \partial_{\theta_j} = -2 \lvert\theta_n\rvert^{-1/3} \lvert\hat{\theta}_j\rvert\partial_{Z_0}\\
\partial_x  = \lvert\theta_n\rvert^{2/3}\partial_S
\end{aligned}
\end{equation}
Using the fact that $a_{\mathbb{F}_{C,0}^0;\pm}$ is smooth in $Z_0$ and $S$, on the current region $Z_0,S$ and $\lvert\hat{\theta}_j\rvert$ are bounded on the region currently considered and the computations in (\ref{derivativeFC00}), we have:
\begin{equation}\label{boundFC0}
\lvert \partial_x^l \partial_{\theta}^{\alpha} a_{\mathbb{F}_{C,0}^0;\pm}\rvert \lesssim  \lvert\theta\rvert^{-\tfrac{1}{3}k +\tfrac{4}{3}\sqrt{\tfrac{n^2}{4}-\lambda}+\tfrac{2}{3} l} 
\end{equation}

\item  On $x\neq 0$ we have smoothness in $x$ for $a_{\mathbb{F}_{C,0}^0} = \sum_{\pm} x^{s_{\pm}} a_{\mathbb{F}_{C,0}^0;\pm} $. Together with the bound in (\ref{boundFC0}), in fact since $\partial_x= x^{-1} S\partial_S$ we have:
$$a_{\mathbb{F}_{C,0}^0} \in S^0_{\tfrac{1}{3},0; \theta} $$
The order of the symbol can be lowered : $S = Z - Z_0 = x\lvert\theta_n\rvert^{2/3}$. On $x\neq 0$, for all $N\in \mathbb{N}$ on
 $$ \lvert\theta_n\rvert^N G(x\lvert\theta_n\rvert^{2/3}, Z_0) 
 = S^{\tfrac{3}{2}N} x^{-\tfrac{3}{2}N} G(x\lvert\theta_n\rvert^{2/3}, Z_0) \in L^{\infty}.$$
Since $\lvert\theta_n\rvert \sim \lvert\theta_n\rvert$, $ \lvert\theta\rvert^N G(x\lvert\theta_n\rvert^{2/3}, Z_0) \in L^{\infty}$.

\end{enumerate}

 \end{proof}

\begin{lemma}\label{conormalFC00}
 $\mathsf{U}^{\mathbb{F}_{C,0}^0}\in x^{s_-} \mathcal{C}^{\infty}_{x,y} + x^{s_+} \mathcal{C}^{\infty}_{x,y} $ on $y_n \neq 0$
\end{lemma}

\begin{proof}
  Let $\chi(x,y) \in \mathcal{C}^{\infty}_0$ with either $\supp \chi \subset \{y_n > 0\}$ or $\supp \chi \subset \{y_n < 0\}$ 
 \begin{equation}
 \begin{aligned}
    ( \mathsf{U}^{\mathbb{F}_{C,0}^0} , \chi)
    &= \lim_{\epsilon \rightarrow 0} (2\pi)^{-n} \int\int\int \psi_{\mathbb{F}_{C,0}^0} \hat{u}_x \chi(x,y) \varpi(\epsilon\theta)  e^{iy\cdot \theta} dx dy d\theta\\
    &=  \lim_{\epsilon \rightarrow 0} (2\pi)^{-n} \int\int\int \sum_{\pm} x^{s_-} a_{\mathbb{F}_{C,0}^0;\pm}   \varpi\, \chi e^{iy\cdot\theta} dx dy d\theta
 \end{aligned}
    \end{equation}
         $\lvert d_{\theta} (y\cdot \theta)\rvert^2$ is bounded away from zero on $\supp \chi \times \RR^n$ since the $n$-th component of the derivative of the phase function,  which is $y_n$, is never zero on this region currently considered. Consider a first order differential operator $L$ whose formal adjoint is :
    \begin{equation}
    ^t L = \dfrac{1}{\lvert y\rvert^2} \sum y_j \dfrac{\partial}{i\partial \theta_j}
    \end{equation}
 so $^tL e^{iy\cdot \theta} = e^{iy\cdot\theta}$. $k$-fold integration by parts gives
   $$( \mathsf{U}^{\mathbb{F}_{C,0}^0} , \chi)
   =  \lim_{\epsilon \rightarrow 0} \sum_{\pm} (2\pi)^{-n} x^{s_{\pm}} \int\int\int  L^k( a_{\mathbb{F}_{C,0}^0;\pm} \varpi) \chi e^{iy\cdot\theta} dx dy d\theta$$
  
 \begin{enumerate}
\item For $k > 3(n+ \tfrac{4}{3}\sqrt{\tfrac{n^2}{4}-\lambda})$ , $\theta \mapsto  L^k( a_{\mathbb{F}_{C,0}^0;\pm} \varpi)\in L^1(\RR^n)$. We can let $\epsilon\rightarrow 0$ under the integral sign :
 $$\mathsf{U}^{\mathbb{F}_{C,0}^0} = \sum_{\pm}(2\pi)^{-n} x^{s_{\pm}} \int L^k(a_{\mathbb{F}_{C,0}^0;\pm}) e^{i y \cdot\theta} \, d\theta = \sum_{\pm} x^{s_{\pm}} \mathsf{U}^{\mathbb{F}_{C,0}^0;\pm} $$
 $$\Rightarrow \lvert \mathsf{U}^{\mathbb{F}_{C,0}^0;\pm}\rvert \leq (2\pi)^{-n} \int \lvert L^k(a_{\mathbb{F}_{C,0}^0;\pm})\rvert  \, d\theta
 \leq  C_k(2\pi)^{-n} \int_{\lvert\theta\rvert \geq 1} \lvert \theta\rvert^{-\tfrac{1}{3}k+\tfrac{4}{3}\sqrt{\tfrac{n^2}{4}-\lambda}} \, d\theta $$
 The last estimate follows from the fact by Lemma \ref{FC00}, 
 $$\lvert L^k a_{\mathbb{F}_{C,0}^0;\pm}\rvert  \leq C_k \lvert\theta\rvert^{-\tfrac{1}{3}k+ \tfrac{4}{3}\sqrt{\tfrac{n^2}{4}-\lambda}};\ \ C_k\ \text{independent of}\  x$$ 
 Thus we have
 $$\mathsf{U}^{\mathbb{F}_{C,0}^0;\pm} \in L^{\infty}(x,y)$$

\item Fix $k$ with $k > 3( n  + \lvert\beta\rvert +\tfrac{4}{3}\sqrt{\tfrac{n^2}{4}-\lambda})$ so
  $$\mathsf{U}^{\mathbb{F}_{C,0}^0} = \sum_{\pm}(2\pi)^{-n} x^{s_{\pm}} \int L^k(a_{\mathbb{F}_{C,0}^0;\pm}) e^{i y \cdot\theta} \, d\theta = \sum_{\pm} x^{s_{\pm}} \mathsf{U}^{\mathbb{F}_{C,0}^0;\pm} $$
  Consider $D_y^{\beta} \mathsf{U}^{\mathbb{F}_{C,0}^0,\pm}$:
$$ D_y^{\beta} \mathsf{U}^{\mathbb{F}_{C,0}^0,\pm}=  (2\pi)^{-n}  \int L^k(a_{\mathbb{F}_{C,0}^0;\pm}) \theta^{\beta} e^{i y \cdot\theta} \, d\theta $$
$$\Rightarrow \lvert D_y^{\beta}  \mathsf{U}^{\mathbb{F}_{C,0}^0;\pm}\rvert\leq 
(2\pi)^{-n} \int \lvert L^k(a_{\mathbb{F}_{C,0}^0;\pm}) \theta^{\beta} \rvert \, d\theta
\stackrel{\text{Lemma \ref{FC00}}}{\leq} (2\pi)^{-n} \int_{\lvert\theta\rvert\geq 1} \lvert\theta\rvert^{-\tfrac{1}{3}k+ \tfrac{4}{3}\sqrt{\tfrac{n^2}{4}-\lambda}+\lvert\beta\rvert}\, d\theta $$
With the current choice of $k$ we have $\theta \mapsto L^k(a_{\mathbb{F}_{C,0}^0;\pm}) \theta^{\beta} \in L^1$. Thus
$$D_y^{\beta} \mathsf{U}^{\mathbb{F}_{C,0}^0;\pm} \in L^{\infty}(x,y)$$

\item Fix $k$ with $k > 3(n +\tfrac{4}{3}\sqrt{\tfrac{n^2}{4}-\lambda} + \tfrac{2}{3}l)$ so that:
  $$\mathsf{U}^{\mathbb{F}_{C,0}^0} = \sum_{\pm}(2\pi)^{-n} x^{s_{\pm}} \int L^k(a_{\mathbb{F}_{C,0}^0;\pm}) e^{i y \cdot\theta} \, d\theta = \sum_{\pm} x^{s_{\pm}} \mathsf{U}^{\mathbb{F}_{C,0}^0;\pm} $$
Consider $D_x^l \mathsf{U}^{\mathbb{F}_{C,0}^0;\pm}$:
$$D_x^l \mathsf{U}^{\mathbb{F}_{C,0}^0;\pm} = (2\pi)^{-n} \int D_x^l L^k(a_{\mathbb{F}_{C,0}^0;\pm}) e^{i y \cdot\theta} \, d\theta$$
From Lemma \ref{FC00}, $ D_x^l L^ka_{\mathbb{F}_{C,0}^0} \lesssim \lvert\theta\rvert^{-\tfrac{1}{3}k+  \tfrac{4}{3}\sqrt{\tfrac{n^2}{4}-\lambda} + \tfrac{2}{3}l}$. With the current choice of $k$ we have:
  $$D_x^l  \mathsf{U}^{\mathbb{F}_{C,0}^0;\pm} \in L^{\infty}$$


\item Same argument is used to deal with mixed derivatives $D_y^{\alpha} D_x^l$. In this case one needs to choose $k$ satisfying $k > 3(n +\tfrac{4}{3}\sqrt{\tfrac{n^2}{4}-\lambda} + \tfrac{2}{3}l+ \lvert\alpha\vert)$ so that
$$D_y^{\alpha} D_x^l\mathsf{U}^{\mathbb{F}_{C,0}^0;\pm} \in L^{\infty}_{x,y}$$

\item We have shown that $ \mathsf{U}^{\mathbb{F}_{C,0}^0}$ can be rewritten as $\sum_{\pm} x^{s_{\pm}}  \mathsf{U}^{\mathbb{F}_{C,0}^0}$ where $D_y^{\alpha} D_x^l\mathsf{U}^{\mathbb{F}_{C,0}^0;\pm} \in L^{\infty}_{x,y}$ for any $\alpha$ and $l$. Thus:
$$\mathsf{U}^{\mathbb{F}_{C,0}^0} \in x^{s_-}\mathcal{C}^{\infty}_{x,y} + x^{s_+}\mathcal{C}^{\infty}_{x,y}$$
\end{enumerate}

\end{proof}

 \begin{lemma}\label{C00x>delta}
    $\mathsf{U}^{\mathbb{F}_{C,0}^0}$ is smooth on $x >\delta$ for any $\delta > 0$.
\end{lemma}

 \begin{proof}
 Denote $\mathsf{U}^{\mathbb{F}_{C,0}^0}_{\delta}$ the restrictions of $\mathsf{U}^{\mathbb{F}_{C,0}^0}$ to $\mathcal{D}'(X_{\delta})$. For $x >\delta$ we have :
 \begin{equation}
   \begin{aligned}
    ( \mathsf{U}^{\mathbb{F}_{C,0}^0}_{\delta} , \chi)
    &= \lim_{\epsilon \rightarrow 0} (2\pi)^{-n} \int\int\int  \psi_{\mathbb{F}_{C,0}^0} \hat{u}_x \chi(x,y) \varpi(\epsilon\theta)  e^{iy\cdot \theta} dx dy d\theta\\
     &= \lim_{\epsilon \rightarrow 0} (2\pi)^{-n} \int\int\int  a_{\mathbb{F}_{C,0}^0} \chi(x,y) \varpi(\epsilon\theta)  e^{iy\cdot \theta} dx dy d\theta
    \end{aligned}
        \end{equation}
  where $a_{\mathbb{F}_{C,0}^0}\in S^{-\infty}_{\tfrac{1}{3},0}$ for $x >\delta$ by Lemma \ref{FC00}. In fact since on $X_{\delta}$ we have rapidly decreasing in $\lvert\theta\rvert$
   $$\mathsf{U}^{\mathbb{F}_{C,0}^0}_{\delta}  = (2\pi)^{-n} \int\int\int  a_{\mathbb{F}_{C,0}^0} e^{iy\cdot\theta} d\theta$$
Note $\mathsf{U}^{\mathbb{F}_{C,0}^0}_{\delta}$ is smooth in $y$. We can similarly consider $\partial_x $ since $\partial_x = x^{-1} S\partial_S$ and  $S\partial_S a_{\mathbb{F}_{C,0}^0} $ is still rapidly decreasing in $\theta$ .
 \end{proof}

\subsubsection{On region $ \mathbb{F}_{C,0}^{\infty}$ }
 \begin{align*}
 \chi \in \mathcal{C}^{\infty}_c ;  \ ( \mathsf{U}^{\mathbb{F}_{C,0}^{\infty}} , \chi)
    &= \lim_{\epsilon \rightarrow 0} (2\pi)^{-n} \int\int\int \psi_{\mathbb{F}_{C,0}^{\infty}} \hat{u}_x \chi(x,y) \varpi(\epsilon\theta)  e^{iy\cdot \theta+i\tilde{\phi}} \, dx dy d\theta\\
     &=\lim_{\epsilon \rightarrow 0} (2\pi)^{-n} \int\int\int a_{\mathbb{F}_{C,0}^{\infty}} \varpi(\epsilon\theta)  \chi(x,y)e^{iy\cdot \theta+i\tilde{\phi}} \, dx dy d\theta\\ 
    \end{align*}
    
 \begin{lemma}[Properties of $a_{\mathbb{F}_{C,0}^{\infty}}$]\label{FC0in}\
 \begin{enumerate}
 \item For $\mathsf{L}$ such that $  ^t \mathsf{L} = \lvert d_{\theta}(y\cdot \theta +\tilde{\phi}) \rvert^{-2} \sum_{j=1}^n \dfrac{\partial ( y \cdot \theta+ \tilde{\phi}) }{\partial \theta_j} \partial_{\theta_j}$, we have
$$\lvert (xD_x)\mathsf{L}^k a_{\mathbb{F}_{C,0}^{\infty}}\rvert  \leq C_k \lvert\theta\rvert^{-\tfrac{1}{3}k-\tfrac{4}{3}\sqrt{\tfrac{n^2}{4}-\lambda}}$$
 where $C_k$ is independent of $x$

 \item  on $x\neq 0$ $a_{\mathbb{F}_{C,0}^{\infty}}\in S^{\tfrac{4}{3}\sqrt{\tfrac{n^2}{4}-\lambda}}_{\tfrac{1}{3}, \tfrac{2}{3}}$. 
\end{enumerate}
 
 \end{lemma}
 \begin{proof}\indent\par\noindent
 \begin{enumerate}
 \item On this region $S > 1, \lvert Z_0 \rvert \leq 2\delta_2$. Choose $\delta_2$ so that $Z= Z_0 + x\lvert\theta_n\rvert^{2/3} > \tfrac{1}{2}$. We have :
 $$\psi_{\mathbb{F}_{C,0}^{\infty}} \hat{u}_x = \exp(-i\tfrac{2}{3} Z^{3/2} \sgn\theta_n) \lvert\theta_n\rvert^{\tfrac{4}{3}\sqrt{\tfrac{n^2}{4}-\lambda}} Z^{-\tfrac{3}{4}}\hat{\mathsf{V}}^{\mathbb{F}_{C,0}^{\infty}}$$
where  $\hat{\mathsf{V}}^{\mathbb{F}_{C,0}^{\infty}}$ is smooth in $Z^{-1/2}$ and $Z_0$. Since $\tilde{a} = -(n-1) < 0$,  
$$\lvert Z^{-\tfrac{1}{4}-\tfrac{\tilde{a}}{2}}\hat{\mathsf{V}}^{\mathbb{F}_{C,0}^{\infty}} \rvert \leq C \lvert\theta_n\rvert^{\tfrac{4}{3}\sqrt{\tfrac{n^2}{4}-\lambda}}.$$

We proceed to consider the bounds on higher derivatives:
Denote $T = Z^{-1/2}$
$$T = \left[\lvert\theta_n\rvert^{2/3} (1 + x - \lvert\hat{\theta}'\rvert^2)\right]^{-1/2}
= \lvert\theta_n\rvert^{-1/3}(1 + x - \lvert\hat{\theta}'\rvert^2)^{-1/2}$$
Using the following computation
\begin{equation}\label{derivativeFinf+}
\begin{aligned}
\partial_{\theta_n} &= 
\left( -\tfrac{1}{3}  \lvert\theta_n\rvert^{-1}   +  T^2  \lvert\hat{\theta}'\rvert^2 \lvert\theta_n\rvert^{-1/3}  \right) T\partial_T + ( Z_0 \lvert\theta_n\rvert^{-1} - 2\lvert\theta_n\rvert^{-1/3} \lvert\hat{\theta}'\rvert^2   )\partial_{Z_0} \\
j  > n ;\ \partial_{\theta_j} & =  T^2  \lvert\hat{\theta}_j\rvert \lvert\theta_n\rvert^{-1/3}   T\partial_T - 2 \lvert\theta_n\rvert^{-1/3} \lvert\hat{\theta}_j\rvert\partial_{Z_0}\\
x\partial_x &= -\tfrac{1}{2} ( T^2 Z_0 + 1) T\partial_T
\end{aligned}
\end{equation}
 and the fact that $T$ and $Z_0$ are bounded on the region currently considered  we have:  
 $$\lvert (xD_x)^l \mathsf{L}^k a_{\mathbb{F}_{C,0}^{\infty}}\rvert  \leq C_k \lvert\theta\rvert^{-\tfrac{1}{3}k+ \tfrac{4}{3}\sqrt{\tfrac{n^2}{4}-\lambda}};\ \ C_k\ \text{independent of}\  x$$
 
 \item On $x\neq 0$ we have smoothness in $x$ of $a_{\mathbb{F}_{C,0}^{\infty}}$ and as a result of the computation in (\ref{derivativeFinf+}) we have:
 $$a_{\mathbb{F}_{C,0}^{\infty}}\in S^{\tfrac{4}{3}\sqrt{\tfrac{n^2}{4}-\lambda}}_{\tfrac{1}{3}, \tfrac{2}{3}}$$
\end{enumerate}
 \end{proof}

\begin{lemma}\label{y<0C0inf}
    For  $y_n < 0$, $\mathsf{U}^{\mathbb{F}_{C,0}^{\infty}}$ is smooth and vanishes to infinite order at $x = 0$.
\end{lemma}
\begin{proof}
Consider $\chi(x,y)$ with $\supp \chi \subset \{ y_n < 0\}$
 \begin{equation}
    ( \mathsf{U}^{\mathbb{F}_{C,0}^{\infty}} , \chi)
    = \lim_{\epsilon \rightarrow 0} (2\pi)^{-n} \int\int\int \psi_{\mathbb{F}_{C,0}^{\infty}} \hat{u}_x \chi(x,y) \varpi(\epsilon\theta)  e^{iy\cdot \theta+i\tilde{\phi}} \, dx dy d\theta
    \end{equation}

We have $\lvert d_{\theta} (y\cdot \theta)\rvert^2$ is bounded away from zero on  $\supp\chi \times \RR^n$ since the $n$-th component of the derivative of the phase function is $y_n - \text{a positive term}$ as follow:       \begin{equation}
  \begin{aligned}
  &y_n - \tfrac{2}{3}\partial_{\theta_n} Z^{3/2} \sgn\theta_n\\
  &= y_n -\tfrac{2}{3}(1 + x-\lvert\hat{\theta}'\rvert^2)^{3/2} 
   -   (1 + x-\lvert\hat{\theta}'\rvert^2)^{1/2} 2\lvert\hat{\theta}'\rvert^2    \\
   &= y_n -  \text{positive term}
  \end{aligned}
  \end{equation}
 If $y_n < 0$, the above expression is never zero. On that region we can consider a first order differential operator $\mathsf{L}$ whose formal adjoint is :
     \begin{equation}
    ^t \mathsf{L} = \lvert d_{\theta}(y\cdot \theta +\tilde{\phi}) \rvert^{-2} \sum_{j=1}^n \dfrac{\partial ( y \cdot \theta+ \tilde{\phi}) }{\partial \theta_j} \partial_{\theta_j}
    \end{equation}
 thus $^t\mathsf{L}  e^{iy\cdot \theta+i\tilde{\phi}} = e^{iy\cdot\theta+i\tilde{\phi}}$.

 $k$-fold integration by parts gives
   $$( \mathsf{U}^{\mathbb{F}_{C,0}^{\infty}} , \chi)
   =  \lim_{\epsilon \rightarrow 0} (2\pi)^{-n} \int\int\int \mathsf{L}^k( a_{\mathbb{F}_{C,0}^{\infty}} \varpi) \chi e^{iy\cdot\theta+i\tilde{\phi}} dx dy d\theta$$
  \begin{enumerate}
\item For $k > 3( n +\tfrac{4}{3}\sqrt{\tfrac{n^2}{4}-\lambda}) $ , $\theta \mapsto \mathsf{L}^k( a_{\mathbb{F}_{C,0}^{\infty}} \varpi)$ is in $L^1(\RR^n)$. We can let $\epsilon\rightarrow 0$ under the integral sign :
 $$\mathsf{U}^{\mathbb{F}_{C,0}^{\infty}} = (2\pi)^{-n} \int \mathsf{L}^k(a_{\mathbb{F}_{C,0}^{\infty}}) e^{iy+ i\tilde{\phi} } \, d\theta$$
 $$\Rightarrow \lvert  \mathsf{U}^{\mathbb{F}_{C,0}^{\infty}}\rvert \leq (2\pi)^{-n}  \int \lvert \mathsf{L}^k (a_{\mathbb{F}_{C,0}^{\infty}})\rvert  \, d\theta
 \stackrel{(\ref{FC0in})}{\leq}  C_k(2\pi)^{-n} \int_{\lvert\theta \rvert\geq 1} \lvert \theta\rvert^{-\tfrac{1}{3}k} \, d\theta $$
  Hence $\mathsf{U}^{\mathbb{F}_{C,0}^{\infty}}\in L^{\infty}(x,y)$

\item Consider $D_y^{\beta} \mathsf{U}^{\mathbb{F}_{C,0}^{\infty}}$. Fix $k$ with $k > 3( n+\tfrac{4}{3}\sqrt{\tfrac{n^2}{4}-\lambda}  + \lvert\beta\rvert)$ so
 $$\mathsf{U}^{\mathbb{F}_{C,0}^{\infty}} = (2\pi)^{-n} \int \mathsf{L}^k (a_{\mathbb{F}_{C,0}^{\infty}}) e^{i y \cdot\theta+i\tilde{\phi}} \, d\theta$$
$$\Rightarrow D_y^{\beta} \mathsf{U}^{\mathbb{F}_{C,0}^{\infty}} = (2\pi)^{-n} \int \mathsf{L}^k(a_{\mathbb{F}_{C,0}^{\infty}}) \theta^{\beta} e^{i y \cdot\theta+i\tilde{\phi}} \, d\theta$$
$$\Rightarrow \lvert D_y^{\beta} \mathsf{U}^{\mathbb{F}_{C,0}^{\infty}}\rvert \leq 
(2\pi)^{-n} \int \lvert \mathsf{L}^k (a_{\mathbb{F}_{C,0}^{\infty}}) \theta^{\beta} \rvert \, d\theta
\leq (2\pi)^{-n} \int_{\lvert\theta\rvert\geq 1} \lvert\theta\rvert^{-\tfrac{1}{3}k +\tfrac{4}{3}\sqrt{\tfrac{n^2}{4}-\lambda}+ \lvert\beta\rvert}\, d\theta$$
With the choice of $k$ we have $\theta \mapsto \mathsf{L}^k(a_{\mathbb{F}_{C,0}^{\infty}}) \theta^{\beta} $ in in $L^1$. Since the constant is independent of $x$
$$D_y^{\beta} \mathsf{U}^{\mathbb{F}_{C,0}^{\infty}} \in L^{\infty}$$

\item Consider $xD_x \mathsf{U}^{\mathbb{F}_{C,0}^{\infty}}$. Fix $k$ with $k > 3( n +\tfrac{4}{3}\sqrt{\tfrac{n^2}{4}-\lambda}+ 1)$
 $$\mathsf{U}^{\mathbb{F}_{C,0}^{\infty}} = (2\pi)^{-n} \int \mathsf{L}^k(a_{\mathbb{F}_{C,0}^{\infty}}) e^{i y \cdot\theta+i\tilde{\phi}} \, d\theta$$
$$\Rightarrow xD_x \mathsf{U}^{\mathbb{F}_{C,0}^{\infty}} = (2\pi)^{-n} \int \left[  xD_x \mathsf{L}^k(a_{\mathbb{F}_{C,0}^{\infty}}) +
 (\mathsf{L}^k a_{\mathbb{F}_{C,0}^{\infty}}) xD_x \tilde{\phi} \right] e^{i y \cdot\theta+i\tilde{\phi}} \, d\theta$$
Consider the second term in the bracket: $x\partial_x \tfrac{2}{3}Z^{3/2} = x \lvert \theta_n\rvert ( 1 + x - \lvert\hat{\theta}'\rvert^2)^{1/2}$. We have $ \lvert\hat{\theta}'\rvert < (1+\kappa_1(\delta_2))$ by (\ref{conicFC0}) and $x <C$, so for  $C$ is independent of $x$:
$$\lvert x \partial_x \tilde{\phi} \rvert \leq C \lvert\theta_n\rvert$$
 From Lemma \ref{FC0in}, $ xD_x L^k(a_{\mathbb{F}_{C,0}^{\infty}}) \lesssim \lvert\theta\rvert^{-\tfrac{1}{3}k+\tfrac{4}{3}\sqrt{\tfrac{n^2}{4}-\lambda}+1}$. With the chosen $k$ we have:
  $$xD_x  \mathsf{U}^{\mathbb{F}_{C,0}^{\infty}} \in L^{\infty}$$
  For higher derivative $(xD_x)^N \mathsf{U}^{\mathbb{F}_{C,0}^{\infty}}$, we simply choose $k$ such that $k > 3( n +\tfrac{4}{3}\sqrt{\tfrac{n^2}{4}-\lambda} + N)$ so
$$(xD_x)^N L^k(a_{\mathbb{F}_{C,0}^{\infty}}) \leq \tilde{C}_N\lvert\theta\rvert^{-\tfrac{1}{3}k+\tfrac{4}{3}\sqrt{\tfrac{n^2}{4}-\lambda}+N}$$
 $$\Rightarrow (xD_x)^N L^k(a_{\mathbb{F}_{C,0}^{\infty}}) \in L^{\infty}$$

\item Consider $x^{-N} \mathsf{U}^{\mathbb{F}_{C,0}^{\infty}}$
$$x^{-1} =  (Z-Z_0)^{-1} \lvert\theta_n\rvert^{2/3}
= T^2 ( 1 - Z_0 T^2)^{-1} \lvert\theta_n\rvert^{2/3}$$
since $\lvert Z_0 \rvert < 2\delta_2$ and $0 \leq T = Z^{-1} < 2$. So for $C_N$ is independent of $x$:
$$x^{-N} \leq C_N \lvert\theta_n\rvert^{\tfrac{2}{3}N}$$
 Fix $k$ with $k > 3( n  +N +\tfrac{4}{3}\sqrt{\tfrac{n^2}{4}-\lambda})$ so
 $$\mathsf{U}^{\mathbb{F}_{C,0}^{\infty}} = (2\pi)^{-n} \int \mathsf{L}^k(a_{\mathbb{F}_{C,0}^{\infty}}) e^{i y \cdot\theta+i\tilde{\phi}} \, d\theta$$
$$\Rightarrow x^{-N} \mathsf{U}^{\mathbb{F}_{C,0}^{\infty}} = (2\pi)^{-n} \int \mathsf{L}^k (a_{\mathbb{F}_{C,0}^{\infty}}) \theta^{\beta} e^{i y \cdot\theta+i\tilde{\phi}} \, d\theta$$
$$\Rightarrow \lvert x^{-N} \mathsf{U}^{\mathbb{F}_{C,0}^{\infty}}\rvert \leq 
(2\pi)^{-n} \int \lvert x^{-N} \mathsf{L}^k(a_{\mathbb{F}_{C,0}^{\infty}}) \theta^{\beta} \rvert \, d\theta
\leq \tilde{C}_N (2\pi)^{-n} \int_{\lvert\theta\rvert\geq 1} \lvert\theta\rvert^{-\tfrac{1}{3}k +\tfrac{4}{3}\sqrt{\tfrac{n^2}{4}-\lambda}+N}\, d\theta$$
with the choice of $k$ then have
$$ x^{-N} \mathsf{U}^{\mathbb{F}_{C,0}^{\infty}} \in L^{\infty}$$
The same reasoning apply if we replace $\mathsf{U}^{\mathbb{F}_{C,0}^{\infty}} $ by $(xD_x)^l \mathsf{U}^{\mathbb{F}_{C,0}^{\infty}}$.

\item Consider $x^{-N} D_y^{\alpha}(xD_x)^{\beta} \mathsf{U}^{\mathbb{F}_{C,0}^{\infty}}
 = D_y^{\alpha} x^{-N} (xD_x)^{\beta} \mathsf{U}^{\mathbb{F}_{C,0}^{\infty}}$. Derivatives in $y$ introduce growth in $\lvert\theta\rvert$, while $x^{-N}$ is polynomial bounded in $\theta$. As above, by the same integration by parts technique on $y_n<0$ we show that we can arbitrarily reduce the growth in $\lvert\theta\rvert$ thus 
$$x^{-N} D_y^{\alpha}(xD_x)^{\beta} \mathsf{U}^{\mathbb{F}_{C,0}^{\infty}} \in L^{\infty}$$

\end{enumerate}

\end{proof}

 \begin{lemma}\label{cornormalFC0inf}[On $y_n > 0$ ]\
 
   For every $\beta>0$, there exists $\alpha > 0$ such that $\mathsf{U}^{\mathbb{F}_{C,0}^{\infty}} $ is smooth and vanishes to infinite order on $y_n >\beta$ and $x < \alpha^2$. In fact $\alpha$ is given by (\ref{inequalityonalphaIII2}) explicitly.
 

 \end{lemma}
 
  \begin{proof}
 
 For $\beta,\alpha >0$ whose relation is to be determined, consider $\chi(y)$ with $\supp \chi \subset \{y_n > \beta , x_n < \alpha^2 \}$. 
 
Together with the constraints on the region currently considered  ie $ \begin{cases} x < \alpha^2 \\
Z_0 + x \lvert\theta_n\rvert^{2/3} > 1 \\
\lvert Z_0\rvert \leq 2\delta_2 \end{cases}$ we have  ( can assume $\delta_2 <\tfrac{1}{4}$)
$$x \lvert\theta_n\rvert^{2/3} =1 - Z_0 >1- \lvert Z_0 \rvert > 1 - 2\delta_2
 \Rightarrow \lvert\theta_n\rvert^{2/3} > \dfrac{1- 2\delta_2}{x}
 \Rightarrow  \lvert\theta_n\rvert >  \left(\dfrac{1-2\delta_2}{\alpha^2}\right)^{3/2} >( \dfrac{2\delta_2}{\alpha^2 })^{3/2}  $$
 $$\Rightarrow \big| 1 - \lvert\hat{\theta}'\rvert^2\big| = \dfrac{\lvert Z_0\rvert}{\lvert\theta_n\rvert^{2/3}}
< \alpha^2\Rightarrow 1 - \lvert\hat{\theta}'\rvert^2 + x< 2  \alpha^2  \  \ ( \text{since}\  x <\alpha^2 )$$
 Using this inequality and 
 $$ 0< 1 - \kappa_0(\delta_2) \leq  \lvert\hat{\theta}'\rvert \leq 1 +\kappa_1(\delta_2) < 2
\Rightarrow \lvert\hat{\theta}'\rvert <2 $$ 
 we obtain a lower bound for:
$$
 d_{\theta_n} (y\cdot\theta + \tilde{\phi}) = y_n - \tfrac{2}{3} (1 + x - \lvert\hat{\theta}'\rvert^2)^{3/2} - 3(1 + x- \lvert\hat{\theta}'\rvert^2)^{1/2} \lvert\hat{\theta}'\rvert^2 
 \geq y_n - \dfrac{2}{3} 2^{3/2} \alpha^3 - 12\times 2^{1/2}\alpha $$
 
 Now choose $\alpha $ such that
 \begin{equation}\label{inequalityonalphaIII2}
  \dfrac{2}{3} 2^{3/2} \alpha^3 + 12\times 2^{1/2}\alpha < \dfrac{\beta}{2}
  \end{equation}
 Thus on $\supp ( \chi(x,y) ($ we have
 $$d_{\theta_n} (y\cdot\theta + \tilde{\phi}) > \beta - \dfrac{1}{2}\beta = \dfrac{1}{2} \beta > 0$$
ie  the above expression is not zero. 
%
%

 Consider $\mathcal{L}$ with
$$ \mathcal{L}^t = \chi(x,y) \lvert d_{\theta}(y\cdot\theta +\tilde{\phi})\rvert \sum_{j=1}^n \dfrac{\partial \tilde{\phi} + y\cdot\theta}{\partial \theta_j} \partial_{\theta_j}$$
  $$ \Rightarrow\   ^t\mathcal{L} e^{iy\cdot \theta+i\tilde{\phi}} = \chi(x,y)e^{iy\cdot\theta+i\tilde{\phi}} $$
  By Lemma \ref{FC0in} we also have
  $$\mathcal{L}^k a_{\mathbb{F}_{C,0}^{\infty}} \leq C_k \lvert\theta_n\rvert^{\tfrac{4}{3}\sqrt{\tfrac{n^2}{4}-\lambda} -\tfrac{1}{3}k } $$
 From this point, the proof uses the same argument as those of Lemma \ref{y<0C0inf} to show on the region currently considered of $\chi$ the boundedness of 
 $$x^{-N} D_y^{\alpha} (xD_x)^{\beta} \mathsf{U}^{\mathbb{F}_{C,0}^{\infty}}.$$
 \end{proof}


 \begin{lemma}[Statement about wavefront set on $x  > 0$]\label{wfC0inf}

   The wavefront set of $\mathsf{U}^{\mathbb{F}_{C,0}^{\infty}}$ is contained in $\Sigma(\delta_2) $ where
   $$\Sigma(\delta_2,\delta) = 
   \{(x,y,\xi,\eta) : \xi = \partial_x \left( y\cdot\theta -\tilde{\phi}\right) , \eta = d_y \left( y\cdot\theta -\tilde{\phi}\right) , d_{\theta} \left( y\cdot \theta-\tilde{\phi}\right) = 0 , \theta \in \text{conesupp}\, a_{\mathbb{F}_{C,0}^{\infty}} , x> \delta \}   $$
   with $\tilde{\phi} = \tfrac{2}{3} Z^{3/2} \sgn\theta_n$ and $$\text{conesupp}\  a_{\mathbb{F}_{C,0}^{\infty}}= \{\lvert\theta\rvert  >0 , \left(1 - \kappa_0(\delta_2)\right) \lvert\theta_n\rvert \leq
   \lvert\theta'\rvert\leq \left(1 + \kappa_1(\delta_2)\right) \lvert\theta_n\rvert   \}.$$
$\kappa_0(\delta_2), \kappa_1(\delta_2)$ are defined in the calculation of $\text{conesupp}\  a_{\mathbb{F}_{C,0}^{\infty}}$ at the end of the subsection.
 \end{lemma}
 
\begin{proof}

$a_{\mathbb{F}_{C,0}^{\infty}} \in S^{\tfrac{4}{3}\sqrt{\tfrac{n^2}{4}-\lambda}}_{\tfrac{1}{3}, 0}$ by Lemma \ref{FC0in}. In addition, since we have the same type of symbol and same phase function as in Friedlander Theorem 6.1 \cite{Fried01}, we employ the same argument to prove our results. The main idea is that on $x >\delta$, $\tilde{\Phi}$ is smooth on $\text{conesupp}\   a_{\mathbb{F}_{C,0}^{\infty}}$ and is homogeneous degree $1$ with respect to $\theta$, thus is a phase function.  
$$( \mathsf{U}^{\mathbb{F}_{C,0}^{\infty}}_{\delta} , \chi) = \lim_{\epsilon\rightarrow 0} (2\pi)^{-n} \int\int\int a_{\mathbb{F}_{C,0}^{\infty}}\chi(x,y) \varpi(\epsilon \theta) e^{i\tilde{\Phi}} \, dx dy d\theta$$
is an oscillatory integral in the sense of Hormander. The wavefront set statement now follows from proposition 2.5.7 in FIO I by H\"ormander \cite{Hor03}.

\end{proof}


  \textbf{Calculation of conic neighborhood}
  \begin{equation}\label{conicFC0}
  \{ \lvert\theta\rvert \geq 1 , \lvert Z_0\rvert \leq 2 \delta_2 \}
 \subset \{\lvert\theta\rvert > 0,   1 - \kappa_0(\delta_2) \leq \lvert\hat{\theta}'\rvert \leq 1 + \kappa_1(\delta_2)    \}
 \end{equation} 
  for some positive constants $\kappa_0$ and $\kappa_1$ depending on $\delta_2$.
  \begin{proof}
  Consider the set $\{ \lvert\theta\rvert \geq 1, \lvert Z_0 \rvert \leq 2\delta_2\}$  
  $$Z_0 = \lvert\theta_n\rvert^{2/3} (1 - \lvert\hat{\theta}'\rvert^2)
   = \lvert\theta\rvert^{\tfrac{2}{3}} \dfrac{1-a^2}{(1+a^2)^{1/3}} ,\  a = \lvert \hat{\theta}'\rvert = \dfrac{\lvert\theta'\rvert}{\lvert\theta_n\rvert}   $$
 
  \begin{enumerate}
\item Consider the case in which $Z_0 \geq 0$ ie $\begin{cases}\lvert\hat{\theta}'\rvert \leq 1  \\ 
 \lvert\theta\rvert^{\tfrac{2}{3}} \dfrac{1-a^2}{(1+a^2)^{1/3}}   \leq 2\delta_2 \end{cases}$
 
Denote $A = (1 + a^2)^{1/3}$ so $a^2 = A^3 - 1$
 $$ f(A) = \dfrac{1-a^2}{(1+a^2)^{1/3}} = \dfrac{2-A^3}{A}
 \Rightarrow f'(A) = \dfrac{-2(A^3+1)}{A^2} $$
 $0 \leq a \leq 1$ so $1\leq A \leq 2^{1/3}$ . On this region $f'(A) < 0$, $f(A)$ is monotone decreasing from $1$ to $0$. Hence for $\delta_2$ such that $0 < 2\delta_2 <1$ , there exists a unique $A$ in this interval such that 
 $$f(A) = 2\delta_2$$
 and since $a > 0$ there exists a unique $\kappa_0(\delta_2)$ such that with $a = 1- \kappa_0(\delta_2)$
 $$\dfrac{1-a^2}{(1+a^2)^{1/3}}   = 2\delta_2$$
 This is the value for $\lvert\hat{\theta}'\rvert$ for $\lvert\theta\rvert = 1$. 
 
 If $\lvert\theta\rvert$ increases from $1$, want $\dfrac{1-a^2}{(1+a^2)^{1/3}} $ to decrease,which happens if $a$ increases. So we have:
 $$\{\lvert\theta\rvert \geq 1 ,0 \leq \lvert\hat{\theta}'\rvert < 1 , Z_0 \leq 2\delta_2\}
  \subset \{ \lvert\theta\rvert\geq 1, 1+\kappa_0(\delta_2) \leq \lvert\hat{\theta}'\rvert < 1\}$$
Note that  $1 - \kappa_0(\delta_2)<1$ so $\kappa_0(\delta_2) >0$ and as $\delta_2 \rightarrow 0$ $1+ \kappa_0(\delta_2)\rightarrow 1$ ie $\kappa_0(\delta_0) \rightarrow 0$

\item  Similarly we consider the case for $Z_0 < 0$ ie $\begin{cases}\lvert\hat{\theta}'\rvert \geq 1  \\ 
 \lvert\theta\rvert^{\tfrac{2}{3}} \dfrac{a^2-1}{(1+a^2)^{1/3}}   \leq 2\delta_2 \end{cases}$

Denote $A = (1 + a^2)^{1/3} \Rightarrow a^2 = A^3 - 1$
 $$ f(A) = \dfrac{a^2-1}{(1+a^2)^{1/3}} = \dfrac{A^3-2}{A}
 \Rightarrow f'(A) = \dfrac{2A^3 +2}{A^2}$$
 $a\geq 1$ so $A\geq 2^{2/3}$ hence $f'(A) >0$. $f(A)$ is monotone, increases from zero. There exists a unique $A$ so that 
 $$f(A) = 2\delta_2$$
 and thus there is a unique $1+ \kappa_1(\delta_2)$ so that for $a = 1+ \kappa_1(\delta_2)$ we have
 $$\dfrac{a^2-1}{(1+a^2)^{1/3}} = 2\delta_2$$
 
 If $\lvert\theta\rvert$ increases from $1$, want $\dfrac{1-a^2}{(1+a^2)^{1/3}} $ to decrease, which happens if $A$ decreases so 
  $$\{\lvert\theta\rvert \geq 1, \lvert\hat{\theta}'\rvert \geq 1  , 0 < -Z_0 < 2\delta_2\}
   \subset \{\lvert\theta\rvert \geq 1 , 1 \leq \lvert\hat{\theta}'\rvert \leq 1 + \kappa_1(\delta_2)\}$$
 Note that $\kappa_1(\delta_2) >0$ and as $\delta_2 \rightarrow 0$ $A \searrow 2^3$ ie $a\searrow 1$ ie $\kappa_1(\delta_2) \searrow 0$
 
 \end{enumerate} 
In short we have:
 $$\{ \lvert\theta\rvert \geq 1 , \lvert Z_0 \rvert \leq 2 \delta_2\}
 \subset \{\lvert\theta\rvert > 0,   1 - \kappa_0(\delta_0) \leq \lvert\hat{\theta}'\rvert \leq 1 + \kappa_1(\delta_2)    \}$$

 \begin{remark}\indent\par\noindent
 \begin{itemize}
 \item The condition $Z_0 < 2\delta_2$ and $\lvert\hat{\theta}'\rvert \leq (1 - 2\delta_1)^{-1}$ does not imply that $Z_0$ is bounded below ie does not imply that $\lvert Z_0 \rvert $ is bounded.
 
 \item The difference between region currently considered and coneregion currently considered : in the region currently considered , with constraint $0 < Z_0 < 2\delta_2$ (correspondingly $0<-Z_0 < 2\delta_2$) , when $\lvert\theta_n\rvert$ increases, this forces $\tfrac{1-a^2}{(1+a^2)^{1/3}}$ (correspondingly $\tfrac{a^2-1}{(1+a^2)^{1/3}}$ to decrease) , while in the coneregion currently considered, we do not have this constraint, hence $\lvert Z_0\rvert$ will not be necessarily bounded.

 \item not the same situation with $Z_0 > \delta_2$ or $-Z_0 > \delta_2$. 
 \end{itemize}
 \end{remark}

  \end{proof}
  
  
\subsection{On region $\mathbb{F}_{C,+}$}
 On the current region we have $\{\lvert\theta\rvert \geq 1, Z_0 \geq \delta_2\}$, whose cone region currently considered is $\{\lvert\theta\rvert  > 0 , \lvert\theta_n\rvert\geq \lvert\theta'\rvert\}$. 
 On region $\mathbb{F}_{C,+}^{\infty}$ we have  $\sigma^{-1} \sim 0$, while we subpartition $\mathbb{F}_{C,+}^0$ into $\mathbb{F}_{C,+}^{0,0}$ (corresponds to $\mu \sim 0$) and $\mathbb{F}_{C,+}^{0,\infty}$ (corresponds to $\mu^{-1} \sim 0$), where $\mu$ is the blow-up variable according to the blow-up at the corner of face $\mathbb{F}_C$ at $\sigma=0$ and $h=0$ . 
 
 For  subregion $\mathbb{F}_{C,+}^{0,0}$ we show
 \begin{enumerate}
\item $\mathsf{U}^{\mathbb{F}_{C,+}^{0,0}}$ is smooth on $x >\delta$ for any $\delta > 0$ ( and any $\delta_2 >0$).
\item $ \mathsf{U}^{\mathbb{F}^{0,0}_{C,+}}\in x^{s_-} \mathcal{C}^{\infty}_{x,y} + x^{s_+} \mathcal{C}^{\infty}_{x,y} $  on $y_n \neq 0$.
\end{enumerate}

For region $\mathbb{F}_{C,+}^{\infty}$ and  $\mathbb{F}_{C,+}^{0,\infty}$ we will show
 \begin{enumerate}
\item  On $y_n < 0$, $ \mathsf{U}^{\mathbb{F}^{\infty}_{C,+}},  \mathsf{U}^{\mathbb{F}_{C,+}^{0,\infty}}$ are smooth and vanishes to infinite order at $x = 0$. 
\item On $y_n>0$ $\mathsf{U}^{\mathbb{F}^{\infty}_{C,+}},  \mathsf{U}^{\mathbb{F}_{C,+}^{0,\infty}}$ are smooth and vanishes to infinite order if we are close enough to the boundary for $y_n > 0$. 
\item wavefront set statement : contained in the forward bicharacteristics from $x = y = 0$.
\end{enumerate}
\subsubsection{On $\mathbb{F}_{C,+}^{0,0}$}
\begin{align}\label{indicialFC+}
    ( \mathsf{U}^{\mathbb{F}_{C,+}^{0,0}}_{\delta} , \chi)
    &= \lim_{\epsilon \rightarrow 0} (2\pi)^{-n} \int\int\int \psi_{\mathbb{F}_{C,+}^{0,0}} \hat{u}_x \chi(x,y) \varpi(\epsilon\theta)  e^{iy\cdot \theta} \, dx\,dy\, d\theta
    \end{align}
    We recall the properties of the approximate solution on the region currently considered: 
    $$\psi_{\mathbb{F}_{C,+}^{0,0}}\hat{u}_x =  \big(\lvert\theta_n\rvert (1 - \lvert\hat{\theta}'\rvert^2)^{1/2}\big)^{\sqrt{\tfrac{n^2}{4}-\lambda}} \hat{\mathsf{U}}^{\mathbb{F}_{C,+}^{0,0}} $$
    where  $ \hat{\mathsf{U}}^{\mathbb{F}_{C,+}^{0,0}}$ is conormal in $t ( = \tfrac{z}{h})$ and smooth in $h$, where $t = x \lvert\theta_n\rvert( 1 - \lvert\hat{\theta}'\rvert^2)^{1/2}$ and $h = \lvert\theta_n\rvert^{-1} ( 1 - \lvert \hat{\theta}'\rvert^2)^{-3/2}$. In more details:
\begin{align*}
 2\sqrt{\tfrac{n^2}{4}-\lambda}\notin \mathbb{Z} :&\  \hat{\mathsf{U}}^{\mathbb{F}_{C,+}^{0,0}} \in x^{\tfrac{n}{2}} t^{\sqrt{\tfrac{n^2}{4}-\lambda}} \mathcal{C}^{\infty}_{t,h} + x^{\tfrac{n}{2}} t^{-\sqrt{\tfrac{n^2}{4}-\lambda}} \mathcal{C}^{\infty}_{t,h}
  \end{align*}
  $$ (1 - \lvert\hat{\theta}'\rvert^2)^{\tfrac{n}{2}} \big[\lvert\theta_n\rvert (1 - \lvert\hat{\theta}'\rvert^2)^{1/2}\big]^{\sqrt{\tfrac{n^2}{4}-\lambda}} \hat{\mathsf{U}}^{\mathbb{F}_{C,+}^{0,0}} \in  \big[\lvert\theta_n\rvert (1 - \lvert\hat{\theta}'\rvert^2)^{1/2} \big]^{2\sqrt{\tfrac{n^2}{4}-\lambda}} x^{s_+} \mathcal{C}^{\infty}_{t,h} + x^{s_-} \mathcal{C}^{\infty}_{t,h}$$
Using this form of $\hat{u}_x$ we rewrite (\ref{indicialFC+}) as
 $$( \mathsf{U}^{\mathbb{F}_{C,+}^{0,0}}_{\delta} , \chi) = \lim_{\epsilon \rightarrow 0} \sum_{\pm} (2\pi)^{-n} x^{s_{\pm}} \int\int\int a_{\mathbb{F}_{C,+}^{0,0};\pm} \chi(x,y) \varpi(\epsilon\theta)  e^{iy\cdot \theta} \,dx\,dy\, d\theta$$

 \begin{lemma}\label{lemssymFC+00}[Properties of $a_{\mathbb{F}_{C,+}^{0,0};\pm}$]\
 
 \begin{enumerate}
 
 \item Let $L$ be the differential operator whose formal adjoint is $^t L = \dfrac{1}{\lvert y\rvert^2} \sum y_j \dfrac{\partial}{i\partial \theta_j}$, then for $C_k$ is independent of $x$
  $$ \lvert D_x^m L^k(a_{\mathbb{F}_{C,+}^{0,0}} ) \rvert\leq C_k \lvert\theta\rvert^{-\tfrac{1}{3}k + 2 \sqrt{\tfrac{n^2}{4}-\lambda}+m} $$

 \item On $x\neq 0$, 
 $$a_{\mathbb{F}_{C,+}^{0,0}} = \sum_{\pm} x^{s_{\pm}} a_{\mathbb{F}_{C,+}^{0,0};\pm }\in S^{-\infty }_{\tfrac{1}{3},0;\theta}$$

 \end{enumerate}

 \end{lemma}
 
 \begin{proof}\indent\par\noindent
 \begin{enumerate}
 \item First we consider a bound on $a_{\mathbb{F}_{C,+}^{0,0};\pm} $.  
    $$\psi_{\mathbb{F}_{C,+}^{0,0}}\hat{u}_x =  \big(\lvert\theta_n\rvert (1 - \lvert\hat{\theta}'\rvert^2)^{1/2}\big)^{\sqrt{\tfrac{n^2}{4}-\lambda}} \hat{\mathsf{U}}^{\mathbb{F}_{C,+}^{0,0}} $$
    $$ t = x \lvert\theta_n\rvert( 1 - \lvert\hat{\theta}'\rvert^2)^{1/2} ;\    h = \lvert\theta_n\rvert^{-1} ( 1 - \lvert \hat{\theta}'\rvert^2)^{-3/2} $$
  $$ (1 - \lvert\hat{\theta}'\rvert^2)^{\tfrac{n}{2}} \big[\lvert\theta_n\rvert (1 - \lvert\hat{\theta}'\rvert^2)^{1/2}\big]^{\sqrt{\tfrac{n^2}{4}-\lambda}} \hat{\mathsf{U}}^{\mathbb{F}_{C,+}^{0,0}} \in  \big[\lvert\theta_n\rvert (1 - \lvert\hat{\theta}'\rvert^2)^{1/2} \big]^{2\sqrt{\tfrac{n^2}{4}-\lambda}} x^{s_+} \mathcal{C}^{\infty}_{t,h} + x^{s_-} \mathcal{C}^{\infty}_{t,h}$$
  In the region currently considered on which $h, t$ and $\lvert\hat{\theta}'\rvert$ are bounded 
$$ \Rightarrow a_{\mathbb{F}_{C,+}^{0,0};\pm} \leq  \lvert \theta_n\rvert^{2 \sqrt{\tfrac{n^2}{4}-\lambda}} $$
%
    
    
  To compute the bound for higher derivatives, we use:
       \begin{align*}
   \lvert\theta_n\rvert\partial_{\theta_n} &=  \left(1-2 \lvert\hat{\theta}'\rvert^2 h^{2/3} \lvert\theta_n\rvert^{-1/3} \right) t\partial_t
   +\left( - \lvert\theta_n\rvert^{-1}   - 3 h^{2/3} \lvert\theta_n\rvert^{-1/3} \lvert\hat{\theta}'\rvert^2  \right)h\partial_h   \\
\lvert\theta_n\rvert \partial_{\theta_j} &= 
 \left(- h^{3/2} \lvert\theta_n\rvert^{2/3} \lvert\theta_n \rvert^{2/3} \lvert\hat{\theta}'\rvert \right)t \partial_t   - 3 h^{2/3} \lvert\theta_n\rvert^{-1/3} \lvert\hat{\theta}'\rvert h\partial_h \\
    x\partial_x &= t\partial_t ;  \partial_x = \lvert\theta_n\rvert (1 -\lvert\hat{\theta}'\rvert^2)^{1/2} \partial_t
\end{align*}
Using the above computation, the fact that $a_{\mathbb{F}_{C,+}^{0,0};\pm}$ is smooth in $t,h$ and on the current support $t,h,\lvert\hat{\theta}'\rvert$ are bounded, we obtain:
\begin{equation}\label{boundFC+00}
\lvert D_x^m L^k(a_{\mathbb{F}_{C,+}^{0,0};\pm} ) \rvert\leq C_k \lvert\theta\rvert^{-\tfrac{1}{3}k + 2 \sqrt{\tfrac{n^2}{4}-\lambda} +m }
\end{equation}

 \item On $x\neq 0$ we have smoothness in $x$ for $a_{\mathbb{F}_{C,+}^{0,0}} = \sum_{\pm} x^{s_{\pm}} a_{\mathbb{F}_{C,+}^{0,0};\pm} $. Together with the bound in (\ref{boundFC+00}), in fact since $\partial_x= x^{-1} t\partial_t$ we have:
$$a_{\mathbb{F}_{C,+}^{0,0}} \in S^{2 \sqrt{\tfrac{n^2}{4}-\lambda}}_{\tfrac{1}{3},0; \theta} $$
 In fact, we can improve the order of the symbol in this case with we are away $\{x =0\}$ : in this region currently considered since $Z >\delta_2$ ie $\lvert\theta_n\rvert > \lvert\theta'\rvert$ so $\lvert\theta_n\rvert $ is equivalent to $\lvert\theta\rvert$ thus it suffices to compare with $\lvert\theta_n\rvert$. Since $\lvert\theta_n\rvert = x^{-\tfrac{3}{2}} t^{3/2} h^{1/2}$,  on $x >\delta$ for all $N$ we have : 
   $$ \lvert\theta_n\rvert^N \hat{\mathsf{U}}^{\mathbb{F}_{C,+}^{0,0}}
    = x^{-\tfrac{3}{2}N} t^{\tfrac{3}{2}N} h^{\tfrac{1}{2}N} \leq C_N$$ 
 so $a_{\mathbb{F}_{C,+}^{0,0}} \in S^{-\infty }_{\tfrac{1}{3},0;\theta}$. 

 \end{enumerate}
 \end{proof}

\begin{lemma}\label{x>deltaC+00}
$\mathsf{U}^{\mathbb{F}_{C,+}^{0,0}}$ is smooth on $x > \delta$ for any $\delta > 0$.
\end{lemma}

\begin{proof}

Denote by $\mathsf{U}^{\mathbb{F}_{C,+}^{0,0}}_{\delta}$ the restrictions of $\mathsf{U}^{\mathbb{F}_{C,+}^{0,0}}$ to $\mathcal{D}'(X_{\delta})$. For $x >\delta$ we have :
   \begin{align*}
    ( \mathsf{U}^{\mathbb{F}_{C,+}^{0,0}}_{\delta} , \chi)
    &= \lim_{\epsilon \rightarrow 0} (2\pi)^{-n} \int\int\int \psi_{\mathbb{F}_{C,+}^{0,0}} \hat{u}_x \chi(x,y) \varpi(\epsilon\theta)  e^{iy\cdot \theta} \, dx\,dy\, d\theta\\
     &= \lim_{\epsilon \rightarrow 0} (2\pi)^{-n} \int\int\int a_{\mathbb{F}_{C,+}^{0,0}} \chi(x,y) \varpi(\epsilon\theta)  e^{iy\cdot \theta} \,dx\,dy\, d\theta
    \end{align*}
 From lemma (\ref{lemssymFC+00})  on $x> \delta$  we have $a_{\mathbb{F}_{C,+}^{0,0}}\in S^{-\infty}_{\tfrac{1}{3},0}$
hence $\mathsf{U}^{\mathbb{F}_{C,+}^{0,0}}_{\delta}, D_y^{\beta} \mathsf{U}^{\mathbb{F}_{C,+}^{0,0}}_{\delta} \in L^{\infty}$.
  The same situation remains when we consider $\partial_x \mathsf{U}^{\mathbb{F}_{C,+}^{0,0}}_{\delta}$.  Since $\partial_x = x^{-1} t\partial_t$, on $x >\delta$, $t\partial_t ( a_{\mathbb{F}_{C,+}^{0,0}}) $ still rapidly decreases in $\theta$.

\end{proof}

\begin{lemma}\label{conormalC+00}
$ \mathsf{U}^{\mathbb{F}_{C,+}^{0,0}} \in x^{s_-} \mathcal{C}^{\infty}_{x,y} + x^{s_+} \mathcal{C}^{s_+} $ on $y_n \neq 0$.

\end{lemma}

\begin{proof}
 The proof is the same that for Lemma \ref{conormalFC00}. Consider $\chi(x,y)\in \mathcal{C}^{\infty}_{(x,y)}$ with either $\supp \chi \subset \{ y_n < 0\}$ or $\supp \chi \subset \{ y_n > 0\}$
 \begin{align*}
    ( \mathsf{U}^{\mathbb{F}_{C,+}^{0,0}} , \chi)
    &= \lim_{\epsilon \rightarrow 0} (2\pi)^{-n} \int\int\int \psi_{\mathbb{F}_{C,+}^{0,0}}\hat{u}_x  \chi(x,y) \varpi(\epsilon\theta)  e^{iy\cdot \theta} dx dy d\theta\\
    &=  \lim_{\epsilon \rightarrow 0} (2\pi)^{-n} \int\int\int \sum_{\pm} x^{s_{\pm}} a_{\mathbb{F}_{C,+}^{0,0};\pm}  \varpi(\epsilon\theta) \chi e^{iy\cdot\theta} dx dy d\theta
 \end{align*}
    
     $\lvert d_{\theta} (y\cdot \theta)\rvert^2$ is bounded away from zero on region currently considered of $\chi \times \RR^n$ since the $n$-th component of the phase function,  which is $y_n$, is never zero on this region currently considered. Consider a first order differential operator $L$ whose formal adjoint is :
 $$^t L = \dfrac{1}{\lvert y\rvert^2} \sum y_j \dfrac{\partial}{i\partial \theta_j}
 \Rightarrow L e^{iy\cdot \theta} = e^{iy\cdot\theta}$$
 Do $k$-fold integration by parts to give :
   \begin{equation}\label{eqnF+00}
   ( \mathsf{U}^{\mathbb{F}_{C,+}^{0,0}}  , \chi)
   =  \lim_{\epsilon \rightarrow 0} \sum_{\pm} (2\pi)^{-n} x^{s_{\pm}} \int\int\int L^k( a_{\mathbb{F}_{C,+}^{0,0};\pm}  \varpi) \chi e^{iy\cdot\theta} \,dx \,dy \,d\theta
   \end{equation}
 
   \begin{enumerate}
   \item For $k > 3( n + 2 \sqrt{\tfrac{n^2}{4}-\lambda}) $, $ \theta \mapsto L^k_{\mathbb{F}_{C,+}^{\infty}} a_{\mathbb{F}_{C,+}^{\infty}} \in L^1(\RR^n)$ by Lemma \ref{lemssymFC+00}. We can let $\epsilon\rightarrow 0$ under the integral sign in (\ref{eqnF+00}) to obtain:
 $$\mathsf{U}^{\mathbb{F}_{C,+}^{0,0}}  = \sum_{\pm} (2\pi)^{-n} x^{s_{\pm}}\int  (L^ka_{\mathbb{F}_{C,+}^{0,0};\pm}  ) e^{i y \cdot\theta} \, d\theta = \sum_{\pm} x^{s_{\pm}}\mathsf{U}^{\mathbb{F}_{C,+}^{0,0};\pm} $$
By Lemma \ref{lemssymFC+00}, $L^ka_{\mathbb{F}_{C,+}^{0,0};\pm} \lesssim \lvert\theta\rvert^{-\tfrac{1}{3}k+2 \sqrt{\tfrac{n^2}{4}-\lambda}+l}$. With the current choice of $k$ we have:
  $$ \mathsf{U}^{\mathbb{F}_{C,+}^{0,0};\pm} \in L^{\infty}_{x,y}$$

 \item Fix $k$ with $k > 3( n  + \lvert\beta\rvert+ 2 \sqrt{\tfrac{n^2}{4}-\lambda})$ in (\ref{eqnF+00}). Let $\epsilon\rightarrow 0$ under the integral sign to obtain:
 $$\mathsf{U}^{\mathbb{F}_{C,+}^{0,0}}  = \sum_{\pm}(2\pi)^{-n} x^{s_{\pm}}\int L^k(a_{\mathbb{F}_{C,+}^{0,0};\pm} ) e^{i y \cdot\theta} \, d\theta = \sum_{\pm} x^{s_{\pm}}\mathsf{U}^{\mathbb{F}_{C,+}^{0,0};\pm} $$
 Consider $D_y^{\beta} \mathsf{U}^{\mathbb{F}_{C,+}^{0,0};\pm} $. 
 $$D_y^{\beta} \mathsf{U}^{\mathbb{F}_{C,+}^{0,0};\pm}  = (2\pi)^{-n} \int   L^k(a_{\mathbb{F}_{C,+}^{0,0};\pm} ) \theta^{\beta} e^{i y \cdot\theta} \, d\theta$$
$$\Rightarrow \lvert D_y^{\beta}  \mathsf{U}^{\mathbb{F}_{C,+}^{0,0}} \rvert \leq 
(2\pi)^{-n} \int \lvert L^k(a_{\mathbb{F}_{C,+}^{0,0};\pm} ) \theta^{\beta} \rvert \, d\theta
\leq (2\pi)^{-n} \int_{\lvert\theta\rvert\geq 1} \lvert\theta\rvert^{-\tfrac{1}{3}k+2 \sqrt{\tfrac{n^2}{4}-\lambda}+  \lvert\beta\rvert}\, d\theta$$
With the current choice of $k$, $\theta \mapsto L^k(a_{\mathbb{F}_{C,+}^{0,0}} ) \theta^{\beta} \in L^1$ by lemma (\ref{lemssymFC+00}), so we have:
$$D_y^{\beta}  \mathsf{U}^{\mathbb{F}_{C,+}^{0,0};\pm}  \in L^{\infty}_{x,y}$$

\item Fix $k$ with $k > 3( n +2 \sqrt{\tfrac{n^2}{4}-\lambda}+l) $ in (\ref{eqnF+00})
 $$\mathsf{U}^{\mathbb{F}_{C,+}^{0,0}}  =\sum_{\pm} (2\pi)^{-n} x^{s_{\pm}}\int L^k(a_{\mathbb{F}_{C,+}^{0,0}} ) e^{i y \cdot\theta} \, d\theta = \sum_{\pm} x^{s_{\pm}} \mathsf{U}^{\mathbb{F}_{C,+}^{0,0};\pm}$$
Consider $D_x^l \mathsf{U}^{\mathbb{F}_{C,+}^{0,0};\pm} $:
$$D_x^l \mathsf{U}^{\mathbb{F}_{C,+}^{0,0}}  = (2\pi)^{-n} \int ( D_x^l L^ka_{\mathbb{F}_{C,+}^{0,0}} ) e^{i y \cdot\theta} \, d\theta$$
 From lemma (\ref{lemssymFC+00}), $D_x^l L^ka_{\mathbb{F}_{C,+}^{0,0};\pm} \lesssim \lvert\theta\rvert^{-\tfrac{1}{3}k+2 \sqrt{\tfrac{n^2}{4}-\lambda}+l}$. With the current choice of $k$, we have:
 $$D_x^l  \mathsf{U}^{\mathbb{F}_{C,+}^{0,0};\pm} \in L^{\infty}.$$

\item Same argument is used to deal with mixed derivatives $D_y^{\alpha} D_x^l$. In this case one needs to choose $k$ satisfying 
$k > 3(n +2\sqrt{\tfrac{n^2}{4}-\lambda} + l + \lvert\alpha\vert)$ so that
$$D_y^{\alpha} D_x^l\mathsf{U}^{\mathbb{F}_{C,+}^{0,0};\pm} \in L^{\infty}_{x,y}$$

\item We have shown that $ \mathsf{U}^{\mathbb{F}_{C,+}^{0,0}}$ can be rewritten as $\sum_{\pm} x^{s_{\pm}}  \mathsf{U}^{\mathbb{F}_{C,+}^{0,0};\pm}$ where $D_y^{\alpha} D_x^l\mathsf{U}^{\mathbb{F}_{C,+}^{0,0};\pm} \in L^{\infty}_{x,y}$ for any $\alpha$ and $l$. This means:
$$\mathsf{U}^{\mathbb{F}_{C,+}^{0,0};\pm} \in x^{s_-}\mathcal{C}^{\infty}_{x,y} + x^{s_+}\mathcal{C}^{\infty}_{x,y}$$

\end{enumerate}

\end{proof}

\

\subsubsection{For subregion $\mathbb{F}_{C,+}^{\infty}$}\
 \begin{align*}
     ( \mathsf{U}^{\mathbb{F}_{C,+}^{\infty}} , \chi)
    &= \lim_{\epsilon \rightarrow 0} (2\pi)^{-n} \int\int\int  \psi_{\mathbb{F}_{C,+}^{\infty}} \hat{u}_x \chi(x,y) \varpi(\epsilon\theta)  e^{iy\cdot \theta} \,dx \,dy\, d\theta\\
    &=  \lim_{\epsilon \rightarrow 0} (2\pi)^{-n} \int\int\int   a_{\mathbb{F}_{C,+}^{\infty}} \varpi \chi e^{iy\cdot\theta+i\phi} \,dx \,dy \, d\theta
    \end{align*}
 
 We list the properties of $a_{\mathbb{F}_{C,+}^{\infty}}$ that will be of later use:
    \begin{lemma}\label{F+inf}\indent\par\noindent
  \begin{enumerate}
  \item $L_{\mathbb{F}_{C,+}^{\infty}}$ whose formal adjoint is $   ^t L_{\mathbb{F}_{C,+}^{\infty}} = \lvert d_{\theta}(y\cdot\theta + \phi))\rvert^{-2} \sum_{j=1}^n \tfrac{\partial (y\cdot\theta+\phi)}{\partial \theta_j} \partial_{\theta_j}$, $\phi = - \tfrac{2}{3}(Z^{3/2} - Z_0^{3/2})\sgn\theta_n$, then: 
 $$ (xD_x)^l L_{\mathbb{F}_{C,+}^{\infty}} ^k a_{\mathbb{F}_{C,+}^{\infty}} \leq\tilde{C}_k\lvert\theta\rvert^{-\tfrac{1}{3}k+ \sqrt{\tfrac{n^2}{4}-\lambda}}$$
 
 \item On $x >\delta$, $a_{\mathbb{F}_{C,+}^{\infty}} \in S^{ \sqrt{\tfrac{n^2}{4}-\lambda}}_{\tfrac{1}{3},0; \theta}$.
\end{enumerate}
  \end{lemma}
  
\begin{proof}\indent\par\noindent
\begin{enumerate}
\item Recall the structure of the approximate solution from the construction
  $$\psi_{\mathbb{F}_{C,+}^{\infty}} \hat{u}_x = \big(\lvert\theta_n\rvert (1 - \lvert\hat{\theta}'\rvert^2)^{1/2}\big)^{\sqrt{\tfrac{n^2}{4}-\lambda}}\hat{\mathsf{U}}^{\mathbb{F}_{C,+}^{\infty}}  $$
  $$ \hat{\mathsf{U}}^{\mathbb{F}_{C,+}^{\infty}}
  = \exp(i\phi) h^{\tilde{\gamma}_{\mathfrak{B}_1}+\tfrac{1}{2}} \sigma^{-\tfrac{1}{2}} 
  \hat{\mathsf{V}}^{\mathbb{F}_{C,+}^{\infty}}$$
  where $\hat{\mathsf{V}}^{\mathbb{F}_{C,+}^{\infty}} \sim \sum_j h^j A_j (\sigma^{-1/6})$ with $A_j$ smooth, $\sigma^{-\tfrac{1}{2}} = x^{-\tfrac{3}{4}} (1 -\lvert\hat{\theta}'\rvert^2)^{\tfrac{3}{4}} $. 
 On the region currently considered $h, \sigma^{-1}$ and $\lvert\hat{\theta}'\rvert$ are bounded, $\tilde{a} < 0$ , and $0 \leq 1 -\lvert\hat{\theta}'\rvert^2\leq 1$ we have
 $$ a_{\mathbb{F}_{C,+}^{\infty}} \leq \lvert\theta_n\rvert^{\sqrt{\tfrac{n^2}{4}-\lambda}}$$

  For estimate on the higher derivatives : from direct computation we have
 \begin{align*}
 \theta_n\partial_{\theta_n}&= \left(- 1 - 3h^{2/3} \lvert\theta_n\rvert^{2/3} \lvert\hat{\theta}'\rvert^2\right)h \partial_h + \tfrac{1}{2} h^{2/3} \lvert\theta_n\rvert^{2/3} \lvert\hat{\theta}'\rvert^2 \sigma^{-1/6} \partial_{\sigma^{-1/6}}\\
 \lvert \theta_n\rvert\partial_{\theta_j}
 &= -3 h^{2/3} \lvert\theta_n\rvert^{2/3} \lvert\hat{\theta_j}\rvert h\partial_h
  + \tfrac{1}{2} h^{2/3} \lvert\theta_n\rvert^{2/3} \hat{\theta}_j \sigma^{-1/6}\partial_{\sigma^{-1/6}}\\
  x\partial_x &= -\tfrac{3}{2}\sigma^{-1}\partial_{\sigma^{-1}} = -\tfrac{1}{4} \sigma^{-1/6} \partial_{\sigma^{-1/6}}
 \end{align*}
 Since $h, \sigma^{-1}$ are bounded (independent of $x$) we have: with $\tilde{C}_k $ is independent of $x$,
  $$ (xD_x)^l L^k(a_{F_{C,+}^{\infty}}) \leq \tilde{C}_k\lvert\theta\rvert^{-\tfrac{1}{3}k+ \sqrt{\tfrac{n^2}{4}-\lambda}}.$$ 
 
\item From the same computation, and on $x \neq 0$ $\partial_x = -  \tfrac{1}{4}x^{-1} \sigma^{-1/6} \partial_{\sigma^{-1/6}}$ we have  
$$a_{\mathbb{F}_{C,+}^{\infty}} \in S^{\sqrt{\tfrac{n^2}{4}-\lambda}}_{\tfrac{1}{3},0; \theta}$$
\end{enumerate}
 \end{proof}

\begin{lemma}\label{y<0C+inf}
    For $y_n <0$ $\mathsf{U}^{\mathbb{F}_{C,+}^{\infty}}$ is smooth and vanishes to infinite order  at $x = 0$.
\end{lemma}
\begin{proof}

 Consider $\chi(x,y)$ with $\supp \chi \subset \{ y_n < 0\}$
 \begin{align*}
    ( \mathsf{U}^{\mathbb{F}_{C,+}^{\infty}} , \chi)
    &= \lim_{\epsilon \rightarrow 0} (2\pi)^{-n} \int\int\int  \psi_{\mathbb{F}_{C,+}^{\infty}} \hat{u}_x \chi(x,y) \varpi(\epsilon\theta)  e^{iy\cdot \theta} dx dy d\theta\\
    &=  \lim_{\epsilon \rightarrow 0} (2\pi)^{-n} \int\int\int   a_{\mathbb{F}_{C,+}^{\infty}} \varpi \chi e^{iy\cdot\theta+i\phi} dx dy d\theta
   \end{align*}
    
     We have $\lvert d_{\theta} (y\cdot \theta)\rvert^2$ is bounded away from zero on region currently considered of $\chi \times \RR^n$ since the $n$-th component of the phase function is:  $y_n - \text{a negative term}$. This is seen as follows:
  \begin{align*}
  &y_n - \tfrac{2}{3}\partial_{\theta_n} ((Z^{3/2} - Z_0^{3/2})\sgn\theta_n)\\
  &=y_n -\tfrac{2}{3}\left[(1 + x-\lvert\hat{\theta}'\rvert^2)^{3/2} - (1 - \lvert\hat{\theta}'\rvert^2)^{3/2}\right]
   - \lvert\theta_n\rvert \left[  (1 + x-\lvert\hat{\theta}'\rvert^2)^{1/2} 2\lvert\hat{\theta}'\rvert^2 \lvert\theta_n\rvert^{-1}    - (1 -\lvert\hat{\theta}'\rvert^2)^{1/2} 2\lvert\hat{\theta}'\rvert^2 \lvert\theta_n\rvert^{-1}   \right]\\
   &= y_n -  \text{positive term}
  \end{align*}
 If $y_n < 0$ then the above expression is never zero. Denote $\bar{\phi} = y \cdot \theta + \phi$. We consider a first order differential operator $L_{\mathbb{F}_{C,+}^{\infty}}$ whose formal adjoint is:
$$^t L_{\mathbb{F}_{C,+}^{\infty}} = \lvert d_{\theta}(y\cdot\theta+ \phi)\rvert^{-2} \sum_{j=1}^n \dfrac{\partial (y\cdot\theta + \phi)}{\partial \theta_j} \partial_{\theta_j}$$
 thus $L_{\mathbb{F}_{C,+}^{\infty}} e^{iy\cdot \theta + i\phi} = e^{iy\cdot\theta+i\phi}$.
 Do $k$-fold integration by parts to give:
   \begin{equation}\label{eqnF+inf}
   ( \mathsf{U}^{\mathbb{F}_{C,+}^{\infty}} , \chi)
   =  \lim_{\epsilon \rightarrow 0} (2\pi)^{-n} \int\int\int L_{\mathbb{F}_{C,+}^{\infty}}^k( a_{\mathbb{F}_{C,+}^{\infty}} \varpi) x^{-\tfrac{1}{2}\tilde{a}}\chi e^{iy\cdot\theta} dx dy d\theta\end{equation}

  \begin{enumerate}
\item   Consider $k > 3( n + \sqrt{\tfrac{n^2}{4}-\lambda} )$ then $ \theta \mapsto L^k_{\mathbb{F}_{C,+}^{\infty}} a_{\mathbb{F}_{C,+}^{\infty}} \in L^1(\RR^n)$. We can let $\epsilon\rightarrow 0$ under the integral sign in (\ref{eqnF+inf}) to obtain
 $$\mathsf{U}^{F_{C,+}^{\infty}} = (2\pi)^{-n} \int L^k_{\mathbb{F}_{C,+}^{\infty}}(a_{\mathbb{F}_{C,+}^{\infty}}) e^{i y \cdot\theta+i\tilde{\phi}} \, d\theta$$
  $$\Rightarrow \lvert  \mathsf{U}^{\mathbb{F}_{C,+}^{\infty}}\rvert \leq (2\pi)^{-n} \int  \lvert L^k_{\mathbb{F}_{C,+}^{\infty}}(a_{\mathbb{F}_{C,+}^{\infty}})\rvert  \, d\theta
 \leq  C_k(2\pi)^{-n} \int_{\lvert\theta\rvert \geq 1} \lvert \theta\rvert^{-\tfrac{1}{3}k} \, d\theta $$
$$\Rightarrow \mathsf{U}^{\mathbb{F}_{C,+}^{\infty}}\in L^{\infty}_{x,y}$$

\item Consider $D_y^{\beta} \mathsf{U}^{\mathbb{F}_{C,+}^{\infty}}$. Fix $k$ with $k > 3( n+ \sqrt{\tfrac{n^2}{4}-\lambda}  + \lvert\beta\rvert)$ so
 $$\mathsf{U}^{\mathbb{F}_{C,+}^{\infty}} = (2\pi)^{-n} \int L^k_{\mathbb{F}_{C,+}^{\infty}}(a_{\mathbb{F}_{C,+}^{\infty}}) e^{i y \cdot\theta+i\tilde{\phi}} \, d\theta$$
$$\Rightarrow D_y^{\beta} \mathsf{U}^{\mathbb{F}_{C,+}^{\infty}} = (2\pi)^{-n}\int L^k_{\mathbb{F}_{C,+}^{\infty}}(a_{\mathbb{F}_{C,+}^{\infty}}) \theta^{\beta} e^{i y \cdot\theta+i\tilde{\phi}} \, d\theta$$
$$\Rightarrow \lvert D_y^{\beta} \mathsf{U}^{\mathbb{F}_{C,+}^{\infty}}\rvert \leq 
(2\pi)^{-n} \int \lvert L^k_{\mathbb{F}_{C,+}^{\infty}}(a_{\mathbb{F}_{C,+}^{\infty}}) \theta^{\beta} \rvert \, d\theta
\leq (2\pi)^{-n} \int_{\lvert\theta\rvert\geq 1} \lvert\theta\rvert^{-\tfrac{1}{3}k+\sqrt{\tfrac{n^2}{4}-\lambda} + \lvert\beta\rvert}\, d\theta$$
With the choice of $k$,  $\theta \mapsto L^k_{\mathbb{F}_{C,+}^{\infty}}(a_{\mathbb{F}_{C,+}^{\infty}}) \theta^{\beta} \in L^1$. Since the constant is independent of $x$
$$D_y^{\beta} \mathsf{U}^{\mathbb{F}_{C,+}^{\infty}} \in L^{\infty}_{x,y}$$

\item Consider $xD_x \mathsf{U}^{\mathbb{F}_{C,+}^{\infty}}$. Fix $k$ with $k > 3( n + \sqrt{\tfrac{n^2}{4}-\lambda} + 1)$
 $$\mathsf{U}^{\mathbb{F}_{C,+}^{\infty}} = (2\pi)^{-n}\int L_{\mathbb{F}_{C,+}^{\infty}}^k(a_{\mathbb{F}_{C,+}^{\infty}}) e^{i y \cdot\theta+i\tilde{\phi}} \, d\theta$$
$$\Rightarrow xD_x \mathsf{U}^{\mathbb{F}_{C,+}^{\infty}} = (2\pi)^{-n} \int \left[ xD_x L_{\mathbb{F}_{C,+}^{\infty}}^k a_{\mathbb{F}_{C,+}^{\infty}} + (L^k a_{\mathbb{F}_{C,+}^{\infty}}) xD_x \phi \right] e^{i y \cdot\theta+i\phi} \, d\theta$$

$$x\partial_x \tfrac{2}{3}(Z^{3/2}-Z_0^{3/2})  = x \lvert \theta_n\rvert ( 1 + x - \lvert\hat{\theta}'\rvert^2)^{1/2}$$

Since $ \lvert\hat{\theta}'\rvert < 1$ and $x <C$  we have $\lvert x \partial_x \tilde{\phi} \rvert \leq C \lvert\theta_n\rvert$ where $C$ is independent of $x$. This together with the bound on $ xD_x L^ka_{\mathbb{F}_{C,+}^{\infty}} \lesssim \tilde{C}_k\lvert\theta\rvert^{-\tfrac{1}{3}k+ \sqrt{\tfrac{n^2}{4}-\lambda}}$ given by Lemma \ref{F+inf}, and the current choice $k$, we have :
  $$xD_x  \mathsf{U}^{\mathbb{F}_{C,+}^{\infty}} \in L^{\infty}$$
Similarly reasoning applies for $(xD_x)^N \mathsf{U}^{\mathbb{F}_{C,+}^{\infty}}$. In that case we need to choose $k$ such that $k > 3( n + \sqrt{\tfrac{n^2}{4}-\lambda}+ N)$ so
$$(xD_x)^N L^k(a_{\mathbb{F}_{C,+}^{\infty}}) \leq \tilde{C}_N\lvert\theta\rvert^{-\tfrac{1}{3}k+ \sqrt{\tfrac{n^2}{4}-\lambda} +N}$$
$$\Rightarrow (xD_x)^N L^k(a_{\mathbb{F}_{C,+}^{\infty}}) \in L^{\infty}(x,y)$$

\item Consider $x^{-N} \mathsf{U}^{\mathbb{F}_{C,+}^{\infty}}$ : We have
$$x^{-1} =  \sigma^{-2/3} h^{2/3} \lvert\theta_n\rvert^{2/3}$$
Since $h < \delta_2^{3/2}$ and $0 \leq \sigma^{-1} < 1$ we have for $C_N$ is independent of $x$
$$x^{-N} \leq C_N \lvert\theta_n\rvert^{\tfrac{2}{3}N}$$
Fix $k$ with $k > 3( n +  \sqrt{\tfrac{n^2}{4}-\lambda} +\tfrac{2}{3}N)$ . 
 $$\mathsf{U}^{\mathbb{F}_{C,+}^{\infty}} = (2\pi)^{-n} \int  L^k_{\mathbb{F}_{C,+}^{\infty}}(a_{\mathbb{F}_{C,+}^{\infty}}) e^{i y \cdot\theta+i\tilde{\phi}} \, d\theta$$
$$\Rightarrow x^{-N} \mathsf{U}^{{\mathbb{F}_{C,+}^{\infty}}} = (2\pi)^{-n}  \int L^k_{{\mathbb{F}_{C,+}^{\infty}}}(a_{\mathbb{F}_{C,+}^{\infty}}) e^{i y \cdot\theta+i\tilde{\phi}} \, d\theta$$
$$\Rightarrow \lvert x^{-N} \mathsf{U}^{\mathbb{F}_{C,+}^{\infty}}\rvert \leq 
(2\pi)^{-n} \int \lvert x^{-N} L^k_{\mathbb{F}_{C,+}^{\infty}}(a_{\mathbb{F}_{C,+}^{\infty}})  \rvert \, d\theta
\leq \tilde{C}_N (2\pi)^{-n} \int_{\lvert\theta\rvert\geq 1} \lvert\theta\rvert^{-\tfrac{1}{3}k + \sqrt{\tfrac{n^2}{4}-\lambda}+\tfrac{2}{3}N}\, d\theta$$
with the choice of $k$ then have
$$ x^{-N} \mathsf{U}^{\mathbb{F}_{C,+}^{\infty}}\in L^{\infty}_{x,y}$$
The same argument goes through if we replace $\mathsf{U}^{\mathbb{F}_{C,+}^{\infty}}$ by $D_y^{\alpha}(xD_x)^N\mathsf{U}^{\mathbb{F}_{C,+}^{\infty}}$. So in fact from here we have
$$x^{-N}D_y^{\alpha} (xD_x)^{\beta} \mathsf{U}^{\mathbb{F}_{C,+}^{\infty}}\in L^{\infty}_{x,y} $$

\end{enumerate}
\end{proof}

\begin{lemma}[Properties of $\mathsf{U}^{\mathbb{F}_{C,+}^{\infty}}$  on $y_n > 0$]\label{y>0C+inf}\

 For every $\beta>0$, there exists $\alpha > 0$ such that $\mathsf{U}^{\mathbb{F}_{C,+}^{\infty}}$ is smooth and vanishes to infinite order at $x = 0$ on $\{y_n >\beta,\ x < \alpha^2\}$. In fact $\alpha$ is given by (\ref{inequalityonalphaC+inf}) explicitly.
 
\end{lemma}
\begin{proof}
 \begin{align*}
    ( \mathsf{U}^{\mathbb{F}_{C,+}^{\infty}} , \chi)
    &= \lim_{\epsilon \rightarrow 0} (2\pi)^{-n} \int\int\int \psi_{\mathbb{F}_{C,+}^{\infty}}  \hat{u}_x \chi(x,y) \varpi(\epsilon\theta)  e^{iy\cdot \theta} dx dy d\theta\\
    &=  \lim_{\epsilon \rightarrow 0} (2\pi)^{-n} \int\int\int a_{\mathbb{F}_{C,+}^{\infty}}\chi(x,y) \varpi(\epsilon\theta)  e^{iy\cdot \theta+i\phi} \, dx dy d\theta
    \end{align*}

 Denote 
 $$\Sigma_0 = \{(x,y) : d_{\theta} (y\cdot \theta+ \phi) = 0 ,\ \text{for some}\  \theta \in \text{conesupp} \ a_{\mathbb{F}_{C,+}^{\infty}} \}$$
For $\beta > 0$, consider $\chi(x,y)$ with $\supp \chi \subset \{y_n > \beta, x < \alpha^2\}$ where 
the relation of $\alpha$ and $\beta$ is to be determined and such that $\supp \chi $ is disjoint from $\Sigma_0$. We show below that if $\alpha$ satisfies (\ref{inequalityonalphaC+inf}) then we have $\lvert d_{\theta} \phi\rvert^2$ is bounded away from zero on region currently considered of $\chi \times \RR^n$.  Once this is achieved, we follow the same argument as those for $y_n  < 0$ in Lemma \ref{y<0C+inf} to show that on the region currently considered of $\chi$, $\mathsf{U}^{\mathbb{F}_{C,+}^{\infty}}$ is smooth and vanishes to infinite order at $x = 0$. \nl In particular, on the support of $\chi$ (described above), we consider the first order differential operator $L$ whose formal adjoint is :
     \begin{equation}
    ^t L_{\mathbb{F}_{C,+}^{\infty}}= \lvert d_{\theta}(y\cdot\theta+\phi))\rvert^{-2} \sum_{j=1}^n \dfrac{\partial (y\cdot\theta+\phi)}{\partial \theta_j} \partial_{\theta_j} \Rightarrow\ \  L_{\mathbb{F}_{C,+}^{\infty}} e^{iy\cdot \theta+i\phi} = e^{iy\cdot\theta+i\phi}
    \end{equation}
\end{proof}
 It remains to show that: \textbf{for $\beta > 0$ if $y_n >\beta$ and if we are close enough the boundary ie $x < \alpha^2$ for some $\alpha$, then $d_{\theta}( y\cdot\theta + \phi)\neq 0$} :
\begin{proof}
  \begin{align*}
  &y_n - \tfrac{2}{3}\partial_{\theta_n} ((Z^{3/2} - Z_0^{3/2})\sgn\theta_n)\\
  &=y_n -\tfrac{2}{3}\left[(1 + x-\lvert\hat{\theta}'\rvert^2)^{3/2} - (1 - \lvert\hat{\theta}'\rvert^2)^{3/2}\right]
   -  \left[  (1 + x-\lvert\hat{\theta}'\rvert^2)^{1/2} 2\lvert\hat{\theta}'\rvert^2    - (1 -\lvert\hat{\theta}'\rvert^2)^{1/2} 2\lvert\hat{\theta}'\rvert^2   \right]
  \end{align*}

Consider the region $\{x < \alpha^2\}$, where $\alpha$ is to be determined.
 
Denote by $G(x)$ the following expression:
 $$G(x) = \tfrac{2}{3}\left[(1 + x-\lvert\hat{\theta}'\rvert^2)^{3/2} - (1 - \lvert\hat{\theta}'\rvert^2)^{3/2}\right]
   + \left[  (1 + x-\lvert\hat{\theta}'\rvert^2)^{1/2} 2\lvert\hat{\theta}'\rvert^2    - (1 -\lvert\hat{\theta}'\rvert^2)^{1/2} 2\lvert\hat{\theta}'\rvert^2 \right]$$

On the region currently considered  we have $\sigma = x^{3/2} (1 - \lvert\hat{\theta}'\rvert^2)^{-3/2}> 1$ so $1 - \lvert\hat{\theta}'\rvert^2 \leq x$. Since $x < \alpha^2$,  $0 < 1 - \lvert\hat{\theta}'\rvert^2 < \alpha^2$. Use this to estimate $\lvert G(x)\rvert$:
$$\lvert G(x) \rvert \leq \dfrac{2}{3}\alpha^3 (2^{3/2} + 1) + 2 \alpha (2^{1/2} + 1) $$
Choose $\alpha < 1$ such that 
\begin{equation}\label{inequalityonalphaC+inf}
\dfrac{2}{3}\alpha^3 (2^{3/2} + 1) + 2 \alpha (2^{1/2} + 1)< \dfrac{1}{2}\beta 
\end{equation}
$$ \Rightarrow y - G(x) > y - \lvert G(x) \rvert > 0$$

\end{proof}

As a corollary we have :
\begin{lemma}\label{SisuC+inf}
  $\text{Singsupp} \ \mathsf{U}^{\mathbb{F}_{C,+}^{\infty}}$ is contained in the union of the forward bicharacteristic from $x=y=0$
\end{lemma}
 
\begin{lemma}[Wavefront set statement]\label{WFC+inf}
   The wavefront set of $\mathsf{U}^{\mathbb{F}_{C,+}^{\infty}}$ is contained in $\Sigma(\delta)$ , where
   $$\Sigma (\delta)= \Big\{(x,y,\xi,\eta) : \xi = \partial_x \left(y\cdot\theta - \phi \right), \eta = d_y \left(y\cdot\theta-\phi\right) , d_{\theta}\left(y\cdot\theta -\phi\right)  = 0 , x > \delta , \lvert\theta_n\rvert \geq \lvert\theta'\rvert > 0 \Big\}$$
   $$\phi = \tfrac{2}{3}( Z^{3/2} - Z_0^{3/2})\sgn\theta_n. $$

\end{lemma}

\begin{proof}
 By Lemma  \ref{F+inf} on $x >\delta$, $\hat{\mathsf{U}}^{\mathbb{F}_{C,+}^{\infty}} \in S^{-\tfrac{1}{3}\tilde{a}+ \sqrt{\tfrac{n^2}{4}-\lambda}}_{\tfrac{1}{3},0; \theta}$. Compared to Lemma 5.2 in Friedlander paper, we have the same type of symbol, the same phase function on the same coneregion currently considered, as well as the same result on the singular region currently considered of $\mathsf{U}^{\mathbb{F}_{C,+}^{\infty}}$. Hence we have the same conclusion on $\text{WF} \,\mathsf{U}^{\mathbb{F}_{C,+}^{\infty}}$. 
\end{proof}

\subsubsection{On region $\mathbb{F}_{C,+}^{0,\infty}$}
  \begin{align*}
    ( \mathsf{U}^{\mathbb{F}_{C,+}^{0,\infty}} , \chi)
    &= \lim_{\epsilon \rightarrow 0} (2\pi)^{-n} \int\int\int  \psi_{\mathbb{F}_{C,+}^{0,\infty}}  \hat{u}_x \chi(x,y) \varpi(\epsilon\theta)  e^{iy\cdot \theta} dx dy d\theta\\
    &=  \lim_{\epsilon \rightarrow 0} (2\pi)^{-n} \int\int\int  a_{\mathbb{F}_{C,+}^{0,\infty}}\chi(x,y) \varpi(\epsilon\theta)  e^{iy\cdot \theta+i\phi} \, dx dy d\theta\\ 
    \end{align*}

\begin{lemma}\indent\par\noindent
  \begin{enumerate}
  \item $L_{\mathbb{F}_{C,+}^{0,\infty}}$ whose formal adjoint is $   ^t L_{\mathbb{F}_{C,+}^{0,\infty}} = \lvert d_{\theta}( y\cdot + \phi)\rvert^{-2} \sum_{j=1}^n \tfrac{\partial(y\cdot + \phi)}{\partial \theta_j} \partial_{\theta_j}$
 then
 $$ (xD_x)^l L_{\mathbb{F}_{C,+}^{0,\infty}}^k a_{\mathbb{F}_{C,+}^{0,\infty}} \leq\tilde{C}_k\lvert\theta\rvert^{-\tfrac{1}{3}k-\tfrac{2}{3} s_-}$$
 
 \item On $x >\delta$, $a_{\mathbb{F}_{C,+}^{0,\infty}} \in S^{-\tfrac{2}{3} s_-}_{\tfrac{1}{3},0; \theta}$.
\end{enumerate}

\end{lemma}

\begin{proof}\indent\par\noindent
\begin{enumerate}
\item Recall the structure of the approximate solutions:
$$\psi_{\mathbb{F}_{C,+}^{0,\infty}} \hat{u}_x =  \big(\lvert\theta_n\rvert (1 - \lvert\hat{\theta}'\rvert^2)^{1/2}\big)^{\sqrt{\tfrac{n^2}{4}-\lambda}} \hat{\mathsf{U}}^{\mathbb{F}_{C,+}^{0,\infty}} $$
$$\hat{\mathsf{U}}^{\mathbb{F}_{C,+}^{0,\infty}} = \exp(i\phi) \nu^{\tilde{\gamma}_{\mathfrak{B}_1} +\tfrac{1}{2}} \sigma^{\tilde{\gamma}_{\mathfrak{B}_1}}\hat{\mathsf{V}}^{\mathbb{F}_{C,+}^{0,\infty}};\ \ \tilde{\gamma}_{\mathfrak{B}_1} =0 $$
 where $\hat{\mathsf{V}}^{\mathbb{F}_{C,+}^{0,\infty}}$ is conormal in $z$ and $s (= \tfrac{h}{z}) $. In fact it is polyhomogeneous with index set $\{j\}$. On the region currently considered  $z$ and $s$ are bounded.\newline
  On the region currently being considered, 
  $$ \lvert a_{\mathbb{F}_{C,+}^{0,\infty}} \rvert \leq \lvert\theta_n\rvert^{ \sqrt{\tfrac{n^2}{4}-\lambda}} $$

  To compute the bound on higher derivatives, we use the following computation and the conormality of $\hat{\mathsf{V}}^{\mathbb{F}_{C,+}^{0,\infty}}$.

\begin{align*}
\theta_n \partial_{\theta_n}
 &= -\left(1+ s^{\tfrac{2}{3}}z^{\tfrac{2}{3}} \lvert\theta_n\rvert^{\tfrac{2}{3}} \lvert\hat{\theta}'\rvert^2 \right)s\partial_s + 2 s^{\tfrac{2}{3}}z^{\tfrac{2}{3}} \lvert\theta_n\rvert^{\tfrac{2}{3}} \lvert\hat{\theta}'\rvert^2 z\partial_z\\
\lvert\theta_n\rvert \partial_{\theta_j} 
&= - s^{\tfrac{2}{3}}z^{\tfrac{2}{3}} \lvert\theta_n\rvert^{\tfrac{2}{3}} \lvert\hat{\theta}'\rvert s\partial_s
+ 2 s^{\tfrac{2}{3}}z^{\tfrac{2}{3}} \hat{\theta_j} \lvert\theta_n\rvert^{\tfrac{2}{3}}z\partial_z\\
\partial_x 
&
=  s^{9/4} \lvert\theta_n\rvert(1 - \lvert\hat{\theta}'\rvert^2)^{\tfrac{1}{2}} \left(-s\partial_s+z\partial_z\right)
\end{align*}
to conclude that 
$$ (xD_x)^l L_{\mathbb{F}_{C,+}^{0,\infty}}^k a_{\mathbb{F}_{C,+}^{0,\infty}} \leq\tilde{C}_k\lvert\theta\rvert^{-\tfrac{1}{3}k +\sqrt{\tfrac{n^2}{4}-\lambda}}$$

\item We use the same computation and the fact that $\partial_x = -x^{-1}s\partial_s + x^{-1} z\partial_z$, to conclude that on $x > \delta$:
$$a_{\mathbb{F}_{C,+}^{0,\infty}} \in S^{\sqrt{\tfrac{n^2}{4}-\lambda}}_{\tfrac{1}{3},0; \theta}$$

\end{enumerate}

\end{proof}

 Consider $\chi$ smooth in $x,y$ with $\supp \chi \subset \{y_n < 0\}$
 \begin{align*}
    ( \mathsf{U}^{\mathbb{F}_{C,+}^{0,\infty}} , \chi)
    &= \lim_{\epsilon \rightarrow 0} (2\pi)^{-n} \int\int\int  \psi_{\mathbb{F}_{C,+}^{0,\infty}}  \hat{u}_x \chi(x,y) \varpi(\epsilon\theta)  e^{iy\cdot \theta} dx dy d\theta\\
    &=  \lim_{\epsilon \rightarrow 0} (2\pi)^{-n} \int\int\int  a_{\mathbb{F}_{C,+}^{0,\infty}}\chi(x,y) \varpi(\epsilon\theta)  e^{iy\cdot \theta+i\phi} \, dx dy d\theta\\ 
    \end{align*}

 Since we have the exact same phase function, same result on the properties of $a_{\mathbb{F}_{C,+}^{0,\infty}}$ (compared to $a_{\mathbb{F}_{C,+}^{\infty}}$) and  $$x^{-1} = s \lvert\theta_n\rvert (1 - \lvert\hat{\theta}'\rvert^2)^{1/2} \leq \text{constant} \ \lvert\theta_n\rvert$$ we use the same method as in Lemma \ref{y<0C+inf}, \ref{y>0C+inf}, \ref{SisuC+inf} and \ref{WFC+inf} to get the following results:
 \begin{lemma}
  On $y_n <0$ $\mathsf{U}^{\mathbb{F}_{C,+}^{0,\infty}}$ is smooth and vanishes to infinite order at $x = 0$. 
\end{lemma} 
 
\begin{lemma}
  $\text{Sing supp}\  \mathsf{U}^{\mathbb{F}_{C,+}^{0,\infty}}$ is contained in the union of the forward bicharacteristic from $x=y=0$. 
\end{lemma}

\begin{lemma}
For every $\beta>0$, there exists $\alpha > 0$ such that  $\mathsf{U}^{\mathbb{F}_{C,+}^{0,\infty}}$ is smooth and vanishes to infinite order on $y_n >\beta$ and $x < \alpha^2$. This means that we have smoothness  for $y_n > 0$ if we are close enough to the boundary, ie we have conormality for $y_n  > 0$. 
\end{lemma}

\begin{lemma}\label{wfC+0inf}
   The wavefront set of $\mathsf{U}^{\mathbb{F}_{C,+}^{0,\infty}}$ is contained in $\Sigma(\delta)$, where
   $$\Sigma (\delta)= \Big\{(x,y,\xi,\eta) : \xi = \partial_x \left(y\cdot\theta - \phi \right), \eta = d_y \left(y\cdot\theta-\phi\right) , d_{\theta}\left(y\cdot\theta -\phi\right)  = 0 , x > \delta , \lvert\theta_n\rvert \geq \lvert\theta'\rvert > 0 \Big\}$$
   $$\phi = \tfrac{2}{3}( Z^{3/2} - Z_0^{3/2})\sgn\theta_n. $$
\end{lemma}


\subsection{On region $\mathbb{F}_{C,-,1}$}
 On the current region $\{Z_0 < -\delta_2 ,\ \lvert \theta_n\rvert \geq (1 - 2\delta_1) \lvert\theta'\rvert\} $ so $\lvert\theta_n\rvert \sim \lvert\theta\rvert$, hence it suffices to consider and compare with the growth of $\lvert\theta_n\rvert$.
We will show:
\begin{enumerate}
\item On $\mathbb{F}_{C,-,1}^{0,0}$: $\mathsf{U}^{\mathbb{F}_{C,-,1}^{0,0}}$ is smooth on $x >\delta$ for any $\delta > 0$ while $\mathsf{U}^{\mathbb{F}_{C,-,1}^{0,0}}\in x^{s_-} \mathcal{C}^{\infty}_{x,y} + x^{s_+} \mathcal{C}^{\infty}_{x,y}$ on $y_n \neq 0$.
\item On $\mathbb{F}_{C,-,1}^{\infty}$
 \begin{enumerate}
 \item $\mathsf{U}^{\mathbb{F}_{C,-,1}^{\infty}}$ is smooth  and vanishes to infinite order at $x = 0$ for $y_n < 0$.
 \item smooth and vanishes to infinite order at $x =0$  for $y_n > 0$ if we are close enough to the boundary.
 \item Wavefront set statement away from $x = 0$ .
 \end{enumerate}
 \item On $\mathbb{F}_{C,-,1}^{0,\infty}$: $\mathsf{U}^{\mathbb{F}_{C,-,1}^{0,\infty}}$ is smooth on $x >\delta$ for any $\delta > 0$. $\mathsf{U}^{\mathbb{F}_{C,-,1}^{0,\infty}}$ is smooth and vanishes to infinite order at $x = 0$ for $y_n \neq 0$.
\end{enumerate}

\subsubsection{Subregion $\mathbb{F}_{C,-,1}^{0,0}$}\
         
    \begin{align}\label{indicialFC-10}
    ( \mathsf{U}^{\mathbb{F}_{C,-,1}^{0,0}} , \chi)
   & = \lim_{\epsilon \rightarrow 0} (2\pi)^{-n} \int\int\int  \psi_{\mathbb{F}_{C,-,1}^{0,0}} \hat{\mathsf{U}}^{\mathbb{F}_{C,-,1}^{0,0}} \chi(x,y) \varpi(\epsilon\theta)  e^{iy\cdot \theta} \,dx\, dy \, d\theta 
   \end{align}
  We recall the structure of $\hat{u}_x$ on the region currently considered:
         $$\psi_{\mathbb{F}_{C,-,1}^{0,0}} \hat{u}_x=  \big[\lvert\theta_n\rvert (\lvert\hat{\theta}'\rvert^2-1)^{1/2}\big]^{\sqrt{\tfrac{n^2}{4}-\lambda}} \hat{\mathsf{U}}^{\mathbb{F}_{C,-,1}^{0,0}} $$ 
      where  $ \hat{\mathsf{U}}^{\mathbb{F}_{C,-,1}^{0,0}}$ is smooth in $t$ and smooth in $h$. Here $t = x\lvert\theta_n\rvert (\lvert\hat{\theta}'\rvert^2-1)^{1/2}, h = \lvert\theta_n\rvert^{-1}  (\lvert\hat{\theta}'\rvert^2-1)^{1/2}$. In more details:
\begin{align*}
  2\sqrt{\tfrac{n^2}{4}-\lambda} \notin \mathbb{Z} :&\  \hat{\mathsf{U}}^{\mathbb{F}_{C,-,1}^{0,0}} \in z^{\tfrac{n}{2}} t^{\sqrt{\tfrac{n^2}{4}-\lambda}} \mathcal{C}^{\infty}_{t,h} + z^{\tfrac{n}{2}} t^{-\sqrt{\tfrac{n^2}{4}-\lambda}} \mathcal{C}^{\infty}_{t,h}
  \end{align*}
  $$\big[\lvert\theta_n\rvert (\lvert\hat{\theta}'\rvert^2-1)^{1/2}\big]^{\sqrt{\tfrac{n^2}{4}-\lambda}} \hat{\mathsf{U}}^{\mathbb{F}_{C,-,1}^{0,0}} \in  \big[\lvert\theta_n\rvert (\lvert\hat{\theta}'\rvert^2-1)^{1/2} \big]^{2\sqrt{\tfrac{n^2}{4}-\lambda}} x^{s_+} \mathcal{C}^{\infty}_{t,h} + x^{s_-} \mathcal{C}^{\infty}_{t,h}$$
    Using this form of $\hat{u}_x$ we rewrite (\ref{indicialFC-10}) as
 $$ ( \mathsf{U}^{\mathbb{F}_{C,-,1}^{0,0}} , \chi)
   = \lim_{\epsilon \rightarrow 0} \sum_{\pm} (2\pi)^{-n} x^{s_{\pm}} \int\int\int  a_{\mathbb{F}_{C,-,1}^{0,0};\pm} \chi(x,y) \varpi(\epsilon\theta)  e^{iy\cdot \theta} \,dx\, dy\, d\theta$$
     
     \begin{lemma}\label{symbolC-100}\indent\par\noindent
     \begin{enumerate}
     \item   $L$ whose formal adjoint is $   ^t L= \lvert y \rvert^{-2} \sum_{j=1}^n \tfrac{\partial \bar{\phi}}{\partial \theta_j} \partial_{\theta_j}$
     $$D_x^l L^k a_{\mathbb{F}_{C,-,1}^{0,0};\pm}  \leq\tilde{C}_k\lvert\theta\rvert^{-\tfrac{1}{3}k+ 2 \sqrt{\tfrac{n^2}{4}-\lambda}+l}.$$
     \item On $x > \delta$, $a_{\mathbb{F}_{C,-,1}^{0,0}} = \sum_{\pm} x^{s_{\pm}}a_{\mathbb{F}_{C,-,1}^{0,0};\pm}  \in S^{-\infty}_{\tfrac{1}{3},0;\theta}$.
     \end{enumerate}
    
        \end{lemma}
       \begin{proof}\indent\par\noindent
       \begin{enumerate}
       \item First we consider the bound on $a_{\mathbb{F}_{C,-,1}^{0,0};\pm}$
                $$\psi_{\mathbb{F}_{C,-,1}^{0,0}} \hat{u}_x=  \big(\lvert\theta_n\rvert (\lvert\hat{\theta}'\rvert^2-1)^{1/2}\big)^{\sqrt{\tfrac{n^2}{4}-\lambda}} \hat{\mathsf{U}}^{\mathbb{F}_{C,-,1}^{0,0}} $$ 
  $$\big[\lvert\theta_n\rvert (\lvert\hat{\theta}'\rvert^2-1)^{1/2}\big]^{\sqrt{\tfrac{n^2}{4}-\lambda}} \hat{\mathsf{U}}^{\mathbb{F}_{C,-,1}^{0,0}} \in  \big[\lvert\theta_n\rvert (\lvert\hat{\theta}'\rvert^2-1)^{1/2} \big]^{2\sqrt{\tfrac{n^2}{4}-\lambda}} x^{s_+} \mathcal{C}^{\infty}_{t,h} + x^{s_-} \mathcal{C}^{\infty}_{t,h}$$
  In the region currently considered on which $h, t$ and $\lvert\hat{\theta}'\rvert$ are bounded and since $s_- \geq 0$, so 
$$a_{\mathbb{F}_{C,-,1}^{0,0};\pm}  \lesssim  \lvert \theta_n\rvert^{2 \sqrt{\tfrac{n^2}{4}-\lambda}} $$
To compute the bound for higher derivatives we use the following computations:
    \begin{align*}
  \lvert \theta_n\rvert \partial_{\theta_n} &= \left(-1 -3(-Z_0)^{-1}\lvert\theta_n\rvert^{2/3} \lvert\hat{\theta}'\rvert^2\right)h \partial_h + \left(1 -  (-Z_0)^{-1} \lvert\theta_n\rvert^{2/3}\lvert\hat{\theta}'\rvert^2\right) t\partial_t\\
  \lvert\theta_n\rvert \partial_{\theta_j} &= 3  (-Z_0)^{-1} \lvert\theta_n\rvert^{-1/3} \lvert\hat{\theta}'\rvert  h\partial_h +  (-Z_0)^{-1} \lvert\theta_n\rvert^{-1/3}\lvert\hat{\theta}_j\rvert t\partial_t\\
   x\partial_x &= t\partial_t  ;\ \partial_x = \lvert\theta_n\rvert( \lvert\hat{\theta}'\rvert^2-1)^{1/2} \partial_t
  \end{align*}
Using the fact that $a_{\mathbb{F}_{C,-,1}^{0,0};\pm}$ is smooth in $t,h$, on the current region $t,h$ and $\lvert\hat{\theta}'\rvert$ are bounded, and the computations above to get:
$$D_x^l L^k a_{\mathbb{F}_{C,-,1}^{0,0}} \leq\tilde{C}_k\lvert\theta\rvert^{-\tfrac{1}{3}k+ 2 \sqrt{\tfrac{n^2}{4}-\lambda} +l }$$


  \
  \
  
  \item We have $\sigma_2(Z_0) , \ \sigma_0(\lvert\theta\rvert) \in S^0_{\tfrac{1}{3},0}$
 . From the prep calculation, on $x\neq 0$, we have
 $$a_{\mathbb{F}_{C,-,1}^{0,0}} \in S^{2 \sqrt{\tfrac{n^2}{4}-\lambda} }_{\tfrac{1}{3},1;\theta}$$
   
  In fact we can improve the order of the symbol 
   $\lvert\theta_n\rvert = x^{-\tfrac{3}{2}} t^{3/2} h^{1/2}$. Thus on $x >\delta$ for all $N$
   $$ \lvert\theta_n\rvert^N \hat{\mathsf{U}}^{\mathbb{F}_{C,-,1}^{0,0}}
    = x^{-\tfrac{3}{2}N} t^{\tfrac{3}{2}N} h^{\tfrac{1}{2}N} \leq C_N$$

    From the previous remark and using the fact that $\lvert\theta_n\rvert \sim \lvert\theta'\rvert$,  on $x > \delta$ we can lower the order of the symbol ie $a_{\mathbb{F}_{C,-,1}^{0,0}} \in S^{-\infty}_{\tfrac{1}{3},0;\theta}$.

  \end{enumerate}
  \end{proof}

\begin{lemma}
    $\mathsf{U}^{\mathbb{F}_{C,-,1}^{0,0}}$ is smooth on $x >\delta$ for any $\delta > 0$.
\end{lemma}

 \begin{proof}
 
 Denote $\mathsf{U}^{\mathbb{F}_{C,-,1}^{0,0}}_{\delta}$ the restrictions of $\mathsf{U}^{\mathbb{F}_{C,-,1}^{0,0}}$ to $\mathcal{D}'(X_{\delta})$ .For $x >\delta$ we have :
   \begin{align*}
    ( \mathsf{U}^{\mathbb{F}_{C,-,1}^{0,0}}_{\delta} , \chi)
    &= \lim_{\epsilon \rightarrow 0} (2\pi)^{-n} \int\int\int \psi_{\mathbb{F}_{C,-,1}^{0,0}} \hat{u}_x \chi(x,y) \varpi(\epsilon\theta)  e^{iy\cdot \theta} \,dx\, dy\, d\theta\\
     &= \lim_{\epsilon \rightarrow 0} (2\pi)^{-n} \int\int\int  a_{\mathbb{F}_{C,-,1}^{0,0}} \chi(x,y) \varpi(\epsilon\theta)  e^{iy\cdot \theta} \,dx\, dy\, d\theta
    \end{align*}
  where $a_{\mathbb{F}_{C,-,1}^{0,0}} \in S^{-\infty}_{\tfrac{1}{3},0}$ for $x >\delta$.  The proof is the same as that for Lemma \ref{x>deltaC+00}.
   \end{proof}

\begin{lemma}\label{conormalFC-1}
     $\mathsf{U}^{\mathbb{F}_{C,-,1}^{0,0}}\in x^{s_-} \mathcal{C}^{\infty}_{x,y} + x^{s_+} \mathcal{C}^{\infty}_{x,y}$ on  $y_n \neq 0$
\end{lemma}
  \begin{proof}
  Same proof as that Lemma \ref{conormalC+00}  \end{proof}
\
\

\subsubsection{Region $\mathbb{F}_{C,-,1}^{\infty}$}\

   On the region currently considered  we have: $\sigma > 1$ (so $Z >0$), $Z_0 \leq-\delta_2$, $1<\hat{\theta}'<(1-2\delta_1)^{-1}$
   $$\psi_{\mathbb{F}_{C,-,1}^{\infty}} \hat{u}_x =  \big(\lvert\theta_n\rvert (\lvert\hat{\theta}'\rvert^2-1)^{1/2}\big)^{\sqrt{\tfrac{n^2}{4}-\lambda}} \hat{\mathsf{U}}^{\mathbb{F}_{C,-,1}^{\infty}}  $$
   $$\hat{\mathsf{U}}^{\mathbb{F}_{C,-,1}^{\infty}} =  \exp(-i\tfrac{2}{3}Z^{3/2}\sgn\theta_n) \exp(-\tfrac{2}{3}(-Z_0)^{3/2}) h^N \sigma^{-\tfrac{1}{2}} \hat{\mathsf{V}}^{\mathbb{F}_{C,-,1}^{\infty}}_N;\ \  \text{for any}\  N$$
 where $\hat{\mathsf{V}}^{\mathbb{F}_{C,-,1}^{\infty}}$ is smooth in $\sigma^{-2/3}$ and $h$. Denote by $\tilde{\phi} = \tfrac{2}{3}Z^{3/2}\sgn\theta_n$.
 \begin{align*}
    ( \mathsf{U}^{\mathbb{F}_{C,-,1}^{\infty}} , \chi)
    &= \lim_{\epsilon \rightarrow 0} (2\pi)^{-n} \int\int\int  \psi_{\mathbb{F}_{C,-,1}^{\infty}}  \hat{u}_x \chi(x,y) \varpi(\epsilon\theta)  e^{iy\cdot \theta} \,dx\, dy\, d\theta\\
    &=  \lim_{\epsilon \rightarrow 0} (2\pi)^{-n} \int\int\int  a_{\mathbb{F}_{C,-,1}^{\infty}}\chi(x,y) \varpi(\epsilon\theta)  e^{iy\cdot \theta+i\tilde{\phi}} \, dx \,dy\, d\theta
    \end{align*}\

 \begin{lemma}\indent\par\noindent
 \begin{enumerate}
\item  For $L_{\mathbb{F}_{C,-,1}^{\infty}}$ whose formal adjoint is $   ^t L_{\mathbb{F}_{C,-,1}^{\infty}} = \lvert d_{\theta}(y\cdot\theta + \tilde{\phi}))\rvert^{-2} \sum_{j=1}^n \tfrac{\partial (y\cdot\theta+ \tilde{\phi})}{\partial \theta_j} \partial_{\theta_j}$ we have
$$(x\partial_x)^l L_{\mathbb{F}_{C,-,1}^{\infty}}^k a_{\mathbb{F}_{C,-,1}^{\infty}} \leq C_k \lvert\theta_n\rvert^{-\tfrac{1}{3}k +  \sqrt{\tfrac{n^2}{4}-\lambda}}.$$
 \item  $\exp\left(-\tfrac{2}{3}(-Z_0)^{3/2}\right)\in S^0_{\tfrac{1}{3},0;\theta}$ and $a_{\mathbb{F}_{C,-,1}^{0,\infty}} \in S^{ \sqrt{\tfrac{n^2}{4}-\lambda}}_{\tfrac{1}{3},0; \theta}$.
 \end{enumerate}
 \end{lemma}
 \begin{proof}
 On the region currently considered 
 $\exp(-\tfrac{2}{3}(-Z_0)^{3/2})$ is a smooth function of $(-Z_0)^{3/2}$ whose derivatives are bounded. Using the same argument as in Friedlander lemma 5.1, we have $\exp\left(-\tfrac{2}{3}(-Z_0)^{3/2}\right)\in S^0_{\tfrac{1}{3},0;\theta}$.
The proof of the statement follows from direct computation similar to those in Lemma \ref{F+inf}.
 
\end{proof}

\begin{lemma}
    $\mathsf{U}^{\mathbb{F}_{C,-,1}^{\infty}}$ is smooth and vanishes to infinite order at $x=0$ for $y_n < 0$.
\end{lemma}

 \begin{proof}
  The proof is similar to that for Lemma \ref{y<0C+inf}. The first ingredient we need is the bound of $x^{-1}$ in terms of $\lvert\theta\rvert$ using $x^{-1} = \sigma^{-1} h^{2/3} \lvert\theta_n\rvert^{2/3}$.  The second is the fact that the $n$-th component of the derivative of the phase function is $y_n -$ positive term as in Lemma \ref{y<0C+inf}:
  \begin{align*}
  &y_n - \tfrac{2}{3}\partial_{\theta_n} Z^{3/2} \sgn\theta_n\\
  &= y_n -\tfrac{2}{3}(1 + x-\lvert\hat{\theta}'\rvert^2)^{3/2} 
   -   (1 + x-\lvert\hat{\theta}'\rvert^2)^{1/2} 2\lvert\hat{\theta}'\rvert^2 \theta_n\sgn\theta_n   \\
   &= y_n -  \text{positive term}
  \end{align*}

 \end{proof}

\begin{lemma}\label{conormalFC-1inf}
    For every $\beta>0$, there exists $\alpha > 0$ such that $\mathsf{U}^{\mathbb{F}_{C,-,1}^{\infty}}$ is smooth and vanishes to infinite order on $y_n >\beta$ and $x < \alpha^2$. In fact $\alpha$ is given by (\ref{ineqOnAlphaC-1inf}) explicitly.
 
\end{lemma}

\begin{proof}
 \begin{align*}
    ( \mathsf{U}^{\mathbb{F}_{C,-,1}^{\infty}} , \chi)
    &= \lim_{\epsilon \rightarrow 0} (2\pi)^{-n} \int\int\int  \psi_{\mathbb{F}_{C,-,1}^{\infty}}  \hat{u}_x \chi(x,y) \varpi(\epsilon\theta)  e^{iy\cdot \theta} \,dx\, dy\, d\theta\\
    &=  \lim_{\epsilon \rightarrow 0} (2\pi)^{-n} \int\int\int  a_{\mathbb{F}_{C,-,1}^{\infty}}\chi(x,y) \varpi(\epsilon\theta)  e^{iy\cdot \theta+i\tilde{\phi}} \, dx \,dy\, d\theta
    \end{align*}

 Denote by
 $$\tilde{\Sigma}_0 = \{(x,y) : d_{\theta} (y\cdot \theta+ \tilde{\phi}) = 0 ,\ \text{for some} \ \theta \in \text{conesupp}\  a_{\mathbb{F}_{C,-,1}^{\infty}} \}$$
For $\beta > 0$, consider $\chi(x,y)$ with $\supp \chi \subset \{y_n > \beta, x < \alpha^2\}$ where the relation of $\alpha$ and $\beta$ is to be determined and such that $\supp \chi $ is disjoint from $\tilde{\Sigma}_0$. We show below that if $\alpha$ satisfies (\ref{ineqOnAlphaC-1inf}) then we have $\lvert d_{\theta} (y\cdot\theta +\tilde{\phi})\rvert^2$ is bounded away from zero on region currently considered of $\chi \times \RR^n$.  Once this is achieved, we follow the same argument as those in Lemma \ref{y<0C+inf} to show that on the region currently considered of $\chi$, $\mathsf{U}^{\mathbb{F}_{C,-,1}^{\infty}}$ is smooth. In particular, on the support of $\chi$, we can consider the first order differential operator $L_{\mathbb{F}_{C,-,1}^{\infty}}$ whose formal adjoint is:
$$^t L_{\mathbb{F}_{C,-,1}^{\infty}} = \lvert d_{\theta} (y\cdot\theta+\tilde{\phi})\rvert^{-2} \sum_{j=1}^n \dfrac{\partial (y\cdot\theta+\tilde{\phi})}{\partial \theta_j} \partial_{\theta_j} \Rightarrow L_{\mathbb{F}_{C,-,1}^{\infty}} e^{iy\cdot \theta+i\tilde{\phi}} = e^{iy\cdot\theta+i\tilde{\phi}}$$
 thus $L_{\mathbb{F}_{C,-,1}^{\infty}} e^{iy\cdot \theta+i\tilde{\phi}} = e^{iy\cdot\theta+i\tilde{\phi}}$. 

\textbf{Show that for $y_n >0$ , if we are close enough the boundary then $d_{\theta}( y\cdot\theta + \tilde{\phi})\neq 0$}:
  \begin{proof}
 $$1\leq \lvert\hat{\theta}'^2 \rvert \leq (1 - 2\delta_1)^{-2}
 \Rightarrow   1 - (1 - 2\delta_1)^{-2} \leq 1 - \lvert\hat{\theta}'^2 \rvert^2  \leq 0$$
 $$\Rightarrow x + 1 - (1 - 2\delta_1)^{-2}\leq x + 1 - \lvert\hat{\theta}'^2 \rvert^2  \leq x$$
 
 Since $\sigma > 1$, $ x (\lvert\hat{\theta}'\rvert^2-1)^{-1} > 1$ so $ x + 1 - \lvert\hat{\theta}'^2>0 \Rightarrow 0 <   x + 1 - \lvert\hat{\theta}'\rvert^2  < x$. 
 
 Using this, we obtain a lower bound for:
\begin{align*} 
 d_{\theta_n} (y\cdot\theta + \tilde{\phi}) &= y_n - \tfrac{2}{3} (1 + x - \lvert\hat{\theta}'\rvert^2)^{3/2} - 3(1 + x- \lvert\hat{\theta}'\rvert^2)^{1/2} \lvert\hat{\theta}'\rvert^2 \sgn\theta_n\\
& \geq y_n - \dfrac{2}{3} x^{3/2} -3 x^{1/2} (1 - 2\delta_1)^{-2} \geq y_n - \dfrac{2}{3} \alpha^3 -3 \alpha (1 - 2\delta_1)^{-2}
\end{align*}
Choose $\alpha$ so that
 \begin{equation}\label{ineqOnAlphaC-1inf}
 \dfrac{2}{3} \alpha^3 +3 \alpha (1 - 2\delta_1)^{-2}<\dfrac{1}{2}\beta
 \end{equation}
 \end{proof}
\end{proof}

\begin{lemma}\label{wfC-1inf}
 Wavefront set of  $\mathsf{U}^{\mathbb{F}_{C,-,1}^{\infty}}$ is contained in
$$
  \Sigma(\delta_1,\delta) = 
   \Big\{(x,y,\xi,\eta) : \xi = \partial_x \left( y\cdot\theta -\tilde{\phi}\right) , \eta = d_y \left( y\cdot\theta -\tilde{\phi}\right) , d_{\theta} \left( y\cdot \theta-\tilde{\phi}\right) = 0 , \theta \in \text{conesupp}\, a_{\mathbb{F}_{C,-,1}^{\infty}} , x> \delta \Big\}  $$
   where $\tilde{\phi} = \tfrac{2}{3} Z^{3/2} \sgn\theta_n$ and 
$$\text{conesupp}\  a_{\mathbb{F}_{C,-,1}^{\infty}} = \{\lvert\theta\rvert > 0 ,  (1-2\delta_1) \lvert\theta'\rvert\leq \lvert\theta_n\rvert \leq \lvert\theta'\rvert \} $$
\end{lemma}
\begin{proof}

The proof follows the same as that for Lemma 6.1 in Friedlander
\end{proof}

\subsubsection{Region $\mathbb{F}_{C,-,1}^{0,\infty}$}
 Recall the structure of the approximate solution on this region : 
  $$\psi_{\mathbb{F}_{C,-,1}^{0,\infty}} \hat{u}_x= \big(\lvert\theta_n\rvert (\lvert\hat{\theta}'\rvert^2-1)^{1/2}\big)^{\sqrt{\tfrac{n^2}{4}-\lambda}} \hat{\mathsf{U}}^{\mathbb{F}_{C,-,1}^{0,\infty}} $$
  $$\hat{\mathsf{U}}^{\mathbb{F}_{C,-,1}^{0,\infty}} = \exp(\tilde{\phi})s^{1/2} z^{\tilde{\gamma}_{\mathfrak{B}_1}} \tilde{\mathsf{U}}^{\mathbb{F}_{C,-,1}^{0,\infty}};\ \tilde{\gamma}_{\mathfrak{B}_1} = 0$$
  where $\tilde{\phi}  = -\dfrac{2}{3}\nu^{-2/3}\sigma^{-2/3} \left[ 1 - (1 - \sigma^{2/3})^{3/2}\right]$ and $\hat{\mathsf{U}}^{\mathbb{F}_{C,-,1}^{0,\infty}}$ is smooth in $z$ and $s$. 
 \begin{align*}
   ( \mathsf{U}^{\mathbb{F}_{C,-,1}^{0,\infty}}_{\delta} , \chi)
    &= \lim_{\epsilon \rightarrow 0} (2\pi)^{-n} \int\int\int \psi_{\mathbb{F}_{C,-,1}^{0,0}} \hat{u}_x \chi(x,y) \varpi(\epsilon\theta)  e^{iy\cdot \theta} \,dx\, dy\, d\theta\\
     &= \lim_{\epsilon \rightarrow 0} (2\pi)^{-n} \int\int\int  a_{\mathbb{F}_{C,-,1}^{0,\infty}} \chi(x,y) \varpi(\epsilon\theta)  e^{iy\cdot \theta} \,dx\, dy\, d\theta
    \end{align*}

   \begin{lemma}\label{upperbound}
     $$e^{- \lvert\theta_n\rvert ( \lvert\hat{\theta}'\rvert^2 - 1)^{3/2} \left(1 - \left[1 -z\right]^{3/2}\right)}   \leq e^{- \lvert\theta_n\rvert x^{3/2} g(x)}  $$
  with $g\geq 0$, $x^{3/2} g(x) \neq 0$ when $x \neq 0$.
   \end{lemma}
   \begin{proof}
  We look for a lower bound for $1 - ( 1 - z)^{3/2}$. We have 
  $$(\lvert\hat{\theta}'\rvert^2  - 1)^{-1} \geq \left[(1 - 2\delta_1)^{-2} - 1\right]^{-3/2}$$ 
  Plugging this into the expression for $\sigma$ we get
  $$z = x ( \lvert\hat{\theta}'\rvert^2 - 1)
  > x \left[(1 - 2\delta_1)^{-2} - 1\right]^{-3/2}$$
 $$\Rightarrow \left[1 -z\right]^{3/2} \leq \left(1- x \left[(1 - 2\delta_1)^{-2} - 1\right]^{-1}\right)^{3/2}$$
  $$\Rightarrow 1 - \left[1 -z\right]^{3/2} \geq 1 - \left[1- x \left[(1 - 2\delta_1)^{-2} - 1\right]^{-1}\right]^{3/2} = g(x) \geq 0$$
    Note that $g(x) \neq 0$ when $x\neq 0$. 

  In the region currently being considered, we have $z < 1$, so we get a lower bound for $(\lvert\hat{\theta}'^2\rvert^2-1)$:
  $$x^{3/2} ( \lvert\hat{\theta}'\rvert^2 - 1)^{-3/2} < 1 \Rightarrow  x^{3/2} <( \lvert\hat{\theta}'\rvert^2 - 1)^{3/2}$$
  
  Using these two lower bounds we obtain:
  $$e^{- \lvert\theta_n\rvert ( \lvert\hat{\theta}'\rvert^2 - 1)^{3/2} \left(1 - \left[1 -z\right]^{3/2}\right)}
  \leq e^{- \lvert\theta_n\rvert ( \lvert\hat{\theta}'\rvert^2 - 1)^{3/2} h(x)}
    \leq e^{- \lvert\theta_n\rvert x^{3/2} g(x)}  $$
  with $x^{3/2} g(x) \neq 0$ when $x \neq 0$
\end{proof}

\begin{lemma}\label{symbolC-10inf}\indent\par\noindent
\begin{enumerate}
\item  For $L$ such that $^t L = \dfrac{1}{\lvert y\rvert^2} \sum y_j \dfrac{\partial}{i\partial \theta_j}$ we have: 
$$\lvert xD_x L^k a_{\mathbb{F}_{C,-,1}^{0,\infty}} \rvert  \leq C_k \lvert\theta\rvert^{\tfrac{1}{3}k+  \sqrt{\tfrac{n^2}{4}-\lambda}}$$
\item On $x\neq 0$, $a_{\mathbb{F}_{C,-,1}^{0,\infty}} \in S^{-\infty}_{\tfrac{1}{3},0}$.
  \end{enumerate}

   \end{lemma}
   
   \begin{proof}\indent\par\noindent
   \begin{itemize}
   \item 
   On the region currently considered  $\lvert\theta_n\rvert \geq (1 -2 \delta_1) \lvert\theta'\rvert$ so $\lvert\theta_n\rvert \sim \lvert\theta\rvert$. We have :
   $$ a_{\mathbb{F}_{C,-,1}^{0,\infty}} \leq \lvert\theta_n\rvert^{\sqrt{\tfrac{n^2}{4}-\lambda}} $$
   Similarly to the proof of Lemma \ref{symbolC-100} we have: 
   $$\lvert xD_x L^k a_{\mathbb{F}_{C,-,1}^{0,\infty}} \rvert  \leq C_k \lvert\theta\rvert^{\tfrac{1}{3}k-\tfrac{2}{3} s_-}$$
 \item  From the Lemma \ref{upperbound} on $x\neq 0$ we have rapidly decreasing in order of $\lvert\theta_n\rvert $ given by $e^{- \lvert\theta_n\rvert ( \lvert\hat{\theta}'\rvert^2 - 1)^{3/2} \left(1 - \left[1 -\sigma^{2/3}\right]^{3/2}\right)}$  so for all $N$ we have:
$$ \lvert\theta_n\rvert^N a_{\mathbb{F}_{C,-,1}^{0,\infty}}  \leq C_N$$
and the same holds for derivatives in $\theta_j$ and $\partial_x$. Hence on $x\neq 0$ 
$$a_{\mathbb{F}_{C,-,1}^{0,\infty}} \in S^{-\infty}_{\tfrac{1}{3},0}$$

\end{itemize}
   \end{proof}

\begin{lemma}
  $\mathsf{U}^{\mathbb{F}_{C,-,1}^{0,\infty}} $ is smooth on $x\neq 0$.
\end{lemma}
  \begin{proof}
  Denote $\mathsf{U}^{\mathbb{F}_{C,-,1}^{0,\infty}}_{\delta}$ the restrictions of $\mathsf{U}^{\mathbb{F}_{C,-,1}^{0,\infty}}$ to $\mathcal{D}'(X_{\delta})$. For $x >\delta$ we have :
  \begin{align*}
   ( \mathsf{U}^{\mathbb{F}_{C,-,1}^{0,\infty}}_{\delta} , \chi)
    &= \lim_{\epsilon \rightarrow 0} (2\pi)^{-n} \int\int\int \psi_{\mathbb{F}_{C,-,1}^{0,0}} \hat{u}_x \chi(x,y) \varpi(\epsilon\theta)  e^{iy\cdot \theta} \,dx\, dy\, d\theta\\
     &= \lim_{\epsilon \rightarrow 0} (2\pi)^{-n} \int\int\int  a_{\mathbb{F}_{C,-,1}^{0,\infty}} \chi(x,y) \varpi(\epsilon\theta)  e^{iy\cdot \theta} \,dx\, dy\, d\theta
    \end{align*}
  where $a_{\mathbb{F}_{C,-,1}^{0,\infty}} \in S^{-\infty}_{\tfrac{1}{3},0}$ for $x >\delta$ by Lemma \ref{symbolC-10inf}. The rest of the proof is the same as that for Lemma \ref{x>deltaC+00}.
\end{proof}

\begin{lemma}\label{C-10inf}
  $\mathsf{U}^{\mathbb{F}_{C,-,1}^{0,\infty}} $ is smooth and vanishes to infinite order at $x=0$ for $y_n \neq 0$. 
\end{lemma}

 \begin{proof}
  
 Consider $\chi(x,y)$ with either $\supp \chi \subset \{ y_n < 0\}$ or $\supp \chi \subset \{ y_n > 0\}$
 \begin{equation}
 \begin{aligned}
    ( \mathsf{U}^{\mathbb{F}_{C,-,1}^{0,\infty}} , \chi)
    &= \lim_{\epsilon \rightarrow 0} (2\pi)^{-n} \int\int\int \psi_{\mathbb{F}_{C,-,1}^{0,\infty}} \hat{u}_x \chi(x,y) \varpi(\epsilon\theta)  e^{iy\cdot \theta} \, dx\,  dy\,  d\theta\\
    &=  \lim_{\epsilon \rightarrow 0} (2\pi)^{-n} \int\int\int a_{\mathbb{F}_{C,-,1}^{0,\infty}} \varpi \chi e^{iy\cdot\theta} \, dx\,  dy\, d\theta
    \end{aligned}
    \end{equation} 
    
    We have $\lvert d_{\theta} (y\cdot \theta)\rvert^2$ is bounded away from zero on the region currently considered of $\chi \times \RR^n$ since the $n$-th component of the phase function is $y_n$, which is never zero there.\nl We use the first order differential operator $L$ whose formal adjoint is 
    \begin{equation}
    ^t L = \dfrac{1}{\lvert y\rvert^2} \sum y_j \dfrac{\partial}{i\partial \theta_j}
    \end{equation}
 thus $^tL e^{iy\cdot \theta} = e^{iy\cdot\theta}$, and is used for the $k$-fold integration as in Lemma \ref{conormalC+00}
$$( \mathsf{U}^{\mathbb{F}_{C,-,1}^{0,\infty}} , \chi)
      =  \lim_{\epsilon \rightarrow 0} (2\pi)^{-n} \int\int\int L^k (a_{\mathbb{F}_{C,-,1}^{0,\infty}}  \varpi) \chi e^{iy\cdot\theta} \,dx \,dy \, d\theta$$
   To show rapid decay in $x$, we use the boundedness of $x^{-1}$ as follows
  $$x^{-1} = s \lvert\theta_n\rvert ( \lvert\hat{\theta}'\rvert^2- 1)^{1/2}  \leq s \lvert\theta_n\rvert  \left[(1 - 2\delta_1)^{-1} -1 \right]$$
 The rest of the proof follows as in that for Lemma \ref{conormalC+00}, on $y_n < 0$ to show the boundedness for
  $$ x^{-N} D_y^{\alpha} (xD_x)^{\beta} \mathsf{U}^{\mathbb{F}_{C,-,1}^{0,\infty}}.$$

  \end{proof}

\


\subsection{On region $\mathbb{F}_{C,-,2}$}
 On the current region $\{Z_0 < -\delta_2, \lvert\theta_n\rvert \leq (1 - \delta_1) \lvert\theta'\rvert\}$. \newline
  Since $ \lvert\hat{\theta}'\rvert \geq (1 - \delta_1)^{-1}$, $z = x (\lvert\hat{\theta}'\rvert^2 - 1 )^{-1} < x \left(( 1- \delta_1)^2 - 1\right)^{-1}$ ie bounded.

 We will show
 \begin{enumerate}
 \item $\mathsf{U}^{\mathbb{F}_{C,-,2}^{0,0}}$ is smooth on $x\neq 0$ and $ \mathsf{U}^{\mathbb{F}_{C,-,2}^{0,0}}\in x^{s_-} \mathcal{C}^{\infty}_{x,y} + x^{s_+}\mathcal{C}^{\infty}_{x,y}$ on $y_n \neq 0$. 
 \item $\mathsf{U}^{\mathbb{F}_{C,-,2}^{0,\infty}}$ is smooth on $x\neq 0$ and is smooth and vanishes to infinite order at $x = 0 $ for $y_n \neq 0$. 
 \end{enumerate}

\subsubsection{$\mathbb{F}_{C,-,2}^{0,0}$}
\begin{align}\label{indicialFC-2}
   ( \mathsf{U}^{\mathbb{F}_{C,-,2}^{0,0}} , \chi)
    &= \lim_{\epsilon \rightarrow 0} (2\pi)^{-n} \int\int\int  \psi_{\mathbb{F}_{C,-,2}^{0,0}} \hat{u}_x \chi(x,y) \varpi(\epsilon\theta)  e^{iy\cdot \theta} \,dx\,dy\, d\theta 
    \end{align}
    We recall the structure of $\hat{u}_x$ on the region currently considered:
         $$\psi_{\mathbb{F}_{C,-,2}^{0,0}} \hat{u}_x=  \big(\lvert\theta_n\rvert (\lvert\hat{\theta}'\rvert^2-1)^{1/2}\big)^{\sqrt{\tfrac{n^2}{4}-\lambda}} \hat{\mathsf{U}}^{\mathbb{F}_{C,-,2}^{0,0}} $$ 
 where  $ \hat{\mathsf{U}}^{\mathbb{F}_{C,-,2}^{0,0}}$   is conormal in $t$ and smooth in $h$ . In more details:
\begin{align*}
  2\sqrt{\tfrac{n^2}{4}-\lambda} \notin \mathbb{Z} :&\  \hat{\mathsf{U}}^{\mathbb{F}_{C,-,2}^{0,0}} \in z^{\tfrac{n}{2}} t^{\sqrt{\tfrac{n^2}{4}-\lambda}} \mathcal{C}^{\infty}_{t,h} + z^{\tfrac{n}{2}} t^{-\sqrt{\tfrac{n^2}{4}-\lambda}} \mathcal{C}^{\infty}_{t,h}
  \end{align*}
  $$\big[\lvert\theta_n\rvert (\lvert\hat{\theta}'\rvert^2-1)^{1/2}\big]^{\sqrt{\tfrac{n^2}{4}-\lambda}} \hat{\mathsf{U}}^{\mathbb{F}_{C,-,2}^{0,0}} \in  \big[ (\lvert\theta'\rvert^2-\lvert\theta_n\rvert)^{1/2} \big]^{2\sqrt{\tfrac{n^2}{4}-\lambda}} x^{s_+} \mathcal{C}^{\infty}_{t,h} + x^{s_-} \mathcal{C}^{\infty}_{t,h}$$
Using this form of $\hat{u}_x$ we rewrite (\ref{indicialFC-2}) as
$$ ( \mathsf{U}^{\mathbb{F}_{C,-,2}^{0,0}} , \chi)
     =  \lim_{\epsilon \rightarrow 0} \sum_{\pm} (2\pi)^{-n} x^{s_{\pm}} \int\int\int a_{\mathbb{F}_{C,-,2}^{0,0};\pm} \varpi \chi e^{iy\cdot\theta} \, dx\,  dy\, d\theta$$

%
   \begin{lemma}\label{00c-2}\indent\par\noindent
    \begin{enumerate}
 \item Let $L$ be the differential operator whose formal adjoint is $^t L = \dfrac{1}{\lvert y\rvert^2} \sum y_j \dfrac{\partial}{i\partial \theta_j}$, then for $C_k$ is independent of $x$
   $$\lvert D_x^m L^k(a_{\mathbb{F}_{C,-,2}^{0,0};\pm} ) \rvert\leq C_k \lvert\theta\rvert^{-\tfrac{1}{3}k + 2 \sqrt{\tfrac{n^2}{4}-\lambda}+m} $$

 \item  Denote by $a_{\mathbb{F}_{C,-,2}^{0,0}} = \sum_{\pm} x^{s_{\pm}} a_{\mathbb{F}_{C,-,2}^{0,0};\pm}$. On $x > \delta$, for all $N$ and all $\alpha$
 $$\Big| \lvert\theta\rvert^N (xD_x)^m\partial_{\theta}^{\alpha}a_{\mathbb{F}_{C,-,2}^{0,0}} \Big|\leq C_{N,m,\alpha}$$
 \end{enumerate}

   \end{lemma}
   
   \begin{proof}\indent\par\noindent
    \begin{enumerate}
\item  First we consider the bound on $a_{\mathbb{F}_{C,-,2}^{0,0};\pm}$
         $$\psi_{\mathbb{F}_{C,-,2}^{0,0}} \hat{u}_x=  \big(\lvert\theta_n\rvert (\lvert\hat{\theta}'\rvert^2-1)^{1/2}\big)^{\sqrt{\tfrac{n^2}{4}-\lambda}} \hat{\mathsf{U}}^{\mathbb{F}_{C,-,2}^{0,0}} $$ 
  $$\big[\lvert\theta_n\rvert (\lvert\hat{\theta}'\rvert^2-1)^{1/2}\big]^{\sqrt{\tfrac{n^2}{4}-\lambda}} \hat{\mathsf{U}}^{\mathbb{F}_{C,-,2}^{0,0}} \in  \big[ (\lvert\theta'\rvert^2-\lvert\theta_n\rvert)^{1/2} \big]^{2\sqrt{\tfrac{n^2}{4}-\lambda}} x^{s_+} \mathcal{C}^{\infty}_{t,h} + x^{s_-} \mathcal{C}^{\infty}_{t,h}$$
  On the region currently considered, $h, t$ and $\lvert\theta_n\rvert$ are bounded so 
$$ \big[\lvert\theta_n\rvert (\lvert\hat{\theta}'\rvert^2-1)^{1/2}\big]^{\sqrt{\tfrac{n^2}{4}-\lambda}} \hat{\mathsf{U}}^{\mathbb{F}_{C,-,2}^{0,0}} \leq  \lvert \theta'\rvert^{2 \sqrt{\tfrac{n^2}{4}-\lambda}} $$
$$\Rightarrow  \lvert a_{\mathbb{F}_{C,-,2}^{0,0};\pm}\rvert \leq \lvert \theta'\rvert^{2 \sqrt{\tfrac{n^2}{4}-\lambda}} $$
For bound on higher derivatives we follow the same method for example in Lemma \ref{symbolC-100} and get
$$\lvert D_x^m L^k(a_{\mathbb{F}_{C,-,2}^{0,0}} ) \rvert\leq C_k \lvert\theta\rvert^{-\tfrac{1}{3}k + 2 \sqrt{\tfrac{n^2}{4}-\lambda}+m}. $$

\item On the region currently considered, $\lvert\theta'\rvert $ is equivalent to $\lvert\theta\rvert$ so it suffices to compare with $\lvert\theta'\rvert$. 
    $$\lvert\theta'\rvert = x^{-1} t ( x^{-1}ht - 1)^{1/2}$$   
  $$\Rightarrow \lvert\theta'\rvert^N G(t, h)=   x^{-N} t^N ( x^{-1}ht - 1)^{\tfrac{1}{2}N}G(t, h)$$
   Note on the region currently considered  we have $x >\delta ,(-Z_0) > \delta_2$ and $h,t$ bounded so the above expression is bounded for $G$ bounded.  Thus we have $x > \delta$: 
   for all $N$ and all $\alpha$
 $$\Big| \lvert\theta\rvert^N (xD_x)^m\partial_{\theta}^{\alpha}a_{\mathbb{F}_{C,-,2}^{0,0}} \Big|\leq C_{N,m,\alpha}$$

 \end{enumerate}

   \end{proof}
As an immediate result of this we have:
   \begin{lemma}
   $\mathsf{U}^{\mathbb{F}_{C,-,2}^{0,0}}$ is smooth on $x\neq 0$.
   \end{lemma}
   
   \begin{lemma}\label{conormalFC-200}
    $ \mathsf{U}^{\mathbb{F}_{C,-,2}^{0,0}}\in x^{s_-}\mathcal{C}^{\infty}_{x,y} + x^{s_+} \mathcal{C}^{\infty}_{x,y}$  for  $y_n \neq 0$. 
   \end{lemma}
   \begin{proof}
   Same proof as for Lemma \ref{conormalC+00}, the main idea is that the integrands in consideration have growth in $\lvert\theta'\rvert \sim \lvert\theta\rvert$ but we can reduce arbitrarily the growth in $\lvert\theta\rvert$ when $y_n\neq 0$. 
   \end{proof}

\subsubsection{Region $\mathbb{F}_{C,-,2}^{0,\infty}$}\
\begin{align*}
   ( \mathsf{U}^{\mathbb{F}_{C,-,2}^{0,\infty}} , \chi)
    &= \lim_{\epsilon \rightarrow 0} (2\pi)^{-n} \int\int\int  \psi_{\mathbb{F}_{C,-,2}^{0,\infty}} \hat{u}_x \chi(x,y) \varpi(\epsilon\theta)  e^{iy\cdot \theta} \,dx\,dy\, d\theta \\
    &=  \lim_{\epsilon \rightarrow 0} (2\pi)^{-n} \int\int\int a_{\mathbb{F}_{C,-,2}^{0,\infty}} \varpi \chi e^{iy\cdot\theta} \, dx\,  dy\, d\theta
    \end{align*}
 Recall the structure of the approximate solution on this region : 
  $$\psi_{\mathbb{F}_{C,-,2}^{0,\infty}} \hat{u}_x= \big(\lvert\theta_n\rvert (\lvert\hat{\theta}'\rvert^2-1)^{1/2}\big)^{\sqrt{\tfrac{n^2}{4}-\lambda}} \hat{\mathsf{U}}^{\mathbb{F}_{C,-,2}^{0,\infty}}$$
  $$\hat{\mathsf{U}}^{\mathbb{F}_{C,-,2}^{0,\infty}} = \exp(\tilde{\phi})s^{1/2} z^{\tilde{\gamma}_{\mathfrak{B}_1}} \tilde{\mathsf{U}}^{\mathbb{F}_{C,-,2}^{0,\infty}} ; \ \tilde{\gamma}_{\mathfrak{B}_1} = 0$$
  where $\tilde{\phi}  = -\tfrac{2}{3}s^{-1}z^{-1} \left[ 1 - (1 - z)^{3/2}\right]$ and $\hat{\mathsf{U}}^{\mathbb{F}_{C,-,1}^{0,\infty}}$ is smooth in $z$ and $s$.

\begin{lemma}
$\mathsf{U}^{\mathbb{F}_{C,-,2}^{0,\infty}} $ is smooth on $x\neq 0$
\end{lemma}

\begin{proof}\indent\par\noindent
%
  $$ \big(\lvert\theta_n\rvert (\lvert\hat{\theta}'\rvert^2-1)^{1/2}\big)^{\sqrt{\tfrac{n^2}{4}-\lambda}} =  \big[ \lvert\hat{\theta}'\rvert^2-\lvert\theta_n\rvert^2\big]^{1/2\sqrt{\tfrac{n^2}{4}-\lambda}}
  \stackrel{<}{\sim} \lvert \theta' \rvert^{2 \sqrt{\tfrac{n^2}{4}-\lambda}} $$
so 
$$(1 -\tilde{\chi}) (\theta_n) \exp(-\tilde{\phi}) a_{\mathbb{F}_{C,-,2}^{0,\infty}} \leq \lvert\theta'\rvert^{2 \sqrt{\tfrac{n^2}{4}-\lambda} }$$
  On the other hand on $x\neq 0$ we have rapid decreasing in $\lvert \theta'\rvert$ coming from the term $\exp(\tilde{\phi})$:   Using $1 - (1 - \mathsf{x})^{3/2} \geq c \mathsf{x}$ for $\mathsf{x} \in [0,1]$ for some $c>0$ to obtain a lower bound for:
  \begin{align*}
 \lvert\theta_n\rvert ( \lvert \hat{\theta}'\rvert^2 - 1)^{3/2} \left(1 - (1 - z)^{3/2}\right) 
 &\geq \lvert\theta_n\rvert ( \lvert \hat{\theta}'\rvert^2 - 1)^{3/2} z\\
 &= \lvert\theta_n\rvert ( \lvert \hat{\theta}'\rvert^2 - 1)^{3/2} x ( \lvert \hat{\theta}'\rvert^2 - 1)^{-1}\\
 &= \lvert\theta_n\rvert ( \lvert \hat{\theta}'\rvert^2 - 1)^{1/2} x
  = ( \lvert\theta'\rvert^2 - \lvert \theta_n\rvert^2 ) x\\
  &\geq  ( \lvert\theta'\rvert^2 - (1 - \delta_1)^2 \lvert \theta'\rvert^2 ) x  = \delta_1  \lvert\theta'\rvert x 
 \end{align*}
   $$\Rightarrow e^{- \lvert\theta_n\rvert ( \lvert\hat{\theta}'\rvert^2 - 1)^{3/2} \left(1 - \left[1 -\sigma^{2/3}\right]^{3/2}\right)}
 \leq e^{ - x \delta_1 \lvert\theta'\rvert}$$
  $$\Rightarrow \lvert\theta'\rvert^N a_{\mathbb{F}_{C,-,2}^{0,\infty}} \leq C_N;\  \forall N \ \in \mathbb{N}$$
Same result if we consider derivatives in $\theta_j$ and $\partial_x$. From here we get the smoothness on $x \neq 0$ (see e.g Lemma \ref{C00x>delta}) 

\end{proof}

\begin{lemma}
 $\mathsf{U}^{\mathbb{F}_{C,-,2}^{0,\infty}} $ is smooth $y_n \neq0$ and in fact vanishes to infinite order at $x=0$ for $y_n \neq 0$.
\end{lemma}

\begin{proof}

Same reasoning as in Lemma \ref{symbolC-10inf} to get for $L$ such that $^t L = \dfrac{1}{\lvert y\rvert^2} \sum y_j \dfrac{\partial}{i\partial \theta_j}$ we have: 
$$(xD_x)^l\lvert L^k a_{\mathbb{F}_{C,-,2}^{0,\infty}} \rvert  \leq C_k \lvert\theta'\rvert^{2 \sqrt{\tfrac{n^2}{4}-\lambda}}$$
 The proof then is the same in (\ref{C-10inf}). The main idea is that the integrand's growth is in terms $\lvert\theta'\rvert \sim \lvert\theta\rvert$ which we have control over using method of stationary phase when $y_n\neq 0$. For the rapid decay in $x$ we use the boundedness of $x^{-1}$ by 
  $$x^{-1} = s \lvert\theta_n\rvert ( \lvert\hat{\theta}'\rvert^2- 1)^{1/2}  
\leq  s \lvert\theta'\rvert^2$$

 
\end{proof}


\section{Main results on wavefront set and conormality}\label{mainresultcomp}

 In this section we prove the main result for the paper which is Theorem \ref{maintheorem} and Corollary \ref{resultOnCorClass1} and Corollary \ref{resultOnCorClass2} which give a description of the singularity behavior of the solution to (\ref{originalProblem})
   $$\begin{cases} PU = 0   \\
 \supp U \subset \{ y_n \geq 0\} \\
  U =  x^{s_-} \delta(y) +  x^{s_-+1} g(x,y) + h(x,y) ; & g(x,y)\in \mathcal{C}^{\infty}(\RR_x, \mathcal{D}'_y); \supp g \subset\{y_n=0\} \\ & h\in H^{1,\alpha}_{0,b,\text{loc}}; \text{for some}\ \alpha\in \RR
 \end{cases}$$
Here $ s_-(\lambda) = \tfrac{n}{2} - \sqrt{\tfrac{n^2}{4} - \lambda}$ are the indicial roots of the normal operator for $P$. The existence and uniqueness for such a solution is given by Lemma \ref{ExistenceNBdy}. Theorem \ref{maintheorem}, Corollary \ref{resultOnCorClass1} and Corollary \ref{resultOnCorClass2} describe singularities of the solution in the interior in terms of wavefront and at the boundary in terms of conormality. These results are used to arrive at Corollary \ref{mainCor2} to address the original goal of the paper: whether or not there are singularities in the `shadow region'.\nl
In this section, we put together all of the lemmas in section \ref{SingComp}, to get results for the `approximate' solution $\mathsf{U} = \mathcal{F}^{-1} ( \hat{u}_x)$, then for the exact solution $U$. Note that from its construction and the computation in section \ref{SingComp}, $\mathsf{U} = \mathcal{F}^{-1}(\hat{u}_x) \in x^{s_-} \mathcal{C}^0(\RR^+_x, \mathcal{D}'_y)$ satisfies (see Theorem \ref{conormality}):
  $$\begin{cases}  P \mathsf{U} = \mathsf{f} ;\ \mathsf{f} \in \dot{\mathcal{C}}^{\infty}( \RR^+_x \times \RR^n_y)   \\  x^{-s_-} \mathsf{U} \, \big|_{x=0} = \delta(y)  \\  \mathsf{U}\, |_{y_n < 0} \in x^{s_-} \mathcal{C}^{\infty}_{x,y} + x^{s_+} \mathcal{C}^{\infty}_{x,y} \end{cases}$$
 
\subsection{Wavefront set statement for the approximate solution $\mathsf{U}$}
The proof follows exactly from that for Theorem 4.1 in Friedlander. Wavefront of $\mathsf{U}_{\delta}$ (which is the restriction on $\mathsf{U}$ to $x > \delta$) is a union of the following wavefront sets:
\begin{itemize}
\item From Lemma \ref{wfC+0inf}, the wavefront set of $\mathsf{U}^{\mathbb{F}_{C,+}^{0,\infty}}$ is contained in $\Sigma(\delta)$, where
   $$\Sigma (\delta)= \Big\{(x,y,\xi,\eta) : \xi = \partial_x \left(y\cdot\theta - \phi \right), \eta = d_y \left(y\cdot\theta-\phi\right) , d_{\theta}\left(y\cdot\theta -\phi\right)  = 0 , x > \delta , \lvert\theta_n\rvert \geq \lvert\theta'\rvert > 0 \Big\}$$
   $$\phi = \tfrac{2}{3}( Z^{3/2} - Z_0^{3/2})\sgn\theta_n. $$

\item From Lemma \ref{WFC+inf}, the wavefront set of $\mathsf{U}^{\mathbb{F}_{C,+}^{\infty}}$ is contained in $\Sigma(\delta)$, where
   $$\Sigma (\delta)= \Big\{(x,y,\xi,\eta) : \xi = \partial_x \left(y\cdot\theta - \phi \right), \eta = d_y \left(y\cdot\theta-\phi\right) , d_{\theta}\left(y\cdot\theta -\phi\right)  = 0 , x > \delta , \lvert\theta_n\rvert \geq \lvert\theta'\rvert > 0 \Big\}$$
   $$\phi = \tfrac{2}{3}( Z^{3/2} - Z_0^{3/2})\sgn\theta_n. $$
\item From Lemma \ref{wfC0inf}, the wavefront set of $\mathsf{U}^{\mathbb{F}_{C,0}^{\infty}}$ is contained in $\Sigma(\delta_2, \delta)$, where
   $$\Sigma(\delta_2,\delta) = 
   \Big\{(x,y,\xi,\eta) : \xi = \partial_x \left( y\cdot\theta -\tilde{\phi}\right) , \eta = d_y \left( y\cdot\theta -\tilde{\phi}\right) , d_{\theta} \left( y\cdot \theta-\tilde{\phi}\right) = 0 , \theta \in \text{conesupp}\, a_{\mathbb{F}_{C,0}^{\infty}} , x> \delta \Big\}   $$
with $\tilde{\phi} = \tfrac{2}{3} Z^{3/2} \sgn\theta_n$ and $$\text{conesupp}\  a_{\mathbb{F}_{C,0}^{\infty}}= \{\lvert\theta\rvert  >0 , \left(1 - \kappa_0(\delta_2)\right) \lvert\theta_n\rvert \leq
   \lvert\theta'\rvert\leq \left(1 + \kappa_1(\delta_2)\right) \lvert\theta_n\rvert   \}.$$

\item From Lemma \ref{wfC-1inf}, wavefront set of  $\mathsf{U}^{\mathbb{F}_{C,-,1}^{\infty}}$ is contained in $\Sigma(\delta_1,\delta)$, where
$$ \Sigma(\delta_1,\delta) = 
   \Big\{(x,y,\xi,\eta) : \xi = \partial_x \left( y\cdot\theta -\tilde{\phi}\right) , \eta = d_y \left( y\cdot\theta -\tilde{\phi}\right) , d_{\theta} \left( y\cdot \theta-\tilde{\phi}\right) = 0 , \theta \in \text{conesupp}\, a_{\mathbb{F}_{C,-,1}^{\infty}} , x> \delta \Big\}  $$
with $\tilde{\phi} = \tfrac{2}{3} Z^{3/2} \sgn\theta_n$ and $$
\text{conesupp} \ a_{\mathbb{F}_{C,-,1}^{\infty}} = \{\lvert\theta\rvert > 0 ,  (1-2\delta_1) \lvert\theta'\rvert \leq \lvert\theta_n\rvert \leq \lvert\theta'\rvert\}.$$
\end{itemize}
In short 
\begin{equation}\label{wfsum1}
\text{WF}\, (\mathsf{U}_{\delta}) \subset \Sigma(\delta) \cup \Sigma(\delta_1, \delta) \cup \Sigma(\delta_2, \delta)
\end{equation}
 $\Sigma(\delta)$ is independent of $\delta_1$ and $\delta_2$. Denote
 $$(\delta_1, \delta_2) \in \Delta = \{(\mathsf{s},\mathsf{t} ) : 0 < s , t < \tfrac{1}{2}\}$$
 We have as $\delta_2 \rightarrow 0$ then $\kappa_1(\delta_2), \kappa_2(\delta_2) \rightarrow 0$ and 
 $$\bigcap_\Delta \ \Sigma(\,\delta_1, \delta\,) = \bigcap_\Delta\  \Sigma(\,\delta_2, \delta\,) $$
 $$=
 \Big\{(x,y,\xi,\eta) : \xi = \partial_x \left( y\cdot\theta -\tilde{\phi}\right) , \eta = d_y \left( y\cdot\theta -\tilde{\phi}\right) , d_{\theta} \left( y\cdot \theta-\tilde{\phi}\right) = 0 ,\lvert\theta'\rvert = \lvert\theta_n\rvert , \lvert\theta\rvert > 0, x> \delta \Big\} 
  $$ 
All the first order derivatives of $\tfrac{2}{3}Z_0^{3/2} = \phi - \tilde{\phi}$ when $\lvert\theta_n\rvert = \lvert\theta'\rvert$ so 
$$\bigcap_\Delta \ \Sigma(\,\delta_1, \delta\,) = \bigcap_\Delta\  \Sigma(\,\delta_2, \delta\,)  \subset \Sigma(\delta)$$
This is in fact the boundary of $\Sigma(\delta)$ which projects on the glancing bicharacteristics. Thus (\ref{wfsum1}) can be replaced by
$$ \text{WF}\, (\mathsf{U}_{\delta}) \subset \Sigma(\delta). $$
Since this holds for all $\delta > 0$ we have
  \begin{theorem}
    With $\phi = \tfrac{2}{3}( Z^{3/2} - Z_0^{3/2})\sgn\theta_n$, 
  $$\text{WF}\,(\mathsf{U} ) \subset \Sigma = \Big\{(x,y,\xi,\eta) : \xi = \partial_x \left(y\cdot\theta - \phi \right), \eta = d_y \left(y\cdot\theta-\phi\right) , d_{\theta}\left(y\cdot\theta -\phi\right)  = 0 , x > 0 , \lvert\theta_n\rvert \geq \lvert\theta'\rvert > 0 \Big\}.$$
     \end{theorem}

\subsection{Conormality statement for the approximate solution $\mathsf{U}$}

\begin{theorem}\label{conormality}\

On $y_n< 0$, $\mathsf{U} \in x^{s_-} \mathcal{C}^{\infty}_{x,y} + x^{s_+} \mathcal{C}^{\infty}_{x,y}$.\nl
On $y_n>0$ and if we are close enough to the boundary $\{x=0\}$, then we have the same conclusion i.e. $\mathsf{U} \in x^{s_-} \mathcal{C}^{\infty}_{x,y} + x^{s_+} \mathcal{C}^{\infty}_{x,y}$
\end{theorem}
\begin{proof}
The proof is a combination of all the results on `classical conormality' (i.e in $ x^{s_-} \mathcal{C}^{\infty}_{x,y} + x^{s_+} \mathcal{C}^{\infty}_{x,y}$) proved in section \ref{SingComp}. \nl
 On every region considered in section \ref{SingComp} we have conormality on $y_n <0$. \nl For $y_n > 0$, on each region in section \ref{SingComp}, either we have smoothness (with infinite order vanishing at $x=0$) or `classical cornomality' behavior when we are close enough to the boundary $\{x=0\}$.
\end{proof}
As a result of Theorem \ref{conormality}, $\mathsf{U} = \mathcal{F}^{-1}(\hat{u}_x) \in x^{s_-} \mathcal{C}^0(\RR^+_x, \mathcal{D}'_y)$ satisfies 
  $$\begin{cases}  P \mathsf{U} = \mathsf{f} ;\ \mathsf{f} \in \dot{\mathcal{C}}^{\infty}( \RR^+_x \times \RR^n_y)   \\  x^{-s_-} \mathsf{U} \, \big|_{x=0} = \delta(y)  \\  \mathsf{U}\, |_{y_n < 0} \in  x^{s_-} \mathcal{C}^{\infty}_{x,y} + x^{s_+} \mathcal{C}^{\infty}_{x,y} \end{cases}$$

  \subsection{Conormality and wavefront set statement for the exact solution $U$}\label{proofOfmainResult}
  \begin{remark}
Conormal in this proof refers to `classical cornormality' i.e $x^{s_-} \mathcal{C}^{\infty}_{x,y} + x^{s_+} \mathcal{C}^{\infty}_{x,y}$ \end{remark}
  \begin{proof}[\textbf{Proof for Theorem \ref{maintheorem} and Theorem \ref{resultOnCorClass1}}]\
  
   Let $U$ solve (i.e the exact solution to) the original problem (\ref{originalProblem}) (See Lemma \ref{ExistenceNBdy} for the more information on boundary condition). 
 $$\begin{cases} PU = 0   \\
 \supp U \subset \{ y_n \geq 0\} \\
  U =  x^{s_-} \delta(y) +  x^{s_-+1} g(x,y) + h(x,y) ; & g(x,y)\in \mathcal{C}^{\infty}(\RR_x, \mathcal{D}'_y); \supp g \subset\{y_n=0\} \\ & h\in H^{1,\alpha}_{0,b,\text{loc}}; \text{for some}\ \alpha\in \RR
 \end{cases}$$
  
  From its construction and Lemma \ref{conormality}, we have $\mathsf{U} = \mathcal{F}^{-1}(\hat{u}_x)\in x^{s_-} \mathcal{C}^0(\RR^+_x, \mathcal{D}'_y)$ solves
  $$\begin{cases}  P \mathsf{U} = \mathsf{f} ;\ \mathsf{f} \in \dot{\mathcal{C}}^{\infty}( \RR^+_x \times \RR^n_y)   \\  x^{-s_-} \mathsf{U} \, \big|_{x=0} = \delta(y)   \\ \mathsf{U}\, |_{y_n < 0} \in  x^{s_-} \mathcal{C}^{\infty}_{x,y} + x^{s_+} \mathcal{C}^{\infty}_{x,y} \end{cases}$$
  We would like to to consider the cut-off of the approximate solution as follows: Let $\chi$ be a smooth function such that $\chi(\mathsf{t}) = 1$ for $\mathsf{t} \leq -1$ and $=0$ for $\mathsf{t} \geq -\tfrac{1}{2}$.  We have 
  $$P \big( \chi(y_n) \mathsf{U} \big) = \chi(y_n) P \mathsf{U} + [P,\chi(y_n)] \mathsf{U}$$
  Note $\supp \, [P, \chi(y_n)] \subset \{-1 \leq \mathsf{t} \leq -\tfrac{1}{2}\}$ i.e away from $y_n=0$. On this region ($y_n < 0$) we have $[P,\chi(y_n)] \mathsf{U}$ is conormal.  Hence
  $$P \Big( \left(1 - \chi(y_n) \right) \mathsf{U} \Big)= P \mathsf{U} - P \left(\chi(y_n) \mathsf{U}\right) \in x^{s_-} \mathcal{C}^{\infty}_{x,y} + x^{s_+} \mathcal{C}^{\infty}_{x,y} $$
  In fact $\left(1 - \chi(y_n)\right) \mathsf{U}$ solves $\begin{cases}  P \Big(\left(1 - \chi(y_n) \right) \mathsf{U} \Big) = \tilde{\mathsf{f}} \in x^{s_-} \mathcal{C}^{\infty}_{x,y} + x^{s_+} \mathcal{C}^{\infty}_{x,y} \\  x^{-s_-} \left(1 - \chi(y_n)\right)\mathsf{U} \, \big|_{x=0} = \delta(y)  \\  \left(1 - \chi(y_n)\right)\mathsf{U}\, |_{y_n < -1} = 0 \end{cases}$.\nl
  Consider the difference $$W = U -  \left(1 - \chi(y_n) \right) \mathsf{U}$$ 
   $W$ solves $$\begin{cases}  P W= - \tilde{\mathsf{f}} \in  x^{s_-} \mathcal{C}^{\infty}_{x,y} + x^{s_+} \mathcal{C}^{\infty}_{x,y}  \\  x^{-s_-} W\, \big|_{x=0} = 0  \\ W |_{y_n < -1} = 0 \end{cases}$$
  By the well-posedness result of Vasy [Theorem 8.11 \cite{Andras02}], $W$ is also conormal. On the other hand, by Theorem \ref{conormality}, $\mathsf{U}$ is conormal on $y_n > 0$ and if close enough the boundary $\{x=0\}$. Thus we have proved theorem \ref{resultOnCorClass1}, which says on $y_n > 0$ and if we are close enough the boundary $\{x=0\}$ then
  $$U \in x^{s_-} \mathcal{C}^{\infty}_{x,y} + x^{s_+} \mathcal{C}^{\infty}_{x,y}  $$
  On the other hand, from the form of boundary condition satisfied by $U$ in (\ref{originalProblem}), on $y_n \neq 0$,  $U\in H^{1,\alpha}_{0,b}$ for some $\alpha \in \RR$. This means that we cannot have the term with $x^{s_-}$ in the expression of $U$ on this region, therefore on $y_n > 0$ and close enough to the boundary $\{x=0\}$ we have:
 $$ U \in x^{s_+} \mathcal{C}^{\infty}_{x,y}$$
  hence we have proved Theorem \ref{resultOnCorClass2}. \nl
  Also from this, since $W$ is smooth on $x\neq 0$, away from the boundary the singularities $U$ are the same as those of $\mathsf{U}$; hence we have proved Theorem \ref{maintheorem}.
\end{proof}
  
  As an immediate consequence of Theorem \ref{resultOnCorClass1} we have Corollary \ref{resultOnCorClass2}:
    \begin{proof}[\textbf{Proof for Corollary \ref{resultOnCorClass2}}]
  For $y_n > 0$ and close enough the boundary $\{x=0\}$, $U \in x^{s_-} \mathcal{C}^{\infty}_{x,y} + x^{s_+} \mathcal{C}^{\infty}_{x,y}$. On the other hand, from the form of boundary condition satisfied by $U$ in (\ref{originalProblem}), on $y_n \neq 0$,  $U\in H^{1,\alpha}_{0,b}$ for some $\alpha \in \RR$. This means that on $y_n \neq 0$ when we are close enough to the boundary, 
  $$U \in H^{1,\infty}_{0,b} $$    \end{proof}